\documentclass[a4paper,11pt]{article}
\usepackage{graphics,color}
\usepackage{amsmath}
\usepackage{amssymb}
\usepackage{mathrsfs}
\usepackage{cite}
\usepackage{verbatim}
\usepackage{float}
\usepackage{graphicx}
\usepackage{amsthm}
\usepackage{textcomp}
\usepackage{subfig}
\usepackage{hyperref}

\newif\ifdraft
\draftfalse 

\newcommand{\draftAAA}[1]{\ifdraft{\color{blue}#1}\fi}
\newcommand{\draftGGG}[1]{\ifdraft{\color{red}#1}\fi}
\newcommand{\draftRRR}[1]{\ifdraft{\color{green}#1}\fi}

\newcommand{\Aone}{\areaonecod}
\newcommand{\area}{\mathcal A}
\newcommand{\areaonecod}{\mathbb A}

\newcommand{\angcoord}{\omega}
\newcommand{\axialcoordofcylinder}{{w_1}}

\newcommand{\badset}{D_k}

\newcommand{\BallR}{{\Omega}}
\newcommand{\Balleps}{{\sourcedisk_{r_k}}}
\newcommand{\BV}{{\rm BV}}

\newcommand{\Coneeps}{{{C}_k}}
\newcommand{\currentgengraphFkJkzero}{{\mathcal G}^+_{k,\eps}}
\newcommand{\currentgengraphminusFkJkzero}{{\mathcal G}_{k,\eps}^-}

\newcommand{\currgraphk} {\jump{G_{u_k}}}

\newcommand{\deltaeps}{\delta_ k}
\newcommand{\diffuku} {d_k}
\newcommand{\dirdatum}{\varphi}

\newcommand{\dts}{\Sigma}
\newcommand{\dtsm}{\Sigma^-}
\newcommand{\dtsn}{\Sigma_n}
\newcommand{\dtsp}{\Sigma^+}
\newcommand{\dtsnm}{\Sigma_n^-}
\newcommand{\dtsnp}{\Sigma_n^+}
\newcommand{\domalaa}{X_{2\longR}^{\rm conv}}
\newcommand{\doubledrectangle} {R_{2\longR}}

\def\eps{\varepsilon}
\newcommand\extpsi{\widehat \psi}

\newcommand{\FB}{{\mathcal F}_{2\longR}}
\newcommand{\FBl}{{\mathcal F}_\longR}

\newcommand{\Fke}{\vartheta_{k,\eps}}

\newcommand{\graphh}{G_ \h}

\newcommand\grad{\nabla}
\newcommand{\GhatFkfour}{\mathcal G_{\hatFke}^{(4)}}
\newcommand{\GFkethree}{\mathcal G_{\Fke}^{(3)}}
\newcommand{\GminushatFkfour}{\mathcal G_{-\hatFke}^{(4)}}
\newcommand{\GplusminushatFkfour}{\mathcal G_{\pm\hatFke}^{(4)}}
\newcommand{\GminusFkethree}{\mathcal G_{-\Fke}^{(3)}}
\newcommand{\GplusminusFkethree}{\mathcal G_{\pm\Fke}^{(3)}}

\newcommand{\h}{h}

\newcommand{\hatFke}{\widehat \vartheta_{k,\eps}}

\newcommand{\hke}{h_{k,\eps}}

\newcommand{\hn}{{\h_n}}
\newcommand{\Hone}{\mathcal{H}^1}
\newcommand{\Htwo}{\mathcal{H}^2}
\newcommand{\Hricc}{\mathcal H_\longR}
\newcommand{\Hspace}{\mathcal H_{2l}}
\newcommand{\hstar}{{\h^{\star}}}
\newcommand{\hstark}{{\h^{\star}_k}}

\newcommand{\intannulus}{\mathsf{a}_{k,\eps}}
\newcommand{\invr}{R_k}
\newcommand{\invtheta}{\Theta_k}
\newcommand{\Jkzerotwopi}{J_{k,\eps}^{0,2\pi}}
\newcommand{\JQke}{J_{\Qke}}
\newcommand{\jump}[1]{\text{{\rm \textlbrackdbl}}{#1}\text{{\rm \textrbrackdbl}}}

\newcommand{\Lh}{L_\h}
\newcommand{\Lip}{{\rm Lip}}

\newcommand{\Lone}{L^1}
\newcommand{\longR}{l}
\newcommand{\limitcurrent}{\mathcal{T}}

\newcommand{\M}{\mathcal{M}}
\newcommand{\maxuk}{\vert u_k\vert^+}

\newcommand{\minuk}{\vert u_k\vert^-}

\newcommand{\modh}{h^\star}
\newcommand{\modpsi}{\psi^\star}
\newcommand{\modhbis}{h^\#}
\newcommand{\modpsibis}{\psi^\#}

\newcommand{\nada}[1]{}
\def\NN{\mathbb{N}}

\newcommand{\OmAlaa}{\doubledrectangle}
\newcommand{\Omegah}{\subgraph_\h}
\newcommand{\Omegahn}{\subgraph_\hn}
\newcommand{\Omegahstar}{\subgraph_{\hstar}}
\newcommand{\Omegahstark}{\subgraph_{\hstark}}

\newcommand{\Oke}{O_{k,\eps}}
\newcommand{\oldDom}{X_\longR}
\newcommand{\Om}{\Omega}

\newcommand{\pmcurrentgengraphFkJkzero}{{\mathcal G}^{\pm}_{k,\eps}}
\newcommand{\pointP}{\mathcal P}
\newcommand{\pointQ}{\mathcal Q}
\newcommand{\psione}{\psi}

\newcommand{\psionen}{\psione_n}
\newcommand{\psionestar}{\psi^\star}
\newcommand{\psionekstar}{\psi^\star_k}
\newcommand{\partialbar}{ \partial_D}

\newcommand{\piPsiDk}{{\mathfrak D}_k}
\newcommand{\piPsiDkXik}{\widehat {\mathfrak D}_k}
\newcommand{\projlambdak}{\pi_{\lambda_k}}
\newcommand{\psike}{\psi_{k,\eps}}

\newcommand{\Qke}{Q_{k,\eps}}

\newcommand{\R}{\mathbb{R}}
\newcommand{\radialdifference}{\lambda}

\newcommand{\relarea}{\overline \area}
\newcommand{\reps}{{r_k}}
\newcommand{\res}{\mathop{\hbox{\vrule height 7pt width 0.5pt depth 0pt
\vrule height 0.5pt width 6pt depth 0pt}}\nolimits}

\newcommand{\scoord}{{w_2}}
\newcommand{\scalarfunction}{\psi}
\newcommand{\scriptHkeps}{{\mathcal H}_{k,\eps}}
\newcommand{\seps}{s_k}
\newcommand{\Sigmake}{\Sigma_{k,\eps}}
\newcommand{\Sone}{\mathbb{S}^1}
\newcommand{\sourcedisk}{{\rm B}}
\newcommand{\sourceradialcoordinate}{r}
\newcommand{\sourceangularcoordinate}{\alpha}
\newcommand{\stripfour}{S^{(4)}_{k,\eps}}
\newcommand{\stripone}{S^{(1)}_{k,\eps}}

\newcommand{\striptwo}{S^{(2)}_{k,\eps}}
\newcommand{\subgraph}{SG}
\newcommand{\subgraphh}{SG_h}
\def\supp{\,{\rm supp \ }}

\newcommand{\tcoord}{{w_1}}
\newcommand{\thetaeps}{\theta_ k}
\newcommand{\thetaepsbar}{\overline \theta_ k}
\newcommand{\teps}{\tau_k}
\newcommand{\Teps}{T_ k}

\newcommand{\vmap}{u}
\newcommand{\veps}{{\vmap_ k}}
\newcommand{\Vke}{\mathcal V_{k,\eps}}
\newcommand{\vortexmap}{u}

\newcommand{\Xik}{{\mathcal W}_k}

\newcommand{\Wspace}{X_\longR^{\rm conv}}

\numberwithin{equation}{section}
\mathchardef\emptyset="001F

\newtheorem{theorem}{Theorem}[section]
\newtheorem{definition}[theorem]{Definition}
\newtheorem{prop}[theorem]{Proposition}
\newtheorem{cor}[theorem]{Corollary}
\newtheorem{lemma}[theorem]{Lemma}

\theoremstyle{definition}
\newtheorem{remark}[theorem]{Remark}
\newtheorem{Remark}[theorem]{Remark}

\title{The $L^1$-relaxed area of the graph of the vortex map}
\author{
Giovanni Bellettini\footnote{
Dipartimento di Ingegneria dell'Informazione e Scienze Matematiche, Universit\`a di Siena, 53100 Siena, Italy,
and International Centre for Theoretical Physics ICTP,
Mathematics Section, 34151 Trieste, Italy.
E-mail: bellettini@diism.unisi.it
                      }\and
Alaa Elshorbagy\footnote{
Technische Universit\"at Dortmund, 
Fakult\"at f\"ur Mathematik, 
44227 Dortmund, Germany. E-mail: alaa.elshorbagy@math.tu-dortmund.de 
                         }
                          \and
Riccardo Scala\footnote{ 
Dipartimento di Ingegneria dell'Informazione e Scienze Matematiche, Universit\`a di Siena, 53100 Siena, Italy.
E-mail: riccardo.scala@unisi.it
}
}

\begin{document}

\maketitle

\begin{abstract}
We compute the value of the $L^1$-relaxed area of the graph of the 
map $\vortexmap : \sourcedisk_\longR(0)\subset \R^2 
\to \R^2$, $\vortexmap(x):= x/\vert x\vert$,
$x \neq 0$, for all values of $\longR>0$. Interestingly, for $\longR$ in 
a certain range, in particular $\longR$ not too large, a Plateau-type 
problem, having as solution 
a sort of catenoid constrained to contain a segment,  has to be solved.
\end{abstract}


\vspace{2mm}

\noindent {\bf AMS (MOS) subject clas\-si\-fi\-ca\-tion:} 
49Q15, 49Q20, 49J45, 58E12.

\tableofcontents
\section{Introduction}\label{sec:introduction}
Determining the domain and the expression of 
the relaxed area functional for graphs of nonsmooth maps 
in codimension greater than $1$ 
is a challenging problem whose solution is far from being reached. 
Given a bounded open set $\Omega \subset \R^n$ 
and a map $v:\Omega \rightarrow\R^N$ of class $C^1$, 
the area of the graph of $v$ over $\Omega$ is given by the classical formula 
\begin{align}\label{area_classical}
 \area(v,\Omega)=\int_\Omega|\mathcal M(\nabla v)|~dx,
\end{align}
where $\mathcal M(\nabla v)$ is the vector whose entries are the 
determinants of the minors of the gradient $\nabla v$ of $v$ 
of all orders\footnote{By convention, the determinant
of order $0$ is $1$.} $k$, $0\leq k\leq\min\{n,N\}$.
Classical methods of relaxation suggest to consider the functional
defined, for any $v\in L^1(\Omega,\R^N)$, as
\begin{equation}\label{area_relax}
 \relarea(v,\Omega):=\inf
\Big\{\liminf_{k\rightarrow +\infty}\area(v_k,\Omega)\Big\},
\end{equation}
and called \emph{(sequential) relaxed area functional}. 
The infimum is computed 
over all sequences of maps $v_k\in C^1(\Omega,\R^N)$ 
approaching $v$ in $L^1(\Omega,\R^N)$. 
The results of Acerbi and Dal Maso \cite{AcDa:94} show that $\relarea(\cdot,\Om)$ 
extends $\area(\cdot,\Om)$ and is $L^1$-lower-semicontinuous. This procedure of relaxation, 
besides extending the notion of graph's area to non-smooth maps, is 
needed also because $\mathcal A(\cdot,\Omega)$ 
is not   $L^1$-lower-semicontinuous\footnote{When $n=N=2$, there 
are sequences $(v_k)\subset W^{1,p}(\Om, \R^2)$,
with $p \in [1,2)$,
weakly converging in $W^{1,p}(\Omega,\R^2)$ 
to a smooth map $v$ for which $\mathcal A(v,\Omega)>
\limsup_{k\rightarrow+\infty}\mathcal A(v_k,\Omega)$,
where $\area(v_k,\Omega)$ is defined as for $C^1$-maps in 
\eqref{area_classical}, with
the determinant of $\grad v_k$ intended in the almost everywhere pointwise sense;
see \cite[Counterexample 7.4]{Ball_Murat:84} and \cite{AcDa:94}.
This counterexample must be slightly modified, considering
$u_k(x) = k x + \lambda (x/\Vert x\Vert_\infty -x)$
for $x \in [-1/k,1/k]$, with $\lambda > 0$ satisfying $(1+\lambda^2)/2 > \sqrt{1+\lambda^2}$,
in order to get the strict inequality above.
}, in 
contrast with  similar polyconvex functionals that enjoy a growth condition of the form $F(u)\geq C|\mathcal M(\nabla u)|^p$ for some $C>0$, and suitable $p>1$ (see, {\it e.g.},  \cite{Morrey:66,Dacorogna:89,FuHu:95}).

When $N=1$ it is possible to characterize 
the domain of $\relarea(\cdot,\Omega)$ 
and its expression \cite{DalMaso:80}:
$\relarea(v,\Omega)$ 
is finite 
if and only if $v\in BV(\Omega)$, in which case 
\begin{align}\label{integral_formula}
\relarea(v,\Omega)= 
\int_\Omega\sqrt{1+|\nabla v|^2}dx+|D^sv|(\Omega),
\end{align}
$\nabla v$ and $D^sv$ representing the absolutely continuous and singular parts of the distributional gradient $Dv$ of $v$. 
Formula \eqref{integral_formula} is a basic 
example of non-parametric variational integral that is a measure 
when considered as a function of $\Omega$
\cite{GS:64}, and is crucial, among others, 
in the study of capillarity
problems \cite{Finn:86}, and in the analysis 
of the Cartesian Plateau problem \cite{Giusti:84}. 
The case $N>1$ (referred here to as the case of codimension greater than $1$) 
is much more involved.  Again, 
one of
its main motivations is 
the study of the Cartesian Plateau problem in higher codimension; in 
addition, from 
the point of view of Calculus of Variations, it is of interest
in those vector minimum problems involving nonconvex integrands with
nonstandard growth \cite{Ball:77}, \cite{Dacorogna:89}, \cite{GiMoSu:98}.

Let us restrict  our attention to the case $n=N=2$.
For a map $v \in C^1(\Omega, \R^2)$ and $\Omega \subset \subset \R^2$,
$\mathcal A(v,\Omega)$
coincides with the area of the graph 
$G_v:=\{(x,y)\in \Omega\times\R^2:y=v(x)\}$ of $v$ seen as a 
Cartesian surface of codimension $2$ in $\Omega\times\R^2$, and is given by
\begin{equation*}
 \area(v,\Omega)
=\int_\Omega\sqrt{1+|\nabla v(x_1,x_2)|^2+|Jv(x_1,x_2)|^2}~dx_1dx_2.
\end{equation*}
Here $\nabla v$ is the gradient of $v$, a $2\times2$ matrix, 
$\vert \nabla v\vert^2$ is the sum of the squares of all
elements of $\nabla v$, 
and $Jv$ is the Jacobian determinant of $v$, {\it i.e.}, the determinant of $\nabla v$.
It is worth to point out once more a couple of relevant difficulties
arising when the codimension is greater than $1$: the 
functional $\mathcal A(\cdot,\Omega)$ is 
no longer convex, but just polyconvex;  in addition
it has a sort of unilateral linear growth, in 
the sense that it is bounded below, but not necessarily above, by the total
variation. 
A characterization of the domain of $\relarea(\cdot, \Omega)$ and of its expression is,
at the moment, not available. Specifically, it is 
only known that the domain of $\relarea(\cdot,\Omega)$ 
is a proper subset of $BV(\Omega,\R^2)$, 
and that integral representation 
formulas such as \eqref{integral_formula} (on the domain of 
$\relarea(\cdot, \Omega)$) are not
possible. This is due to the 
additional difficulty that in general, for a fixed map $v$, 
the set function $A\subseteq \Omega \mapsto \relarea (v,A)$ 
may be not subadditive, and in particular it cannot be a measure (as opposite
to what happens in codimension $1$ for a large class of non-parametric variational integrals \cite{GS:64}). 
This interesting phenomenon 
was conjectured by De Giorgi \cite{DeGiorgi:92} for the 
triple junction map $u_T:\Omega=\sourcedisk_\longR(0)\rightarrow\R^2$, and proved in \cite{AcDa:94}, where the authors exhibited 
three subsets $\Omega_1, \Omega_2, \Omega_3$ of the open disk
$\sourcedisk_\longR(0)$ of radius $\longR$ centered at $0$,
such that 
\begin{align}\label{non-locality}
 \Omega_1\subset \Omega_2\cup\Omega_3\qquad\text{and }\qquad \relarea 
(u_T,\Omega_1)>\relarea(u_T,\Omega_2)+\relarea(u_T,\Omega_3).
\end{align}
The triple junction map $u_T \in BV(\Omega,\R^2)$ 
takes only three values $\alpha,\beta,\gamma\in \R^2$, 
the vertices of an equilateral triangle, in three circular $120^o$-degree
sectors of $\Omega$ meeting at $0$.
The same authors show that  the non-locality property \eqref{non-locality} holds also for the Sobolev
map $u(x)=\frac{x}{|x|}$, called here the vortex map,  where $\Omega$ is a 
ball of radius $\longR$ centered at the origin,
the singular point, and $n=N \geq 3$.
For these two maps $u_T$ and $u$ much effort
has been done to understand the exact value of the area 
functional;
the corresponding geometric problem stands in finding the optimal
way, in terms of area, to ``fill the holes'' of the graph 
of $u_T$ and $u$ (two non-smooth $2$-dimensional sets of 
codimension two) with limits
of sequences of smooth two-dimensional graphs. 
 In \cite{AcDa:94} it is proved that both $u_T$ and $u$ have finite relaxed area, but 
only lower and upper bounds were available for $u_T$, whereas the sharp estimate for $u$ is provided only for $l$ large enough. For the 
triple junction map
 $u_T$ an improvement is obtained in 
\cite{BePa:10}, where it is exhibited a sequence $(u_k)$ of Lipschitz maps 
$u_k:\sourcedisk_\longR(0)\rightarrow\R^2$ 
converging to $u$ in $L^1(\Om, \R^2)$, such 
that
\begin{align*}
\lim_{k \to +\infty}
 \area(u_k,\sourcedisk_\longR(0)) = |\mathcal G_{u_T}|+3m_\longR,
\end{align*}
where $|\mathcal G_{u_T}|$ is the area of the graph of $u_T$ 
out of the jump set, and $m_\longR$ 
is the area of an area-minimizing surface, solution of a Plateau-type problem in $\R^3$. 
Roughly speaking, three entangled area-minimizing surfaces with area $m_\longR$ 
(each sitting in a copy of $\R^3 \subset \R^4$, the three $\R^3$'s being 
mutually nonparallel)
are needed in $\sourcedisk_\longR(0)
\times \R^2$ to ``fill the holes'' left by the graph $\mathcal G_{u_T} $ of $u_T$, which is not boundaryless ({\it i.e.}, 
the boundary as a 
current is nonzero). The optimality of $(u_k)$ was also conjectured
in \cite{BePa:10}, and proven subsequently in  
\cite{Scala:19}, where a crucial tool is 
a symmetrization technique for boudaryless 
integral currents.

In the present paper we compute the value of the relaxed area functional 
for the \emph{vortex map} $u$ in two dimensions.
That is,  
\begin{equation}\label{vortexmapdef}
 \vortexmap(x):=\frac{x}{|x|}, \qquad x \in \Om \setminus \{0\}, \ 
\Om = \sourcedisk_\longR(0) \subset \R^2. 
\end{equation}
Observe that $u$ belongs to $W^{1,p}(\Omega,\R^2)$ for all $p \in [1,2)$,
and that the image of $u$ is the one-dimensional unit circle $\mathbb S^1\subset\R^2$, 
so that  $Ju(x)= {\rm det}(\grad u(x)) =0$ for all 
$x\in \Omega \setminus \{0\}$. 
In \cite[Lemma 5.2]{AcDa:94}, the authors show\footnote{In \cite{AcDa:94} the 
proof of \eqref{valueA_llarge} is given also for $N=n\geq 2$, where now $\pi$ 
in \eqref{valueA_llarge} is replaced by $\omega_n$.} that, 
for $\longR$ large enough,
\begin{align}\label{valueA_llarge}
 \relarea(u,\sourcedisk_\longR(0))
=|\mathcal G_u|+\pi=\int_{\sourcedisk_\longR(0)}\sqrt{1+|\nabla u|^2}dx+\pi.
\end{align}
With the aid of an example, they also show that $\relarea(u,
\sourcedisk_l(0))$ must be strictly smaller than the right-hand side
of \eqref{valueA_llarge}, 
since there is a sequence of $C^1$-maps 
approximating $u$ and having, asymptotically, a lower
value of $\area(\cdot, \Om)$. 
We anticipate here 
that, when $\longR$ is small, the above mentioned 
sequence is not optimal, and the construction 
of a recovery sequence for $\relarea(u,\sourcedisk_\longR)$ 
is much more involved and requires to solve a sort of Plateau-type 
problem in $\R^3$ with singular boundary, with a part of multiplicity $2$. 
Equivalently, with a reflection argument
with respect to a plane,  it can be seen as a non-parametric Plateau-type problem with 
a partial free boundary; one of our results (Theorem \ref{prop_zero}, valid for 
any $\longR>0$)
consists in the analysis of solutions of this problem, 
in particular we show that, excluding a singular configuration\footnote{This
corresponds to assumption (iii) in Lemma \ref{lemma_reg2_bis}.}, there is a non-parametric solution
attaining a zero boundary condition on the free part. 

In order to give an idea of how
the value $\pi$ in \eqref{valueA_llarge} pops up, it is convenient to introduce the tool of Cartesian currents.
One can regard the graphs $G_v=\{(x,y)\in \Omega\times \R^2:y=v(x)\}$ of $C^1$ maps $v:\Omega\rightarrow\R^2$ as integer multiplicity 
$2$-currents in $\Omega\times \R^2$. It is seen that a 
sequence $(G_{u_k})$ with $u_k$ 
approaching $u$ and with $\sup_k \area(u_k,\Om)<+\infty$, 
converges\footnote{This is a consequence of Federer-Fleming closure theorem.}, up to subsequences, to a Cartesian current $T$ which splits as $T=\mathcal G_u+S$, with $S$ a vertical integral current such that $\partial S=-\partial \mathcal G_u$.
A direct computation shows that 
$$\partial \mathcal G_u=-\delta_0\times \partial \jump{B_1}$$ 
(see \cite[Section 3.2.2]{GiMoSu:98}),
so that the problem of determining the value of $\relarea(u,\Omega)$ is somehow related to the 
computation of the mass of a {mass-minimizing} vertical current $S_\textrm{min}\in \mathcal D_2(\Omega\times \R^2)$ satisfying 
\begin{align}\label{boundary_of_u}
 \partial S_\textrm{min}=\delta_0\times \partial \jump{B_1}\qquad \text{in }\mathcal D_1(\Omega\times \R^2).
\end{align}
In some cases, and in particular for $\longR$ large, these two problems are related, and it turns out that $S_{\rm min}
=\delta_0\times\jump{B_1}$, whose mass is $\pi$. However $S_\textrm{min}\neq \delta_0\times\jump{B_1}$ for $\longR$ small. Moreover, the two problems of determining $S_\textrm{min}$ and the value of the relaxed area 
functional are, unfortunately, not
 related in general. This is mainly due  to the following two obstructions:
\begin{itemize}
 \item we have to guarantee that the current $\mathcal G_u+S_\textrm{min}$ is obtained as a limit of smooth graphs, that is not easy to establish since not all Cartesian currents can be obtained as such limits (see \cite[Section 4.2.2]{GiMoSu:98});
 \item even if $\mathcal G_u+S_\textrm{min}$ is 
limit of graphs $\mathcal G_{u_k}$ of smooth maps $u_k$, nothing ensures 
that $\area(u_k,\Omega)\rightarrow\relarea(u,\Omega)$, 
due to possible cancellations of the currents $\mathcal G_{u_k}$ that, 
in the limit, might overlap with opposite orientation.
\end{itemize}

Actually, in many cases, as in the one considered in this paper, 
for an optimal sequence $(u_k)$ 
realizing the value of $\relarea(u,\Omega)$, it holds
\begin{align}\label{Sopt}
\mathcal G_{u_k}\rightharpoonup\mathcal G_u+S_{\textrm{opt}}\neq \mathcal G_u+S_{\textrm{min}},
\end{align}
and the limit vertical part $S_{\textrm{opt}}$ 
satisfies $|S_{\textrm{opt}}|>|S_{\textrm{min}}|$. 
For instance, if $l$ is small, 
it is possible to construct a sequence $(\widehat u_k)$ 
approaching $u$ which is not a recovery sequence 
for $\relarea(u,\Om)$, but whose limit vertical part $S_{\textrm{min}}$ has mass strictly smaller than the one of $S_{\textrm{opt}}$ (see Section \ref{sec:catenoid_with_a_flap}). In this case, a 
suitable projection of $S_\textrm{min}$ in $\mathbb R^3$ is half of a classical area-minimizing catenoid between two unit circles at distance $2l$ from each other.

An additional source of difficulties 
in the computation of $\relarea(u, \Om)$ is due to an 
example \cite{Scala:19}
valid for the triple junction map $u_T$, 
 and showing that the equality
\begin{align}\label{triplevalue}
 \relarea(u_T,\Omega)=|\mathcal G_{u_T}|+3m_\longR
\end{align}
holds only under some additional requirements; for instance if the 
triple junction point is exactly located at
the origin $0$ and the domain is a 
disc $\Omega=\sourcedisk_\longR(0)$ around it. In particular, for different domains, \eqref{triplevalue} is no longer valid, and $S_\textrm{opt}$ is a vertical current whose support projection on $\Omega$ is a set connecting the triple point with $\partial\Omega$, and which does not coincide
 with (neither is a subset of) the jump set of $u_T$ (see 
\cite[Example in Section 6]{Scala:19} and also \cite{BeElPaSc:19} for other non-symmetric settings).

A similar behaviour of the vertical part $S_\textrm{opt}$ 
holds for $u$: when $\longR$ is small, 
the projection of 
$S_\textrm{opt}$ on $\sourcedisk_\longR(0)$ concentrates over a radius connecting $0$ to $\partial \sourcedisk_\longR(0)$,
see Fig. \ref{elliss_1}, left. However, if the domain $\Omega$ loses its symmetry, 
almost nothing is known about $S_\textrm{opt}$.

This kind of phenomena have been observed also in 
other cases, as in \cite{BePaTe:15,BePaTe:16} where $BV$-maps  $u:\Omega\rightarrow\R^2$
with a prescribed discontinuity on a curve (jump set) are considered. The creation of such ``phantom bridges'' 
between the singularities of the 
map $u$ and the boundary of the domain is very specific of the choice of the $L^1$ topology 
in the computation of $\relarea(\cdot, \Omega)$. 
Other choices are possible, giving rise to different relaxed 
functionals\footnote{Relaxing $\area(\cdot,\Om)$ in stronger
	topologies $\tau$ is possible; however, this would make more
	difficult to prove, eventually, $\tau$-coercivity of 
$\relarea(\cdot,\Omega)$. In 
	addition, it could destroy the interesting
	nonlocal phenomena related to the 
	appearence of certain nonstandard Plateau problems,
	which are the focus of this paper.} 
(see  \cite{BePaTe:15,BePaTe:16}).  

The nonlocality and the uncontrollability of $S_\textrm{opt}$ are more and more evident if we try to generalize \eqref{triplevalue} dropping the assumption that the range of $u_T$ 
consists of the vertices of an equilateral triangle. If we assume that $u_T$ takes values in $\{\alpha,\beta,\gamma\}$, three generic (not aligned)
points in $\R^2$ then, also if the domain of $u_T$ is symmetric, there is no 
sharp computation of $\relarea(u_T,\Omega)$.
In this case, the analysis  is related to an entangled Plateau problem, where three area-minimizing discs have as partial free boundary three curves connecting the couples of points in $\{\alpha,\beta,\gamma\}$, respectively, and where these three curves are forced to overlap. Some  partial results had been obtained in \cite{BeElPaSc:19}, where the authors find an upper bound for $\relarea(u_T,\Omega)$. However the question of finding the value of $\relarea(\cdot, \Omega)$ for this piecewise constant maps seems to be 
difficult.

Before stating our main results we need to fix some 
notation\footnote{The relation with the map $u$ will be clear after Section \ref{sec:lower_bound}; at this point we remark that  $R_{2\longR}$ has its first coordinate which is essentially the radial coordinate in the source $\Omega=\sourcedisk_\longR(0)$, 
and the second coordinate is instead the first coordinate $\rho$ in the target space $\mathbb R^2$. The 
graph of the function $\varphi$ is (half of) the lateral boundary of a cylinder, which coincides 
with (one half of) the image of the map $\Omega\ni x=(\rho,\theta)\mapsto (\rho,u(x))$.}. 
For $\longR>0$ we denote $R_{2\longR}:=(0,2l)\times (-1,1)$
and let $\partial_DR_{2l}:=(\{0,2l\}\times[-1,1])\cup((0,2l)\times \{-1\})$ be
what we call the Dirichlet boundary  of $R_{2\longR}$.
Define  
$\varphi: \partial_D R_{2\longR} \rightarrow [0,1]$ as
$\varphi(t,s)
:= \sqrt{1-s^2}$ 
if $(t,s) \in \{0,2\longR\} \times [-1,1]$
and $\varphi(t,s):=0$ if $(t,s) \in (0,2\longR) \times \{-1\}$.
Let 
\begin{align*}
&\widetilde {\mathcal H}_{2l}:=\{
h:[0,2l]\rightarrow [-1,1], \ h ~{\rm continuous},~ h(0)=h(2l)=1
\},
\\
&\mathcal X_{D,\varphi}:=\{\psi\in W^{1,1}(R_{2l}):\psi=\varphi\text{ on }\partial_DR_{2l}\},
\end{align*}
and for any $h\in \widetilde {\mathcal H}_{2l}$ set
$ G_h:=\{(t,s)\in R_{2l}:s=h(t)\}$ and $SG_h:=\{(t,s)\in R_{2l}:s\leq h(t)\}$ (see Fig. \ref{fig:graphofphi} for a view of the setting). The 
main result of the present paper (see Theorems \ref{teo:step1} and \ref{Thm:maintheorem}) reads as follows: 

\begin{theorem}\label{teo_intro}
 Let $N=n=2$, $l>0$ and $u:\sourcedisk_\longR(0)\rightarrow\R^2$ be the vortex map defined in \eqref{vortexmapdef}. Then 
 \begin{align}\label{value_main}
 \relarea (u,\sourcedisk_\longR(0))
=\int_{\sourcedisk_\longR(0)}\sqrt{1+|\nabla u|^2}dx+\inf\{ \mathcal A(\psi,SG_h):(h,\psi)\in \widetilde {\mathcal H}_{2l}\times 
\mathcal X_{D,\varphi},\;\psi=0\text{ on }G_h\}. 
 \end{align}
\end{theorem}
We show that for $\longR$ large enough the infimum is not attained in $\widetilde {\mathcal H}_{2l}\times \mathcal X_{D,\varphi}$ and equals $\pi$. We prove that a minimizer instead exists for $\longR$ small, hence $\psi$ is 
real analytic 
in the interior of $SG_h$;
furthermore, we show that 
$h$ is smooth and convex, and 
$\psi$ has vanishing 
trace on the graph of $h$ (Theorem \ref{prop_zero}). 

We also show that the infimum on the right-hand side 
of \eqref{value_main}
can be equivalently rewritten in many ways. Let us first point out 
that the functional $\area(\cdot, SG_\cdot)$ 
is not lower-semicontinuous, so also at this stage we need a relaxation of the infimum problem, and we introduce the functional $\FB$, defined on a 
new class $X_{2l}^{\textrm{conv}}$ 
of admissible pairs of functions, 
which is obtained after specializing the choice of $h\in \widetilde {\mathcal H}_{2l}$ and then generalizing the choice of $\psi\in \mathcal X_{D,\varphi}$ (see Definition \ref{def:FB} in Section \ref{sec:structure_of_minimizers}). 

An equivalent formulation for this minimum problem is the following: Let us consider any Lipschitz 
simple curve $\gamma:[0,1]\rightarrow \overline R_{2l}$ with $\gamma(0)=(0,1,0)$
 and $\gamma(1)=(2l,1,0)$, and then the closed curve 
$\Gamma\subset\R^3$ defined by glueing the trace of $\gamma$ with the graph of 
$\varphi$ over $\partial_D R_{2l}$. 
We can then consider an  area-minimizing disc $\Sigma^+$ spanning $\Gamma$, solution of the classical Plateau problem. 
Assuming for simplicity that $\gamma([0,1])$ 
is the graph of a function $h\in\widetilde{\mathcal H}_{2l}$, then
\begin{align}\label{min_problem_plateau}
\inf\{ \mathcal A(\psi,SG_h):(h,\psi)\in \widetilde {\mathcal H}_{2l}\times 
\mathcal X_{D,\varphi},\;\psi=0\text{ on }G_h\}=\inf \mathcal H^2(\Sigma^+),
\end{align} 
where the infimum on the right-hand side is computed over the set of all such 
curves $\gamma$ (see Corollary \ref{main_cor}). For $\longR$ sufficiently
small, say $l
\in (0,l_0)$, the infimum is attained by a disc-type surface $\Sigma^+$, and $\gamma([0,1])$ coincides with the graph of a smooth convex 
function $h\in \widetilde {\mathcal H}_{2l}$. On the contrary, for $l\geq l_0$, $\gamma$ is degenerate, in the sense that if $\sigma_n\subset R_{2l}$ is the free-boundary of $\Sigma^+_n$, where $(\Sigma_n^+)$ is 
a minimizing sequence of discs for the Plateau problem, 
then $\sigma_n$ converges to 
the set $\partial_D R_{2l}$ and $\Sigma^+_n$ converges to 
two distinct half-circles of radius $1$, whose total area is $\pi$.

We do not know the explicit value of the threshold $l_0$.
However, it is clear that $l_0>\frac12$ 
(see the discussion at the end of Section \ref{subsec:a_Plateau_problem_for_a_self-intersecting_boundary_space_curve} and Remark \ref{rem:threshold}).
Furthermore, doubling the surface $\Sigma^+$ by considering 
its symmetric with respect to the plane containing $R_{2l}$, and then 
taking the union $\Sigma$ of these two area-minimizing surfaces, 
it turns out that $\Sigma$ solves a non-standard Plateau problem, spanning a nonsimple curve which shows self-intersections 
(this is the union of $\Gamma$ with its symmetric with 
respect to $R_{2l}$, the obtained curve is the union of two circles connected by a segment, see Section \ref{subsec:a_Plateau_problem_for_a_self-intersecting_boundary_space_curve} and Fig. \ref{fig:curvagamma}). 
Again, the obtained
area-minimizing surface is a sort of catenoid forced to contain a segment
(see Fig. \ref{elliss_1}, left) for $\longR$ small, and two distinct discs spanning the two circles for $\longR$ large (Fig. \ref{elliss_1}, right). 
The restriction of $\Sigma$ to the set $\overline B_1\times [0,l]$ is a suitable projection in $\R^3$ of the aforementioned vertical current $S_{\rm opt}$.

We will discuss on the appearence of this Plateau problem in the end of this introduction:
Let us first spend some words on how we prove Theorem \ref{teo_intro}. The proof is divided into two parts, namely the lower bound 
(Sections \ref{sec:cylindrical_Steiner_symmetrization}-\ref{sec:lower_bound} 
excluding 
Section \ref{sec:three_examples}) and the 
upper bound (Sections \ref{sec:structure_of_minimizers} and 
\ref{sec:upper_bound}). The
proof of the lower bound, {\it i.e.}, 
the inequality $\geq$ 
in \eqref{value_main}, 
 is extremely involved:
we assume $(u_k)$ to be a recovery sequence converging to $u$, 
so that $\mathcal A(u_k,\sourcedisk_\longR(0))\rightarrow\relarea (u,\sourcedisk_\longR(0))$, and we analyse the behaviour of the graphs $G_{u_k}$ 
over two distinct subsets of $\sourcedisk_\longR(0)$, respectively one on which $u_k$ converges uniformly to $u$, and one where concentration phenomena are allowed (let us call this the 
``bad set'', denoted $\badset$ in the sequel). In the former, studied in 
Section 
\ref{sec:lower_bound:first_reductions_on_a_recovery_sequence},
we see that, up to small errors, the contribution of the areas of $G_{u_k}$ gives the first term on
 the right-hand side of \eqref{value_main}. In the set $\badset$, the graphs $G_{u_k}$ might behave very badly. In order to 
detect their behaviour we introduce 
suitable projections in $\R^3$ (the maps $\Psi_k$ in Definition 
\ref{def:the_map_Psikk}
and the maps $\pi_{\lambda_k}$ in 
Definition \ref{def:the_projection_pi_k}) and use 
them to reduce the currents carried by the graphs $G_{u_k}$ to 
integral $2$-currents supported 
in the cylinder $[0,l]
\times \overline \sourcedisk_1 (0)
\subset\R^3$. It is necessary to use a cylindrical Steiner-type 
symmetrization technique for these integral currents, 
described in Section \ref{sec:cylindrical_Steiner_symmetrization}.
Afterwards, an additional partition of the domain is needed, and we focus on what happens far from the origin and in a neighbourhood $\sourcedisk_\eps(0)$ of it. The first analysis is carried on in Sections \ref{sec:the_maps}, \ref{subsec:preliminary_estimates}, and \ref{subsec:est_Omegaeps_Dk}.  
The analysis around $0$ is instead done in Section \ref{subsec:symmetrization_of_the_image_of_D_k_cap_B_eps}. Roughly speaking, we 
construct 
a cylindrically symmetric integral $2$-current in 
$[0,\longR] \times \overline \sourcedisk_1(0)$ 
whose area, up to small errors, is equal to the area of $G_{u_k}$ 
over $\badset$. In order to relate the area of this current with the 
second term on the right-hand side of \eqref{value_main}, we 
have to artificially add some rectifiable sets to this current 
(see Section \ref{subsec:gluing})\footnote{To elucidate the meaning 
of all the objects we introduce, we have complemented this 
section with some examples contained in 
Section \ref{sec:three_examples}. Note that the 
construction in Section \ref{subsec:an_approximating_sequence_of_maps_with_degree_zero:cylinder}, as well as the 
catenoid
with flap in Fig. \ref{fig:cat_flap} analysed in Section 
\ref{sec:catenoid_with_a_flap}, does not lead to a recovery sequence, 
for any value of $\longR$. Nevertheless, we believe the examples
to be useful in order to follow the construction made
to prove the lower bound.}, in such a way to force the new
integral current 
to be a candidate for the minimum problem on 
the left-hand side of \eqref{min_problem_plateau}. 
Some additional rearrangements are needed here,  which are
described in Section \ref{sec:lower_bound}. The passage to the  
limit as $k\rightarrow +\infty$ is then performed in 
Theorem \ref{teo:step1}, where we also show that all the errors 
in the estimates of the previous sections are negligible.

The second part of the paper concerns the upper bound in \eqref{value_main}. This 
consists in a careful definition of a recovery sequence $(u_k)$
converging to the vortex map, and thus such that 
$\area(u_k,
\sourcedisk_\longR(0))$
approaches the value on the right-hand side of \eqref{value_main}
as $k \to +\infty$. 
In order to explicitely construct $u_k$, we need first to show that the minimum problem stated in Theorem \ref{teo:step1} is in fact equivalent to the non-parametric Plateau-type problem in \eqref{value_main},
{\it i.e.}, we have to prove \eqref{min_problem_plateau}. This is 
done in Section \ref{sec:structure_of_minimizers}, where we exploit the convexity of the domain together with some well-known regularity results for the solution of the Plateau problem in this setting. This analysis leads us to Theorem \ref{Thm:existenceofminimizer}, which 
characterizes the solution of 
\eqref{min_problem_plateau}, and which is based on a regularity result for the minimizing pair $(h^\star,\psi^\star)\in \widetilde {\mathcal H}_{2l}\times 
\mathcal X_{D,\varphi}$.  Finally, thanks to the regularity 
results that we have obtained (especially, boundary
regularity), in Section \ref{sec:upper_bound} we define explicitely the 
maps $u_k$,
making a crucial use of rescaled versions 
of the area-minimizing surface $\Sigma$ in a vertical copy of $\R^3$ inside
$\R^4$,
 and prove the upper bound in Theorem \ref{Thm:maintheorem}.

From this discussion 
the appearence of the aforementioned nonstandard Plateau problem
should be more clear.
The shape of the solution $\Sigma$ of this problem (after a suitable projection from $\R^4$ to $\R^3$) is related to the graph $\mathcal G_u$ upon the ``limit of the bad set'' (in turn related to $S_{\textrm{opt}}$ in \eqref{Sopt}). More precisely, let us fix a map $u_k$ in the 
recovery sequence for $\relarea(u,\sourcedisk_\longR(0))$, 
and let us call $D_{u_k}$ the corresponding
bad set, roughly 
the set where the values of  $u_k$ remain ``far'' from those of $u$. 
Essentially, the slice of the 
half catenoid-type (containing the segment)
surface $\Sigma\subset \overline \sourcedisk_1(0)\times [0,l]$ 
with respect to a plane $ \R^2\times\{t\}$, $t\in(0,l)$,
 is a closed curve touching the lateral boundary of $\overline \sourcedisk_1(0)\times [0,l]$ at a point. This will be the limit of the 
image of the restriction
${u_k}_{\vert D_{u_k}\cap \partial \sourcedisk_t(0)}$ in $\R^2$ 
(identified with $\R^2\times\{t\}$).
Indeed, 
$u_k(
(\sourcedisk_1(0) \setminus D_{u_k}) \cap
\partial \sourcedisk_t(0))$ 
is a closed
curve that lies very close to $\mathbb S^1$,  
whereas
$u_k(
D_{u_k}\cap 
\partial \sourcedisk_t(0))$ 
makes a trip in $(\overline \sourcedisk_1(0)\times [0,l])
\cap (\R^2\times\{t\})$ in order to approach the shape of a $t$-slice 
of  $\Sigma$. Since $u_k(\partial \sourcedisk_t(0))$ is a closed curve, 
after passing to the limit 
as $k\rightarrow+\infty$, we 
obtain a closed curve which overlaps $\mathbb S^1$ 
(limit of the images of the complements of the bad sets) and then is a 
closed curve (limit of the images of $D_{u_k}$) attached to $\mathbb S^1$ 
at a point whose shape is the slice of $\Sigma$.

\section{Preliminaries}\label{sec:preliminaries}

\subsection{Notation and conventions}\label{subsec:notation}
The symbol $\area(v,\Omega)$ denotes  the classical area of the graph of a smooth map $v:\Omega\subset\R^n\rightarrow \R^N$, given by \eqref{area_classical}.  We will deal with the case $n=2$ and mostly with the cases $(n,N)=(2,1)$ and $(n,N)=(2,2)$. The relaxed area functional (with respect to the $L^1$-convergence) 
is denoted by $\relarea(v,\Omega)$ and is defined in \eqref{area_relax}.

We first remark that the infimum in \eqref{area_relax} can be considered as taken over the class of sequences $v_k\in \textrm{Lip}(\Omega;\R^2)$. 
This does not change the value of $\relarea(\cdot,\Om)$, as observed in \cite{BePa:10}. 

Recall that in formula \eqref{area_classical} the symbol $\mathcal M(\nabla v)$ denotes the vector whose entries are all 
determinants of the minors of $\nabla v$. Precisely, let $\alpha$ and $\beta$ be subsets of $\{1,2\}$, let $\bar\alpha$ denote the complementary set of $\alpha$, namely $\bar\alpha=\{1,2\}\setminus\alpha$, let $|\cdot|$ denote the cardinality, and let $A\in \R^{2\times2}$ be a matrix. Then, if $|\alpha|+|\beta|=2$, we denote by  
\begin{align}\label{eq:subdeterminant}
M_{\bar\alpha}^\beta(A)
\end{align}
 the determinant of the submatrix of $A$ whose lines are those with index in $\beta$, and columns with index in $\bar\alpha$. By convention $M_\emptyset^\emptyset(A)=1$ and  moreover
$$M_{j}^i=a_{ij},\qquad i,j\in \{1,2\},\qquad \qquad M_{\{1,2\}}^{\{1,2\}}(A)=\det A,$$
and the vector $\mathcal M(A)$ will take the form $$\mathcal M(A)=(M_{\bar\alpha}^\beta)(A))=(1,a_{11},a_{12},a_{21},a_{22},\det A),$$
where $\alpha$ and $\beta$ run over all the subsets of $\{1,2\}$ with the constraint $|\alpha|+|\beta|=2$. We will identify $\alpha$ and $\beta$ as multi-indeces in $\{1,2\}$.

\subsubsection{Area in cylindrical coordinates}
Polar coordinates
in $\R^2_{{\rm source}}$ are denoted by $(\sourceradialcoordinate,
\sourceangularcoordinate)$.
Polar coordinates
in the target space $\R^2_{{\rm target}}$ are denoted by $(\rho,\theta)$. 

Assume that $B=\{(r,\sourceangularcoordinate)\in \R^2: r \in (r_0,r_1),
~\sourceangularcoordinate \in (\sourceangularcoordinate_0,\sourceangularcoordinate_1)\}$; then the area of the graph of 
$v=(v_1,v_2)$ in polar coordinates is given by 
 \begin{equation*}
  \area (v,B)=\int_{r_0}^{r_1}\int_{\sourceangularcoordinate_0}^{\sourceangularcoordinate_1} \vert \M(\nabla v)\vert (r,\sourceangularcoordinate)~ r dr d\sourceangularcoordinate.
 \end{equation*}
Recall that, for $i \in \{1,2\}$, we have 
\begin{align*}
\partial_{x_1}v_i=\cos \sourceangularcoordinate
 \partial_r v_i - \frac{1}{r} \sin \sourceangularcoordinate
  \partial_\sourceangularcoordinate v_i,\qquad 
\partial_{x_2}v_i=\sin\sourceangularcoordinate
 \partial_r v_i + \frac{1}{r} \cos \sourceangularcoordinate
 \partial_\sourceangularcoordinate v_i.
\end{align*}
Hence 
\begin{align}
&\vert \nabla v_i\vert ^2 = 
(\partial _r v_i)^2 +\frac{1}{r^2} (\partial_\sourceangularcoordinate v_i)^2,\qquad i \in \{1,2\},\label{eqn:gradinpolar}
\\
&\partial_ {x_1} v_1 \partial_{ x_2} v_2-\partial_{ x_2}v_1
\partial_{ x_1}v_2 =\frac{1}{r}\Big( 
\partial_ {r} v_1 \partial_{\sourceangularcoordinate} v_2-\partial_{\sourceangularcoordinate}v_1
\partial_{ r}v_2\Big).
\nonumber
\end{align}

Thus the area of the graph of $v$ on $B$ is given by 
\begin{equation}\label{eqn:areapolarexpression}
\begin{aligned}
& \area (v,B)
\\
=& \int_{r_0}^{r_1}\int_{\sourceangularcoordinate_0}^{\sourceangularcoordinate_1} 
\sqrt{1+(\partial_r v_1)^2+(\partial_r v_2)^2+\frac{1}{r^2}
\left\{
(\partial_\sourceangularcoordinate
 v_1)^2+
(\partial_\sourceangularcoordinate
 v_2)^2+\Big( 
\partial_ {r} v_1 \partial_{\sourceangularcoordinate} v_2-\partial_{\sourceangularcoordinate}v_1
\partial_{ r}v_2\Big)^2\right\}}~ r dr d\sourceangularcoordinate.
\end{aligned}
\end{equation}
%
We denote by $\sourcedisk_r=\sourcedisk_r(0)\subset\R^2 = \R^2_{{\rm source}}$ the open disc centered at $0$ 
with radius $r>0$ in the source space. 
Our reference domain is $\Omega=\sourcedisk_\longR \subset \R^2_{{\rm source}}= \R^2_{(x_1,x_2)}$
where $l>0$ is fixed once for all. The symbol $u$ will be used to note the vortex map in \eqref{vortexmapdef}, which we assume to be defined on
 $\sourcedisk_\longR$.

For any $\varrho>0$, 
it is convenient to introduce the (portion of) cylinder $C_\longR(\varrho)$, 
as 
\begin{equation}\label{eq:portion_of_cylinder}
\begin{aligned}
 C_\longR(\varrho):=(-1,l)\times B_\varrho=\{(t,\rho,\theta)\in (-1,l)\times \R^+\times 
(-\pi,\pi]:\rho<\varrho\} 
\subset \R^3 = \R_t 
\times \R^2_{{\rm target}},
\end{aligned}
\end{equation}
where $(t,\rho,\theta)$ 
are cylindrical coordinates in $\R^3$, with the cylinder axis the $t$-axis. For $\varrho=1$ we 
simply write 
\begin{equation}\label{eq:cyl_one}
C_\longR(1)=C_\longR.
\end{equation}
For a fixed parameter $\eps \in (0,\longR)$, we introduce
the cylinders 
\begin{equation}
\label{eq:C_l_eps_r}
 C_\longR^\eps(\varrho):=(\eps,l)\times B_\varrho=
\{(t,\rho,\theta)
\in (0,\longR)\times \R^+\times (-\pi,\pi]:\eps<\rho<\varrho\}
\subset \R_t 
\times \R^2_{{\rm target}}.
\end{equation}
Also in this case we use the notation 
\begin{equation}\label{eq:C_l_eps_1}
C_\longR^\eps(1)=C_\longR^\eps.
\end{equation}
The closure of 
$C_\longR(\rho)$ (resp. $C_\longR^\eps(\rho)$)
is denoted by 
$\overline C_\longR(\rho)$ (resp. $\overline C_\longR^\eps(\rho)$),
and the lateral boundary
of 
$C_\longR(\rho)$ (resp. $C_\longR^\eps(\rho)$)
is denoted by 
$\partial_{\rm lat}C_\longR(\rho)$ (resp. $\partial_{\rm lat}C_\longR^\eps(\rho)$).

We will often deal with integral currents whose support is in the cylinder 
$$[0,\longR]\times \overline B_1\subset \overline C_\longR. $$
\begin{Remark}\rm
The choice of $C_\longR=(-1,l)\times B_1$ covering also 
certain negative values of the first coordinate $t$ is useful to control and detect the behaviour of these currents on the plane $\{t=0\}$.
\end{Remark}
\subsubsection{Area formula}\label{subsec:area_formula}
Let $f:U\subset\R^k\rightarrow \R^n$ be Lipschitz continuous, with $k\leq n$. 
The area of the image $f(U)$ of $U$ by $f$ is given by
\begin{align*}
\int_UJf(x)dx
\end{align*}
with the Jacobian matrix of $f$ given by 
\begin{align*}
Jf= \sqrt{\det\big((\grad f)^T \grad f\big)}
=
\sqrt{\sum(\det A)^2} \qquad {\rm a.e.in~} U,
\end{align*}
where, for almost every $x \in U$, the sum is made on all 
submatrices $A(x)$ of $\nabla f(x)$ of dimension $k\times k$.

\subsection{Currents}\label{subsec:currents}

For the reader convenience we recall some basic notion on currents. We refer to \cite{Krantz_Parks:08} and \cite{GiMoSu:98} for a more exhaustive discussion (see also \cite{Federer:69}).

Given an open set $U\subset \R^n$ we denote by $\mathcal D^k(U)$ the 
space of smooth $k$-forms compactly 
supported in $U$ and by $\mathcal D_k(U)$ the space of $k$-dimensional currents, for $0\leq k\leq n$. If $T\in \mathcal D_k(\R^n)$, the symbol $\vert T\vert$ denotes the mass of the current $T$, and if $U\subset\R^n$ is an open set, the symbol $\vert T\vert_U$ will denote the mass of $T$ in $U$, namely
$$|T|_U:=\sup T(\omega),$$
 the supremum being over all $\omega\in \mathcal D^k(U)$ with $\|\omega\|\leq 1$. 
 
 For $k\geq 1$ it is defined the boundary $\partial T\in \mathcal D_{k-1}(U)$ of a current $T\in \mathcal D_k(U)$ 
by the formula
 $$\partial T(\omega):=T(d \omega) \text{ for all }\omega\in \mathcal D^{k-1}(U),$$
 where $d\omega$ is the external differential of $\omega$. 
For $T\in \mathcal D_0(U)$ one sets $\partial T:=0$.
 
 If $F:U\rightarrow V$ a Lipschitz 
map between open sets, and $T\in \mathcal D_k(U)$, we denote by $F_\sharp T\in \mathcal D_k(V)$ the push-forward of $T$ by $F$ (see \cite[Section 7.4.2]{Krantz_Parks:08}).
 
 Given a $k$-dimensional rectifiable set $S\subset U$ and a tangent 
unit simple $k$-vector $\tau$ to it, we denote by 
$\jump{S}$ the current given by integration over $S$, namely
$$
\jump{S}(\omega)=\int_S\textlangle\tau(x),\omega(x)\textrangle~ 
d\mathcal H^{k}(x),\qquad \omega\in \mathcal D^k(U).
$$ 
We will often omit specifying which is the vector $\tau$ if it is clear from the context. 
We will often deal with the case $k=2$, and $U\subset \R^3$ where there are only two possible orientations. Moreover in the case $k=3$ and $U\subset \R^3$ the current $\jump{S}$ reduces to the integration over the  $3$-dimensional set $S\subset\R^3$, and $\tau=e_1\wedge e_2\wedge e_3$. 
 
We call $T\in \mathcal D_k(U)$ an integral current if it is rectifiable with integer multiplicity and if both  $|T|_U$ and $|\partial T|_U$ are finite. The Federer-Fleming theorem for integral currents then states that a 
sequence of integral currents $T_i\in \mathcal D_k(U)$ with $\sup_i (T_i|+|\partial T_i|)<+\infty$ admits a 
subsequence converging weakly in the sense of currents to an integral current $T$.
  
A finite perimeter set is a subset $E\subset\R^n$ such that the current $\jump{E}\in \mathcal D_n(U)$ is integral. The symbol $\partial^*E$ denotes the 
reduced boundary of $E$. 
$E$ is unique up to negligible sets, so that we always choose 
a representative of $E$ for 
which the closure of the reduced boundary equals the topological boundary \cite{Maggi:12}.
  
An integral current $T\in \mathcal D_k(U)$ is called indecomposable if there is no integral current $R\in \mathcal D_k(U)$ such that $R\neq0\neq T-R$ with 
$$|T|_U+|\partial T|_U=|R|_U+|\partial R|_U+|T-R|_U+|\partial (T- R)|_U.$$  
We will often use the following decomposition theorem for integer multiplicity currents: For every integral current $T\in \mathcal D_k(U)$ there is a sequence of indecomposable integral currents  
$T_i\in \mathcal D_k(U)$ with $T=\sum_iT_i$ and  
$|T|+|\partial T|=\sum_i|T_i|+\sum_i|\partial T_i|$ 
(see \cite[Section 4.2.25]{Federer:69}). 
In  the case that $T\in \mathcal D_n(U)$, $U\subseteq\R^n$, 
the previous  decomposition theorem can be stated as follows: There is a sequence of finite perimeter sets with $\{E_i\}_{i\in \mathbb Z}$ such that $T=\sum_{i\geq0}\jump{E_i\cap U}-\sum_{i<0}\jump{E_i\cap U}$ with $\sum_i|E_i\cap U|
+\sum_i\mathcal H^{n-1}(U \cap \partial^* E_i)=|T|+|\partial T|$ 
(see \cite[Theorem 7.5.5]{Krantz_Parks:08} and its proof). Moreover, the decomposition theorem applied to $E_i$ allows us to assume that the 
sequence $(\jump{E_i})$ consists of indecomposable currents. In the case of $1$-dimensional currents, 
it is possible also to characterize indecomposable currents, namely $T\in \mathcal D_1(\R^n)$ is indecomposable if $T=\gamma_\sharp \jump{[0,|T|]}$ with $\gamma:[0,|T|]\rightarrow\R^n$ a $1$-Lipschitz simple curve, 
{\it i.e.}, injective on $[0,|T|)$. If moreover $\partial T=0$ then $\gamma(0)=\gamma(|T|)$.

We will exploit the property that any boudaryless current $T\in \mathcal D_{n-1}(\R^n)$ is the boundary of a sum of currents given by integration over locally finite perimeter sets $E_i$, 
{\it i.e.}, $T=\sum_i\partial \jump{E_i}$. This is a consequence of the cone construction, and for integral currents can be obtained also from the isoperimetric inequality (see \cite[Formula (7.26)]{Krantz_Parks:08} and \cite[Theorem 7.9.1]{Krantz_Parks:08}).

We need also the concept of slice of an integral current with respect to a Lipschitz 
function $f$ (see \cite[Section 7.6]{Krantz_Parks:08}). Since we only employ it for slices with respect to parallel planes, 
the function $f$ will be $f(x)=x_h$ where $x_h$ is the coordinate in $\R^n$ whose axis is orthogonal to the considered planes. We denote by $T_t\in \mathcal D_{k-1}(\R^n)$ the slices of $T\in \mathcal D_k(\R^n)$ on the plane $\{x_h=t\}$, which will be 
supported on this plane. We will also use that, if $T$ is boundaryless, then 
$$\partial (T\res \{x_h<t\})=T_t
\qquad
{\rm for~ a.e.}~ t\in \R.
$$

\subsection{Generalized graphs in codimension $1$}\label{sec:generalized_graphs_in_codimension_one}
Let $v\in L^1(\Omega)$.
We denote by $R_v\subseteq\Omega$ the set of regular points of $v$, 
{\it i.e.}, the set consisting of points $x$ which are Lebesgue points for 
$v$, $v(x)$ 
coincides with the Lebesgue value of $v$ at $x$, and $v$ 
is approximately differentiable at $x$. We also set
\begin{align*}
 &G_v^R:=\{(x,v(x))\in R_v\times \R\},
\\
 &SG_v^R:=\{(x,y)\in R_v\times \R:y< v(x)\}.
 \end{align*}
 We often will identify $SG_v^R$ with the integral $3$-current $\jump{SG_v}\in \mathcal D_3(\Omega\times\R)$. If $v$ is a 
function of bounded variation, $\Omega\setminus R_v$ has 
zero Lebesgue measure, so that the current $ \jump{SG_v}$ 
coincides with the integration over the subgraph 
 \begin{align*}
 &SG_v:=\{(x,y)\in \Omega\times \R:y< v(x)\}.
 \end{align*}
For this reason we often identify $SG_v=SG_v^R$.
It is well-known that the  perimeter of $SG_v$ in $\Omega\times \R$ 
coincides with $\relarea(v,\Omega)$.

The support of the boundary of $\jump{SG_v}$ 
includes the graph $G_v^R$, but in general consists also of additional parts, called vertical. We denote by
$$\mathcal G_v:=\partial \jump{SG_v}\res(\Omega\times\R),$$
the generalized graph of $u$, which is a $2$-integral current supported on $\partial^*SG_v$, the reduced boundary of $SG_v$ in $\Omega\times\R$. 

Let $\widehat \Omega \subset \R^2$ 
be a bounded  open set such that $\Omega\subseteq \widehat\Omega$, 
and suppose that 
$L := \widehat \Omega\cap
\partial \Omega$ is a rectifiable 
curve. Given $\scalarfunction 
\in BV(\Omega)$ and  a $W^{1,1}$ function $\varphi:
\widehat \Omega \to \R$, we can consider 
\begin{align*}
 \overline{\scalarfunction}:=\begin{cases}
               f & \text{on }\Omega,
\\
               \varphi&\text{on }\widehat \Omega\setminus\Omega.
               \end{cases}
\end{align*}
Then (see \cite{Giusti:84}, \cite{AmFuPa:00})
\begin{align*}
 \relarea(\overline{\scalarfunction},\widehat\Om)=
\relarea(\scalarfunction,\Om)+\int_{L}|\scalarfunction-\varphi|d\mathcal H^{1}
+\relarea(\varphi,\widehat\Om
\setminus\overline{\Om}).
\end{align*}

\subsection{Polar graphs in a cylinder}\label{sec:polar_graphs_in_a_cylinder}
Consider the (portion of) cylinder $C_\longR=(-1,l)\times B_1$ 
defined in \eqref{eq:cyl_one}, endowed with cylindrical 
coordinates $(t,\rho,\theta)\in (-1,l)\times [0,1)\times (-\pi,\pi]$. 
Take the rectangle $H=\{(t,\rho,\theta) \in C_\longR: \theta=0\}$,
 which is 
endowed with Cartesian coordinates $(t,\rho)\in (-1,l)\times (0,1)$.
If $\Theta:H\rightarrow[0,\pi]$ 
is a 
function defined on $H$, we can associate to it the map
${\rm id} \bowtie \Theta : H \to C_\longR$ defined as
\begin{align*}
 (t,\rho)\to(t,\rho,\Theta(t,\rho)), \qquad (t,\rho) \in H.
\end{align*}
The polar graph of $\Theta$ 
is defined as 
\begin{align*}
 G^{\rm pol}_\Theta:=\{(t,\rho,\Theta(t,\rho)):t\in(-1,l),\;\rho\in (0,1)\} = 
{\rm id} \bowtie \Theta(H),
\end{align*}
where again we have used cylindrical coordinates. 

We define a sort of polar subgraph of $\Theta$ as 
\begin{align*}
 SG^{\rm pol}_\Theta:=
\{(t,\rho,\theta):t\in(-1,l),\;\rho\in [0,1),\;\theta\in (-\eta,\Theta(t,\rho))\}.
\end{align*}
Here $\eta>0$ is a small number introduced for convenience, and
 it will suffice to take $\eta<1$.
If the set $SG^{\rm pol}_{\Theta}$ has finite perimeter, its reduced
boundary 
in $\{-\eta<\theta<\pi+\eta\}\cap C_\longR$ 
coincides with the generalized 
polar graph $\mathcal G_\Theta$ of $\Theta$, 
\begin{equation}\label{eq:generalized_polar_graph}
\mathcal G_\Theta
= (\partial^* SG^{\rm pol}_\Theta)\cap (\{-\eta<\theta<\pi+\eta\} \cap C_\longR).
\end{equation}
This set includes, up to $\mathcal H^2$-negligible sets, the polar graph $G^{\rm pol}_\Theta$. 
When $SG^{\rm pol}_{\Theta}$ has finite perimeter, we see that the current $\jump{SG^{\rm pol}_\Theta}\in \mathcal D_3(C_\longR)$ is integral and its boundary in  $\{-\eta<\theta<\pi+\eta\}\cap C_\longR$ is the integration over the generalized polar graph of $\Theta$, 
{\it i.e.},
\begin{align*}
\partial \jump{SG^{\rm pol}_{\Theta}}\res (\{-\eta<\theta<\pi+\eta\}\cap C_\longR)=\jump{\mathcal G_\Theta},
\end{align*}
where $\mathcal G_\Theta$ is naturally oriented by the outer normal to $\partial^* SG^{\rm pol}_{\Theta}$. 

Notice also that since $\Theta\in[0,\pi]$ the current $\jump{\mathcal G_\Theta}$ 
carried by the generalized polar graph $\mathcal G_\Theta$ is 
supported in $\{0\leq\theta\leq\pi\} \cap C_\longR$.


\subsection{Plateau problem in parametric form}

We report here some 
results about the classical solution to the 
disc-type Plateau problem. If $\Gamma\subset\R^3$ is a closed rectifiable Jordan curve, the Plateau problem consists into minimize the functional
\begin{align}
\label{plateau}
\mathcal P_\Gamma(X):= \int_{B_1}|\partial_{x_1}X\wedge\partial_{x_2} X|dx_1dx_2,
\end{align}
on the class of all functions $X\in C^0(\overline{B}_1;\R^3)\cap H^1(B_1;\R^3)$ with
 $X\res\partial B_1$ being a weakly monotonic parametrization of the curve $\Gamma$. 
The functional \eqref{plateau} measures the area (with multiplicity) of the surface $X(B_1)$. 
We can always associate to a 
map $X$ the current $X_\sharp\jump{B_1}$, the integration over the surface $X(B_1)$. Notice in general 
$$|X_\sharp\jump{B_1}|\leq \mathcal P_\Gamma(X),$$
and strict inequality can occur if for instance the map $X$ parametrizes two times and with opposite orientation a part of the surface $X(B_1)$.

A solution $X_\Gamma$ to the Plateau problem exists and satisfies the properties: 
it is harmonic (hence analytic)
\begin{align*}
 \Delta X_\Gamma=0\qquad\text{ in }B_1,
\end{align*}
it is a conformal parametrization 
\begin{align*}
 |\partial_{x_1}X_\Gamma|^2=|\partial_{x_2}X_\Gamma|^2,
\qquad \partial_{x_1}X_\Gamma\cdot \partial_{x_2}
X_\Gamma=0\qquad\text{ in }B_1,
\end{align*}
and $X_\Gamma\res\partial B_1$ is a strictly monotonic parametrization of $\Gamma$.
We will say that the surface $X_\Gamma(B_1)$ has the topology of the disc.

Thanks to the properties above it is always possible, with the aid of a conformal change of variables, to parametrize $X(B_1)$ over any simply connected bounded domain. In other words, if $U$ is any such domain, and if $\Phi:U\rightarrow B_1$ is any conformal homeomorphism, then $X\circ\Phi$ is a solution to the Plateau problem on $U$.

\subsection{A Plateau problem for a self-intersecting boundary space curve}
\label{subsec:a_Plateau_problem_for_a_self-intersecting_boundary_space_curve}

The classical disc-type 
Plateau problem is solved for boundary value a simple Jordan space curve, 
in particular $\Gamma$ 
does not have self-intersections. Here we will treat a specific Plateau problem where the curve $\Gamma$ has non-trivial intersections, and it overlaps itself on a segment which is parametrized two times with opposite directions.

Specifically, we consider the cylinder $(0,2l)\times B_1$ 
and 
two circles $\mathcal C_1, \mathcal C_2$
 which are the boundaries of its two circular bases, namely 
$\mathcal C_1:=\{0\}\times \partial B_1$ and $\mathcal C_{2}:=\{2l\}\times \partial B_1$. 
Then we take the segment $(0,2l)\times\{1\}\times\{0\}$. If $\gamma_0$ 
is a monotonic parametrization of this segment, 
starting from $(0,1,0)$ up to $(2l,1,0)$, $\gamma_1$ is a 
monotonic parametrization of $\mathcal C_1$ starting from the point $(0,1,0)$ and ending at the same point, and  $\gamma_2$ a parametrization of $\mathcal C_2$ with initial and final point $(2l,1,0)$
with the same orientation of $\mathcal C_1$, then we consider the parametrization 
\begin{align}\label{curvegamma_def}
 \gamma:=\gamma_1\star\gamma_0\star(-\gamma_2)\star(-\gamma_0),
\end{align}
(read from left to right)
which is a closed curve in $\R^3$ which 
travels two times across the segment $(0,2l)\times\{1\}\times\{0\}$ with opposite directions (the orientation of this curve is depicted in Fig. \ref{fig:curvagamma}).
We want to solve the Plateau problem with $\Gamma$ 
to be the image of $\gamma$.

\begin{figure}
\begin{center}
    \includegraphics[width=0.3\textwidth]{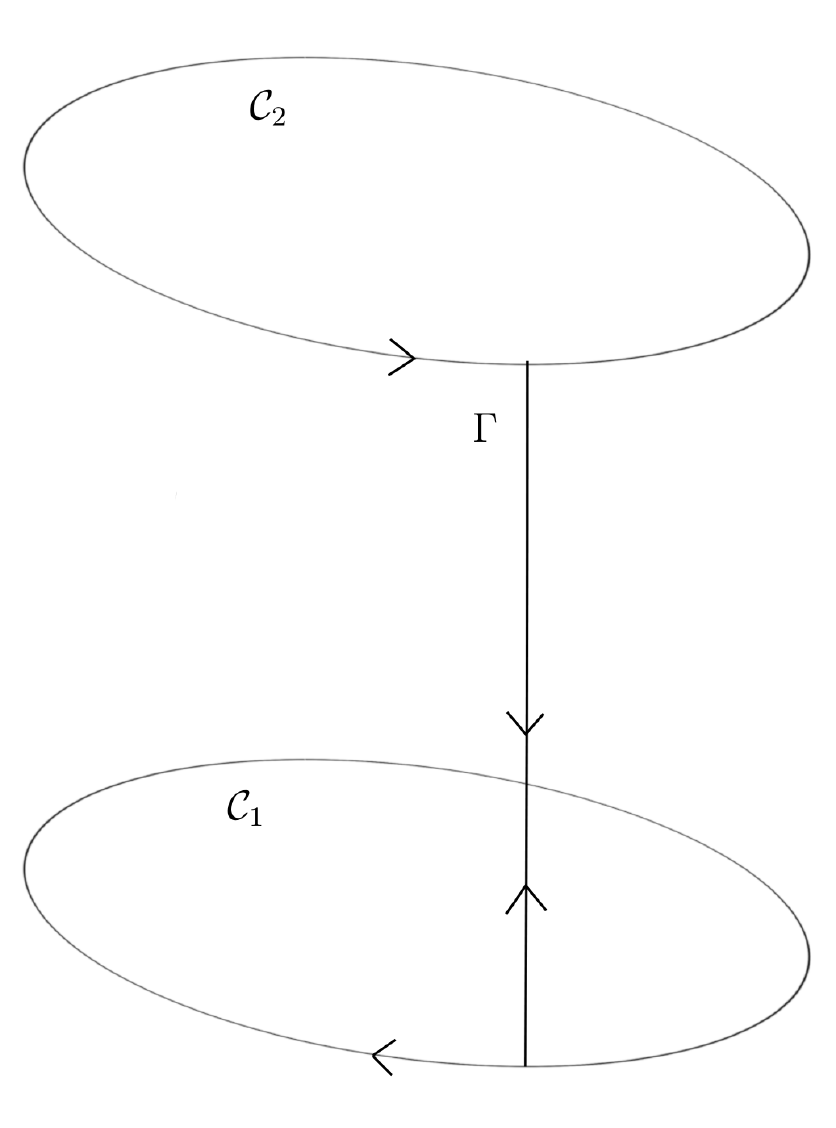}
 \caption{The self-overlapping curve $\Gamma$ with its orientation. }
\label{fig:curvagamma}
\end{center}
\end{figure}

The existence of solutions to the 
Plateau problem spanning self-intersecting boundaries has been addressed in \cite{Hass:91}, whose results have been recently improved in \cite{Creutz:20}. Without entering deeply into the details, it is known that, depending on the geometry of $\gamma$ (in this case, depending on the distance between the two circles $\mathcal C_1$ and $\mathcal C_2$) two kind of solutions are expected:
\begin{itemize}
 \item[(a)] The solution 
consists of two discs filling $\mathcal C_1$ and $\mathcal C_2$, see Fig. \ref{elliss_1}, right. In this case, a parametrization of it  $X:\overline B_1\rightarrow\R^3$ can be chosen so that, if $L_1$ and $L_2$ are two parallel chords in $B_1$ dividing $B_1$ in three sectors, then $X$ restricted to the sector enclosed between $L_1$ and $L_2$ parametrizes the segment $\gamma_0$ (and then its resulting area is null), $X(L_1)=P_1$ and $X(L_2)=P_2$ are the two endpoints of $\gamma_0$, and $X$ restricted to the sectors between $L_i$, $i=1,2$, and $\partial B_1$ parametrizes the disc filling $\mathcal C_i$, $i=1,2$. Moreover the map $X$ can be still taken Sobolev regular 
(see \cite{Creutz:20} for details).

 \item[(b)] There is a classical solution, 
{\it i.e.}, there is a harmonic and conformal map $X:B_1\rightarrow\R^3$,
continuous up to the boundary of $B_1$, such that $X\res \partial B_1$ is a weakly monotonic parametrization of $\Gamma$. 
In this case the resulting minimal surface is a sort of catenoid attached to the segment $(0,2l)\times \{(1,0)\}$ (see Fig. \ref{elliss_1} left).
\end{itemize}

\begin{remark}\label{rem:threshold}
We expect that there is a threshold $l_0$ such that if $l< l_0$ an area-minimizing disc with boundary $\gamma$ is of the form (b),
and for values $l>l_0$ the two discs have minimal area. We do not find explicitly $l_0$ but 
it is easy to see that if $l\leq \frac12$ an area-minimizing disc with boundary $\gamma$ has always less area than the solution with two discs. Indeed, the area of the two discs is $2\pi$, whereas we can always compare 
the area of the surface  $\Sigma$ as in (b) with the area  of the lateral surface of the cylinder $(0,2l)\times B_1$, that is $4l\pi$. Hence $\mathcal H^2(\Sigma)<4l\pi\leq 2\pi$ for $l\leq \frac12$.	
\end{remark}

		\begin{figure}
		\includegraphics[width=0.8\textwidth]{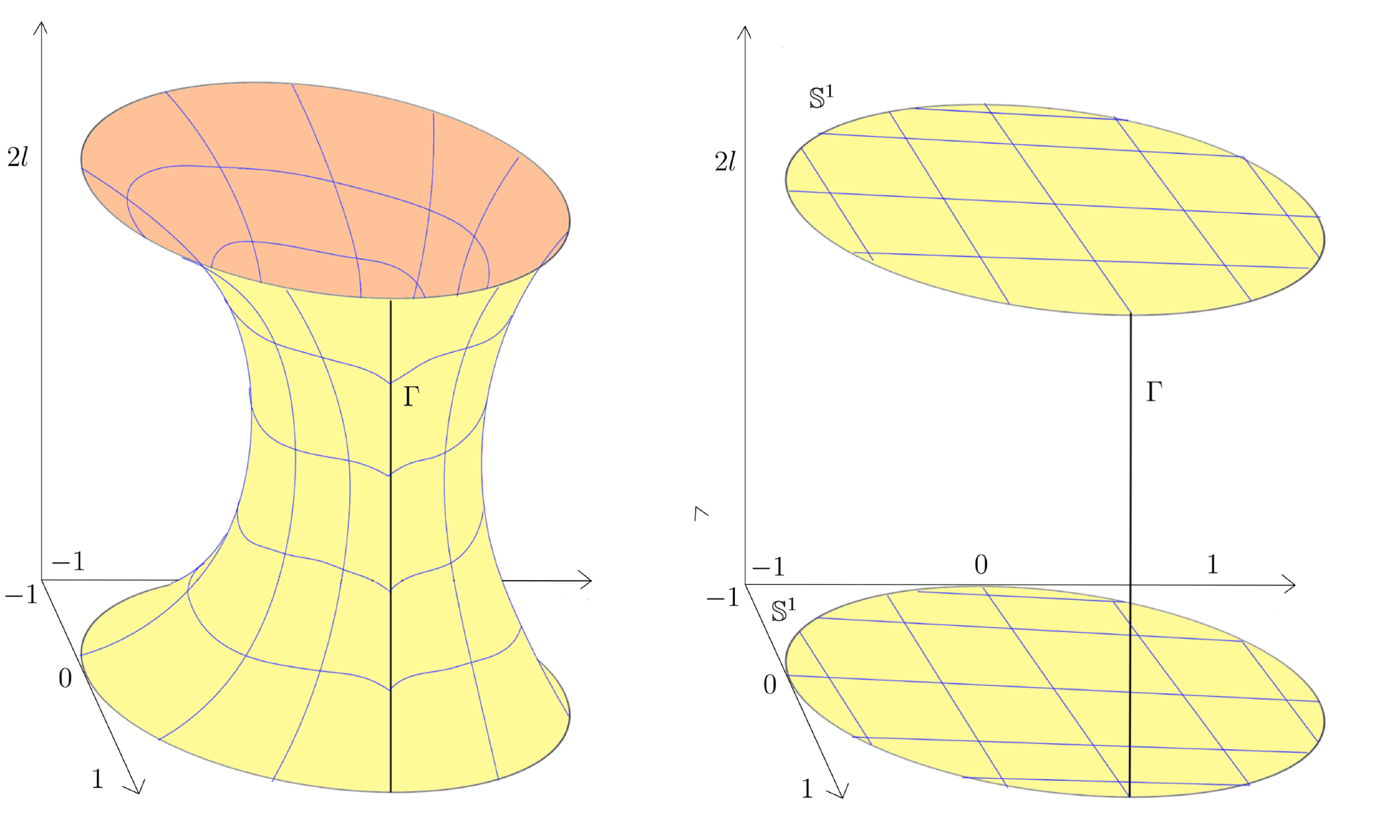}
		\caption{on the left the shape of a possible solution to the Plateau problem with boundary $\Gamma$. 
On the right another solution to the Plateau problem with boundary $\Gamma$.
See Section \ref{subsec:a_Plateau_problem_for_a_self-intersecting_boundary_space_curve}}\label{elliss_1}
		\end{figure}

\section{Cylindrical Steiner symmetrization}\label{sec:cylindrical_Steiner_symmetrization}
In this section we introduce  the cylindrical Steiner symmetrization of a
finite perimeter\footnote{Recall that we choose a representative of 
$U$ such that the closure of its reduced boundary $\partial^* U$  
equals the topological boundary.} set $U\subseteq C_\longR=(-1,l)\times B_1(0)$. 
This rearrangement is obtained slice by slice by spherical (two dimensional) symmetrization, a technique introduced first by P\`olya. We refer to \cite{CPS} and references therein for a exhaustive description of the subject. Here we collect the main properties we will use in the sequel of the paper.
Furthermore we will introduce a generalization of this symmetrization in order to apply it to $2$-integral currents.

Let us recall that $C_\longR$ is endowed with cylindrical coordinates 
$(t,\rho,\theta)\in (-1,l)\times[0,1)\times (-\pi,\pi]$. If $x_1,x_2,x_3$ are cartesian coordinates, we have
 $x_1=t$, $x_2=\rho\cos\theta$, $x_2=\rho \sin\theta$. 
Sometimes it will be convenient to 
extend $2\pi$-periodically
the values of $\theta$ 
on the whole of $\R$.

For every $t\in(-1,l)$ let $U_t
:=U\cap(\{t\}\times\R^2)$ 
the slice of $U$ on the plane 
with first coordinate $t$,
and for every $\rho\in (0,1)$ let $U_t
(\rho):=U_t\cap (\{t\}
\times \{\rho\}\times (-\pi,\pi])$ be the slice of $U_t$
with the circle of radius $\rho$.

\begin{definition}[\textbf{Symmetrization of solid sets in $C_\longR$}]
\label{def:symmetrization_of_solid_sets_in_C_l} 
For every $t\in(-1,l)$ 
and $\rho\in (0,1)$ we let
\begin{equation}\label{eq:Theta_t_rho}
\Theta(t,\rho) = 
\Theta_U(t,\rho)
:=\frac{1}{\rho}\mathcal H^1(U_t
(\rho)),
\end{equation}
and we define
the cylindrically symmetrized set 
$\mathbb S(U)\subseteq C_\longR$ as
\begin{equation}
\label{eq:def_S(U)}
 \mathbb S(U):=\left\{(t,\rho,\theta): t\in (-1,l), 
~ \rho\in (0,1),
~ \theta\in 
\big(-\Theta(t,\rho)/2,\Theta(t,\rho)/2\big)\right\}.
\end{equation}
\end{definition}
Notice that $\Theta_U = \Theta_{\mathbb S(U)}$.
The set $\mathbb S(U)$ enjoys the following properties:
\begin{itemize}
 \item[(1)]
$\mathcal H^2(\mathbb S(U)_t)=
\mathcal H^2(U_t)$
and 
$\mathcal H^1(\partial^*(\mathbb S(U)_t)) \leq
\mathcal H^1(\partial^* (U_t))$
for every $t\in (-1,l)$;
 \item[(2)] $|\mathbb S(U)|=|U|$ and $\mathcal H^2(\partial^*\mathbb S(U)) \leq 
\mathcal H^2(\partial^* U)$.
\end{itemize}

A proof of these properties is contained in \cite[Theorem 1.4]{CPS}.
In particular, since $U$ has
finite perimeter, so is 
$\mathbb S(U)$ and its perimeter cannot increase after symmetrization.
We will need to apply it 
to  $3$-integral currents in $C_\longR$. That is,
(possibly infinite) sums of finite perimeter 
sets with integer coefficients.
For this reason we introduce the following generalization of 
cylindrical symmetrization.

 Let $\mathcal E \in \mathcal D_3(C_\longR)$ 
be an integral $3$-current.
By Federer decomposition theorem \cite[Section 4.2.25, p. 420]{Federer:69} (see
also \cite[Section 4.5.9]{Federer:69} and \cite[Theorem 7.5.5]{Krantz_Parks:08}) 
it follows that there is a sequence $(E_i)_{i\in\mathbb N}$ 
of finite perimeter sets 
such that 
 \begin{align}
\label{eq:dec_mathcalE}
  \mathcal E=\sum_i(-1)^{\sigma_i}\jump{E_i},
 \end{align}
for suitable $\sigma_i\in \{0,1\}$. 
We can also assume the decomposition is done 
in undecomposable components, so that 
\begin{align}\label{eq:126}
|\mathcal E|=
 \sum_i|E_i|
\qquad \text{ and }\qquad |\partial \mathcal E|=
\sum_i\mathcal H^2(\partial^* E_i).
\end{align}
According to Definition \ref{def:symmetrization_of_solid_sets_in_C_l}, we can symmetrize each
set $E_i$ into $\mathbb S(E_i)$.

\begin{definition}[\textbf{Symmetrization of an integer $3$-current}]
\label{def:symmetrization_of_an_integer_3_current}
Let $E:=\textrm{supp}(\mathcal E)$ denote the support of the current 
$\mathcal E {\in \mathcal D_3(C_\longR)}$.
We let 
$$
\mathbb S(E)
:=
\bigcup_i
\mathbb S(E_i),
$$
which will be referred to as the symmetrized support of $\mathcal E$.
The symmetrized current $\mathbb S(\mathcal E)
{\in \mathcal D_3(C_\longR)}
$ is defined as
\begin{equation}\label{eq:def_symmetrization_of_mathcal_E}
\mathbb S(\mathcal E):=\jump{\mathbb S(E)}. 
\end{equation}
\end{definition}
Notice that the multiplicity of 
$\jump{\mathbb S(E)}$ is one, hence $\jump{\mathbb S(E)}$ is the integration over a finite perimeter set.

\subsection{Cylindrical symmetrization of a two-current. Slicings}\label{subsec:slicing_and_cylindrical_symmetrization}

Let us focus on a 
slice $\mathcal E_t$ 
of the current 
$\mathcal E$ 
with respect to a plane 
$\{t\}
\times \R^2_{{\rm target}}$. 
Suppose 
for the moment that $\mathcal E$ is the integration 
over a finite perimeter set (that we identify with $E$)
in $C_\longR$; 
$\mathcal E_t$ is the integration over the slice $E_t$
of $E$, and suppose that the boundary of $E_t$ 
is the trace $\sigma$
of a rectifiable Jordan curve.
Applying Definition 
\ref{def:symmetrization_of_an_integer_3_current}
to the set $ E$ 
we see that $E_t$ is 
transformed into the symmetrized set 
$\mathbb S(E_t)$ 
whose 
boundary is again\footnote{$\mathbb S(E_t)$ 
is simply connected. Indeed the support of $\rho\mapsto \Theta(t,\rho)$ 
is a connected subset of $(0,1)$.}
 the trace $\sigma_s$
of a Jordan curve. 
By the properties of the symmetrization we 
infer $\mathcal H^1(\sigma) \geq 
\mathcal H^1(\sigma_s)$.

However, if
 the boundary of $E_t$ is the trace $\sigma$ of a nonsimple
curve, then the procedure is more involved. More generally, from
Definition \ref{def:symmetrization_of_an_integer_3_current}, we see that for a.e. $t\in (-1,l)$ 
the slice $\mathcal E_t$ 
of $\mathcal E$ is an integral  $2$-current, and it can 
be represented by integration over finite perimeter sets 
$(E_i)_t$ (with suitable signs) which are 
exactly the slices of the sets $E_i$ 
in \eqref{eq:dec_mathcalE}. Moreover for a.e. $t\in (-1,l)$ the boundary of $\mathcal E_t$ is a $1$-dimensional integral current with finite mass, and it 
coincides with the integration    (with 
suitable signs) over the boundaries of $(E_i)_t$, namely 
$$
\partial \mathcal E_t=
\sum_i(-1)^{\sigma_i}\partial\jump{ (E_i)_t}.
$$
Let us call this boundary $\sigma$ (which, with a little abuse of notation, we identify with an integral  $1$-current, an at most
countable sum of simple curves), 
and set
$\sigma_s:=\partial \jump{\mathbb S(E)_t}$. 
By
Definition 
\ref{def:symmetrization_of_an_integer_3_current}
it then follows that $\mathbb S(\mathcal E)_t
=\jump{\mathbb S(E)_t}$.
Now, by the properties of the symmetrization, 
we see that $\mathcal H^1(\textrm{supp}(\sigma)) \geq
\mathcal H^1({\rm supp}(\sigma_s))$.
Also in this case it turns out that $\sigma_s$ is the integration over countable many simple curves (with 
suitable orientation).

We have described  so far how the boundary of $\mathcal E$ is trasformed slice by slice. 
In general if $\mathcal E$ is a $3$-integral current in $C_\longR$, then the current $\mathcal S:=\partial \mathcal E$ has the property that 
$$|\mathcal S|\geq\mathcal H^2(\partial^*\mathbb S(E)).$$
There is also a viceversa.  
Precisely assume that  
$\mathcal S$ is any boundaryless integral $2$-current in $C_\longR$. 
Then there is an integral $3$-current $\mathcal E$ whose boundary is $\mathcal S$. 
So that we can define the symmetrization of 
$\mathcal S$ by symmetrizing $\mathcal E$. 

\begin{definition}[\textbf{Cylindrical symmetrization of 
the boundary of a three-current}]
\label{def:symmetrization_of_the_boundary_of_a_three_current}
 The symmetrization of $\mathcal S = \partial \mathcal E$ is 
defined as 
$$
\mathbb S(\mathcal S):=\partial \mathbb S(\mathcal E).
$$
 \end{definition}

The next lemma will be useful in Section \ref{subsec:symmetrization_of_the_image_of_D_k_cap_B_eps}. 

\begin{lemma}\label{lem:sections_of_symm_current}
 Let $t\in(-1,l)$ be such that $\mathcal S\res (\{t\}\times\R^2)=0$. Then 
 \begin{equation}\label{eq:sec_S_t}
\mathbb S(\mathcal S)\res (\{t\}\times\R^2)=0.
\end{equation}
\end{lemma}

\begin{proof}
 We know that $\mathcal S=\partial \mathcal E$.
 By the properties of the cylindrical symmetrization (see item (2) above)
for  each set $E_{i}$ we have  
 \begin{align*}
  \mathcal H^2\Big(
(\{t\}\times \R^2)\cap \partial^* E_{i}
\Big)\geq\mathcal 
H^2\Big(
(\{t\}\times \R^2) \cap
\partial^*\mathbb S(E_{i})
\Big).
 \end{align*}
{}From our assumption 
it follows\footnote{
This follows since the decomposition is done in 
undecomposable components: if there is some boundary of some $E_i$ 
then it cannot cancel with some other boundary (opposite oriented) 
of some $E_j$.
}
 that for 
all $i$ 
we have
$\mathcal H^2(
(\{t\}\times \R^2)
\cap \partial^* E_{i})=0$, and thus 
$$ \mathcal H^2(
(\{t\}\times \R^2)\cap 
\partial^* \mathbb S(E_{i})
)=0, 
\qquad i\in \mathbb N.
$$
To conclude the proof we have to show that 
 \begin{align}\label{claim_conclusion}
\mathcal H^2\Big(
(\{t\}\times \R^2)\cap 
\partial^*\mathbb S(E)\Big)=
\mathcal H^2
\Big((\{t\}\times \R^2)\cap \partial^*(\cup_i \mathbb S(E_{i}))\Big)=0.  
 \end{align}
 The conclusion easily follows if the family $\{E_{i}\}$ is finite, since in 
this case $\partial (\cup_i \mathbb S(E_{i}))\subseteq 
\cup_i\partial\mathbb S(E_{i})$. 
If this family is not finite we argue as follows: 
fix $\eps>0$ and $N_\eps\in \mathbb N$ so that (see \eqref{eq:126}) 
 \begin{align}\label{control_perimeters}
  \sum_{i=N_\eps+1}^{+\infty}\mathcal H^2(\partial^* E_{i})\leq \eps.
 \end{align}
We have $$\mathbb S (E)=\cup_i\mathbb S(E_{i})=
\big(\cup_{i=1}^{N_\eps}\mathbb S(E_{i})\big)\cup 
\big(\cup_{i=N_\eps+1}^{+\infty}\mathbb S(E_{i})\big)=:A\cup B,$$
 thus 
 $$
(\{t\}\times \R^2)\cap
\partial \mathbb S (E)\subseteq 
\Big(
(\{t\}\times \R^2)\cap \partial^* A\
\Big)\cup 
\Big(
(\{t\}\times \R^2)
\cap 
\partial^* B\Big).$$
 By the previous observations 
$\mathcal H^2(
(\{t\}\times \R^2)\cap
\partial^* A
)=0$; we will prove that 
 $$\mathcal H^2(
(\{t\}\times \R^2)
\cap \partial^* B)=
\mathcal 
H^2\Big((\{t\}\times \R^2)
\cap 
\partial^* (\cup_{i=N_\eps+1}^{+\infty}\mathbb S(E_{i}))\Big)
\leq \eps,$$
 so that \eqref{claim_conclusion} follows by arbitrariness of $\eps>0$. 
To do so, it suffices to  write
 $$
\mathcal 
H^2\Big((\{t\}\times \R^2)
\cap 
\partial^* (\cup_{i=N_\eps+1}^{+\infty}\mathbb S(E_{i}))\Big)
\leq \mathcal H^2\big(\partial^* (\cup_{i=N_\eps+1}^{+\infty}\mathbb S(E_{i}))\big)\leq \sum_{i=N_\eps+1}^{+\infty}
\mathcal H^2(\partial^* \mathbb S(E_{i}))\leq \eps.$$
The last inequality follows from \eqref{control_perimeters} and from the fact that symmetrization does not increase the perimeter. As for the second inequality, it follows from the lower semicontinuity of the perimeter. Indeed, 
setting $F_k:=\cup_{i=N_\eps+1}^{k}\mathbb S(E_{i})$ for $k \geq N_\eps+1$, we see that $F_k\rightarrow F_\infty:=\cup_{i=N_\eps+1}^{\infty}\mathbb S(E_{i})$ in $L^1(C_\longR)$, 
and since $F_k$ has finite perimeter we infer
\begin{align*}
 \mathcal H^2(\partial^* F_\infty)\leq \liminf_{k\rightarrow +\infty}
\mathcal H^2(\partial^* F_k)\leq \liminf_{k\rightarrow +\infty}\sum_{i=N_\eps+1}^{k}\mathcal H^2(\partial^* \mathbb S(E_{i})). 
\end{align*}
 \end{proof}

As before, 
we can look at what happens to the current $\mathcal S$ slice by slice. 
If 
$\partial \mathcal E=\mathcal S$, then 
$\mathcal S_t=-\partial (\mathcal E_t)$
for a.e. $t\in(-1,l)$. 
Assume that $\mathcal E$ decomposes as 
in \eqref{eq:dec_mathcalE}, then 
\begin{align}\label{eq:dec_slice}
  \mathcal E_t=\sum_i(-1)^{\sigma_i}\jump{(E_i)_t}\qquad
\textrm{for~ a.e.~} t \in (-1,l).
\end{align}
Now the sets $(E_i)_t$ are symmetrized as before, and their union, 
denoted $\mathbb S(E_t)$ (so
that $\mathbb S(\mathcal E)_t=\jump{\mathbb S(E_t)}$) satisfies 
\begin{align*}
 \partial \jump{\mathbb S(E_t)}=-\mathbb S(\mathcal S)_t
\end{align*}
and 
\begin{align*}
 |\mathcal S_t|\geq \mathcal H^1(\partial^* \mathbb S(E)_t).
\end{align*}

Let us go back to \eqref{eq:dec_slice}. In general
\begin{equation}\label{eq:slice_ineq}
|\mathcal E_t|\leq \sum_i\mathcal H^2((E_i)_t); 
\end{equation}
 however, since the decomposition is made of undecomposable components, \eqref{eq:126} holds 
and hence
\begin{align}\label{eq:dec_slice_mass}
 |\mathcal E_t|= \sum_i\mathcal H^2((E_i)_t)
\qquad
\textrm{~ for~ a.e.~} t\in(-1,l). 
\end{align}
This can be seen integrating in $t$ formula \eqref{eq:slice_ineq}, so that if strict inequality holds for a positive measured set of $t\in(-1,l)$ 
we would get strict inequality in the first equation of \eqref{eq:126}, which 
is a contradiction. 

Moreover, by construction, $\mathcal H^2((E_i)_t)=\mathcal H^2(\mathbb S(E_i)_t)$ for all $i$, and since $\mathbb S(E)_t=\cup_i\mathbb S(E_i)_t$ it also follows $$|\mathcal E_t|= \sum_i\mathcal H^2((E_i)_t)=\sum_i\mathcal H^2(\mathbb S(E_i)_t)\geq\mathcal H^2(\mathbb S(E)_t).$$

Now we fix $t$ such that \eqref{eq:dec_slice_mass} holds 
and set $F_i:=(E_i)_t$, $\mathcal{F}:=\mathcal E_t$, 
$F := {\rm supp}(\mathcal F)$, 
$\mathbb S(F)=\mathbb S(E)_t$.
The set $F_i\in B_1$ can be sliced with respect to the radial coordinate $\rho\in (0,1)$, so that exploiting that 
$$
(\mathcal E_t)_\rho = 
\sum_i (-1)^{\sigma_i}
\jump{((E_i)_t)_\rho}
$$
holds for a.e. $\rho$, we can repeat the same argument as before 
to obtain 
\begin{align*}
 |\mathcal F_\rho|= \sum_i\mathcal H^1((F_i)_\rho) \qquad
\textrm{for~ a.e.~} \rho\in(0,1).
\end{align*}
Again we have $\sum_i\mathcal H^1((F_i)_\rho)\geq \mathcal H^1(\mathbb S(F)_\rho)$.
Recalling that $\mathbb S(F)_\rho=\mathbb S(E)_t\cap \partial B_\rho$, 
we conclude that, for a.e. $t\in(-1,l)$ and for a.e. $\rho\in(0,1)$ the slice 
$(\mathcal E_t)_\rho$ satisfies
\begin{equation}\label{eq:in_conclusion}
 |(\mathcal E_t)_\rho|\geq \mathcal H^1(\mathbb S(E)_t\cap \partial B_\rho)=\rho\Theta(t,\rho),
\end{equation}
where we have defined 
$\Theta(t,\rho):=\rho^{-1}\mathcal H^1(\mathbb S(E)_t\cap \partial B_\rho)$ the 
measure in radiants of the arc $\mathbb S(E)_t\cap \partial B_\rho$.

\begin{figure}
\begin{center}
    \includegraphics[width=0.7\textwidth]{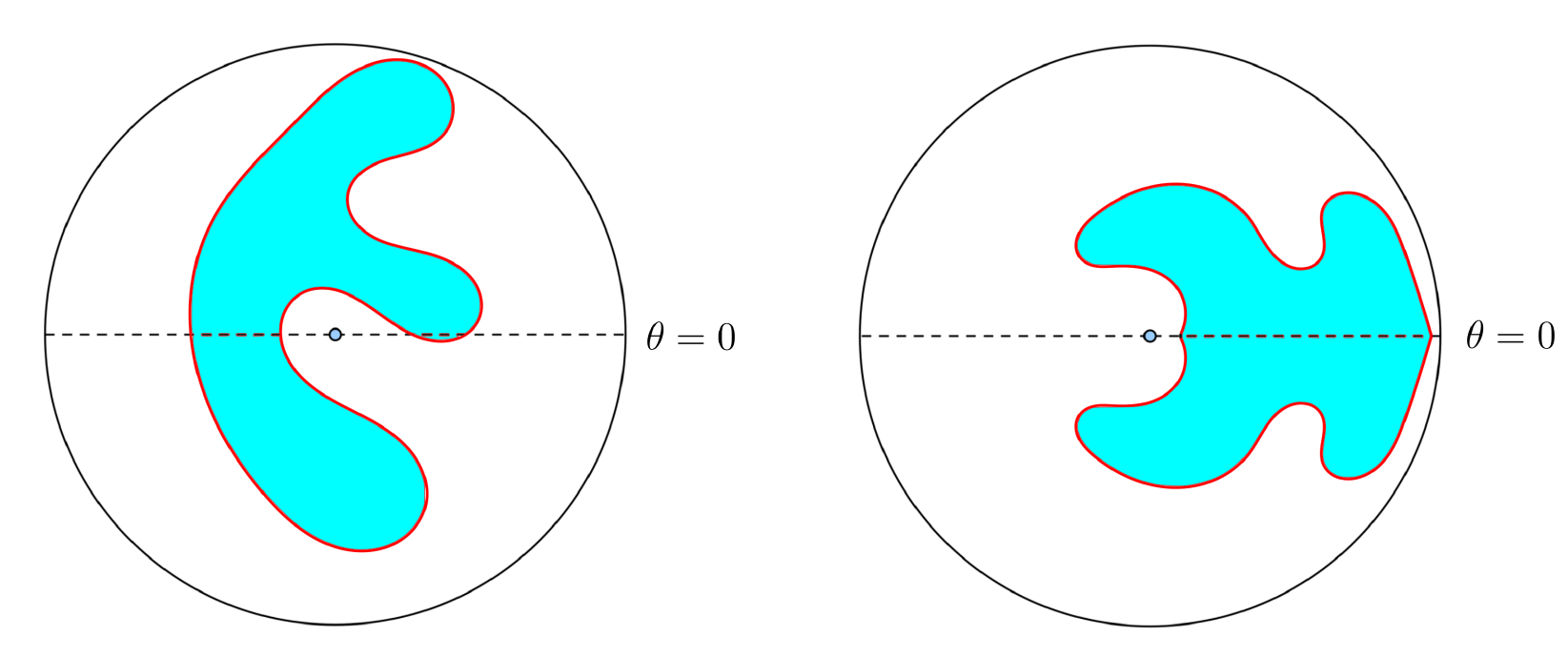}
 \caption{The 
symmetrization of a subset of $B_1$ bounded by a Jordan curve,
with the respect to the radius $\{\theta=0\}$; see formula \eqref{eq:Theta_t_rho}.
}\label{fig2}
\end{center}
\end{figure}

\begin{remark}
 In the sequel we are going to apply the cylindrical symmetrization to a current supported in the portion of the cylinder 
$(0,\longR)\times B_1\subset C_\longR$. The fact that we set the symmetrization in $C_\longR=(-1,l)\times B_1$ will be useful to avoid possible creation of boundary on the disc $\{0\}\times B_1$.  
\end{remark}

\section{Lower bound: first reductions on a recovery sequence}
\label{sec:lower_bound:first_reductions_on_a_recovery_sequence}
Let $\vortexmap(x) = x/\vert x\vert$, $x\neq 0$, be the vortex map and
$\Omega = \sourcedisk_\longR$;
we aim to prove that 
\begin{align*}
 \relarea(\vortexmap,\Omega)\geq \int_\Omega
|\mathcal M(\nabla \vortexmap)|~dx+\frac12\mathcal P_\Gamma(X),
\end{align*}
where $\Gamma$ is the image of 
the self-intersecting curve parametrized in \eqref{curvegamma_def}, 
see Fig. \ref{elliss_1}, 
and $X$ 
is a disc-type solution of the Plateau problem for $\Gamma$.

Let $(u_k)\subset C^1(\Omega,\R^2)$ be a recovery sequence 
for the area of the graph of $\vortexmap$, {\it i.e.}, 
$u_k\rightarrow \vortexmap$ in $L^1(\Omega,\R^2)$ and
\begin{align*}
\liminf_{k\rightarrow +\infty}\area(u_k,\Omega) = 
\relarea(\vortexmap,\Omega);
\end{align*}
with no loss of generality 
we can suppose that $u_k \to \vortexmap$ almost everywhere in $\Omega$
and
\begin{equation}\label{eq:with_no_loss}
\liminf_{k\rightarrow +\infty}\area(u_k,\Omega) = 
\displaystyle
\lim_{k\rightarrow+ \infty}\area(\vortexmap_k,\Omega)< +\infty.
\end{equation}
If $\Pi:\R^2_{{\rm target}}\rightarrow
\overline B_1 \subset \R^2_{{\rm target}}$
 is the projection map onto $\overline B_1$,
\begin{align}\label{eq:Pi}
\Pi(x):=\begin{cases}
        \frac{x}{|x|}&\text{if }|x|>1
\\
        x&\text{if }\vert x\vert \leq 1,
       \end{cases}
\end{align}
then
\begin{align*}
 \area(v,\Omega)\geq \area(\Pi\circ v,\Omega) \qquad \forall 
v\in C^1(\Omega,\R^2). 
\end{align*}
Notice that in general $\Pi\circ v 
\notin C^1(\Omega,\R^2)$;  
however $\Pi\circ v$ is of class $C^1$ on the set $\{x \in \Om:|v(x)|<1\}$ and Lipschitz continuous in $\Omega$. Therefore, possibly
replacing $u_k$ by $\Pi\circ u_k$, we can assume that 
$u_k$ takes values in $\overline B_1$ for all $k\in \mathbb N$. 

We start by dividing the source disc
$\Omega$ in several suitable subsets. 
First we observe that from \eqref{eq:with_no_loss}
there exists a constant $C>0$ such that
\begin{equation}
\label{eq:uniform_bound_of_grad_u_k}
 C\geq \int_{\Omega}|\nabla u_k|~dx
=\int_0^l\int_{\partial \sourcedisk_\sourceradialcoordinate}
|\nabla u_k(\sourceradialcoordinate,\sourceangularcoordinate)|
~d\mathcal H^1(\sourceangularcoordinate)d\sourceradialcoordinate \qquad \forall k \in \NN.
\end{equation}
By Fatou's lemma, we then infer
\begin{align*}
\int_0^l L(\sourceradialcoordinate)~d\sourceradialcoordinate\leq C,
\end{align*}
where 
$$
L(\sourceradialcoordinate):=
\liminf_{k\rightarrow+\infty} \int_{\partial \sourcedisk_\sourceradialcoordinate}|\nabla u_k(\sourceradialcoordinate,\sourceangularcoordinate)|
~d\mathcal H^1(\sourceangularcoordinate)\qquad 
{\rm for~ a.e.}~\sourceradialcoordinate \in (0,\longR).
$$ 
In particular, $L(\sourceradialcoordinate)$ is finite for 
almost every $\sourceradialcoordinate\in (0,\longR)$. 
Since $u_k\rightarrow \vortexmap$ 
almost everywhere
in $\Omega$, we have that 
for almost every
 $\sourceradialcoordinate\in (0,\longR)$ 
$$
u_k(\sourceradialcoordinate,\sourceangularcoordinate)\rightarrow u(\sourceradialcoordinate,\sourceangularcoordinate) \qquad {\rm 
for~} \mathcal H^1-{\rm a.e.}~
\sourceangularcoordinate \in \partial \sourcedisk_\sourceradialcoordinate. 
$$
Thus we can choose $\eps \in (0,1)$ arbitrarily small 
such that the two following properties
are satisfied:
\begin{align}
& L(\eps)
\leq C_\eps
\ \ ~{\rm for~ a~ constant~} C_\eps>0 ~{\rm depending~ on~} \eps;
\label{eq:H1}
\\
& \lim_{k \to +\infty}
u_k(\eps,\sourceangularcoordinate) = \vortexmap(\eps,\sourceangularcoordinate) 
\ \ ~{\rm for~} \mathcal H^1-{\rm a.e.}~ \sourceangularcoordinate \in \partial \sourcedisk_\eps.
\label{eq:H2}
\end{align}

\subsection{The functions $\diffuku$, 
the subdomains $A_n$ and $D_k^\delta$, and selection of $(\lambda_k)$}\label{subsec:selection_of_subdomains}
By Egorov lemma, there exists a sequence $(A_n)$ of 
measurable subsets of $\Omega$ such that, for any $n \in \mathbb N$,
$A_{n+1}\subseteq A_n$, 
\begin{align}\label{eq:measure_of_A_n}
 |A_n|<\frac{1}{n},
\end{align} 
and
\begin{align}
\label{eq:u_k_converges_to_u_uniformly_out_of_A_n}
 u_k\rightarrow u \;\text{ in }L^\infty(\Omega\setminus A_n,\R^2) \ 
{\rm as}~ k \to +\infty.
\end{align}

\begin{definition}[\textbf{The function $\diffuku$ and the set $D_k^\delta$}]
We indicate by $\diffuku:\Omega\setminus \{0\}\rightarrow [0,2]$ the function 
\begin{equation}\label{eq:w_k}
 \diffuku:=\big|u_k-\vortexmap\big|,
\end{equation}
and 
for any $\delta>0$ we set
\begin{equation}\label{eq:D_k_delta}
 D^\delta_k:=\{x\in \Omega \setminus \{0\}:\diffuku(x)>\delta\} =: \{\diffuku > \delta\}.
\end{equation}
\end{definition}
Notice that 
\begin{equation}\label{eq:partial_D_k_inside_level_set}
 \partial D^\delta_k\subseteq \{x\in \Omega \setminus \{0\}: \diffuku(x)=\delta\} =:
\{\diffuku = \delta\}.
\end{equation}
For $\eps$ satisfying 
\eqref{eq:H1} 
and \eqref{eq:H2}, we have
 $\diffuku \in {\rm Lip}(\Om \setminus \overline\sourcedisk_\eps;\R^2) \cap 
W^{1,1}(\Om; \R^2)$.
For any $n \in \mathbb N$, 
from
\eqref{eq:u_k_converges_to_u_uniformly_out_of_A_n}
it follows that 
for any $\delta>0$ there exists
$k_{\delta,n} \in \NN$ 
such that $\diffuku<\frac{\delta}{2}$ in $\Omega\setminus A_n$ 
for any $k\geq k_{\delta,n}$, and thus  
\begin{equation*}
\Omega\setminus A_n\subseteq \Big\{\diffuku<\frac{\delta}{2}\Big\}
\subseteq\Omega\setminus D^\delta_k
\qquad \forall k > k_{\delta,n}.
\end{equation*}
Passing to the complement, from \eqref{eq:partial_D_k_inside_level_set}
 and the inclusion 
$\{\diffuku = \delta\} \subseteq
\{\diffuku\geq \delta/2\}$,
we get 
\begin{align}\label{inclusion}
 D^\delta_k\subseteq A_n\qquad \text{and }\qquad \partial D^\delta_k\subseteq A_n
\qquad 
\forall k > k_{\delta,n}.
\end{align}

\begin{lemma}[\textbf{Choice of $\lambda_k$ and estimates on $D_k^{\lambda_k}$}]
\label{lem:choice_of_u_k_and_t_k} 
Let $\eps \in (0,1)$ satisfy 
\eqref{eq:H1}  and \eqref{eq:H2}. 
Let $n>0$ and $A_n \subset \Om$ be a measurable 
set satisfying
 \eqref{eq:measure_of_A_n} and \eqref{eq:u_k_converges_to_u_uniformly_out_of_A_n}.
Then there are a (not relabelled) subsequence of  
$(u_k)$ and  a decreasing
 infinitesimal sequence $(\radialdifference_k)$ of positive numbers,
both depending on $n$ and $\eps$, such that the following properties hold:
\begin{itemize}
 \item[(i)] for all $k\in \mathbb N$ we have 
$\radialdifference_k\neq1-|u_k(0)|$ 
and the boundary of the set $D^{\radialdifference_k}_k 
= \{\diffuku > \radialdifference_k\}$ 
consists of an at most countable 
number of continuous curves 
which are either closed or with endpoints on $\partial\Omega$, and 
whose total length is finite;
 \item[(ii)] $D^{\radialdifference_k}_k \cup \partial D^{\radialdifference_k}_k\subseteq A_n$
for all $k \in \mathbb N$;
 \item[(iii)] 
$\displaystyle 
\lim_{k\to +\infty}  
\displaystyle 
\int_{\partial D^{\radialdifference_k}_k}\diffuku ~d\mathcal H^1 = 
\lim_{k\to +\infty}  
\left(\radialdifference_k \mathcal H^1(\partial D_k^{\radialdifference_k})\right) 
=0$;
 \item[(iv)] 
$\partial D^{\radialdifference_k}_k\cap 
\partial \sourcedisk_\eps$ consists of a finite set 
of points. Hence\footnote{
The relative boundary of 
$D^{\radialdifference_k}_k\cap 
\partial \sourcedisk_\eps$ is contained 
in 
$\partial D^{\radialdifference_k}_k\cap 
\partial \sourcedisk_\eps$.}, 
also the relative boundary of $D^{\radialdifference_k}_k\cap 
\partial \sourcedisk_\eps$ in $\partial \sourcedisk_\eps$
consists of a finite set
$\{x_i\}_{i \in I_k}$ of points which are the endpoints 
of the corresponding finite number of arcs forming 
$D^{\radialdifference_k}_k \cap
\partial \sourcedisk_\eps$, and
\begin{align}
\label{sum_xi}
\lim_{k \to +\infty} \sum_{i \in I_k}\diffuku(x_i) = 
0;
\end{align}
\item[(v)] 
$  \mathcal {H}^1(D^{\radialdifference_k}_k \cap\partial \sourcedisk_\eps
)\leq \frac{1}{n}$
for all $k \in \NN$.
 \end{itemize}
\end{lemma}
\begin{proof}
Let 
$$
I:=(0,2)\setminus \bigcup_{k\in \mathbb N}\{1-|u_k(0)|\},
$$
which is of full measure in $(0,2)$.

We have, for an absolute positive constant $\alpha$, recalling
the definition of $\diffuku$ in \eqref{eq:w_k},
\begin{equation}\label{eq:alpha}
 \int_{\Om}|\nabla u_k-\nabla\vortexmap|~dx\geq
\alpha \int_{\Om}|\grad \diffuku|~dx= \alpha \int_0^{2}\mathcal  H^1(\{\diffuku=\radialdifference\})~d\radialdifference,
\end{equation}
where the last equality 
follows from the coarea formula, recalling also that
$u_k$ takes values in $\overline B_1$. 
The left-hand side is uniformly bounded with respect to $k$, thanks to
\eqref{eq:uniform_bound_of_grad_u_k} and the fact that $\grad 
\vortexmap \in L^1(\Omega, \R^2)$. Thus, denoting 
\begin{equation}\label{eq:varphi}
\varphi_k(\cdot):=\mathcal  H^1(\{\diffuku=\cdot\}), \qquad 
\varphi:=\liminf_{k \to+\infty} \varphi_k,
\end{equation}
 we get, from Fatou's lemma,
\begin{align}\label{const}
 \int_0^2\varphi(\radialdifference)~d\radialdifference=
\int_I\varphi(\radialdifference)~d\radialdifference\leq C_1,
\end{align}
for some constant $C_1>0$.

Let us now focus attention on the set $\partial \sourcedisk_\eps$. 
We apply the tangential coarea formula to $\partial \sourcedisk_\eps$ (see 
for instance \cite[Theorems 11.4, 18.8]{Maggi:12}) so that, if
 $\partial_{\rm tg}$ 
stands for the tangential derivative
 along $\partial \sourcedisk_\eps$, we have
\begin{align*}
 \int_{\partial \sourcedisk_\eps}\big|\partial_{\rm tg} \diffuku \big|
~d\mathcal H^1=
\int_0^2\mathcal H^0(\{\diffuku=\radialdifference\}\cap \partial \sourcedisk_\eps)~d\radialdifference.
\end{align*}
Arguing in a similar manner as before,
denoting 
\begin{equation}\label{eq:psi}
\psi_k(\cdot):=\mathcal H^0(\{\diffuku=\cdot\}\cap \partial \sourcedisk_\eps), \qquad
\psi:=\liminf_{k\rightarrow+\infty}\psi_k,
\end{equation}
it follows that, exploiting condition \eqref{eq:H1},
there exists a constant $C_\eps'>0$ 
such that 
\begin{equation}\label{eq:integrability_of_psi}
 \int_I\psi(\radialdifference)~d\radialdifference\leq C_\eps'.
\end{equation}
We now claim that
\begin{equation}\label{claim1}
\exists 
(\radialdifference_m)\subset I: \ 
 \lim_{m\to +\infty} \lambda_m=0, \ \ 
\lim_{m\rightarrow +\infty} (\varphi(\radialdifference_m) \radialdifference_m)=0=
\lim_{m\rightarrow + \infty} (\psi(\radialdifference_m) \radialdifference_m).
\end{equation}
Recalling that $I$ is of full measure, assume \eqref{claim1} is false, 
so that either there are $c_0>0$ and $\delta_0>0$ such that 
\begin{align}\label{c_0estimate}
 \varphi(\radialdifference)>\frac{c_0}{\radialdifference} \qquad \forall \radialdifference
 \in (0,\delta_0)\cap I,
\end{align}
or 
there are $c_0'>0$ and $\delta_0'>0$ such that 
\begin{align}\label{c_0estimate'}
 \psi(\radialdifference)>\frac{c_0'}{\radialdifference} \qquad 
\forall \radialdifference \in (0,\delta_0')\cap I.
\end{align}
Suppose for instance we are in case \eqref{c_0estimate}:
since $I$ has full measure, this contradicts \eqref{const}; 
the same argument applied to \eqref{c_0estimate'}
 leads to contradict 
\eqref{eq:integrability_of_psi}. Hence claim
\eqref{claim1} is proven, and therefore,
upon passing to a (not relabelled) subsequence 
we might assume that $(\lambda_m)$ is decreasing, and  
\begin{align*}
 \varphi(\radialdifference_m)\radialdifference_m
<\frac{1}{m},\qquad \psi(\radialdifference_m)\radialdifference_m<\frac{1}{m} \qquad \qquad
\forall m \in \NN.
\end{align*}
Thus, recalling \eqref{eq:varphi} and \eqref{eq:psi}, 
 for any $m \in \NN$ there are infinitely many $l \in \mathbb N$  such that 
\begin{equation}\label{conditions_tm}
 \varphi_{l}(\radialdifference_m)\radialdifference_m
<\frac{2}{m},\qquad \psi_{l}(\radialdifference_m)\radialdifference_m<\frac{2}{m}.
\end{equation}
Moreover, for any $n \in \NN$ and $m \in \NN$ 
there exists $k(n, \radialdifference_m) \in \NN$ 
such that 
\begin{align}\label{inclusion2}
D^{\radialdifference_m}_h \cup
 \partial D^{\radialdifference_m}_h
\subseteq A_n\qquad \text{ and }\qquad 
\mathcal {H}^1(D^{\radialdifference_m}_h \cap \partial \sourcedisk_\eps)
\leq \frac{1}{n} \qquad
\forall h\geq k(n,\radialdifference_m),
\end{align}
where the inclusion follows from 
 \eqref{inclusion} and 
the inequality being a consequence of \eqref{eq:H2}.
For any $m \in \mathbb N$ we can choose $h_m \in \mathbb N$ 
(depending also on $n$)
such that $h_m < h_{m+1}$, 
 $h_m\geq k(n,\radialdifference_m)$, and \eqref{conditions_tm} is verified for $l=h_m$.
Therefore
\begin{align}
& \lim_{m\rightarrow+\infty}(\varphi_{h_m}(\radialdifference_m)\radialdifference_m)=0,
\label{eq:phi_h_m_t_m_and_inclusion}
\\
& D^{\radialdifference_m}_{h_m}\cup 
 \partial 
D^{\radialdifference_m}_{h_m}
\subseteq A_n \ \ \text{ for all}~n,m\in \NN,
\label{eq:phi_h_m_t_m_and_inclusion_1}
\\
&  \lim_{m\rightarrow + \infty}
(\psi_{h_m}(\radialdifference_m)\radialdifference_m)=0,
\label{eq:psi_h_m_t_m_and_inclusion}
\\
& 
\mathcal {H}^1
(D^{\radialdifference_m}_{h_m}\cap
\partial \sourcedisk_\eps)\leq 
\frac{1}{n} \ \ \text{ for all}~n,m \in \NN.
\label{eq:psi_h_m_t_m_and_inclusion_1}
\end{align}
Notice also that from \eqref{conditions_tm}  
we have $\psi_{h_m}(\radialdifference_m)<+\infty$, 
so that $\{d_{h_m}=\radialdifference_m\}\cap \partial \sourcedisk_\eps$ is a finite
set $\{\widetilde x_i\}$ of points.
The relative
boundary $\partial(D^{\radialdifference_m}_{h_m} \cap \partial \sourcedisk_\eps)$ 
of 
$D^{\radialdifference_m}_{h_m} \cap \partial \sourcedisk_\eps$ 
in $\partial \sourcedisk_\eps$ must belong to $\partial D^{\radialdifference_m}_{h_m} 
\cap \partial \sourcedisk_\eps\subseteq
\{d_{h_m}=\radialdifference_m\}\cap \partial \sourcedisk_\eps
=\{\widetilde x_i\}$. 
Hence, let 
$\{x_i\}\subseteq\{\widetilde x_i\}$ be 
the set of boundary points of $D^{\radialdifference_m}_{h_m} 
\cap \partial \sourcedisk_\eps$ 
in $\partial \sourcedisk_\eps$. 

Since $D^{\radialdifference_m}_{h_m}\cap \partial \sourcedisk_\eps$ is 
open in $\partial \sourcedisk_\eps$, we have that (whenever it is nonempty)
it consists either of the union of arcs 
with endpoints $\{x_i\}$ or is the whole of $\partial
\sourcedisk_\eps$, and  statements (ii) and (v) follow.
Notice also that 
$$
\sum_{x \in \partial (D_{h_m}^{\radialdifference_m}\cap \partial
\sourcedisk_\eps)} d_{h_m}(x) =
 \mathcal H^0(\partial (D_{h_m}^{\radialdifference_m}\cap \partial
\sourcedisk_\eps)) \radialdifference_m \leq
\psi_{h_m}(\radialdifference_m) \radialdifference_m,
$$
and \eqref{sum_xi} follows from \eqref{eq:psi_h_m_t_m_and_inclusion}. 

To prove (iii) 
we see that, by definition of $\varphi_k$ in \eqref{eq:varphi}
and 
recalling \eqref{eq:phi_h_m_t_m_and_inclusion}, we obtain
\begin{align*}
 \lim_{m\rightarrow +\infty}\int_{\{d_{h_m} = \radialdifference_m\}}
d_{h_m} ~d\mathcal H^1 = 
\lim_{m\rightarrow+\infty}\left(
\mathcal H^1(\{d_{h_m}=\radialdifference_m\})\radialdifference_m
\right)
=\lim_{m\rightarrow+\infty}(\varphi_{h_m}(\radialdifference_m)\radialdifference_m)=0.
\end{align*}
A similar argument holds for $\psi_k$ using
\eqref{eq:psi_h_m_t_m_and_inclusion},  and also (v) follows.

It remains to prove (i). 
The first assertion follows since $\radialdifference_m\in I$ 
from \eqref{claim1}. 
As for the second assertion, we see that $D_{h_m}^{\radialdifference_m}$ 
is a subset of $\Omega\setminus\{0\}$ whose perimeter is finite:
indeed, by definition the reduced boundary of $D_{h_m}^{\radialdifference_m}$ is a
subset of $\{d_{h_m}=\radialdifference_m\}$, which has finite $\mathcal H^1$
measure by \eqref{conditions_tm}.
Thus $\partial D_{h_m}^{\radialdifference_m}$  is a closed $1$-integral current 
in $\Omega\setminus \{0\}$ 
and by the decomposition theorem for $1$-dimensional currents 
it is the sum of integration on simple curves 
\cite[pag. 420, 421]{Federer:69}, 
either closed or with endopoints on the 
boundary of $\Omega\setminus\{0\}$, {\it i.e.}, $\{0\} \cup \partial\Om$. 
The finiteness of the total length of these curves follows,
 since $D_{h_m}^{\lambda_m}$
is a set of finite perimeter.
This concludes the proof of (i), and of the lemma.
\end{proof}

\begin{cor}\label{cor:t_k_H^1}
Let $\eps$, $n$ and $(\radialdifference_k)$
 be as in Lemma \ref{lem:choice_of_u_k_and_t_k}.  Then 
 \begin{align*}
\lim_{k \to +\infty}
\left(
\mathcal H^1(\{\diffuku=\radialdifference_k\})\radialdifference_k  
\right)=0, \qquad 
\lim_{k \to +\infty}
\left(
\mathcal H^0\big(\{\diffuku=\radialdifference_k\}
\cap \partial \sourcedisk_\eps(0)\big)\radialdifference_k
\right)=0.
 \end{align*}
\end{cor}
\begin{proof}
It follows from the proof of Lemma \ref{lem:choice_of_u_k_and_t_k}.
\end{proof}

Once for all we fix the sequence $(\radialdifference_k)$ as in Lemma 
\ref{lem:choice_of_u_k_and_t_k}
and, 
in order to shorten the notation, we give the following:
\begin{definition}[\textbf{Definite choice of $\badset$}] We set
\begin{equation}\label{eq:D_k_once_for_all}
\badset
:=D_k^{\radialdifference_k}.
\end{equation}
\end{definition}
Let us recall that
\begin{align}\label{inclusion_partialDk}
 \partial \badset
\subseteq\{\diffuku=\radialdifference_k\}.
\end{align}
Also, observe that,
 upon extracting a further  (not relabelled) subsequence, 
we might assume that the  characteristic functions $\chi^{}_{\badset}$ converge 
${\rm weakly}^*$ 
in $L^\infty(\Omega)$ to some $\zeta_n 
\in L^\infty(\Omega; [0,1])$ (the 
sequence $(\badset$) depends on $n$, and so $\zeta_n$ depends on $n$).
Since the limit holds also weakly in $L^1(\Omega)$ we see that 
\begin{align}
\label{eq:estimate_zeta}
 \|\zeta_n\|_{L^1(\Omega)}\leq \liminf_{k\rightarrow+\infty}\|\chi^{}_{\badset}\|_{L^1(\Omega)}\leq \frac1n.
\end{align}
Recalling the definition of $M_{\bar\alpha}^\beta(A)$ in \eqref{eq:subdeterminant}, we prove the following statement.

\begin{lemma}[\textbf{The currents $T_k$ and the limit
current $\mathcal T_n$}]\label{lemma_conv_currents_unif2}
 Let $n\in \mathbb N$ be fixed and let $A_n$ satisfy  
\eqref{eq:measure_of_A_n} and \eqref{eq:u_k_converges_to_u_uniformly_out_of_A_n}. 
 For any $k \in \mathbb N$ 
define the current $T_k\in \mathcal D_2(\Omega\times\R^2)$ as 
 \begin{align*}
  T_k(\omega):=\begin{cases}
                \displaystyle
\int_{\Omega\setminus \badset}
\varphi(x,u_k(x))M^\beta_{\bar\alpha}(\nabla u_k(x))~dx&\text{if }|\beta|\leq 1,\\
                0&\text{if }|\beta|=2,
               \end{cases}
 \end{align*}
where $\omega\in \mathcal D^2(\Om \times \R^2)$ is a $2$-form that writes as
\begin{equation}\label{eq:omega}
\omega(x,y)=
\varphi(x,y)dx^\alpha\wedge dy^\beta, \quad 
\varphi \in C^\infty_c(\Omega\times \R^2),  \quad |\alpha|+|\beta|=2.
\end{equation}
 Then 
$$
\lim_{k \to +\infty}T_k
=\limitcurrent_n\in \mathcal D_2(\Omega\times\R^2)
\qquad {\rm weakly~ in~ the~ sense~ of~ currents},
$$
where
 \begin{align*}
\limitcurrent_n(\omega):=\int_\Omega\varphi(x,\vortexmap(x))M^\beta_{\bar\alpha}
(\nabla \vortexmap(x))(1-\zeta_n(x))~dx
\qquad 
\forall \omega {\rm ~as~in~} 
\eqref{eq:omega}.
 \end{align*}
\end{lemma}

\begin{proof}
Since the Jacobian of $\vortexmap$ vanishes almost everywhere it follows that 
$\limitcurrent_n(\varphi dy^1\wedge dy^2)=0$ for all $\varphi$ as in 
\eqref{eq:omega}. Then for $2$-forms $\omega=\varphi dy^1\wedge dy^2$ 
the convergence
$T_k(\omega)\rightarrow \limitcurrent_n(\omega)$ is achieved.
We are then left to prove that for all $2$-forms $\omega$ with $\omega(x,y)=\varphi(x,y)dx^\alpha\wedge dy^\beta$, $\varphi \in C^\infty_c(\Omega\times \R^2)$, $|\alpha|+|\beta|=2$, and $|\beta|\leq 1$, it holds
\begin{align}\label{claim_lemma2}
 \lim_{k\rightarrow 
+\infty}\int_{\Omega\setminus \badset}\varphi(x,u_k(x))M^\beta_{\bar\alpha}
(\nabla u_k(x))~dx
= 
\int_{\Omega}\varphi(x,\vortexmap(x))M^\beta_{\bar\alpha}(\nabla 
\vortexmap)(1-\zeta_n(x))~dx.
\end{align}
To simplify the argument we treat separately the cases $\omega=\varphi(x,y)dx^1\wedge dx^2$ and $\omega=\varphi(x,y)dx^i\wedge dy^j$ for some $i,j\in\{1,2\}$. In the former case we simply have
\begin{align*}
\int_{\Omega\setminus \badset}\varphi(x,u_k(x))~dx
=
\int_{\Omega}\varphi(x,u_k(x))\chi_{\Omega\setminus \badset}(x)~dx.
\end{align*}
Then, using that $u_k\rightarrow u$ uniformly in $\Omega\setminus \badset$
(see \eqref{eq:u_k_converges_to_u_uniformly_out_of_A_n}, Lemma 
\ref{lem:choice_of_u_k_and_t_k}(ii) and \eqref{eq:D_k_once_for_all})
and 
$\chi_{\Omega\setminus \badset} \to \chi_\Omega-\zeta_n$ 
${\rm weakly}^*$ in $L^\infty(\Omega)$, it follows
\begin{align*}
\lim_{k\rightarrow +\infty}\int_{\Omega\setminus \badset}\varphi(x,u_k(x))~dx
=\int_{\Omega}\varphi(x,\vortexmap(x))(1-\zeta_n(x))~dx=\mathcal T_n(\omega).
\end{align*}
Assume now $\omega=\varphi(x,y)dx^i\wedge dy^j$, $i,j
\in \{1,2\}$, $i \neq j$.
In this case \eqref{claim_lemma2} reads as
\begin{align*}
 \lim_{k\rightarrow +\infty}\Big(
\int_{\Omega\setminus 
\badset}\varphi(x,u_k(x))D_{\bar i}
[( u_k(x))_j]~dx-  \int_{\Omega}\varphi(x,\vortexmap(x))D_{\bar i} \vortexmap_j(x)(1-\zeta_n(x))~dx\Big)=0,
\end{align*}
with $\bar i=\{1,2\}\setminus \{i\}$.
Since 
$\chi_{\badset}^{} \to \zeta_n$ ${\rm weakly}^*$ in $L^\infty(\Omega)$, 
this is equivalent to proving
\begin{align*}
 \lim_{k\rightarrow +\infty}
\Big(\int_{\Omega\setminus \badset}\varphi(x,u_k(x))D_{\bar i}[( u_k(x))_j]~dx
-  \int_{\Omega\setminus \badset}\varphi(x,\vortexmap(x))D_{\bar i} \vortexmap_j(x)~dx\Big)=0.
\end{align*}
The quantity between parentheses on the left-hand side can be written as
\begin{align*}
 \int_{\Omega\setminus \badset}\Big(
\varphi(x,u_k(x))-\varphi(x,\vortexmap(x))\Big)
D_{\bar i}[( u_k(x))_j]~dx+\int_{\Omega\setminus \badset}
\varphi(x,u(x))\Big(D_{\bar i}[( u_k(x))_j]-D_{\bar i} \vortexmap_j(x)\Big)~dx,  
\end{align*}
and we see that the first integral tends to zero 
as $k \to +\infty$, since $u_k\rightarrow u$ uniformly in 
$\Omega\setminus \badset$, $\varphi$ is Lipschitz continuous, 
and the $L^1(\Omega)$-norm
of
$D_{\bar i}[( u_k)_j]$ is uniformly bounded with respect to 
$k$. 
The second integral can be instead 
integrated by parts\footnote{From Lemma \ref{lem:choice_of_u_k_and_t_k}(i),  
$\badset$  has rectifiable boundary; 
moreover, $\varphi(\cdot, u_k(\cdot))$ is Lipschitz.
We can then apply a version of the Gauss-Green 
theorem, see for instance
\cite[pag. 124, exercise 12.12]{Maggi:12}.
},
 obtaining
 \begin{align*}
 &\int_{\Omega\setminus \badset}\varphi(x,u(x))(D_{\bar i}[( u_k(x))_j]-D_{\bar i}u_j(x))~dx\\
   =&\int_{\partial \badset}\varphi(x,u(x))(( u_k(x))_j-
\vortexmap_j(x))\nu_{\bar i}(x)~d\mathcal H^1(x)-
\int_{\Omega\setminus \badset}D_{\bar i}(\varphi(x,u(x)))(( u_k(x))_j-\vortexmap_j(x))~dx
\\
=: & 
~\textrm{I}_k+
\textrm{II}_k.
 \end{align*}
Thanks to the fact that $\varphi$ is bounded 
and that $|( u_k)_j(x)-u_j(x)|\leq \diffuku(x)=\lambda_k$ on $\partial \badset$, we 
conclude by Corollary \ref{cor:t_k_H^1} that 
$\lim_{k \to +
\infty} {\rm I}_k=0$.
Moreover
$$
\begin{aligned}
 {\rm II}_k = 
&-\int_{\Omega\setminus \badset}\partial_{x_{\bar i}}\varphi(x,u(x))(( u_k(x))_j-u_j(x))dx\nonumber\\
 &-\sum_{l=1}^2\int_{\Omega\setminus \badset}\partial_{y_l}\varphi(x,\vortexmap(x))D_{\bar i}u_l(x)(( u_k)_j(x)-u_j(x))dx
=:  {\rm II}_{k,1} +
 {\rm II}_{k,2}.
\end{aligned}
$$
Then 
$\lim_{k \to +
\infty} {\rm II}_{k,1}=
\lim_{k \to +
\infty} {\rm II}_{k,2}=0$,
since
the partial derivatives of $\varphi$ are bounded,
$D_{\bar i} \vortexmap \in L^1(\Omega\setminus \badset, \R^2)$, 
$|( u_k)_j-\vortexmap_j|\leq \diffuku\leq \lambda_k$ on
$\Om \setminus \badset$, and $\lim_{k \to +\infty} 
\lambda_k =0$. 
\end{proof}

\begin{Remark}\rm
The mass of the current $T_k$ is given by 
\begin{align}\label{eq:mass_of_Tk}
 |T_k|=\int_{\Omega\setminus \badset}\sqrt{1+|\nabla u_k|^2}dx.
\end{align}
To see \eqref{eq:mass_of_Tk} we choose a $2$-form 
$\omega \in \mathcal D^2(\Om \times \R^2)$ as
$$\omega:=\sum_{|\alpha|+|\beta|=2}\varphi_{\bar\alpha\beta}dx^\alpha\wedge dy^\beta,\qquad \Vert\omega\Vert\leq 1,$$
set\footnote{Here $\alpha$ and $\beta$ run over all the multi-indeces in $\{1,2\}$ with the constraint $|\alpha|+|\beta|=2$.} $\widehat \omega(x,y)=:(\varphi_{\bar\alpha\beta}(x,y))\in \R^6,$ and 
\begin{align*}
\widetilde{\mathcal M}(\nabla u_k(x)):=(1,D_1[(u_k(x))_1],D_2[(u_k(x))_1],D_1[(u_k(x))_2],D_2[(u_k(x))_2],0)\in \R^6 = \R \times \R^4 \times \R,
\end{align*}
so that 
\begin{equation}\label{eq:trivial_inequality}
\begin{aligned}
 T_k(\omega)&
=\int_{\Omega\setminus \badset}\langle\widehat \omega(x,u_k(x)),\widetilde{\mathcal M}(\nabla u_k(x))\rangle dx
\\
 &\leq \Vert\widehat \omega\Vert\int_{\Omega\setminus \badset}|\widetilde{\mathcal M}(\nabla u_k(x))|dx\leq \int_{\Omega\setminus \badset}\sqrt{1+|\nabla u_k|^2}dx.
\end{aligned}
\end{equation}
To prove the converse inequality, 
choosing $\widehat \omega(x,y)=\frac{\widetilde{\mathcal M}(\nabla u_k(x))}{|\widetilde{\mathcal M}(\nabla u_k(x))|}$ 
would give the equality in \eqref{eq:trivial_inequality}. 
However, 
$\frac{\widetilde{\mathcal M}(\nabla u_k)}{|\widetilde{\mathcal M}(\nabla u_k)|}$ is not necessarily of class $C^\infty_c$, so 
we need to use the density of $C^\infty_c(\Omega\times \R^2)$ in 
$L^1(\Omega\times \R^2)$ 
(here we use that ${\widetilde{\mathcal M}(\nabla u_k)}\in 
L^\infty(\Omega, \R^6)$ since $u_k $ is Lipschitz continuous).

With a similar argument, setting 
$$\widetilde{\mathcal M}(\nabla u(x)):=(1-\zeta_n(x)){\mathcal M}(\nabla u(x))\in \R^6, \quad x \in \Omega 
\setminus \overline \sourcedisk_\eps
$$
we can show that the total mass of
$\limitcurrent_n$ in $(\Omega\setminus \overline \sourcedisk_\eps)
\times\R^2$ is given by
\begin{align}\label{eq:mass_of_T}
 |\limitcurrent_n|_{(\Omega \setminus \overline \sourcedisk_\eps)\times\R^2}
=\int_{\Omega \setminus \overline \sourcedisk_\eps}|\mathcal M(\nabla u)||1-\zeta_n|~dx.
\end{align}
\end{Remark}
%
\subsection{Estimate of the mass of 
$\currgraphk$ over $\Omega \setminus \badset$}
\label{sec:estimate_over_D_k^c}
We denote by $\Phi_k = \Phi_{u_k} = 
{\rm Id} \bowtie u_k 
:\Omega\rightarrow \Omega\times \R^2$ the map
\begin{equation}\label{eq:Phi_k}
\begin{aligned}
 \Phi_k(x)=(x,u_k(x)),
\end{aligned}
\end{equation}
in such a way that $\Phi_k(\Omega)=G_{u_k}$, with 
$G_{u_k}=\{(x,y)\in\Omega\times\R^2:y=u_k(x)\}$
the graph of $u_k$. 

We denote as usual by 
\begin{equation}\label{eq:G_u_k}
\currgraphk
\in \mathcal D_2(\Omega\times\R^2)
\end{equation}
the integral current 
supported by the graph of $u_k$.

We now want to estimate the area of the graph of $u_k$ over 
the set $(\Om  \setminus \overline \sourcedisk_\eps)\setminus \badset$. 
\begin{prop}
Let $\eps \in (0,\longR)$ satisfy \eqref{eq:H1}  and \eqref{eq:H2},
$n \in \mathbb N$,  $(\lambda_k)$ be as in Lemma 
\ref{lem:choice_of_u_k_and_t_k}, and let $D_k$ be as in 
\eqref{eq:D_k_once_for_all}. Then
\begin{align}\label{estimate_Dc}
 \liminf_{k\rightarrow 
+\infty}\int_{\Omega \setminus \badset}|\mathcal M(\nabla u_k)|~dx
\geq \int_{
\Omega \setminus \overline \sourcedisk_\eps
}|\mathcal M(\nabla u)|~dx-\frac{1}{n}-\frac{2}{\eps n}.
\end{align}
\end{prop}
\begin{proof}
Set
$\Omega_\eps:=\Omega \setminus \overline \sourcedisk_\eps$.
Since, by definition, $T_k$ vanishes on smooth $2$-forms
supported in $(\badset \cap \Omega_\eps)\times\R^2$, 
we employ 
\eqref{eq:mass_of_Tk} 
to obtain
\begin{equation}
\label{eq:liminf_T_k_on_Omega_eps}
\begin{aligned}
&  
\liminf_{k\rightarrow 
+\infty}
\int_{\Omega \setminus \badset}|\mathcal M(\nabla u_k)|~dx
\geq 
& \liminf_{k\rightarrow 
+\infty}
\int_{\Omega \setminus \badset}\sqrt{1 + \vert
\nabla u_k|^2}~dx \geq
\liminf_{k\rightarrow +\infty}|T_k|_{(\Omega_\eps{\setminus \badset})\times\R^2}
\\
= &
\liminf_{k\rightarrow +\infty}|T_k|_{\Omega_\eps\times\R^2} \geq 
 |\limitcurrent_n|_{\Omega_\eps\times\R^2},
\end{aligned}
\end{equation}
where we use that $(T_k)$ weakly converges to $\limitcurrent_n$
(Lemma \ref{lemma_conv_currents_unif2}), and the weak lower semicontinuity of the mass.
In turn, from \eqref{eq:mass_of_T} and \eqref{eq:estimate_zeta},
\begin{align}
\label{eq:T_n_Omega_eps_seconda}
|\limitcurrent_n|_{\Omega_\eps\times\R^2}
= &\int_{\Omega_\eps}|\mathcal M(\nabla u)||1-\zeta_n|~dx\geq \int_{\Omega_\eps}|\mathcal M(\nabla u)|~dx-\int_{\Omega_\eps}|\mathcal M(\nabla u)||\zeta_n| ~dx
\nonumber
\\
\geq & 
\int_{\Omega_\eps}|\mathcal M(\nabla u)|~dx
-
\|\mathcal M(\nabla u)\|_{L^\infty(\Omega_\eps)}
\|\zeta_n\|_{L^1(\Omega_\eps)}
\\
\geq &\nonumber
\int_{\Omega_\eps}|\mathcal M(\nabla u)|~dx
-
\frac{1}{n} 
\|\mathcal M(\nabla u)\|_{L^\infty(\Omega_\eps)}.
\end{align}
Next, using
$\sqrt{1+z^2}\leq 1+|z|$ and\footnote{
$D_i u_j(x) 
= \frac{\delta_{ij}}{\vert x\vert} - \frac{x_i x_j}{\vert x\vert^3}$,
hence $\sum_{ij}(D_i u_j(x))^2 = \frac{2}{\vert x\vert^2} + 2 \frac{x_1^2 x_2^2}{\vert 
x\vert^6} \leq \frac{4}{\vert x\vert^2}$.} 
$|\nabla u(x)|\leq \frac{2}{|x|}$ 
which, on $\Omega_\eps$,
is bounded by $2/\eps$, we also get
$$
\|\mathcal M(\nabla u)\|_{L^\infty(\Omega_\eps)} 
= \|\sqrt{1+\vert\nabla u\vert^2}\|_{L^\infty(\Omega_\eps)} 
\leq
1 + \frac{2}{\eps}.
$$
We deduce 
$$
\vert 
\limitcurrent_n \vert_{\Omega_\eps\times\R^2}
\geq 
\int_{\Omega_\eps}|\mathcal M(\nabla u)|~dx
-
\frac{1}{n} - \frac{2}{\eps n}.
$$
From 
\eqref{eq:T_n_Omega_eps_seconda} and 
\eqref{eq:liminf_T_k_on_Omega_eps}
 inequality \eqref{estimate_Dc} follows.
\end{proof}
%

\section{
The maps $\Psi_k$, $\projlambdak$, and the currents $\piPsiDk$,
$\piPsiDkXik$, $\mathcal E_k$}
\label{sec:the_maps}
Recalling that 
$\badset$ is
defined in \eqref{eq:D_k_once_for_all} and \eqref{eq:D_k_delta},
in  Section 
\ref{sec:estimate_over_D_k^c}
we have estimated the area of the graph of $u_k$ over 
$\Omega\setminus \badset$. The next step, which is considerably more
difficult, 
is to estimate this area over $\badset$, 
and this will be splitted in several parts (Sections 
\ref{subsec:preliminary_estimates}-\ref{subsec:gluing}). 
After introducing some preliminaries in 
Section \ref{subsec:the_set_ecc},
the first step is to reduce the graph of $u_k$ 
(a surface of codimension $2$ in $\R^4$) 
to a suitable rectifiable set ($\Psi_k(\badset)$ and their projections) 
of codimension $1$ sitting 
in $\overline C_\longR\subset \R^3$. 
In this section we introduce all various objects needed to prove 
the lower bound.

\begin{definition}[\textbf{The map $\Psi_k$}]
\label{def:the_map_Psikk}
For all $k\in \mathbb N$, 
we define the map
$\Psi_k = \Psi_{u_k}:\Omega\rightarrow\R^3=\R_{\vert x\vert}\times \R^2_{{\rm target}}$ 
as
\begin{align}
\Psi_k(x):= &
(|x|,u_k(x))
\qquad \forall x \in \Omega.
\label{eq:Psi_k}
\end{align}
\end{definition}
Notice that 
$\Psi_k$ takes values in 
$\overline{C_\longR}$, and is Lipschitz continuous.
Moreover $\Psi_k=R\circ \Phi_k$, 
where $\Phi_k = {\rm Id} \bowtie u_k 
:\Omega\rightarrow\R^4$ is defined 
in \eqref{eq:Phi_k}, 
and $R:\R^4\ni (x,y)\mapsto (|x|,y)\in \R^3$. 
By the area formula and since ${\rm Lip}(R)=1$ 
we have
\begin{align*}
\int_B(\grad \Psi_k^T \grad \Psi_k)^{\frac12}~dx\leq \int_B(\grad
\Phi_k^T \grad \Phi_k)^{\frac12}~
dx =\int_B|\mathcal M(\nabla u_k)|~dx=
|\mathcal G_{u_k}|_{B\times \R^2},
\end{align*}
for any Borel set $B \subseteq \Om$.

\subsection{The sets $\Psi_k(\badset)$ and the currents
$(\Psi_k)_\sharp\jump{\badset}$
}\label{subsec:the_set_ecc}
We start noticing that 
\begin{equation}
\label{eq:Psi_k_in_R_3_minus_C_l_complement}
 \Psi_k(\Omega \setminus \badset)\subset \overline C_\longR\setminus 
C_\longR(1-\radialdifference_k),
\qquad 
k \in \mathbb N,
\end{equation}
where we recall that 
$C_\longR(1-\radialdifference_k)$ is defined in \eqref{eq:portion_of_cylinder}.
Indeed, since $\Omega \setminus \badset
\subseteq\{\diffuku\leq \radialdifference_k\}$ for any $k \in \mathbb N$ we have
\begin{equation}\label{eq:oss_utile_complement}
\radialdifference_k\geq |u_k(x)-\frac{x}{|x|}|\geq \textrm{dist}(u_k(x),\mathbb S^1) = 
1- \vert u_k(x)\vert,
\qquad x \in 
\Omega \setminus \badset,
\end{equation}
so that $|u_k(x)|\geq 1-\radialdifference_k$.
In particular
\begin{equation}
\label{eq:Psi_k_in_R_3_minus_C_l}
  \Psi_k(\partial \badset)\subset \overline C_\longR\setminus 
C_\longR(1-\radialdifference_k), 
\qquad 
k \in \mathbb N.
\end{equation}
As a consequence, since the map $\Psi_k$ is Lipschitz continuous, we have:
\begin{cor}
For all $k \in \mathbb N$
the integral $2$-current 
$(\Psi_k)_\sharp\jump{\badset}$
is boundaryless in $C_\longR(1-\lambda_k)$.
\end{cor}
Observe that $\Psi_k(\badset)$ is rectifiable  
 and contains\footnote{It could be
different because of 
possible cancellations.}
the support of $(\Psi_k)_\sharp\jump{\badset}$;
also $\Psi_k(\badset)$ is contained in 
$[0,\longR)\times \overline{B}_{1}$. Specifically, the fact that  $C_\longR$ has axial 
coordinate in $(-1,l)$ and not in $(0,\longR)$ 
will be convenient in order to control the behaviour of $(\Psi_k)_\sharp\jump{\badset}$ on $\{0\}\times \R^2$.

\begin{definition}[\textbf{The projection $\projlambdak$}]
\label{def:the_projection_pi_k}
We let
\begin{equation}\label{eq:pi_k}
\projlambdak = \pi_{\lambda_k}:\R^3\rightarrow\overline{C}_\longR(1-\lambda_k)
\end{equation}
be the
orthogonal projection onto the compact convex set
 $\overline{C}_\longR(1-\lambda_k)$.
\end{definition} 
In Section \ref{subsec:construction_of_S_k} we project
$\Psi_k(\badset)$ on $\overline{C}_\longR(1-\lambda_k)$ in order to get a 
rectifiable set (and its associated current)
whose area (counted with multiplicity) is less than 
or equal to the area of the original set; the area
of the projected set, in turn, gives a lower bound for 
the mass of  $\currgraphk$ over $\badset$ 
(see formulas \eqref{eq:does_not_increase_the_area} and \eqref{eq:spa}). 
 Then, as a second step, 
 we symmetrize $\projlambdak\circ\Psi_k(\badset)$ using the cylindrical  
rearrangement introduced in Section \ref{sec:cylindrical_Steiner_symmetrization} to get a still smaller (in area) object. The
estimate of the area of the symmetrized object is divided in two parts: 
the first one (Section \ref{subsec:est_Omegaeps_Dk}) 
deals with  
$\projlambdak\circ\Psi_k(\badset\cap (\Omega \setminus \sourcedisk_\eps))$ whose symmetrized 
set can be seen as the generalized graph of a 
suitable polar function.
In Section \ref{subsec:symmetrization_of_the_image_of_D_k_cap_B_eps} 
we deal with the second part,
where we  estimate the area of the symmetrization obtained from $\projlambdak\circ\Psi_k(\badset\cap \sourcedisk_\eps)$. 
In Sections \ref{subsec:gluing} and \ref{sec:lower_bound}, 
we collect our estimates 
and we utilize the symmetrized object
as a competitor for a suitable non-parametric Plateau problem. To do this we 
need to glue to the obtained rectifiable set some artificial surfaces,
 whose areas are controlled and are infinitesimal 
in the limit as $k\rightarrow +\infty$. This limit is taken 
only at the end of Section \ref{sec:lower_bound}, allowing us to analyse a non-parametric Plateau problem 
whose boundary condition does not depend on $k$, so that also its solution 
does not depend on $k$. 
The area of such a solution will be the lower bound for the area of 
the rectifiable set 
$\projlambdak\circ\Psi_k(\badset\cap (\Omega \setminus \overline \sourcedisk_\eps))\cup \projlambdak\circ\Psi_k(\badset\cap \sourcedisk_\eps)$, 
and then finally for the area of the graph of $u_k$ on $\badset$.

\subsection{Construction of the current $\piPsiDkXik$
via the currents $\mathcal D_k$ and $\mathcal W_k$}
\label{subsec:construction_of_S_k}
We are interested in the part of the set  $\Psi_k(\badset)$ included 
in $\overline{C}_\longR(1-\lambda_k)$; 
we need an explicit description of the boundary of $\Psi_k(\badset)$, 
and to this aim
we compose $\Psi_k$ with the projection $\projlambdak$ in \eqref{eq:pi_k}. 

\begin{definition}[\textbf{Projection of $\Psi_k(\badset)$: 
the current
$\piPsiDk$}]
We define 
the current $\piPsiDk 
{\in \mathcal D_2(C_\longR)}$ 
as
\begin{equation}
\label{eq:S_hat_k}
\piPsiDk
:=
(\projlambdak\circ \Psi_k)_\sharp\jump{\badset}.
\end{equation}
\end{definition}
\begin{Remark}\rm 
In general 
$\Psi_k(\badset) \subseteq
\overline {C_\longR}(1-\lambda_k) \cup (\overline {C_\longR} \setminus
C_\longR(1-\lambda_k))$, 
while ${\rm spt}(\piPsiDk) \subseteq
\overline{C_\longR}(1-\lambda_k)$.
\end{Remark}

Since ${\rm Lip}(\projlambdak)=1$,
the map $\projlambdak$  does not increase the area, and 
therefore 
\begin{equation}\label{eq:does_not_increase_the_area}
 \int_{\badset}
\vert 
J(\projlambdak\circ\Psi_k)\vert 
~dx
\leq\int_{\badset}
\vert 
J(\Psi_k)\vert 
~dx\leq |\currgraphk|_{\badset\times \R^2},
\end{equation}
\begin{equation}\label{eq:masses_localized}
 \int_{\badset\cap (\Om \setminus \overline \sourcedisk_\eps)}
\vert 
J(\projlambdak\circ\Psi_k)\vert 
~dx\leq\int_{\badset\cap (\Om \setminus \overline \sourcedisk_\eps)}
\vert 
J(\Psi_k)\vert 
~dx\leq |\currgraphk|_{(\badset\cap (\Om \setminus \overline \sourcedisk_\eps))\times \R^2},
\end{equation}
The same holds for the mass of the current
$\piPsiDk$, 
{\it i.e.}, 
$$
\vert 
 \piPsiDk \vert \leq 
\vert 
(\Psi_k)_\sharp\jump{\badset}
\vert,
$$
and recalling also \eqref{eq:C_l_eps_1}, 
\begin{equation}\label{eq:masses}
|\piPsiDk|_{{\overline C}^\eps_\longR}\leq 
|(\Psi_k)_\sharp\jump{\badset}|_{ {\overline C}^\eps_\longR }\leq 
|\currgraphk|_{(\badset\cap (\Om \setminus \overline \sourcedisk_\eps))\times \R^2}.
\end{equation}
\begin{remark}\label{rem:hat_Z_k}
The area, counted with multiplicity, 
 of the $2$-rectifiable set  
$\projlambdak\circ \Psi_k(\badset)$ is
greater than or equal to  the mass of the current $\piPsiDk$, 
more specifically
\begin{equation}\label{eq:more_spec}
 \int_{\badset}
\vert 
J(\projlambdak\circ\Psi_k)
\vert ~
dx\geq |\piPsiDk|_{\overline{C}_\longR(1-\lambda_k)}
\qquad \text{ and }\qquad \int_{\badset\cap 
(\Om \setminus \overline \sourcedisk_\eps)
}
\vert 
J(\projlambdak\circ\Psi_k)\vert 
~dx\geq |\piPsiDk|_{\overline{C}^\eps_\longR(1-\lambda_k)}.
\end{equation}
This is due to the fact that 
$\projlambdak\circ \Psi_k(\badset)$ might overlap with opposite orientations 
so that the multiplicity of $\piPsiDk$ vanishes, 
and the overlappings do not contribute to its mass.
In particular, ${\rm spt}(\piPsiDk) \subseteq
 \projlambdak\circ \Psi_k(\badset)$. 
\end{remark}

From \eqref{eq:does_not_increase_the_area}
 and \eqref{eq:more_spec}
it follows
\begin{equation}\label{eq:spa}
\vert \jump{G_{u_k}}
\vert_{\badset \times \R^2}
\geq 
\vert 
\piPsiDk
\vert_{\overline C_\longR(1-\lambda_k)},
\qquad
\vert \jump{G_{u_k}}
\vert_{(\badset \cap (\Omega \setminus \overline
\sourcedisk_\eps))\times \R^2}
\geq 
\vert 
\piPsiDk
\vert_{\overline C_\longR^\eps(1-\lambda_k)}.
 \end{equation}

We now analyse the boundary of $\piPsiDk$. 
Up to small modifications, we will prove that it is boundaryless
in $C_\longR(1-\lambda_k')$ 
(see \eqref{eq:mass_of_Xik_2} and \eqref{eq:S_k_is_boundaryless_in_C_l}, where $\lambda_k'$ are suitable small numbers in $(0,\lambda_k)$ chosen below in Definition \ref{def_lambdak'})
and so $\piPsiDk$ can be symmetrized 
 according to Definition \ref{def:symmetrization_of_the_boundary_of_a_three_current}. Before proceeding to the symmetrization 
we need some preliminaries.
We build suitable currents 
$\Xik$,
with their support sets denoted by $W_k$
(see \eqref{eq:def_Xi_k} and \eqref{eq:W_k}), with 
$\partial\Xik$ 
coinciding with  $\partial \piPsiDk$ (see \eqref{partial_Xi},
\eqref{eq:S_k}, and \eqref{eq:S_k_is_boundaryless_in_C_l}).

\begin{Remark}\rm 
By \eqref{eq:Psi_k_in_R_3_minus_C_l},
$\projlambdak\circ\Psi_k(\partial \badset)$ is contained in $\partial_{\rm lat} C_\longR(1-\lambda_k)$. 
By Lemma \ref{lem:choice_of_u_k_and_t_k}(i),
$\projlambdak\circ\Psi_k(\partial \badset)$ is the union of the image of 
at most countably many curves, 
and this union, counted with multiplicities, has 
finite $\mathcal H^1$ 
measure: specifically, if we 
define
$$
M(\projlambdak\circ\Psi_k(\partial \badset)):=\int_{\partial \badset}
\Big|
\partial_{\rm tg} ~(\projlambdak\circ\Psi_k)\Big|~d\mathcal H^1,
$$
where
 $\partial_{\rm tg}$ 
stands for the tangential derivative
 along 
$\partial \badset$,  then $M(\projlambdak\circ\Psi_k(\partial \badset)) < +\infty$
 since $\mathcal H^1(\partial \badset)<+\infty$ (still by
Lemma \ref{lem:choice_of_u_k_and_t_k}(i))
and $u_k$ is Lipschitz continuous.

Moreover
\begin{equation}
\label{eq:boundary_hat_S_k_equals_Gamma_k}
\partial 
\piPsiDk = 
(\projlambdak\circ\Psi_k)_\sharp\partial\jump{\badset} {\in \mathcal D_1(C_\longR)}
\qquad {\rm in}~ C_\longR.
\end{equation}
\end{Remark}
It is convenient to introduce a suitable map
$\tau$  parametrizing the region 
$\overline C_\longR \setminus C_\longR(1-\lambda_k)$
in between
the two concentric cylinders; this map can 
then be pulled back by $\projlambdak\circ \Psi_k$, but only in $\Omega \setminus \badset$,
to get the map $\widetilde \tau$.

\begin{definition}[\textbf{The maps $\tau$, $\widetilde \tau$}]\label{def:the_map_tau}
We set
\begin{equation}
\label{eq:tau}
\begin{aligned}
& \tau = \tau_{\lambda_k}:
[1-\lambda_k,1]
\times \partial C_\longR(1-\lambda_k)\rightarrow 
\overline C_\longR
 \setminus C_\longR(1-\lambda_k)
\subset\R^3,
\\
& \tau(\rho,t,y):=\Big(t,\frac{y}{|y|} \rho \Big),
\qquad
\ \ \rho \in [1-\lambda_k,1], \ 
\ \ (t,y)\in \partial C_\longR(1-\lambda_k)=[-1,l]\times \partial B_{1-\lambda_k}. 
\end{aligned}
\end{equation}
By \eqref{eq:Psi_k_in_R_3_minus_C_l_complement} it follows
$\projlambdak \circ \Psi_k(\Omega \setminus \badset) \subset \partial C_\longR(1-\lambda_k)$,
hence we can also set
\begin{equation}
\label{eq:widetilde_tau}
\widetilde \tau(\rho,x) = 
\widetilde \tau_{u_k,\lambda_k}(\rho,x)
:=\tau(\rho,\projlambdak\circ\Psi_k(x)), \qquad \rho
\in [1-\lambda_k,1],\;x\in \Omega\setminus \badset.
\end{equation}
\end{definition}

Notice that $\tau(\rho,\cdot,\cdot)$ 
takes values in $\partial C_\longR(\rho)$ for any $\rho\in [1-\lambda_k,1]$,
that $\tau(\cdot,t,y)$ moves along the normal to the lateral boundary of
$\partial C_\longR(1-\lambda_k)$  at the point $(t,y)$,
and  $\tau(1-\lambda_k, \cdot,\cdot)$ is the identity. We also 
observe that, due to the fact that $\projlambdak\circ\Psi_k$ 
takes values in $[0,\longR)\times \overline 
B_1$, the same holds for $\widetilde \tau$.
\begin{remark}\rm
 If $\lambda_k>0$ is small enough 
(which is true for $k$ large enough), 
the Jacobian of $\tau$ 
is close to $1$ 
so that the $\mathcal H^1$-measure, counted
with multiplicities,
of the set
 $\tau(\rho,\projlambdak\circ\Psi_k(\partial \badset))$ is, for fixed $\rho$, 
bounded by two times 
the $\mathcal H^1$-measure 
of $\projlambdak\circ\Psi_k(\partial \badset)$, 
still counted with multiplicities. 
More precisely, 
\begin{equation}
\label{eq:bound_by_two}
\begin{aligned}
 2\int_{\partial \badset}\Big|\partial_{\rm tg}
~(\projlambdak\circ\Psi_k)\Big|~
d\mathcal H^1
\geq
\int_{\partial \badset}\Big|\partial_{\rm tg}
~\tau(\rho,\projlambdak\circ\Psi_k)\Big|~d\mathcal H^1,
\qquad
\rho\in [1-\lambda_k,1],
\end{aligned}
\end{equation}
\begin{align*}
 2\int_{(\Omega \setminus \overline \sourcedisk_\eps)  \cap
\partial \badset
}\Big|\partial_{\rm tg}~(\projlambdak\circ\Psi_k)\Big|~
~d\mathcal H^1\geq\int_{
(\Omega \setminus \overline \sourcedisk_\eps)  \cap
\partial \badset
}\Big|
\partial_{\rm tg}~\tau(\rho,\projlambdak\circ\Psi_k)\Big|~d\mathcal H^1,
\end{align*}
for all $\rho\in [1-\lambda_k,1]$ and $k \in \mathbb N$ large enough,
where we recall that, from Lemma \ref{lem:choice_of_u_k_and_t_k}(i), 
$\partial \badset$ is rectifiable.
\end{remark}

Now we take a sequence\footnote{The sequence 
$(\lambda_k')$ depends on $\eps$ and $n$.} of numbers 
$\lambda'_k\in(0,\lambda_k)$, which will be fixed in the sequel
(see Definition \ref{def:the_current_S_k}).
\begin{definition}[\textbf{The set $W_k$ and the current 
$\Xik$}]\label{def:W_k}
We define the 
$2$-rectifiable set\footnote{The set $W_k$ 
consists of ``vertical'' walls, normal to $\partial 
C_\longR(1-\lambda_k)$, build on $\projlambdak\circ\Psi_k(\partial \badset)$, with height $\lambda_k' -
\lambda_k$: see Fig. \ref{fig:tricilindro}.}  
\begin{equation}\label{eq:W_k}
W_k:=\tau\big([1-\lambda_k,1-\lambda'_k]\times \projlambdak\circ\Psi_k(\partial \badset)\big)
=\widetilde \tau\big(
[1-\lambda_k,1-\lambda'_k]\times \partial \badset\big),
\end{equation}
and the $2$-current 
\begin{align}
\label{eq:def_Xi_k} 
\Xik:=\widetilde \tau_\sharp\jump{[1-\lambda_k,1-\lambda'_k]
\times\partial \badset} \in {\mathcal D}_2(C_\longR).
\end{align}
\end{definition}
Clearly ${\rm spt}(\Xik) \subseteq W_k$; 
Again, although $\Xik$ is defined as a current in $C_\longR$, it is supported in $[0,\longR]\times B_1$.
\begin{remark}[\textbf{Use of $\jump{\cdot}$ for not 
top-dimensional currents}]\label{remark:orientationofboundary}
$\partial \badset$ is endowed with a natural orientation, inherited 
from the fact that it is the boundary of the set $\badset$; 
consistently, we sometimes use the identification
$\jump{\partial \badset}=\partial \jump{\badset}$.
With a little abuse of notation we have noted the current
integration over $[1-\lambda_k,1-\lambda'_k]
\times\partial \badset$, meaning that $\partial \badset$ is endowed with 
this natural orientation.
Finally, recalling that $\tilde \tau(\rho,\cdot )$ takes 
values in $\partial C_\longR(\rho)$, we can do the following identification:
\begin{align*}
 \widetilde \tau_\sharp\jump{[1-\lambda_k,1-\lambda'_k]
\times\partial \badset} = \widetilde \tau_\sharp\partial\jump{[1-\lambda_k,1-\lambda'_k]
\times \badset}\res\Big(C_\longR(1-\lambda_k')\setminus \overline{C}_\longR(1-\lambda_k)\Big).  
\end{align*}

\end{remark}
We denote 
\begin{equation}
\label{eq:area_of_W_k_counted_with_multiplicities}
M(W_k):=\int_{[1-\lambda_k,1-\lambda'_k]\times\partial \badset}
\vert 
J(\widetilde \tau(\rho,x))\vert 
~d\rho~ d{\mathcal H^1}(x)
\end{equation}
the area of $W_k$ counted
with multiplicities.
By the area formula and 
using \eqref{eq:bound_by_two} we infer
\begin{align}\label{eq:mass_of_Xik_1}
|\Xik|\leq M(W_k) \leq 
2(\lambda_k-\lambda_k')
\int_{\partial \badset}\Big|\partial_{\rm tg}
~(\projlambdak\circ\Psi_k)\Big|~d\mathcal H^1
=2(\lambda_k-\lambda_k')M(\projlambdak\circ\Psi_k(\partial \badset)).
\end{align}
 Then we are led to the following
\begin{definition}[\textbf{The sequence $(\lambda_k')$}] \label{def_lambdak'}
We select $\lambda_k'\in (0,\lambda_k)$ so that
\begin{align}
\label{eq:mass_of_Xik_2}
2(\lambda_k-\lambda_k')M(\projlambdak\circ\Psi_k(\partial \badset))\leq \frac1n 
\qquad \forall k \in \mathbb N.
\end{align}
\end{definition}
Finally we observe that 
\begin{align}\label{partial_Xi}
 \partial \Xik
=\tau(1-\lambda_k',\cdot,\cdot)_\sharp 
\Big((\projlambdak\circ\Psi_k)_\sharp\partial\jump{\badset}\Big) 
-(\projlambdak\circ\Psi_k)_\sharp\partial\jump{\badset}.
\end{align}
\begin{definition}[\textbf{The current $\piPsiDkXik$}]
\label{def:the_current_S_k}
We define
\begin{align}
\label{eq:S_k}
\piPsiDkXik
:=
\piPsiDk+ \Xik {\in \mathcal D_2(C_\longR)}.
 \end{align}
 \end{definition}
The next result will be useful to select a primitive of 
$\piPsiDkXik$.

\begin{cor}\label{boundaryat0_first}
The current 
$\piPsiDkXik$ 
is supported 
in $[0,\longR]\times \overline{B}_{1-\lambda_k'}$ 
and 
\begin{equation}
\label{eq:S_k_is_boundaryless_in_C_l}
\piPsiDkXik
\quad \textrm{is~ boundaryless~ in~the~open~cylinder~}
C_\longR(1-\lambda_k').
\end{equation}
In particular
 $\partial 
\piPsiDkXik
=0$ in 
$\mathcal D_1((-\infty,l)\times B_{1-\lambda_k'})$.
\end{cor}

\begin{proof}
The statement follows by construction, and 
noticing that, 
since $\tau(1-\lambda_k',\cdot,\cdot)_\sharp \Big((\projlambdak\circ\Psi_k)_\sharp\partial\jump{\badset}\Big)$ 
has support in $\partial_{\rm lat} C_\longR(1-\lambda_k')$, 
one can use
\eqref{eq:boundary_hat_S_k_equals_Gamma_k} to deduce 
\eqref{eq:S_k_is_boundaryless_in_C_l}.
\end{proof}

\subsection{
The $3$-current $\mathcal E_k$
and the symmetrization of $\piPsiDkXik$} 
\label{subsec:symmetrization_of_S_k}
Since we want to symmetrize 
$\piPsiDkXik$
 according to Definition \ref{def:symmetrization_of_the_boundary_of_a_three_current}, we need to identify a unique 
primitive $3$-current $\mathcal E_k$ such that $\partial \mathcal E_k=\piPsiDkXik$. 

The restriction of the map $\projlambdak\circ \Psi_k$
to 
$\Omega\setminus \badset$ 
takes 
$\Omega\setminus \badset$ into
$\partial C_\longR(1-\lambda_k)$ (see \eqref{eq:Psi_k_in_R_3_minus_C_l_complement}), 
and  can also be written as
\begin{equation}\label{eq:P_k}
 \projlambdak\circ \Psi_k(x)=\Big(|x|,\frac{u_k(x)}{|u_k(x)|}(1-\lambda_k)\Big),
\qquad x \in \Omega \setminus \badset.
\end{equation}
The current $(\projlambdak\circ \Psi_k)_\sharp\jump{\Omega\setminus \badset}$ has boundary 
\begin{align}\label{eq:273}
 \partial(\projlambdak\circ \Psi_k)_\sharp\jump{\Omega\setminus \badset}=-(\projlambdak\circ \Psi_k)_\sharp\partial\jump{\badset}.
\end{align}
\begin{definition}[\textbf{The currents $\mathcal Y_k$ and $\mathcal X_k$}]
\label{def:the_currents_mathcal_Y_k_and_mathcal_x_k}
Recalling the definition of $\tau$ (see \eqref{eq:widetilde_tau}, \eqref{eq:tau}) we set
\begin{align}
\label{eq:defY_k}
\mathcal Y_k:=& \widetilde \tau_\sharp\jump{[1-\lambda_k,1-\lambda_k']\times (\Omega\setminus \badset)} 
\in 
\mathcal D_3(C_\longR),
\\
 \mathcal X_k:=& \jump{C_\longR(1-\lambda_k')\setminus \overline{C_\longR}
(1-\lambda_k)}-\mathcal Y_k
\in {\mathcal D}_3(C_\longR).
\label{eq:defX_k}
\end{align}
\end{definition}
Notice that $\mathcal X_k$ cannot be directly defined as a push-forward
via the map $\widetilde \tau$, for 
part of $\Psi_k(\badset)$ could be contained in $C_\longR(1-\lambda_k)$, 
and for this reason 
we are led to define it as a difference. 

The current $\mathcal Y_k$ could have multiplicity 
different from $0$ and $1$, and in particular could 
not be the integration over a finite perimeter set. 
This depends on the fact that the map $\Psi_k$ could generate 
overlappings and 
self-intersections of the set $\projlambdak\circ \Psi_k(\Omega\setminus \badset)$. 
If the multiplicity of $\mathcal Y_k$ is only $1$ or $0$ then the same holds
 for $\mathcal X_k$.
Also, $\mathcal Y_k$ might be null, and in this case $\mathcal X_k$ 
coincides with the integration over the region 
$C_\longR(1-\lambda_k')\setminus \overline C_\longR(1-\lambda_k)$.
A finer description of these two currents will be necessary later,  
and this will be done by a slicing argument in Lemma \ref{lemma_theta} below.

Recalling \eqref{eq:def_Xi_k}, 
\begin{align*}
 \partial \mathcal Y_k=- \Xik
=
-\partial \mathcal X_k \qquad \text{in } \ C_\longR(1-\lambda_k')\setminus \overline{C_\longR}(1-\lambda_k),
\end{align*}
as it can be seen by considering the push-forward by $\tau$ of \eqref{eq:273}.
We proceed to the symmetrization  in $C_\longR(1-\lambda_k')$
of the current 
$\piPsiDkXik$ in \eqref{eq:S_k}.
By \eqref{eq:S_k_is_boundaryless_in_C_l}
it follows the existence of an integer multiplicity $3$-current $\mathcal E_k\in \mathcal D_3(C_\longR(1-\lambda_k'))$ such that 
\begin{align}\label{eq:primitive}
 \partial \mathcal E_k=\piPsiDkXik \qquad {\rm in}~  \ C_\longR(1-\lambda_k').
\end{align}
The current $\mathcal E_k$ is unique up to a constant, that we might assume to be integer, since $\mathcal E_k$ has integer multiplicity. 
Hence we choose such a constant\footnote{
The fact that this choice is possible is a 
consequence of the 
constancy theorem (see for instance 
\cite[Proposition 7.3.1]{Krantz_Parks:08}). 
Indeed, let $\widehat {\mathcal E}_k$ have the same boundary 
({\it i.e.}, $\Xik$) 
of $\mathcal X_k$  in $C_\longR(1-\lambda_k')
\setminus \overline{C_\longR}(1-\lambda_k)$.
Thus $\widehat {\mathcal E}_k-\mathcal X_k$ is boundaryless, 
and must be an integer multiple of the integration 
over $C_\longR(1-\lambda_k')\setminus \overline{C_\longR}(1-\lambda_k)$, {\it
i.e.}, 
$\widehat {\mathcal E}_k-\mathcal X_k
=h\jump{C_\longR(1-\lambda_k')\setminus \overline{C_\longR}(1-\lambda_k)}$. We then set $\mathcal E_k:=\widehat {\mathcal E}_k-h\jump{C_\longR(1-\lambda_k')}$ so that $\mathcal E_k=\mathcal X_k$ in ${C_\longR(1-\lambda_k')\setminus \overline{C_\longR}(1-\lambda_k)}$.
} 
so that  
\begin{align}\label{crucial_identification}
\mathcal E_k\res \Big(C_\longR(1-\lambda_k')\setminus \overline{C_\longR}(1-\lambda_k)
\Big)=\mathcal X_k. 
\end{align}
 Let $E_k$ denote 
the support of $\mathcal E_k$;  by decomposition,
 \begin{align}\label{eq:dec_Ek}
 \mathcal E_k=\sum_i(-1)^{\sigma_i}\jump{E_{k,i}}\qquad 
\text{ in } \ C_\longR(1-\lambda_k'),  
 \end{align}
 with $E_{k,i}\subset C_\longR(1-\lambda_k')$ finite perimeter sets, the decomposition 
done with undecomposable components, see \eqref{eq:dec_mathcalE}, \eqref{eq:126}. We denote
$$
\mathbb S(E_k):=\cup_i \mathbb S(E_{k,i})
$$
the union of the cylindrical symmetrizations of the sets $E_{k,i}$,
see \eqref{eq:def_S(U)}.
Recalling
\eqref{eq:primitive},
Definition 
\ref{def:symmetrization_of_the_boundary_of_a_three_current}
and \eqref{eq:def_symmetrization_of_mathcal_E}, 
the symmetrization of the  current $\piPsiDkXik$ is 
\begin{align}\label{symmetrization_of_Sk}
 \mathbb S(\piPsiDkXik)=
\partial \mathbb S(\mathcal E_k)=
\partial\jump{  \mathbb S(E_k)}\res C_\longR(1-\lambda_k').
\end{align}
Formula \eqref{symmetrization_of_Sk}
contains the needed information about the symmetrization of $\Psi_k(\badset)$,
since by construction $\piPsiDk= \piPsiDkXik \res
\overline{C_\longR}(1-\lambda_k)$ (recall \eqref{eq:S_hat_k}).

We have
\begin{align}
\label{dec_Ubis}
 \piPsiDkXik=\sum_i(-1)^{\sigma_i}\jump{\partial^* E_{k,i}} \qquad\text{ in } C_\longR(1-\lambda_k'), 
\end{align}
and since  the decomposition in \eqref{eq:dec_Ek} is 
done by undecomposable components, by \eqref{eq:126} it follows, in  $C_\longR(1-\lambda_k')$, 

\begin{align*}
 |\piPsiDkXik|=\sum_i\mathcal H^2(\partial^* E_{k,i})\qquad \text{ and }\qquad
\partial^* E_{k,i}\subseteq {\rm spt}(\piPsiDkXik).
\end{align*}

\begin{remark}[\textbf{Nonuniqueness of the decomposition}]
 Once the decomposition \eqref{eq:dec_Ek} is fixed, the symmetrization is uniquely determined. However, the decomposition might not be unique, and the resulting symmetrized current in general depends on the choice of the decomposition. This will not be an issue, 
since our procedure will lead to a minimization problem which will not depend on this step.
\end{remark}

Since 
\begin{align*}
 \mathcal H^2(\partial^*\mathbb S( E_{k,i}))\leq \mathcal H^2(\partial^* E_{k,i}) \qquad \text{ for all }i\in \mathbb N,
\end{align*}
and $\mathbb S(E_k)=\cup_i\mathbb S(E_{k,i})$, we also have 
\begin{align}\label{1.78}
 \mathcal H^2(\partial^*\mathbb S( E_k))\leq \sum_i\mathcal H^2(\partial^* E_{k,i})=|\piPsiDkXik|.
\end{align}
The same inequalities hold if we restrict the mass to the set $C^\eps_\longR$,
namely
\begin{equation}\label{eq:same_inequality_for_the_mass}
 |\mathbb S(\piPsiDkXik)|_{C^\eps_\longR}\leq |\piPsiDkXik|_{C^\eps_\longR}.
\end{equation}

Now we want to understand whether $\mathbb S(\piPsiDkXik)$ has some boundary on $\{0\}\times \R^2$.
We have already observed (Corollary \ref{boundaryat0_first}) that  $\piPsiDkXik$ has no boundary in $C_\longR(1-\lambda_k')$. The same holds for the symmetrized current:

\begin{cor}[\textbf{Closedness of  
$\mathbb S(\piPsiDkXik)$ in $(-\infty,l)\times B_{1-\lambda_k'}$}] 
\label{boundaryat0}
The current $\mathbb S(\piPsiDkXik)$ is supported in $[0,\longR]\times \overline{B}_{1-\lambda_k'}$ and 
$ \partial
\mathbb S(\piPsiDkXik)
=0$ in 
 $\mathcal D_1((-\infty,l)\times B_{1-\lambda_k'})$.
\end{cor}

\begin{proof}
 By definition, $\mathbb S(\piPsiDkXik)$ is the 
 boundary of the current carried by the integration
over the finite perimeter set $\mathbb S(E_k)$ in $C_\longR(1-\lambda_k')$. Hence $\partial \mathbb S(\piPsiDkXik)=0$
in $\mathcal D_1(C_\longR(1-\lambda_k'))$. 
The conclusion then follows from the fact that $\mathbb S(\piPsiDkXik)$ 
 is supported in $[0,\longR)\times B_{1-\lambda_k'}\subset C_\longR(1-\lambda_k')$.
\end{proof}

%
\section
{Towards an estimate of $\vert \mathbb S(\piPsiDkXik)\vert$: two
useful lemmas}
\label{subsec:preliminary_estimates}
Now that the symmetrization $\mathbb S(\piPsiDkXik)$ of the current 
$\piPsiDkXik$ 
in 
$C_\longR(1-\lambda_k')$ is obtained
(see \eqref{symmetrization_of_Sk}), 
we need to estimate its mass. 
This will be done separately
in $\overline C_l^\eps (1-\lambda_k) = [\eps, \longR] 
\times \overline B_{1-\lambda_k}$
and in $\overline C_\eps(1-\lambda_k) = [-1,\eps] \times \overline B_{1-\lambda_k}$.
In formula \eqref{symmetrizedset} of Section \ref{subsec:est_Omegaeps_Dk} 
we express the restriction of 
$\mathbb S(\piPsiDkXik)$ to $C_l^\eps$ as generalized graph of 
suitable functions $\Fke$ and $-\Fke$ and estimate the area of these graphs
(see Proposition \ref{prop_Gf}, below). 
In addition, we need a fine description of the 
trace of the symmetrized set boundary 
$\partial \mathbb S(E_k)$ on 
the lateral part of $\partial C_l(1-\lambda_k')$;
this will be done in Section \ref{subsec:symmetrization_of_the_image_of_D_k_cap_B_eps}.

We start by collecting
in Lemma \ref{lem:estimate_of_Q_k} and Lemma \ref{lemma_theta}
two important preliminary estimates; 
we need to introduce the functions $\minuk$, $\maxuk$.

For any $\sourceradialcoordinate\in (\eps,l)$ 
we consider the 
closed curve\footnote{We use here polar coordinates $(r,\alpha)$.} 
$\alpha \in (0,2\pi]
\mapsto \Psi_k(\sourceradialcoordinate,\sourceangularcoordinate)
\in \{\sourceradialcoordinate\}\times \overline{B}_1$; 
the image of $\Psi_k(\sourceradialcoordinate,\cdot)$
is the slice 
of $\Psi_k(\Omega \setminus \overline \sourcedisk_\eps)$ {with the plane $\{t=\sourceradialcoordinate\}$}.
\begin{definition}[\textbf{The functions $|u_k|^\pm$}]\label{def:q_tilde_k}
For all $\sourceradialcoordinate\in(\eps,l)$ we define
\begin{equation}\label{eq:q_-q_+}
\minuk(\sourceradialcoordinate):=
\min_{\sourceangularcoordinate\in(0,2\pi]}|u_k(\sourceradialcoordinate,\sourceangularcoordinate)|,
\qquad  \maxuk(\sourceradialcoordinate):=\max_{\sourceangularcoordinate\in(0,2\pi]}
|u_k(\sourceradialcoordinate,\sourceangularcoordinate)|,
\end{equation}
\end{definition}
Thus the map $\Psi_k(r,\cdot)$ defined in \eqref{eq:Psi_k} takes values in 
$$\{\sourceradialcoordinate\}\times(\overline{B}_{\maxuk(\sourceradialcoordinate)}\setminus B_{\minuk(\sourceradialcoordinate)}).
$$
Let us remark that $\minuk(\sourceradialcoordinate)$ might be equal to $0$, 
that $\maxuk(\sourceradialcoordinate)\leq 1$, 
and that it might happen that 
$\maxuk(\sourceradialcoordinate)=\minuk(\sourceradialcoordinate)$,
see Fig. \ref{fig_q}. 
Moreover, from 
\eqref{eq:Psi_k_in_R_3_minus_C_l},
\begin{equation}\label{uk_outside_Dk}
 |u_k|\geq 1-\lambda_k\qquad \text{in}~\Omega \setminus \badset,
\end{equation}
so that
\begin{align*}
 \maxuk(\sourceradialcoordinate)\geq 1-\lambda_k \qquad \text{ if ~} 
r ~ \text{is~such~that~} \ 
(\Omega\setminus \badset)\cap \partial \sourcedisk_\sourceradialcoordinate\neq \emptyset,
\end{align*}
whereas it might happen that 
\begin{align}\label{q+small}
 \maxuk(\sourceradialcoordinate)< 1-\lambda_k \qquad \text{ if ~}
r ~ \text{is~such~that~} \ 
(\Omega\setminus \badset)\cap \partial \sourcedisk_\sourceradialcoordinate= \emptyset.
\end{align}
In such a case, since 
$\badset\subseteq A_n$ (Lemma \ref{lem:choice_of_u_k_and_t_k} (ii)), 
this can happen only if $\partial \sourcedisk_\sourceradialcoordinate\subseteq A_n$. 
\begin{definition}[\textbf{The set $\Qke$}]
\label{def:the_set_Q_k}
We define 
\begin{equation}\label{Q_k}
\Qke := \{\sourceradialcoordinate \in (\eps,l): 
 \maxuk(\sourceradialcoordinate)< 1-\lambda_k 
\}.
\end{equation}
\end{definition}
Then 
\begin{equation}\label{eq:Q_k_subset_A_n}
\Qke \subseteq \{\sourceradialcoordinate \in (\eps,l) : 
\partial \sourcedisk_\sourceradialcoordinate \subseteq A_n\}.
\end{equation}
The next lemma, that will be used in 
Section 
\ref{subsec:gluing},
shows that the measure of $\Qke$ is small
(see Fig. \ref{fig_q}).

\begin{lemma}[\textbf{Estimate of $\Qke$}]\label{lem:estimate_of_Q_k}
We have
 \begin{align*}
  \mathcal H^1(\Qke)< \frac{1}{2\pi \eps n}.
 \end{align*}
\end{lemma}
\begin{proof}
If $t \in \Qke$ then $\partial \sourcedisk_t\subseteq A_n$.
Then
\begin{align*}
 \mathcal H^1(\Qke)=\int_{\Qke}1 dt\leq \frac{1}{2\pi\eps}\int_{\Qke}2\pi t~dt
=\frac{1}{2\pi\eps}\int_{\Qke}\mathcal H^1(\partial \sourcedisk_t)dt\leq \frac{1}{2\pi\eps}|A_n|,
\end{align*}
where the last inequality is a consequence of the coarea
formula and \eqref{eq:Q_k_subset_A_n}. 
The thesis then follows recalling that 
 $|A_n|<\frac1n$, see \eqref{eq:measure_of_A_n}.
\end{proof}

\begin{figure}
\begin{center}
    \includegraphics[width=0.7\textwidth]{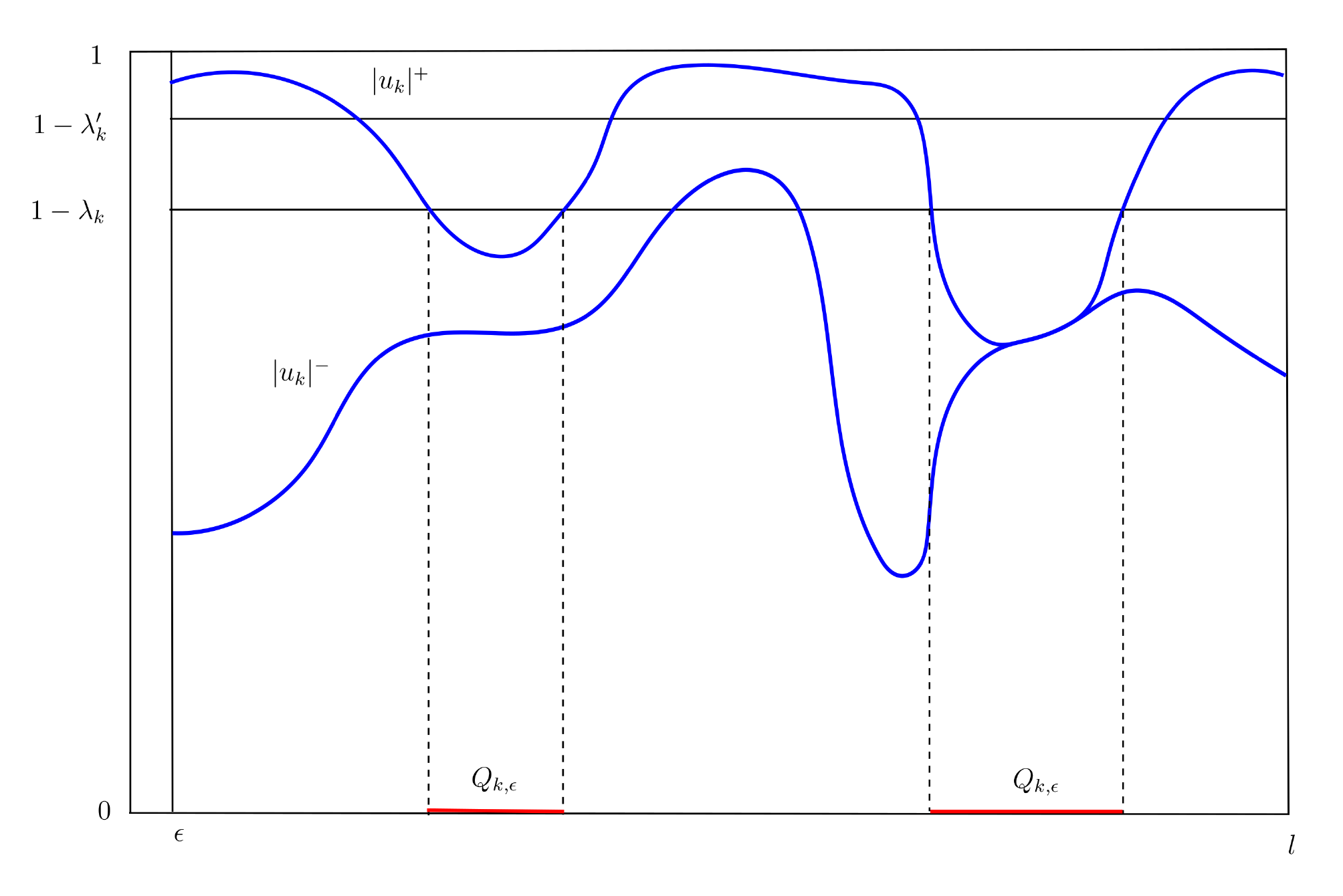}
 \caption{The graphs of the functions $\maxuk$ and $\minuk$ 
defined in \eqref{eq:q_-q_+}, 
and the set $\Qke$ in \eqref{Q_k}.
}
\label{fig_q}
\end{center}
\end{figure}

By slicing and from \eqref{crucial_identification},
\eqref{eq:dec_Ek}, we have for almost every $t\in(0,\longR)$ and
almost every $\rho\in (1-\lambda_k,1-\lambda_k')$,
 \begin{align*}
  (\mathcal X_k)_{t,\rho}=\sum_i(-1)^{\sigma_i}\jump{E_{k,i}\cap(\{t\}\times \partial B_\rho)},
 \end{align*}
 and
\begin{align}\label{eq:ineq_slice_Xk}
 \mathcal H^1(\mathbb S(E_k)\cap (\{t\}\times \partial B_\rho))
\leq \sum_i\mathcal H^1(E_{k,i}\cap(\{t\}\times \partial B_\rho))=|(\mathcal X_k)_{t,\rho}|,
\end{align}
since the decomposition is done in undecomposable components (see
\eqref{eq:in_conclusion}). 

Recalling the definition of $\Theta$
in \eqref{eq:Theta_t_rho}
 we have,  
for fixed $t \in (0,\longR)$ and for any $\rho\in (0,1-\lambda_k']$, 
\begin{equation}\label{eq:Theta_k}
\Theta_k(t,\rho):= \Theta_{\mathbb S(E_k)}(t,\rho)
=\frac{1}{\rho}\mathcal H^1(\mathbb S(E_k)\cap (\{t\}\times \partial B_\rho))   
\end{equation}
denotes the measure (in radiants) of the slice 
$\mathbb S(E_k)\cap (\{t\}\times\partial B_\rho)$.
By construction, 
\begin{align*}
 \Theta_k(t,\rho)=\Theta_k(t,\varrho)\qquad 
\text{for any }\rho,\varrho\in (1-\lambda_k,1-\lambda_k'),
\end{align*}
since the slices of 
$\mathcal X_k$, and hence of the sets $E_{k,i}$, are 
radially symmetric\footnote{Each radial section is (suitably rescaled) the 
same since, by definition, function 
$\tau$ 
in \eqref{eq:tau}
is radial.} in 
$C_l(1-\lambda_k')\setminus \overline{C_l}(1-\lambda_k)$.
Also, the right-hand side of \eqref{eq:Theta_k} 
vanishes for $\rho \in (1-\lambda_k, 1-\lambda_k')$.

We now look
for an estimate of $\Theta_k(t,\rho)$, 
for $t\in(\eps,l)$ 
and $\rho\in (1-\lambda_k,1-\lambda_k')$:
the next lemma will be used in Section \ref{subsec:gluing}.

\begin{lemma}[\textbf{$L^1$-estimate of the angular slices}]\label{lemma_theta}
We have
\begin{equation}\label{eq:L^1_angular_estimate}
 \int_{\eps}^l\Theta_k(t,\rho)~dt\leq \frac{1}{\eps n}+o_k(1)
\qquad
\forall \rho\in (1-\lambda_k,1-\lambda_k'), 
\end{equation}
where $o_k(1)$ is a nonnegative function, depending on $\eps$ and $n$,
and infinitesimal as $k\rightarrow +\infty$.
\end{lemma}

\begin{proof}
 It is convenient to set
\begin{align}\label{1.111}
 H_{k,t}:=\badset\cap \partial \sourcedisk_t,\qquad H_{k,t}^c
:=(\Omega\setminus \badset)\cap \partial \sourcedisk_t
\qquad 
\forall t\in (\eps,l).
\end{align}
Observe that
the relative boundary
of $H_{k,t}$, \textit{i.e.}, the boundary of $H_{k,t}$ when considered as a 
subset of $\partial \sourcedisk_t$, is contained in 
$\partial \badset \cap \partial \sourcedisk_t$.

We fix $t\in(\eps,l)$ such that the relative boundary of 
$H_{k,t}$ is 
a finite set of points 
(this happens for $\mathcal H^1$-a.e. $t$, since
$\mathcal H^1(\partial \badset) < +\infty$ from Lemma
\ref{lem:choice_of_u_k_and_t_k}(i)) and 
fix any $\rho\in (1-\lambda_k,1-\lambda_k')$.
By inequality \eqref{eq:ineq_slice_Xk} we have
\begin{align}\label{theta_arcs}
 \Theta_k(t,\rho)\leq \frac{1}{\rho}|(\mathcal X_k)_{t,\rho}|,
\end{align}
so it is sufficient to estimate the mass of a (slice of a) slice of the 
$3$-current $\mathcal X_k$ defined in \eqref{eq:defX_k}. We recall that 
by \eqref{eq:defX_k} we have\footnote{The orientation of $\partial B_\rho$ is taken counterclockwise.}
\begin{align}\label{eq:slice_t_rho}
 (\mathcal X_k)_{t,\rho}=\jump{\{t\}\times \partial B_\rho}-(\mathcal Y_k)_{t,\rho},
\end{align}
where $\jump{\{t\}\times \partial B_\rho}$ has a natural 
orientation\footnote{The orientation of the 3-current 
$\jump{C_l(\rho)}$ induces an orientation of its slice 
$\jump{\{t\}
\times B_\rho}$. This orientation
induces an orientation of 
$\jump{\{t\}
\times \partial B_\rho}$, which coincides
with the orientation
of $\jump{\{t\}\times \R^2}_{\rho}$ induced by the slicing by $\rho$.}
 inherited by the fact that it is the boundary of 
$\jump{\{t\}\times B_\rho}$ in $\{t\}\times\R^2$, which 
in turn is a slice of 
$\jump{C_l(\rho)}$. 
By \eqref{eq:defY_k}
\begin{align}\label{eq:slice_ipsilon}
 (\mathcal Y_k)_{t,\rho}=
\widetilde{\tau}_\sharp\jump{\{\rho\}\times ((\Omega\setminus \badset)
\cap \partial \sourcedisk_t)}=
\tau(\rho,\projlambdak \circ \Psi_k(\cdot))_\sharp\jump{H_{k,t}^c},
\end{align}
see \eqref{eq:tau}, \eqref{eq:widetilde_tau}, \eqref{eq:P_k}, and 
Remark \ref{remark:orientationofboundary} for the orientation of $\jump{\{\rho\}\times ((\Omega\setminus \badset)
\cap \partial \sourcedisk_t)}$. As for $ \jump{H_{k,t}^c}$ we endow 
the set $H_{k,t}^c\subset \partial \sourcedisk_t$ with the orientation 
inherited by $\partial \sourcedisk_t$, {\it i.e.}, by a counterclockwise
tangent unit vector.
Now, since 
the restriction of 
$\tau(\rho,\projlambdak \circ \Psi_k(\cdot))$ 
to $\partial \sourcedisk_t$ takes values in $\{t\}\times \partial B_\rho$, 
the current $(\mathcal Y_k)_{t,\rho}$ is 
the integration over arcs\footnote{Such arcs could overlap, since in general the multiplicity of $\mathcal Y_k$ might be different from $1$.} in $\{t\}\times \partial B_\rho$. 
To identify these arcs we 
distinguish the following three cases (A), (B), (C): 

(A) 
\nada{$H_{k,t} =\partial\sourcedisk_t$, so that $\partial 
\sourcedisk_t\subset \badset$. 
In this case\footnote{Notice that $\partial \sourcedisk_t\subset \badset$ might happen only for a set of $t\in (\eps,l)$ that is small in measure, see also Lemma \ref{lem:estimate_of_Q_k} (this is a consequence of the fact that  $\badset$ is a small set).}
} $H_{k,t}^c=\emptyset$. From \eqref{eq:slice_ipsilon}
it follows
$(\mathcal Y_k)_{t,\rho}=0$ and 
 $(\mathcal X_k)_{t,\rho}=\jump{\{t\}\times \partial B_\rho}$ from \eqref{eq:slice_t_rho}.
 Thus
 \begin{align}\label{7.12}
  \Theta_k(t,\rho)={2\pi}\leq 2\pi\frac{t}{\eps}=\frac1\eps \mathcal H^1(H_{k,t}).
 \end{align}

(B) 
$H_{k,t}^c = \partial \sourcedisk_t\subset \Omega\setminus \badset$, hence 
$({\mathcal Y_k})_{t,\rho} =
\tau(\rho,\projlambdak \circ \Psi_k(\cdot))_\sharp\jump{\partial \sourcedisk_t}$
from \eqref{eq:slice_ipsilon}. 
Then
\begin{equation}\label{eq:non_trivial}
({\mathcal Y_k})_{t,\rho} =
\jump{\{t\}\times \partial B_\rho}.
\end{equation}
Indeed, fix three points $x_1,x_2,x_3\in \partial
\sourcedisk_t$ in counterclockwise order 
such that $|\frac{x_i}{|x_i|}-\frac{x_j}{|x_j|}|>4\lambda_k$ for $i\neq j$. 
Since $\diffuku(x)=|\frac{x}{|x|}-u_k(x)|<\lambda_k$ for $x\in \Omega\setminus \badset$, $x\neq0$,  
the points $z_i:=\projlambdak\circ\Psi_k(x_i)$ are 
still in counterclockwise order in $\{t\}\times \partial B_{1-\lambda_k}$ 
(the image of the arc $\overline{x_1x_2}$ covers the arc $\overline{z_1z_2}$ 
that does not contain $z_3$). 
Therefore $(\projlambdak\circ\Psi_k(\cdot))_\sharp\jump{\overline{x_ix_{i+1}}}
=\jump{\overline{z_iz_{i+1}}}$ for $i=1,2,3$\footnote{The boundary of $(\projlambdak\circ\Psi_k(\cdot))_\sharp\jump{\overline{x_ix_{i+1}}}$ is
 $\delta_{z_{i+1}}-\delta_{z_i}$, 
hence $(\projlambdak\circ\Psi_k(\cdot))_\sharp\jump{\overline{x_ix_{i+1}}}$ 
is an arc connecting 
$z_i$ and $z_{i+1}$. Since this arc cannot contain the third point, it must be $\jump{\overline{z_iz_{i+1}}}$, counterclockwise oriented.} 
(with the convention $x_4=x_1$, $z_4=z_1$), and hence
\begin{align*}
(\projlambdak\circ\Psi_k(\cdot))_\sharp\jump{\partial \sourcedisk_t}=\sum_{i=1}^3(\projlambdak\circ\Psi_k(\cdot))_\sharp\jump{\overline{x_ix_{i+1}}} =\sum_{i=1}^3\jump{\overline{z_iz_{i+1}}}=\jump{\{t\}\times \partial B_{1-\lambda_k}}. 
\end{align*}
Taking the push-forward by $\tau$  we get 
\eqref{eq:non_trivial}. 
From this and \eqref{eq:slice_t_rho} we deduce
$(\mathcal X_k)_{t,\rho}=0$,
 and 
$\Theta_k(t,\rho)=0$.

Before passing to case (C), we anticipate an observation which will be useful to deal with it.
Let $\overline{x_1 x_{2}}\subset (\Omega \setminus \badset) \cap \partial \sourcedisk_t 
$ be an arc oriented counterclockwise. We want to identify the current $(\projlambdak\circ\Psi_k)_\sharp\jump{\overline{x_1x_{2}}}$; to do that we consider three different cases for $\overline{x_1 x_{2}}$. Case 1: $|\frac{x_1}{|x_1|}-\frac{x_{2}}{|x_{2}|}|>2\lambda_k$. Hence $z_1:=\projlambdak \circ\Psi_k(x_1)$ and $z_2:=\projlambdak \circ\Psi_k(x_2)$ must have the same order on $\partial B_{1-\lambda_k}$ of $x_1$ and $x_2$, moreover $(\projlambdak\circ\Psi_k)_\sharp\jump{\overline{x_1x_{2}}}=\jump{\overline{z_1 z_2}}$, where $\overline{z_1 z_2}$ is the arc connecting $z_1, z_2$, starting from $z_1$ and oriented counterclockwise. Case 2: $|\frac{x_1}{|x_1|}-\frac{x_{2}}{|x_{2}|}|\leq 2\lambda_k$ (that implies  $|z_1-z_2|\leq 4\lambda_k$, and $z_1, z_2$ could have reversed order of $x_1$ and $x_2$). Let $z_1, z_2$ have the same order of $x_1, x_2$, then $(\projlambdak\circ\Psi_k)_\sharp\jump{\overline{x_1x_{2}}}=\jump{\overline{z_1 z_2}}$ where $\overline{z_1 z_2}$ is the arc connecting $z_1, z_2$, starting from $z_1$ and oriented counterclockwise. Now let $z_1, z_2$ have the reversed order of $x_1, x_2$. If $\overline{x_1 x_2}$ is the short path arc connecting $x_1, x_2$, then 
$(\projlambdak\circ\Psi_k)_\sharp\jump{\overline{x_1x_{2}}}=\jump{\overline{z_1 z_2}}$, where $\overline{z_1 z_2}$ is the (short path) arc connecting $z_1$, $z_2$, starting  from $z_1$, and oriented clockwise. If $\overline{x_1 x_2}$ is instead the long path arc joining $x_1, x_2$, then $(\projlambdak\circ\Psi_k)_\sharp\jump{\overline{x_1x_{2}}}=\jump{\{t\}\times \partial B_{1-\lambda_k}}+\jump{\overline{z_1 z_2}}$, where $\jump{\{t\}\times \partial B_{1-\lambda_k}}$ is oriented counterclockwise, and $\overline{z_1 z_2}$ is the (short path) arc starting from $z_1$ and oriented counterclockwise. Notice also that in case 2 we always have $\mathcal H^1(\overline{z_1 z_2})<8\lambda_k$.

Now, we analyse the third case.

(C) $H_{k,t}^c$ is union of finitely many arcs.
Let us denote by $\left\{\overline{x_1^i x_2^i}\right\}_i$ these distinct arcs with
endpoints\footnote{In polar coordinates.} $x_j^i=(t,\alpha_j^i)\in \partial\sourcedisk_t$,
with the index $i \in \{1,\dots,h = h_{k,t}\}$ varying in a finite set,
so that
\begin{align*}
 \jump{H_{k,t}^c}=\sum_{i=1}^{h}\jump{\overline{x_1^i x_2^i}}\qquad \text{and }\qquad\jump{H_{k,t}}=\sum_{i=1}^{h}\jump{\overline{x_2^i x_1^{i+1}}},
\end{align*}
where again the orientation of $\overline{x_1^i x_2^i}$ is the one 
inherited by the
counterclockwise orientation of
$\partial \sourcedisk_t$ 
and, by convention, $h+1=1$. 
Being $H^c_{k,t}$ relatively closed set in $\partial \sourcedisk_t$, 
it might happen that $x_1^i=x_2^i$ for some $i$.
Notice that $x_j^i$ belongs to the relative
boundary of $H_{k,t}$ which, in turn, is a subset of 
$\partial \badset\cap 
\partial \sourcedisk_t$.

 We denote $$
z_j^i:=\projlambdak \circ\Psi_k(x_j^i)\in \{t\}\times \partial B_{1-\lambda_k}.
$$
After applying  $\projlambdak\circ\Psi_k(\cdot)$,
the points $x_j^i$ might also reverse their order, \textit{i.e.}, 
the orientation of the 
arc $\projlambdak\circ
\Psi_k\left(\overline{x_1^i x_2^i}\right)$
could be the opposite of the orientation of 
$\overline{x_1^i x_2^i}$.

In order to describe the current $(\mathcal X_k)_{t,\rho}$ we need first 
to extend $\tau(\rho,\projlambdak \circ \Psi_k(\cdot)) $ to $H_{k,t}$:
note carefully 
that $\projlambdak \circ \Psi_k(\cdot)$ is well-defined 
in $H_{k,t}^c$, but not necessarily in $H_{k,t}$, 
since  $\projlambdak \circ \Psi_k(H_{k,t}) \cap C_l(1-\lambda_k)$ may
not be empty, and in such a case it is not in the domain of $\tau(\rho,\cdot)$. 
The extension we get (see \eqref{eq:extension})
will allow to 
write a specific double slice of $\mathcal X_k$
as push-forward, see \eqref{eq:X_k_is_a_push_forward}.
We stress  that this extension is done for a fixed slice $\{t\}\times \R^2$ 
and in general 
it cannot be done globally\footnote{We do not need a 
global extension since we aim to obtain an estimate which holds for a fixed $t$.} for all $t \in (\eps,l)$.

\begin{figure}
\begin{center}
    \includegraphics[width=0.63\textwidth]{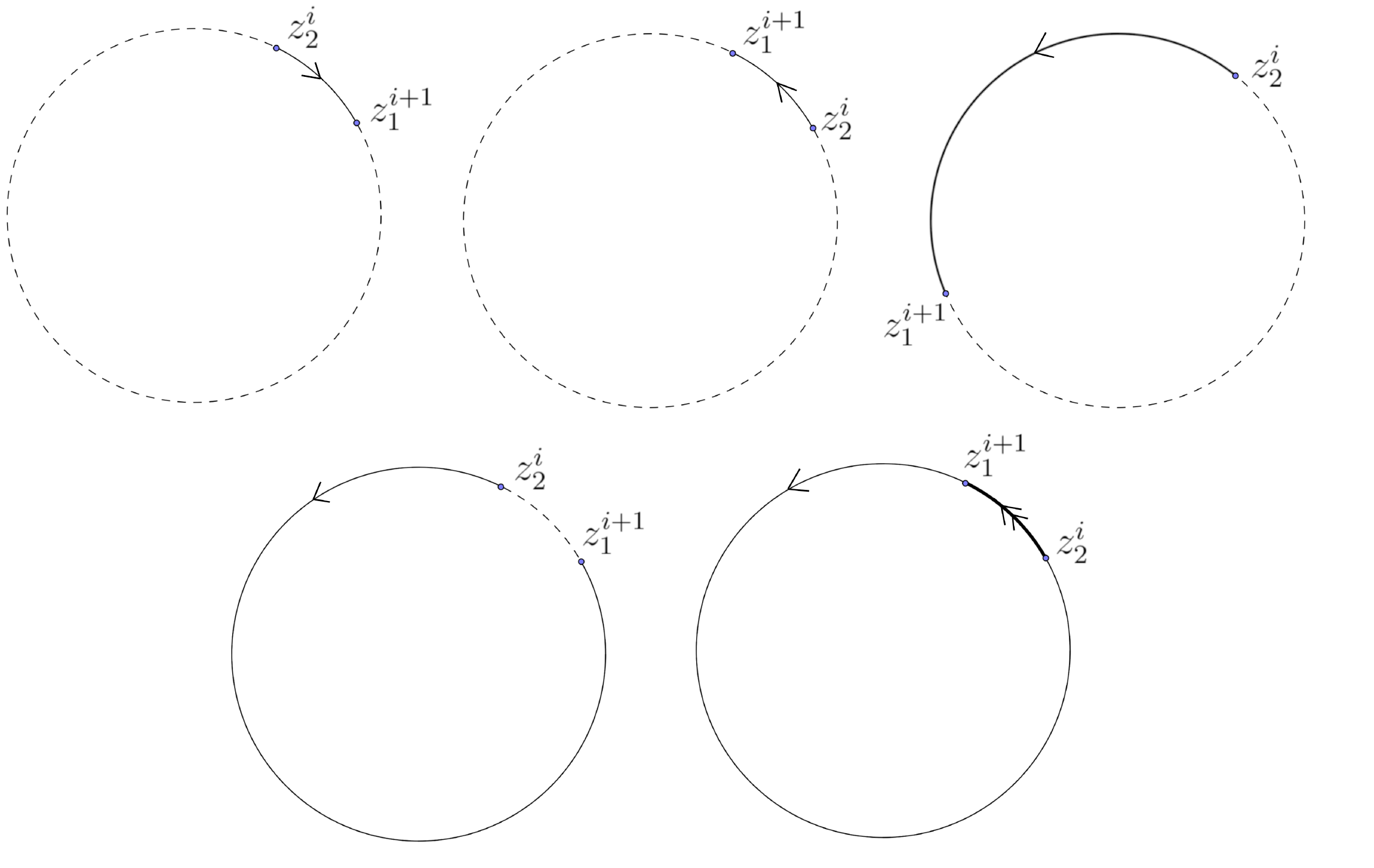}
 \caption{The choice of the arc between $z^i_2$ 
and  $z^{i+1}_1$. The correct arc is the one in bold on the 
dashed
 circle $\{t\}\times \partial B_{1-\lambda_k}$. On 
top left the case  
$|z^i_2-z^{i+1}_1|\leq 2\lambda_k$ and the arc is clockwise oriented;
on 
top center  again case  $|z^i_2-z^{i+1}_1|\leq 2\lambda_k$ 
and the arc counterclockwise oriented; on top right 
the case $|z^i_2-z^{i+1}_1|>2\lambda_k$; 
on bottom left 
again the case $|z^i_2-z^{i+1}_1|\leq 2\lambda_k$ when the oriented arc  
$\overline{z^i_2z^{i+1}_1}$ is the long one. Finally on bottom right it is depicted again the
case 
$|z^i_2-z^{i+1}_1|\leq 2\lambda_k$ but the counterclockwise arc between $z^i_2$ and  $z^{i+1}_1$ has reversed order with
respect to $\alpha_2^i$ and $\alpha_1^{i+i}$, so that $\beta_i$ is the long arc; in this case the correct arc such 
that $|\widehat \beta_i-\beta_i|\leq 2\widehat{\lambda}_k$ is the  
short one connecting $z^i_2$ and  $z^{i+1}_1$ (double bold) together with a complete
 turn around the circle. }\label{figarcs}
\end{center}
\end{figure}

For  $t$ fixed such that case (C) holds, 
we extend the function $\projlambdak \circ \Psi_k(\cdot)$ to $ H_{k,t}$ as follows. Let $\overline{x_2^i x_1^{i+1}}$ 
be an arc of $ H_{k,t}$; we want to map this arc on an arc in $\{t\}
\times \partial B_{1-\lambda_k}$ joining 
the two image points 
$z^i_2,z^{i+1}_1$, with the orientation
from $z^i_2$ to $z^{i+1}_1$. 
However there are infinitely many\footnote{
We can for instance join $z^i_2$ to $z^{i+1}_1$ travelling 
along an oriented arc connecting them, and then travelling along the whole circle an arbitrary number of times (thus considering a self-overlapping
arc).}  
choices of an arc connecting $z^i_2$ to $z^{i+1}_1$.
To specify which arc we choose we distinguish two possibilities: 
$\vert z^i_2 - z^{i+1}_1 \vert \leq 
2\lambda_k$, and 
$\vert z^i_2 - z^{i+1}_1 \vert >
2\lambda_k$.
Notice that 
$\vert z^i_2 - z^{i+1}_1 \vert \leq 
2\lambda_k$
is the only case in which the 
points 
$x_2^i$ 
and $ x_1^{i+1}$ could have
images $z_2^i$ and $z_1^{i+1}$ with a reversed order 
 on $\{t\}\times \partial B_{1-\lambda_k}$. Indeed,
since 
$x_2^i, 
 x_1^{i+1}
\in \partial \sourcedisk_t \cap \partial \badset$, 
we have 
$\diffuku(x_j^i)=|\frac{x_j^i}{|x_j^i|}-u_k(x_j^i)|=\lambda_k$. 
In particular, if the distance between $z_2^i$ 
and $ z_1^{i+1}$ 
is larger than  $2\lambda_k$, it means that the distance between $u_k(x_2^i)$ 
and $ u_k (x_1^{i+1})$
were larger than $2\lambda_k$ ($\projlambdak$ does not increase the distance), so that $z_2^i$ 
and $ z_1^{i+1}$  must have the 
same order of $\frac{x_2^i}{|x_2^i|}$ 
and $\frac{x_1^{i+1}}{|x_{1}^{i+1}|}$ on $\partial B_1$, which is the same order of $x_2^i$ and $ x_1^{i+1}$ on $\partial\sourcedisk_t$.

We are now in a position to specify the arc:
when
$\vert z^i_2 - z^{i+1}_1 \vert >
2\lambda_k$
we define
$\overline{z^i_2z^{i+1}_1}$ 
to be the counterclockwise oriented arc\footnote{Likewise
the orientation from $x_2^i$ to $x_1^{i+1}$.} from $z^i_2$ to $z^{i+1}_1$.
When
$\vert z^i_2 - z^{i+1}_1 \vert \leq 
2\lambda_k$ we argue as follows: Let $\beta_i$ be the angular 
amplitude of the arc $\overline{x_2^ix_1^{i+1}}$. 
We define
$\overline{z^i_2z^{i+1}_1}$ as
the unique oriented arc from $z_2^i$ to $z_1^{i+1}$ satisfying the 
following property: If $\widehat \beta_i$ is its oriented angular amplitude (positive if counterclockwise oriented, negative otherwise), 
then
 \begin{align}\label{differenza_amplitude_beta}
  |\widehat \beta_i-\beta_i|\leq 2\widehat{\lambda}_k,
 \end{align}
where $\widehat{\lambda}_k$ is the angular amplitude of a chord 
on $\partial B_{1-\lambda_k}$ of length $\lambda_k$ (see Fig. \ref{figarcs}). It is easy to check that there is a unique arc $\overline{z^i_2z^{i+1}_1}$ satisfying this property. Moreover 
the same property holds for $\beta_i$ and $\widehat \beta_i$ in the case 
that $\vert z^i_2 - z^{i+1}_1 \vert > 
2\lambda_k$, since 
$\projlambdak\circ \Psi_k(\cdot)$ does not change the angular 
coordinate of a point $x_j^i$ of a quantity larger than $\widehat \lambda_k$.

Once we have specified the image arc, 
we can  define $\widehat P_{k,i}:\overline{x_2^i x_1^{i+1}}\rightarrow \overline{z^i_2z^{i+1}_1}$ 
to be the affine (with respect to the angular coordinate) function 
mapping $x_2^i$ to $z^i_2$ and 
$x_1^{i+1}$ to $z^{i+1}_1$.
We then introduce $P_k = P_{k,t}:\partial \sourcedisk_t\rightarrow \{t\}\times \partial B_{1-\lambda_k}$ as follows:
\begin{align}\label{eq:extension}
 P_k(x):=\begin{cases}
              \projlambdak\circ \Psi_k(x)&\text{if }x\in H^c_{k,t},
\\
\\
              \widehat P_{k,i}(x)&\text{if }x\in \overline{x_2^i x_1^{i+1}} \text{~for some }i.
             \end{cases}
\end{align}
We claim that 
\begin{align}\label{claim_lemma_archi}
 \tau(\rho,P_k(\cdot))_\sharp \jump{\partial \sourcedisk_t}=\jump{\{t\}\times \partial B_\rho}.
\end{align}
 Since the map $\tau(\rho,\cdot)$ is an orientation preserving
homeomorphism between $\partial B_{1-\lambda_k}$ and $\partial
B_\rho$, it is sufficient to show that 
\begin{align}
\label{subclaim_lemma_archi}
 P_k(\cdot)_\sharp\jump{\partial \sourcedisk_t}=\jump{\{t\}\times \partial B_{1-\lambda_k}}.
\end{align}
Equivalently, we will prove that 
\begin{align}
\label{subclaim_lemma_archi_bis}
 \sum_{i=1}^h(\jump{\overline{z_1^iz_2^i}}+\jump{\overline{z_2^iz_1^{i+1}}})=\jump{\{t\}\times \partial B_{1-\lambda_k}}.
\end{align}
Let $\angcoord_i$ (resp. 
$\beta_i$) be the angular amplitude,
in counterclockwise order, of the arc $\overline{x_1^ix_2^i}$ (resp.
$\overline{x_2^ix_1^{i+1}}$). Trivially we have 
$\sum_{i=1}^h(\angcoord_i+\beta_i)=2\pi$.
If  $\widehat \angcoord_i$ (resp. $\widehat \beta_i$) is
 the angular amplitude of $\overline{z_1^iz_2^i}$ 
(resp. $\overline{z_2^iz_1^{i+1}}$),
taken with sign $\pm1$ according to their orientation, we 
see that to prove \eqref{subclaim_lemma_archi_bis} it suffices to show 
\begin{align}\label{subclaim_lemma_archi_final}
 \sum_{i=1}^h(\widehat \angcoord_i+\widehat \beta_i)=2\pi.
\end{align}
To do this we use \eqref{differenza_amplitude_beta}; notice first  
that  the counterpart of \eqref{differenza_amplitude_beta} holds for 
the arc between $x_1^i$ and $x_2^i$: Namely 
the map 
$\projlambdak\circ\Psi_k$ transforms 
the arc $\overline{x_1^ix_2^i}$ 
of angular amplitude $\omega_i$,
in the arc $\overline{z_1^{i}z_2^i}$ of amplitude $\widehat \omega_i$ in such a way that 
\begin{align}\label{differenza_amplitude_omega}
 |\widehat \omega_i-\omega_i|\leq 2\widehat \lambda_k.
\end{align}
Now, if $\theta_j^i$ is the angular coordinate of $z_j^i$, and 
$ \sourceangularcoordinate^i_j$ is  the angular coordinate of $x_j^i$, 
we know that 
\begin{equation}\label{theta_i}
\theta_j^i=\alpha_j^i+r_j^i,\qquad\text{ with }|r_j^i|\leq \widehat \lambda_k. 
\end{equation}
Here again $\widehat \lambda_k$ is the angle of a chord of length $\lambda_k$ on 
$\partial B_{1-\lambda_k}$.
To prove  \eqref{subclaim_lemma_archi_final} we reduce ourselves to show that 
\begin{align}
 &\widehat \omega_i=\omega_i+r_2^i-r_1^i,\label{1_claim_final}\\\label{2_claim_final}
 &\widehat \beta_i=\beta_i+r_1^{i+1}-r_2^i,
\end{align}
for all $i$. 
Fix $i$; we can assume $\alpha_2^i=\alpha_1^i+\omega_i$, and by \eqref{theta_i} we get
$$\widehat \omega_i=\omega_i+r_2^i-r_1^i+2k_i\pi,$$
with $k_i\in \mathbb Z$ accordingly to the number of oriented complete turns
 around the circle $\partial B_{1-\lambda_k}$. From \eqref{differenza_amplitude_omega} we have $k_i=0$ for all $i$, and \eqref{1_claim_final} follows. 
A similar argument, 
using \eqref{differenza_amplitude_beta}, 
leads  to \eqref{2_claim_final}, hence \eqref{subclaim_lemma_archi_final} is proved, and \eqref{claim_lemma_archi} follows at once.
\medskip
Define
$$
y_j^i:=
\tau(\rho,z_j^i)\in \{t\}\times \partial B_\rho.
$$
From \eqref{claim_lemma_archi}, 
\eqref{eq:slice_t_rho}, and \eqref{eq:slice_ipsilon} it follows that 
\begin{align}
\label{eq:X_k_is_a_push_forward}
 (\mathcal X_k)_{t,\rho}=\tau(\rho,P_k(\cdot))_\sharp \jump{\partial \sourcedisk_t}-\tau(\rho,\projlambdak \circ \Psi_k(\cdot))_\sharp\jump{H_{k,t}^c}=\tau(\rho,P_k(\cdot))_\sharp\jump{H_{k,t}},
\end{align}
so that, 
since the maps $\tau(\rho,\cdot)$ send the arcs $\overline{z^i_2z^{i+1}_1}$ onto $\overline{y^i_2y^{i+1}_1}$,
 we have
\begin{align}\label{current_X_doubleslice}
 (\mathcal X_k)_{t,\rho}=\sum_{i=1}^{h}\jump{\overline{y_2^iy_1^{i+1}}},
\end{align}
hence
\begin{align}\label{eq:290}
 |(\mathcal X_k)_{t,\rho}|\leq \sum_{i=1}^{h}\mathcal H^1(\overline{y_2^iy_1^{i+1}}).
\end{align}
We now estimate the length of the arcs $\overline{y_2^iy_1^{i+1}}$. 
For simplicity we fix $i$ and 
set $Y_1:=y_2^i$, $Y_2:=y_1^{i+1}$,
$X_1:=x_2^i$ and $X_2:=x_1^{i+1}$.
Let  ${\rm d}(\cdot,\cdot)$ denote the distance  
between points of 
 $\{t\}\times \partial B_{\rho}$ ({\it i.e.}, the length of the minimal arc connecting the two points), let $\pi_\rho$ be the orthogonal 
projection of $\R^2_{\rm target}$ 
 onto the convex set $\overline B_\rho$,
and write
$Y_i = (t,\widetilde Y_i)$ 
with $\widetilde Y_i \in \overline B_\rho$,
for $i=1,2$.
Then, setting $\widehat X_i := \frac{X_i}{|X_i|}$
and denoting
$ \overline{\widehat X_1 \widehat X_2}$ the arc between 
$\widehat X_1$ and $\widehat X_2$ on 
$\{t\}\times \partial B_1$, we have
\begin{equation}
\label{passages}
\begin{aligned}
\mathcal H^1(\overline{Y_1Y_2})&\leq \mathcal 
H^1\Big(\overline{\pi_\rho(X_1)
\pi_\rho(X_2)}\Big)+{\rm d}(\pi_\rho(X_1),\widetilde Y_1)+{\rm d}(\pi_\rho(X_2),\widetilde Y_2)
\\
&=\rho\mathcal H^1\Big(\overline{
\widehat X_1\widehat X_2}\Big)+{\rm d}(\pi_\rho(X_1),
\widetilde Y_1)+{\rm d}(\pi_\rho(X_2),\widetilde Y_2)
\\
&\leq \rho\mathcal H^1\Big(\overline{\widehat X_1
\widehat X_2}\Big)+\frac\pi2|\pi_\rho(X_1)-\widetilde Y_1|+
\frac\pi2|\pi_\rho(X_2)-\widetilde Y_2|
\\
&\leq \rho\mathcal H^1\Big(\overline{\widehat X_1
\widehat X_2}\Big)+\frac\pi2|\pi_\rho(X_1)-
\pi_\rho\circ u_k(X_1)|+\frac\pi2|\pi_\rho\circ u_k(X_1)-\widetilde Y_1|
\\
&\;\;\;\;+\frac\pi2|\pi_\rho(X_2)-\pi_\rho\circ u_k(X_2)|+\frac{\pi}{2}|\pi_\rho\circ u_k(X_2)-\widetilde Y_2|
\\
&\leq\frac{\rho}{\eps}\mathcal H^1\Big(\overline{X_1X_2}\Big)+\frac\pi2\big( \diffuku(X_1)+\diffuku(X_2)\big)+\pi(\lambda_k-\lambda_k'),
\end{aligned}
\end{equation}
where we use that, for $x \neq 0$,
$$\diffuku(x)=|\frac{x}{|x|}-u_k(x)|=|u(x)-u_k(x)|\geq |\pi_\rho\circ u(x)-\pi_\rho\circ u_k(x)|,$$
because ${\rm Lip}(\pi_\rho)=1$, 
$|\pi_\rho\circ u_k(X_i)-\widetilde Y_i|\leq \lambda_k-\lambda_k'$ 
for $i=1,2$, and $X_i \in \partial \sourcedisk_t$, ~$t>\eps$. 
By \eqref{theta_arcs} 
\eqref{eq:290} 
and \eqref{passages}, we infer
\begin{align}\label{estimate_theta}
 \Theta_k(t,\rho)\leq \frac{1}{\eps}\mathcal H^1(H_{k,t})+\frac{\pi}{2\rho}\sum_{x\in \partial H_{k,t}}(\diffuku(x)+\lambda_k).
\end{align}
Estimate \eqref{estimate_theta} holds for $\mathcal H^1$-almost
every $t\in(\eps,l)$ 
such that neither case (A) nor (B) happens. Moreover, by \eqref{7.12} it holds also in case (A).
Case (B) does not contribute to the $L^1$ norm of $\Theta_k(\cdot,\rho)$, 
and therefore \eqref{estimate_theta} holds for $\mathcal H^1$-almost every $t\in(\eps,l)$.

Denoting by $m(x)=|x|$, so that $|\nabla m|=1$ out of the origin, the coarea formula allows us to write
\begin{align*}
\int_{\partial \badset}\diffuku(\sigma)d\mathcal H^1(\sigma)\geq\int_{\partial \badset}|\frac{\partial  m}{\partial \sigma}|\diffuku(\sigma)d\mathcal H^1(\sigma)=\int_\eps^l\sum_{x\in m^{-1}(t)\cap \partial \badset}\diffuku(x)dt=\int_\eps^l\sum_{x\in\partial H_{k,t}}\diffuku(x)dt.
\end{align*}
Similarly 
\begin{align*}
 \int_{\partial \badset}\lambda_k~d\mathcal H^1(\sigma)
\geq\int_\eps^l\lambda_k\mathcal H^0(\{x\in\partial H_{k,t}\})~ dt.
\end{align*}
Recalling \eqref{eq:measure_of_A_n}, from \eqref{estimate_theta} we 
finally get
\begin{align*}
 \int_{\eps}^l\Theta_k(t,\rho)dt
\leq \frac{1}{\eps}|\badset|+\frac{\pi}{2(1-\lambda_k)}
\int_{\partial \badset}\left(\diffuku(\sigma)+
2\lambda_k\right)~ d\mathcal H^1(\sigma)\leq \frac{1}{\eps n}+o_k(1),
\end{align*}
where $o_k(1)$ depends on $\eps$ and $n$ (since $\lambda_k$ does) and 
vanishes as $k\rightarrow +\infty$, thanks to Lemma 
\ref{lem:choice_of_u_k_and_t_k} (iii).
\end{proof}

\section{
Estimate from below of the mass of $\currgraphk$ 
over $\badset \cap (\Om \setminus \overline \sourcedisk_\eps)$}
\label{subsec:est_Omegaeps_Dk}
Now we want to identify the current $\mathbb S(\piPsiDkXik)$ in 
\eqref{symmetrization_of_Sk}
as 
sum of polar graphs (Section \ref{sec:polar_graphs_in_a_cylinder}), and 
to do this we need some preliminaries.

\begin{definition}[\textbf{The function $\Fke$}]
\label{def:the_functions_F_k}
Recalling the definition of $\Theta_k = \Theta_{\mathbb S(E_k)}$ in \eqref{eq:Theta_k}, 
we set
\begin{align}
&\Fke:(\eps,l)\times (0,1-\lambda_k']\times \{0\}\rightarrow [0,\pi], \qquad 
\Fke(t,\rho,0):=\frac{\Theta_k(t,\rho)}{2}.
\label{eq:Fke}
\end{align}
\end{definition}

Note that ${\rm dom}(\Fke)\subsetneq {\rm dom}(\Theta_k)$.
The polar graph of $\Fke$ is the set
$G^{{\rm pol}}_{\Fke}= \{(t,\rho,\Fke(t,\rho,0)):(t,\rho,0)\in(\eps,l)\times (0,1-\lambda_k']\times\{0\}\}$.
By construction $\mathbb S(E_k)$ is 
the polar subgraph of $\Fke$ restricted to the half-cylinder 
$\{(t,\rho,\theta): t \in (\eps,\longR), \theta\in (0,\pi)\}$. 
More precisely, let $\eta$ be any number\footnote{$\eta=0$ is not allowed,
since in this case the boundary of the subgraph (as a current) does
 not include the set where $\theta=0$.} 
 with $0<\eta<\frac{\pi}{4}$;
then the polar subgraph $$
SG^{\rm pol}_{\Fke}:=\{(t,\rho,\theta)
\in (\eps,l)\times (0,1-\lambda_k'] \times (-\pi/4,\pi) 
:\theta\in (-\eta,\Fke(t,\rho,0))\}$$
satisfies 
\begin{align}\label{subgraph1}
 SG^{\rm pol}_{\Fke}\cap\{\theta\in (0,\pi)\}=\mathbb S(E_k)\cap \{\theta\in (0,\pi)\},
\end{align}
and similarly (for the polar epigraph), setting 
$$UG^{\rm pol}_{-\Fke}:=\{(t,\rho,\theta)
\in (\eps,l)\times (0,1-\lambda_k'] \times (-\pi,\pi/4) 
:\theta\in (-\Fke(t,\rho,0),\eta)\},$$
we have
\begin{align}\label{subgraph2}
 UG^{\rm pol}_{-\Fke}\cap\{\theta\in (-\pi,0)\}=\mathbb S(E_k)\cap \{\theta\in (-\pi,0)\}.
\end{align}

\begin{remark}[\textbf{The sets $\Fke=0$, $\Fke=\pi$}]\rm
Careful attention must be paid to the sets 
$\{\Fke=0\}$ and $\{\Fke=\pi\}$. Indeed on such sets the 
two graphs of $\Fke$ and $-\Fke$ overlap and then, when 
considered as integral currents, they cancel each other. 
Moreover the set $\partial^* \mathbb S(E_k)$ 
includes the two graphs of $\Fke$ and $-\Fke$ with the exception of these two sets.
In other words, from \eqref{subgraph1} and \eqref{subgraph2} we have
\begin{align}\label{symmetrizedset}
 \mathbb S(E_k)\cap C^\eps_l=\big(SG^{\rm pol}_{\Fke}\cap\{\theta\in (0,\pi)\}\big)\cup\big(UG^{\rm pol}_{-\Fke}\cap\{\theta\in (-\pi,0)\}\big)
\end{align}
up to $\mathcal{H}^3$-negligible sets.
From this formula it is evident that  the graphs of $\Fke$ and $-\Fke$  over 
$\{\Theta_k=0\}\cup\{\Theta_k=2\pi\}$ cancel each other, 
and thus they do not belong to the {reduced} boundary of $\mathbb S(E_k)$. 
Moreover, the polar subgraph and the polar epigraph are sets 
of finite perimeter, as is their union in \eqref{symmetrizedset}. 
\end{remark}

\begin{definition}[\textbf{The polar projection map $\pi_0^{\rm pol}$}]
We let
$\pi_0^{\rm pol} = \pi^{{\rm pol}}_{0,\lambda_k',\eps}:\overline {C}_l^\eps(1-\lambda_k')\rightarrow \overline {C}_l^\eps
(1-\lambda_k')\cap\{\theta=0\}$ be the polar projection defined by 
\begin{align}\label{def_pi0}
 \pi_0^{\rm pol}(t,\rho, \theta):=(t,\rho, 0).
\end{align}
\end{definition}
We now introduce various subsets of 
$(0,\longR)\times (0,1)\times \{0\}$ 
in cylindrical coordinates, namely
$ \pi_0^{\rm pol}(\pi_{\lambda_k}\circ \Psi_k(\Omega \setminus 
\overline \sourcedisk_\eps))\subseteq \striptwo\subseteq \striptwo\cup \JQke
\subseteq \stripfour$.
We start with $\striptwo$
 (see also formulas \eqref{eq:stripone} and \eqref{newJ_k} below), and note
preliminarly that 
\begin{equation}
\label{stripe_Jhat}                  
\pi_0^{\rm pol}(\projlambdak \circ \Psi_k(\Omega \setminus \overline \sourcedisk_\eps))
\end{equation} 
coincides with
$$
\left\{(t,\rho,0)\in C_l: t\in (\eps,l),\;\rho\in [\minuk(t)\wedge 
(1-\lambda_k), \maxuk(t)\wedge (1-\lambda_k)]\right\},
$$ 
 $\minuk, \maxuk$ being the functions introduced in \eqref{eq:q_-q_+}.

\begin{definition}
[\textbf{The set $\striptwo$}]\label{def:J_k}
Recalling the expression of $\Qke$ in \eqref{Q_k}, we
define
\begin{align}
 &  \striptwo
 := 
\pi_0^{\rm pol}(\projlambdak \circ \Psi_k(\Omega \setminus \overline \sourcedisk_\eps))
\cup \Big(
((\eps,l)\setminus \Qke)\times [1-\lambda_k,1-\lambda_k']\times\{0\}
\Big),
\label{eq:strip_two}                  
\end{align}
see Fig. \ref{fig_q2}. 
\end{definition}
We have $\striptwo = \pi_0^{\rm pol}
\Big(\projlambdak\circ \Psi_k(\Omega \setminus \overline\sourcedisk_\eps)
\cup \tau\big([1-\lambda_k,1-\lambda'_k]\times A\big)\Big)$, where  
$\tau$ is defined in \eqref{eq:tau}, and 
$A:= \{(t,y) 
\in \projlambdak\circ\Psi_k(\Omega\setminus
\overline \sourcedisk_\eps) 
: t \in (\eps,l), y \in \partial B_{1-\lambda_k}\}$,
since $(t,y) 
\in \projlambdak\circ\Psi_k(\Omega\setminus
\overline \sourcedisk_\eps)$ and $y \in \partial B_{1-\lambda_k}$
implies that $t \in (\eps,\longR) \setminus \Qke$, see Fig. 
\ref{fig_q} and \eqref{eq:pi_k}.

\begin{figure}
	\begin{center}
		\includegraphics[width=0.7\textwidth]{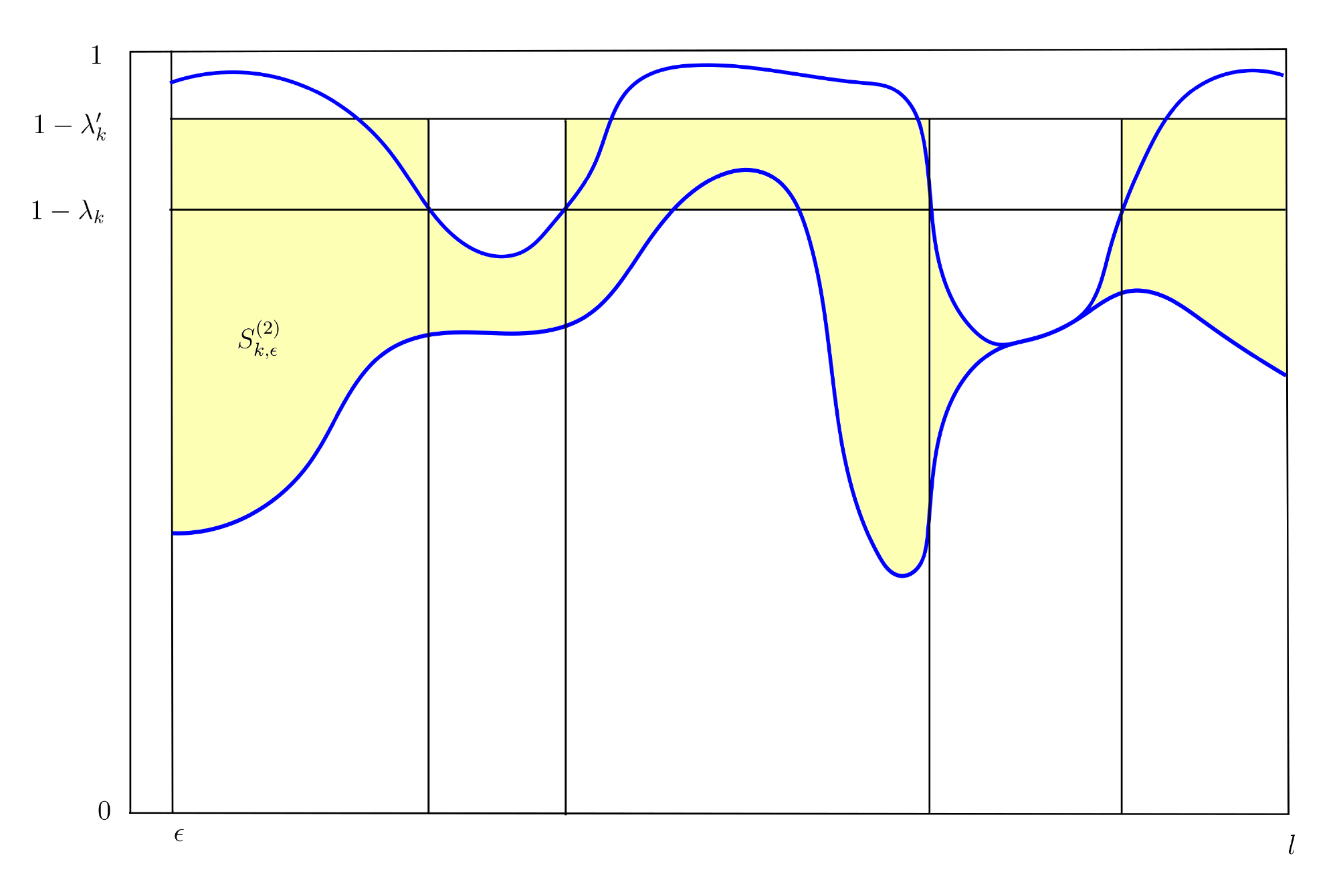}
		\caption{The graphs of the functions $\maxuk$ and $\minuk$ 
			and the set $\striptwo$ in Definition \ref{def:J_k}.
See also Fig. \ref{fig_q}.		}
		\label{fig_q2}
	\end{center}
\end{figure}

\begin{remark}\rm \label{rmk:8.6}
\begin{itemize}
\item[(i)] 
It might happen that $\projlambdak\circ\Psi_k(\badset\setminus \overline{\sourcedisk}_\eps)
= \emptyset$. By construction we have
$$
\pi_0^{\rm pol}(\projlambdak \circ \Psi_k(\Omega \setminus \overline \sourcedisk_\eps))
\cap \Big(
((\eps,l)\setminus \Qke)\times [1-\lambda_k,1-\lambda_k']\times\{0\}
\Big) = ((\eps,l) \setminus \Qke) \times \{1-\lambda_k\} \times 
\{0\}.$$
The two functions $\minuk$ and $\maxuk$ could coincide in large portions 
of $(\eps,l)$ (and even 
everywhere), so that $\pi_0^{\rm pol}(\projlambdak \circ \Psi_k(\Omega \setminus \overline \sourcedisk_\eps))$
could collapse to a curve (for instance if $\Psi_k(\Omega) \subset \partial C_l$);
see also the example in Section \ref{subsec:an_approximating_sequence_of_maps_with_degree_zero:cylinder}. 
On the other hand, 
 $\mathcal H^2(\striptwo)>0$ (see Lemma \ref{lem:estimate_of_Q_k}).
\item[(ii)]
Notice that $A \supseteq\projlambdak\circ\Psi_k(\partial \badset \setminus
\overline \sourcedisk_\eps) \cup \projlambdak\circ\Psi_k((\Omega 
\setminus \overline \sourcedisk_\eps) \setminus \badset)$. 
Moreover $A \cap \projlambdak\circ\Psi_k(\badset)$ may not be empty. 
\item[(iii)] Inside the cylinder $C^\eps_l(1-\lambda_k)$,
 $\striptwo$ is exactly the $\pi_0^{\rm pol}$-projection of 
$\projlambdak\circ \Psi_k(\badset \setminus \overline 
\sourcedisk_\eps)$; 
remember also that $(\projlambdak\circ 
\Psi_k)_{\sharp}(\jump{\badset \setminus \overline 
\sourcedisk_\eps}) = (\partial \mathcal E_k) \res
C_\longR^\eps(1-\lambda_k)$, by \eqref{eq:S_hat_k}, \eqref{eq:S_k} 
and \eqref{eq:primitive}.
\item[(iv)] Recalling the definition of $W_k$ in \eqref{eq:W_k},
\begin{align}\label{projectionofimage}
 \pi_0^{\rm pol}
\Big(\projlambdak\circ\Psi_k(\badset\setminus \overline \sourcedisk_\eps)\cup W_k\Big)\subseteq \striptwo,
\end{align}
and the above inclusion might be strict.
\item[(v)] If $\Fke(t,\rho,0)\in (0,\pi)$ then $(t,\rho,0)\in \striptwo$, by \eqref{projectionofimage}. Indeed in this case the 
circle $(\pi_0^{\rm pol})^{-1}(t,\rho,0)$ intersects both some sets in $\{E_{k,i}\}$ (see \eqref{eq:dec_Ek}) and their complement, so in particular $(\pi_0^{\rm pol})^{-1}(t,\rho,0)$ 
must intersect the {reduced} boundary of some of the sets in $\{E_{k,i}\}$, 
namely  $\projlambdak\circ\Psi_k(\badset\setminus \overline \sourcedisk_\eps)\cup W_k$, for $\mathcal H^2$-a.e. $(t,\rho,0)\in \striptwo$.  Furthermore  
\begin{align*}
 \{(t,\rho,0):\Fke(t,\rho,0)\in (0,\pi)\}=\pi_0^{\rm pol}
\big({\rm spt}(\mathbb S(\piPsiDkXik))\big),
\end{align*}
up to $\mathcal{H}^2-$ negligible sets\footnote{
There could be $(t,\rho,0)
\in \pi_0^{\rm pol}
\big({\rm spt}(\mathbb S(\piPsiDkXik))\big)$ such that 
$\Fke(t,\rho,0)\notin (0,\pi)$.
Indeed take $t \in
(\eps, l)$ such that $\partial \sourcedisk_t \subset  \badset$ and
assume that $\Psi_k(\partial \sourcedisk_t)=\{t\}\times \partial B_\rho$,
$\rho<1-\lambda_k$. Then $\Fke(t,\rho)=\pi$ and $\{t\}\times
\partial B_\rho \subset \partial^*\mathbb S(E_k)$; 
however  this can only happen for $(t,\rho)$ in a negligible
${\mathcal H}^2$-set.}
\end{itemize}
\end{remark}

\begin{remark}\label{remark_Qke}
\begin{itemize}

\item[(i)] $\Theta_k = 2\pi$ on 
$\{(t,\rho,0):t\in\Qke,\rho\in (|u_k|^+(t),1-\lambda_k')\}$. 
Notice that the part of the cylinder $\{(t,\rho,\theta):t\in\Qke,\rho\in (|u_k|^+(t),1-\lambda_k'),\theta\in (-\pi,\pi]\}$ does not intersect 
$\projlambdak\circ \Psi_k(\badset)$, and neither $W_k$, by construction. 
As a consequence it does not intersect ${\rm spt}(\mathbb S(\piPsiDkXik))$.

\item[(ii)] We write
$\{(t,\rho,0)\in \striptwo:~{\rm either }~\Theta_k(t,\rho)=0 ~{\rm or~ }\Theta_k(t,\rho)=2\pi\}
=
\striptwo\cap \{\Theta_k \in \{0,2\pi\}\}$. 
Then
$\striptwo\cap \{\Theta_k \in \{0,2\pi\}\}$ 
corresponds to the values of $t$ 
and $\rho$ for which $(t,\rho,0)\in \striptwo$ and the slice $\mathbb S(E_k)_{(t,\rho)}=\mathbb S(E_k)\cap (\{t\}\times\partial B_\rho)$ is either empty or the whole circle $\{t\}\times \partial B_\rho$ (up to $\mathcal H^1$-negligible sets). Notice also that the intersection $\pi_0^{\rm pol}
\Big(\projlambdak\circ\Psi_k(\badset\setminus \overline \sourcedisk_\eps)\cup W_k\Big) \cap \striptwo\cap \{\Theta_k \in \{0,2\pi\}\} $ may not be empty on a set of positive $\mathcal{H}^2-$measure. Indeed in the proof of Proposition \ref{prop_Gf}, we 
show that the $\pi_0^{\rm pol}$-projection of\footnote{This is the set 
where $\projlambdak\circ\Psi_k(\badset)$ overlaps 
itself with opposite orientation; this set might have positive area, see Fig. \ref{fig:set_cancellation}.} 
$\Big(\projlambdak\circ\Psi_k(\badset)\Big) \setminus
{\rm spt}(\piPsiDkXik)$ 
is contained in $\striptwo\cap \{\Theta_k \in \{0,2\pi\}\}$. 
\end{itemize}
\end{remark}
%
%
\subsection{The current ${\mathbb S}(\piPsiDkXik)$ 
as sum of a polar subgraph and a polar epigraph}
Let $G^{\rm pol}_{\pm \Fke\res \left(
\striptwo\cap \{\Theta_k \in \{0,2\pi\}\}\right)}$ 
be the polar graph of $\pm \Fke\res \left(
\striptwo\cap \{\Theta_k \in \{0,2\pi\}\}\right)$; these two sets, by symmetry, overlap,
and
$$
\jump{G^{\rm pol}_{-\Fke\res \left(\striptwo\cap 
\{\Theta_k \in \{0,2\pi\}\}\right)}}+
\jump{G^{\rm pol}_{\Fke\res \left(\striptwo\cap \{\Theta_k \in \{0,2\pi\}\}\right)}}=0,
$$
due to the fact that $\jump{G^{\rm pol}_{\Fke\res
\left(\striptwo\cap \{\Theta_k \in \{0,2\pi\}\}\right)}}$ 
and $\jump{G^{\rm pol}_{-\Fke\res \left(\striptwo\cap \{\Theta_k \in \{0,2\pi\}\}\right)}}$ are oriented in 
opposite way. Indeed we endow $\jump{G^{\rm pol}_{\Fke\res
\left( \striptwo\cap \{\Theta_k \in \{0,2\pi\}\}\right)}}$ with the orientation inherited by looking at it as the 
boundary of the polar subgraph of $\Fke$, and we endow $\jump{G^{\rm pol}_{-\Fke\res \left(\striptwo\cap \{\Theta_k \in \{0,2\pi\}\}\right)}}$ with 
the opposite orientation,
since we look at it as boundary of an epigraph.

\begin{definition}[\textbf{The currents $\pmcurrentgengraphFkJkzero$}]
\label{def:the_currents_jump_gengraphplusminusFkJkzero}
We 
set
\begin{equation}\label{eq:def_generalizedgraphs_F}
\begin{aligned}
 &\currentgengraphFkJkzero
:=(\partial \jump{SG^{\rm pol}_{\Fke}})\res\big(\{\theta\in (0,\pi)\}\cap C^\eps_l(1-\lambda_k')\big)+
\jump{G^{\rm pol}_{\Fke\res \left(\striptwo\cap \{\Theta_k \in \{0,2\pi\}\}\right)}},  
\\
 &\currentgengraphminusFkJkzero:=
(\partial \jump{UG^{\rm pol}_{-\Fke}})\res\big(\{\theta\in (-\pi,0)\}\cap C^\eps_l(1-\lambda_k')\big)+\jump{G^{\rm pol}_{-\Fke\res \left(\striptwo\cap \{\Theta_k \in \{0,2\pi\}\}\right)}}.
\end{aligned}
\end{equation}
\end{definition}

 The non standard orientation of $
\currentgengraphminusFkJkzero$ 
is chosen in such a way that condition \eqref{eq:punto1} 
in Proposition \ref{prop_Gf} below takes place.
In this proposition we will also see that, being $\mathbb S(E_k)$ a finite perimeter set in $C_l^\eps$, its 
{reduced} boundary, seen as a current, has finite mass. In turn, the integration on its boundary is exactly 
 $ \currentgengraphFkJkzero
+\currentgengraphminusFkJkzero$ (see also \eqref{symmetrizedset}).

\begin{remark}\label{rmk:8.8}\rm
The generalized polar graph of $\Fke$ might have large parts on which $\Fke\in \{0,\pi\}$;
for this reason we neglected this part in the currents introduced in \eqref{eq:def_generalizedgraphs_F} by 
restricting the boundary of the subgraph 
in $\{\theta\in (0,\pi)\}\cap C^\eps_l(1-\lambda_k')$ (and similarly for the epigraph). However we want to consider the graph above the set $\Fke\in \{0,\pi\}$ on the strip $\striptwo$, in particular the 
projection of the set where $\projlambdak\circ\Psi_k(\badset)$ overlaps itself (which may have positive area)\footnote{See the example in Section \ref{sec:catenoid_with_a_flap}.}, for this reason we have to add the term $\jump{G^{\rm pol}_{\Fke\res\left( 
\striptwo\cap \{\Theta_k \in \{0,2\pi\}\}\right)}}$ in formulas \eqref{eq:def_generalizedgraphs_F}. The reason why we have to get rid of the graph of $\Fke$ on $\{\Fke\in \{0,\pi\}\}$ outside $\striptwo$ is that this term is not controlled by the area of $(\projlambdak \circ \Psi_k(\Omega \setminus \overline \sourcedisk_\eps))$ (see also Remark \ref{remark_Qke} (iv)).
\end{remark}

\begin{prop}[\textbf{Estimate of the mass
of $\pmcurrentgengraphFkJkzero$
         }]
\label{prop_Gf}
Let $\eps$ be fixed as in \eqref{eq:H1}  and \eqref{eq:H2},
and recall the definition 
\eqref{symmetrization_of_Sk} of 
$\mathbb S(\piPsiDkXik)$.
 Then the following 
properties hold:
\begin{equation}\label{eq:punto1}
  \currentgengraphFkJkzero
+\currentgengraphminusFkJkzero
=
\mathbb S(\piPsiDkXik)
\res{C^\eps_l(1-\lambda_k')},
 \end{equation}
 \begin{equation}
\label{formulaii}
|  \currentgengraphFkJkzero|+|\currentgengraphminusFkJkzero|
=|\mathbb S(\piPsiDkXik)|_{C^\eps_l(1-\lambda_k')}+2\mathcal H^2\left(
\striptwo\cap \{\Theta_k \in \{0,2\pi\}\}\right),
\end{equation}
\begin{equation}
\label{eq:estimate_useful}
|  \currentgengraphFkJkzero|+|\currentgengraphminusFkJkzero|\leq \int_{\badset\cap 
(\Omega \setminus \overline \sourcedisk_\eps)
}|J(\projlambdak\circ\Psi_k)|~dx+\frac{1}{ n}+o_k(1), 
\end{equation}
where $o_k(1)$ is a nonnegative infinitesimal sequence as $k\rightarrow+\infty$, depending on $n$ and $\eps$.
\end{prop}

\begin{proof} Identity
\eqref{eq:punto1} 
follows by definition and from \eqref{symmetrizedset}.
Concerning 
\eqref{formulaii}, setting for simplicity
\begin{equation}\label{J^i}
\Jkzerotwopi:=\striptwo\cap \{\Theta_k \in \{0,2\pi\}\},
\end{equation}
it is sufficient to 
observe that 
${\rm spt}(\jump{\currentgengraphFkJkzero})$ 
and ${\rm spt}(\jump{\currentgengraphminusFkJkzero})$ 
coincide
on the set $\Fke(\Jkzerotwopi)$ 
(whose measure is equal\footnote{
Indeed $\Fke$ restricted to
$\Jkzerotwopi\cap\{\Theta_k=0\}$ is the identity map and $\Fke$ restricted to 
$\Jkzerotwopi\cap\{\Theta_k=2\pi\}$ is 
a $\pi$-rotation.}
 to the measure of $\Jkzerotwopi$). 
Thus, the currents $\currentgengraphFkJkzero$ and 
$\currentgengraphminusFkJkzero$ cancel each other on this set, since they are endowed with opposite orientation. Hence 
\begin{align}
 | \currentgengraphFkJkzero|+| \currentgengraphminusFkJkzero|
=| \currentgengraphFkJkzero+ \currentgengraphminusFkJkzero|+2\mathcal H^2(\Jkzerotwopi),
\end{align}
and \eqref{formulaii} follows from \eqref{eq:punto1}.

Let us prove 
\eqref{eq:estimate_useful}.
 We recall that the rectifiable set 
$\projlambdak\circ\Psi_k(\badset)\cup W_k$ 
includes 
the support of the current $\piPsiDkXik$. 
There might be parts of $\projlambdak\circ\Psi_k(\badset)\cup W_k$ 
where the multiplicity of $\piPsiDkXik$ is zero, and this happens for 
instance where two pieces of $\projlambdak\circ\Psi_k(\badset)\cup W_k$ 
overlap with opposite orientations. 
We decompose 
$\projlambdak\circ\Psi_k(\badset)\cup W_k$ as follows:
\begin{align}\label{dec_sk}
 \projlambdak\circ\Psi_k(\badset)\cup W_k= Z_k^0\cup 
{\rm spt}(\piPsiDkXik)
=Z_k^0 \cup {\rm spt}(\piPsiDk)
\cup {\rm spt}(\Xik),
\end{align}
where 
$$
Z_k^0:= \Big(\projlambdak \circ \Psi_k(\badset) \setminus {\rm spt}(\piPsiDk)\Big) \cup \Big(W_k\setminus {\rm spt}(\Xik)\Big)
$$
is the set where 
$\piPsiDkXik$ has 
vanishing multiplicity.
It is convenient to introduce
the following notation for the set in 
\eqref{stripe_Jhat}: 
\begin{equation}\label{eq:stripone}
\stripone:=\pi_0^{\rm pol}(\projlambdak \circ \Psi_k(\Omega \setminus \overline \sourcedisk_\eps)).
\end{equation}
We claim that 
\begin{equation}\label{claim_iii}
\stripone
\cap \Jkzerotwopi\subseteq \pi_0^{\rm pol}
\big(Z_k^0\big), 
\end{equation}
where $\pi_0^{\rm pol}
$ is the projection introduced in \eqref{def_pi0} (again, here the inclusion is intended up to $\mathcal{H}^2$-negligible sets).
To prove this we argue by slicing: for $t\in(\eps,l)$ set
$$(\stripone
\cap \Jkzerotwopi)_t:=(\stripone
\cap \Jkzerotwopi)\cap (\{t\}\times\R^2).$$
It is sufficient to show that
\begin{align}\label{claimslice}
 (\stripone
\cap \Jkzerotwopi)_t\subseteq \pi_0^{\rm pol}
\big(Z_k^0\big)
\qquad
{\rm for}~ \mathcal H^1-{\rm a.e.} ~t\in(\eps,l).
\end{align}
In turn, denoting $(Z_k^0)_t:=Z_k^0\cap (\{t\}\times\R^2)$ we will 
prove\footnote{The only  fact we 
will use is that the $\pi_0^{\rm pol}$-projection 
of the set $\projlambdak\circ \Psi_k(\badset)$ is 
surjective on $\striptwo$ (essentially by definition) 
and then the inverse image of a point where 
$\mathbb S(\piPsiDkXik)$ is null is covered at least two times. }
\begin{align}\label{1.80}
 (\stripone\cap \Jkzerotwopi)_t\subseteq \pi_0^{\rm pol}
\big((Z_k^0)_t\big)
\qquad
{\rm for}~ \mathcal H^1-{\rm a.e.}~ t\in(\eps,l).
\end{align}
Now, $(Z_k^0)_t$ is, for $\mathcal H^1$-a.e. $t\in (\eps,l)$, the set  where the coefficient 
of the integral current $(\piPsiDkXik)_t $ is zero. 
Recalling \eqref{J^i}, we have that\footnote{See Fig. \ref{fig:set_cancellation}, 
the two bold segments: on one $\Theta_k=0$ and 
on the other one $\Theta_k=2\pi$}  $\Theta_k(t,\rho)\in\{0,2\pi\}$ for $\rho\in 
[\minuk(t)\wedge (1-\lambda_k),\maxuk(t)\wedge (1-\lambda_k)]$ such that $(t,\rho,0)\in \Jkzerotwopi$. This means that 
 either
\begin{itemize}
 \item for all $i$
the intersection between $E_{k,i}$ (see \eqref{eq:dec_Ek}) and $\{t\}\times \partial B_\rho$ is empty (up to $\mathcal H^1$-negligible sets),
or
 \item for at least  one $i$, 
it happens 
$E_{k,i}\cap( \{t\}\times\partial B_\rho)= \{t\}\times\partial B_\rho$ 
(up to $\mathcal H^1$-negligible sets).
\end{itemize}
In both cases, for $\mathcal H^1$-a.e. 
$\rho\in (\stripone
 \cap \Jkzerotwopi)_t $, the current 
$(\mathbb S(\piPsiDkXik))_t$ is 
null on the set $$\{(t,\rho,\theta):\theta\in (-\pi,\pi),\rho\in (\stripone
 \cap \Jkzerotwopi)_t\}.
$$
Indeed, recalling that $\mathbb S(\piPsiDkXik)=\partial \jump{\mathbb S(E_k)}$, 
in the first case this is obvious, 
in the second one it is sufficient to remember that $E_k=\cup_iE_{k,i}$. 
In other words, the set $(\projlambdak\circ\Psi_k(\badset))_t$ must overlap itself with opposite directions in this set, because the multiplicity of $(\piPsiDkXik)_t$ is null there.
Hence we have proved \eqref{1.80}, and claim \eqref{claim_iii} follows.
 
As a consequence of \eqref{claim_iii} and of its proof, we have
\begin{align}\label{1.81}
 2\mathcal H^2(\Jkzerotwopi)&\leq 2\mathcal H^2\big(\pi_0^{\rm pol}
(Z_k^0\cap C^\eps_l(1-\lambda_k))\big)+2(\lambda_k-\lambda_k')l\nonumber\\
 &\leq \int_{\badset\cap (\Omega \setminus \overline \sourcedisk_\eps)}|J(\projlambdak\circ\Psi_k)|dx+M(W_k\cap C^\eps_l(1-\lambda_k'))-|\piPsiDkXik|_{C^\eps_l(1-\lambda_k')}+o_k(1),
\end{align}
see \eqref{eq:area_of_W_k_counted_with_multiplicities}. Indeed, the first inequality is easy to see, recalling that $\Jkzerotwopi$ is the union of $\stripone
 \cap \Jkzerotwopi$ and $\Jkzerotwopi\setminus \stripone
$, and the latter has measure less than $(\lambda_k-\lambda_k')l$ that is 
infinitesimal as $k\rightarrow+\infty$ (we denote it by $o_k(1)$). 
To see the second inequality we use decomposition \eqref{dec_sk}.
Since $Z_k^0$ is covered at least two times (with opposite directions), 
the area 
$M((\projlambdak\circ\Psi_k(\badset)\cup W_k) \cap 
C^\eps_l(1-\lambda_k'))$ 
of $\projlambdak\circ\Psi_k(\badset)\cup W_k$ in $C^\eps_l(1-\lambda_k')$ 
counted with multiplicity, {\it i.e.},
$$
M((\projlambdak\circ\Psi_k(\badset)\cup W_k)\cap C^\eps_l(1-\lambda_k')):=
\int_{\badset\cap (\Omega \setminus \overline \sourcedisk_\eps)}|J(\projlambdak\circ\Psi_k)|dx+ M(W_k\cap{C^\eps_l(1-\lambda_k')}),$$
 satisfies
$$
\begin{aligned}
M((\projlambdak\circ\Psi_k(\badset)\cup W_k)\cap C^\eps_l(1-\lambda_k'))
\geq &
2\mathcal H^2(Z_k^0\cap C^\eps_l(1-\lambda_k'))+|\piPsiDkXik|_{C^\eps_l(1-\lambda_k')}
\\
\geq &
2\mathcal H^2\big(\pi_0^{\rm pol}
(Z_k^0\cap C^\eps_l(1-\lambda_k'))\big)+|\piPsiDkXik|_{C^\eps_l(1-\lambda_k')},
\\
\geq &
2\mathcal H^2\big(\pi_0^{\rm pol}
(Z_k^0\cap C^\eps_l(1-\lambda_k))\big)+|\piPsiDkXik|_{C^\eps_l(1-\lambda_k')},
\end{aligned}
$$
and \eqref{1.81} follows.

In order to prove \eqref{eq:estimate_useful} it is 
now sufficient to observe that 
\begin{equation*}
\begin{aligned}
& |  \currentgengraphFkJkzero|+|\currentgengraphminusFkJkzero|
= 
\vert \mathbb S(\piPsiDkXik) \vert_{C_l^\eps(1-\lambda_k')}
+2\mathcal H^2(\Jkzerotwopi)
\\
\leq & \vert \mathbb S(\piPsiDkXik) \vert_{C_l^\eps(1-\lambda_k')}
+ 
\int_{\badset\cap (\Omega \setminus \overline \sourcedisk_\eps)}|J(\projlambdak\circ\Psi_k)|dx
+
M(W_k\cap C^\eps_l(1-\lambda_k'))-|\piPsiDkXik|_{C^\eps_l(1-\lambda_k')}
+o_k(1)
\\
\leq & \int_{\badset\cap (\Omega \setminus \overline \sourcedisk_\eps)}|J(\projlambdak\circ\Psi_k)|dx+\frac{1}{n}+o_k(1),
\end{aligned}
\end{equation*}
where we have used \eqref{eq:mass_of_Xik_2} and \eqref{eq:same_inequality_for_the_mass} localized in the cylinder $C_l^\eps(1-\lambda_k')$.

\end{proof}

\begin{figure}
\begin{center}
    \includegraphics[width=0.85
    \textwidth]{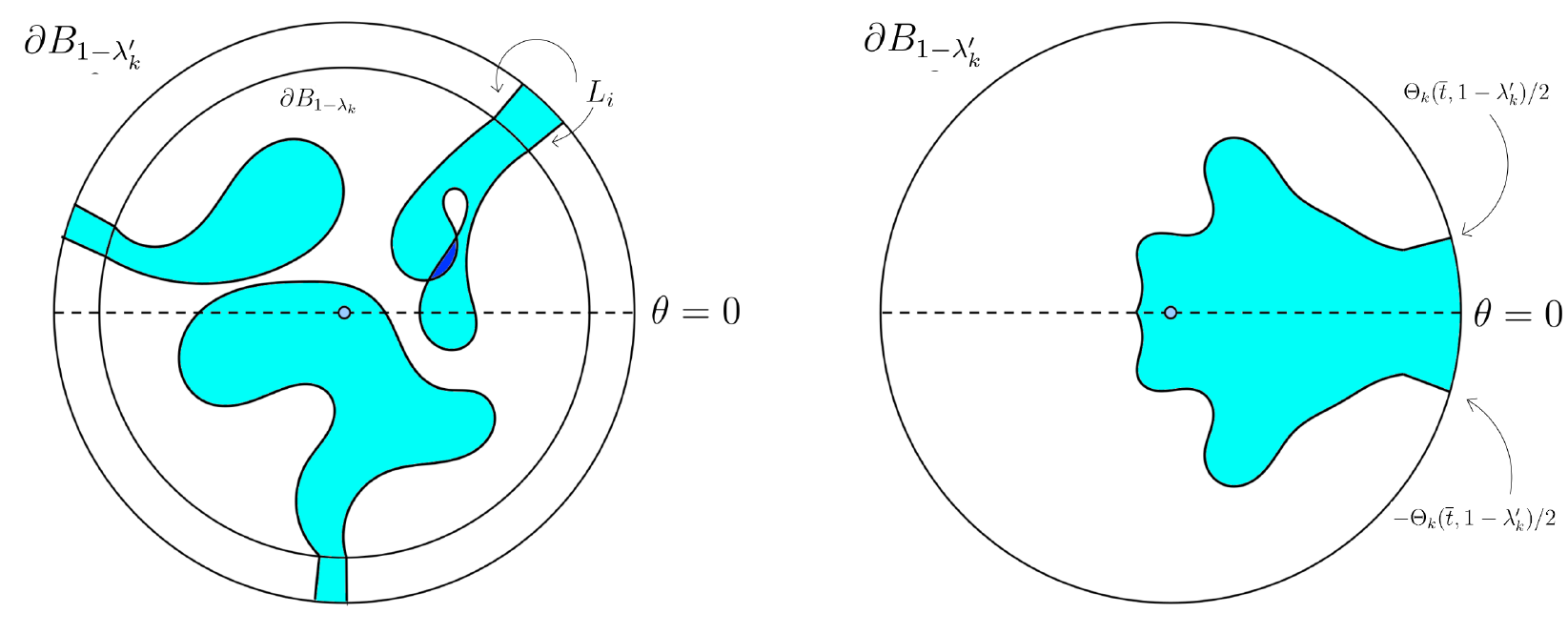}
 \caption{Intersection of the cylinder $C_l(1-\lambda_k)$ with $\{t=\overline t\}\times \R^2$. The 
symmetrization of a 
closed current in $B_{1-\lambda_k'}$, which on the left is emphasized in grey,
 and with dark grey the area in which the multiplicity of the current is $2$. 
The set is bounded by a generic curve with endpoints on  
$\partial B_{1-\lambda_k}$, in turn these endpoints have been 
joined with $\partial B_{1-\lambda_k'}$ by radial segments $L_i$. 
The area emphasized has been symmetrized with the respect to the 
radius $\{\theta=0\}$ in the right picture. In the  picture on the right, we have indicated the angles $\pm \Theta_k(\overline t, 1-\lambda_k')/2$.
}\label{fig3}
\end{center}
\end{figure}

\begin{cor}
We have
$$
\vert \jump{G_{u_k}}\vert_{\badset \cap (\Omega \setminus\overline \sourcedisk_\eps))\times \R^2} 
\geq 
\vert \currentgengraphFkJkzero\vert + \vert \currentgengraphminusFkJkzero \vert -\frac{1}{n} - o_k(1).
$$
\end{cor}
\begin{proof}
It follows from \eqref{eq:estimate_useful}
and \eqref{eq:masses_localized}.
\end{proof}
Now we restrict our attention to the 
rectifiable sets 
${\rm spt}(\pmcurrentgengraphFkJkzero)$,
the supports of the currents in \eqref{eq:def_generalizedgraphs_F}.
We recall that the function $\Fke$ might take values in $(0,\pi)$ 
only in the ``strip'' $\striptwo$, see Remark \ref{rmk:8.6} (v), 
and 
\begin{align*}
 \striptwo
\subset (\eps,l)\times [0,1]\times \{0\}\subset C_l.
\end{align*}
Now we add to 
$\currentgengraphFkJkzero$ 
a graph on some additional  set outside $\striptwo$, see Fig.\ref{fig_q3}.

\begin{definition}
We let
\begin{equation}
\label{strip_Jextension1}
\JQke  
:=  \{(t,\rho,0)\in C_l: t\in \Qke,\;\rho\in [\maxuk(t),1-\lambda_k']\}.
\end{equation}
\end{definition}

\begin{figure}
	\begin{center}
		\includegraphics[width=0.7\textwidth]{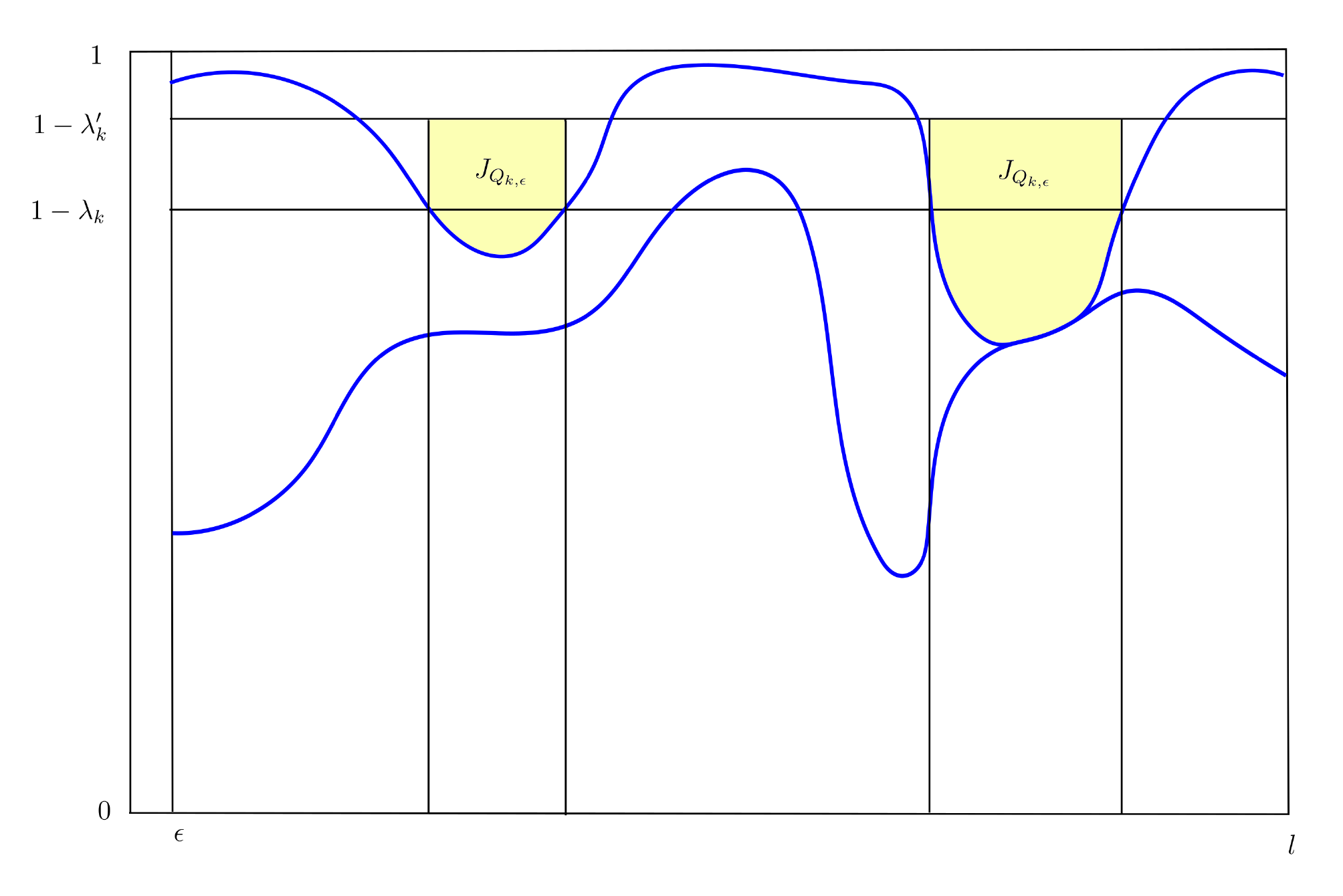}
		\caption{The graphs of the functions $\maxuk$ and $\minuk$ 
			and the set $\JQke$ in \eqref{strip_Jextension1}.
See also Fig. \ref{fig_q}.
		}
		\label{fig_q3}
	\end{center}
\end{figure}

By definition of $\Qke$ in \eqref{Q_k}, we have that for $\mathcal H^2$-a.e. $(t,\rho,0)\in \JQke $ it holds     $(\pi_0^{\rm pol})^{-1}((t,\rho,0))
\cap {\rm spt}(\mathbb S(\piPsiDkXik))=\emptyset$, so that $\Fke\in \{0,\pi\}$ on $\JQke$. Recalling \eqref{crucial_identification} and \eqref{eq:defX_k}, it is not difficult to see that $(\pi_0^{\rm pol})^{-1}(\JQke)\subseteq
\mathbb S(E_k)$. Hence 
$$
\Fke=\pi \qquad {\rm in}~ \JQke
$$
 (see also Remark \ref{remark_Qke}). 

Now, we want to add to the currents 
$\pmcurrentgengraphFkJkzero$ in \eqref{def:the_currents_jump_gengraphplusminusFkJkzero} a new part 
above a region that becomes, in Section
\ref{sec:lower_bound}, 
the subgraph of the function $h$.
\begin{definition}[\textbf{The currents $\GFkethree$ and $\GminusFkethree$}] \label{Def: The current g3}
We define 
\begin{align*}
\GFkethree
:=& 
\currentgengraphFkJkzero
+
\jump{G^{{\rm pol}}_{\Fke\res \JQke }}
\in \mathcal D_2(C^\eps_l(1-\lambda_k')),
\\
\GminusFkethree
:=& 
\currentgengraphminusFkJkzero+\jump{G^{{\rm pol}}_{- \Fke\res \JQke }}\in \mathcal D_2(C^\eps_l(1-\lambda_k')).
\end{align*}
\end{definition}

\begin{lemma}\label{lemma10.3}
The following assertions hold:
\begin{itemize}
\item[(i)] 
\begin{align}\label{123}
 |\GFkethree|=
|\currentgengraphFkJkzero|
+\mathcal H^2(\JQke);
\end{align}
\item[(ii)]
\begin{equation}\label{124}
 \mathcal H^2(\JQke)\leq |\Qke|\leq \frac{1}{2\pi\eps n};
\end{equation}

\item[(iii)]
\begin{equation}\label{eq:G_three_plus_G_minus_three}
 \GFkethree+
\GminusFkethree=\partial\jump{ \mathbb S(E_k)}
{\res C_l^\eps(1-\lambda_k')}.
\end{equation}
\end{itemize}
\end{lemma}

\begin{proof} 
(i) follows from the fact that $\currentgengraphFkJkzero$ and $\jump{G_{\Fke\res \JQke}}$ have disjoint supports, and $|\jump{G_{\Fke\res \JQke}}|=\mathcal H^2(\JQke)$.
(ii) follows from
\begin{equation*}
\mathcal H^2(\JQke)=\int_{\Qke}(1
-\lambda_k'-\maxuk(t))~dt\leq |\Qke|\leq \frac{1}{2\pi\eps n},
\end{equation*}
 where the last inequality is a consequence of Lemma \ref{lem:estimate_of_Q_k}.
(iii) follows as in  Proposition \ref{prop_Gf} using the fact that $\jump{G_{\Fke\res \JQke}}$ and $\jump{G_{- \Fke\res \JQke}}$ have opposite orientation. 
\end{proof}
Current $\GFkethree+
\GminusFkethree$ is closed  in $C_l^\eps(1-\lambda_k')$. 
We can  look at its boundary 
as a current in $\mathcal D_2((\eps,l)\times\R^2)$, 
which stands on the lateral boundary of the cylinder $C_l^\eps(1-\lambda_k')$. To this aim we study the trace of $\Fke$ (that is $\Theta_k(t,\rho)/2$) on the segment  
\begin{equation}\label{eq:L_k}
(\eps, l) \times \{1-\lambda_k'\} \times \{0\}.
\end{equation}
Observe that by definition 
$$
\Fke=\pi \quad {\rm on}~
\Qke\times \{1-\lambda_k'\}\times \{0\}\subseteq (\eps, l) \times \{1-\lambda_k'\} \times \{0\},
$$
whereas on $((\eps, l)\setminus \Qke) \times \{1-\lambda_k'\} \times \{0\}$
we have 
$$\Fke(t,1-\lambda_k',0)=\Theta_k(t,1-\lambda_k')/2=\Theta_k(t,\rho)/2,
\qquad t\in(\eps,l)\setminus \Qke,$$
for all $\rho\in (1-\lambda_k,1-\lambda_k')$.
\begin{definition}[\textbf{The $2$-rectifiable set $\Sigmake$}]
\label{def:Sigmake}
We let 
\begin{equation}\label{eq:tr_p_osp}
\Sigmake:= \Big\{
(t,\rho,\theta):t\in (\eps,l),\;\rho=1-\lambda_k',\;\theta\in (-\Theta_k(t,1-\lambda_k')/2,\Theta_k(t,1-\lambda_k')/2)\Big\}.
\end{equation}
\end{definition}
Referring to the right picture in Figure \ref{fig3}, 
the section of $\Sigma_k^\eps$ is the short arc 
connecting the points 
$(\overline t, 1-\lambda_k', - \Theta_k(\overline t, 1-\lambda_k')/2)$ 
and  
$(\overline t, 1-\lambda_k',  \Theta_k(\overline t, 1-\lambda_k')/2)$;
see also Fig. \ref{fig:tricilindro}.

If we denote by $\jump{\Sigmake}$ the current given by 
integration over $\Sigmake$ (suitably oriented), its boundary  
coincides with the boundary of $\GFkethree+
\GminusFkethree$ on $\partial_{{\rm lat}} C_l^\eps(1-\lambda_k')$. 

\begin{lemma}[\textbf{Properties of $\Sigmake$}]
\label{lem:properties_of_Sigmake}
$\Sigmake$,  
oriented by the outward unit normal to the lateral boundary 
of $C_l(1-\lambda_k')$, is such that
\begin{align*}
 \GFkethree+
\GminusFkethree+\jump{\Sigmake}\in \mathcal D_2((\eps,l)\times\R^2)
\qquad {\rm is ~boundaryless}.
\end{align*}
Moreover
\begin{align}\label{mass_Sigma}
\mathcal H^2(\Sigmake)\leq \frac{1}{\eps n}+o_k(1),
\end{align}
where the sequence $o_k(1)\geq0$ depends on  $n$ and $\eps$, and is 
infinitesimal as $k \to +\infty$.
Finally  
\begin{align}\label{orientation_ofsmallarc}
(\partial\jump{\Sigmake})\res(\{\eps\}\times \R^2)=
\jump{
\{\eps\}\times \{1-\lambda_k' \}
\times \left[
\frac{-\Theta_k(\eps,1-\lambda_k')}{2},\frac{\Theta_k(\eps,1-\lambda_k')}{2}\right]},
\end{align}
oriented counterclockwise\footnote{Looking at the plane
$\{\eps\}\times \R^2$ from the side $t > \eps$.}.
\end{lemma}
\begin{proof} 
The fact that the current $\GFkethree+
\GminusFkethree+\jump{\Sigmake}$ is boundaryless in $\mathcal D_2((\eps,l)\times\R^2)$ is a consequence of the fact that $\Sigma_{k,\eps}$ is 
a subset of the polar subgraph of the trace of $\Fke$ on 
$(\eps,l)\times \{1-\lambda_k'\} \times \{0\}$. Concerning
  \eqref{mass_Sigma} {we have 
\begin{equation}
\mathcal H^2(\Sigmake)= \int_{\eps}^l\int_{-\Theta_k(t,1-\lambda_k')/2}^{\Theta_k(t,1-\lambda_k')/2}(1-\lambda_k')d\theta dt\leq \frac{1}{\eps n}+o_k(1),
\end{equation}
where the last inequality follows from Lemma \ref{lemma_theta}}. As for the last assertion, we have to understand which is the orientation of $\jump{\Sigmake}$, which has been 
chosen in such a way that  $\GFkethree+
\GminusFkethree+\jump{\Sigmake}=
(\partial \jump{\mathbb S(E_k)}) \res ((\eps,l) \times [0,1-\lambda_k'] \times 
\{\theta \in (-\pi,\pi]\})$.
Hence, since $\mathbb S(E_k)$ is contained in $C_l(1-\lambda_k')$, 
the orientation of $\jump{\Sigmake}$ is the one inherited by the external 
normal to $\partial \jump{\mathbb S(E_k)}$, namely the outward
unit normal to the lateral
boundary of $C_l(1-\lambda_k')$.
\end{proof}
\section{Estimate from below of the mass of $\currgraphk$ 
over $\badset\cap \sourcedisk_\eps$}
\label{subsec:symmetrization_of_the_image_of_D_k_cap_B_eps}

We now analyse the image of $\badset\cap \sourcedisk_\eps$ through $\Psi_k$.
We want to reduce this set to a current $\mathcal V_k\in \mathcal D_2(\{\eps\}\times \R^2)$ (defined in \eqref{eq:def_Vk}), in order that it contains the necessary information on the area of $\Psi_k(\badset\cap \sourcedisk_\eps)$. To this aim we need first to describe the boundary of $\mathcal V_k$ and then show that 
its mass gives a lower bound for the area of the graph
of $u_k$ (see formula \eqref{estimate_of_Vk}).

Borrowing the notation from the proof of Lemma \ref{lemma_theta}, 
the set $\partial \sourcedisk_\eps$ 
is splitted as:
\begin{equation}
\label{eq:H_k_eps}
 \partial \sourcedisk_\eps=(\badset\cap \partial \sourcedisk_\eps)\cup ((\Omega\setminus \badset)\cap \partial \sourcedisk_\eps)=:H_{k,\eps}\cup H_{k,\eps}^c.
\end{equation}
We denote by
\begin{align}\label{points_xi}
 \{x_i\}_{i=1}^{I_k}\subseteq 
\{\widehat x_i\}_{i=1}^{J_k}:= \partial \sourcedisk_\eps\cap \partial \badset,
\end{align}
the finite family of points 
(see Lemma \ref{lem:choice_of_u_k_and_t_k} (v)) which 
represents the relative boundary of $H_{k,\eps}$ in 
$\partial \sourcedisk_\eps$.
Recall that $\{\widehat x_i\}_{i=1}^{J_k}$ is finite 
as well by Lemma \ref{lem:choice_of_u_k_and_t_k} (iv). For notational simplicity,
we skip the dependence on $\epsilon$. 

Recalling the definition of $W_k$ in \eqref{eq:W_k},
the following crucial lemma states that 
$(\projlambdak\circ\Psi_k(\badset))\cup W_k$ does not intersect the plane $\{\eps\}\times \R^2$ 
in a set of positive $\mathcal H^2$-measure. 

\begin{lemma}\label{lem:Z_k_verticalpart}
 The rectifiable set $(\projlambdak\circ\Psi_k(\badset))\cup {W_k}$ 
satisfies
 $$\mathcal H^2\Big(\big(\projlambdak\circ\Psi_k(\badset)\cup W_k\big)\cap \{t=\eps\}\Big)=0.$$
\end{lemma}
\begin{proof}
 It is sufficient to show that $\mathcal H^2((\projlambdak\circ\Psi_k(\badset))
\cap \{t=\eps\})=0$ and $\mathcal H^2({W_k}\cap \{t=\eps\})=0$. To show the 
first equality, suppose
$\mathcal H^2(\projlambdak\circ\Psi_k(\badset)\cap \{t=\eps\})>0$. 
Since ${\rm Lip}(\projlambdak)=1$ 
and $\projlambdak$ takes on the plane $\{t=\eps\}$ into itself, we have
$\mathcal H^2(\Psi_k(\badset)\cap \{t=\eps\})>0$. Again, 
being $\Psi_k$ Lipschitz continuous, we deduce that
 $\Psi_k^{-1}(\Psi_k(\badset)\cap \{t=\eps\})$ has positive measure. 
But $\Psi_k^{-1}(\Psi_k(\badset)\cap \{t=\eps\})\subset \Psi_k^{-1}(\{t=\eps\})
=\partial \sourcedisk_\eps$ which has obviously $\mathcal H^2$ null measure.
 
 Let us prove that $\mathcal H^2({W_k}\cap \{t=\eps\})=0$. 
Recalling (see \eqref{eq:W_k}) 
that $W_k=\tau([1-\lambda_k,1-\lambda_k']\times \gamma_k)$ 
with $\gamma_k:=\projlambdak\circ\Psi_k(\partial \badset)$,
and since $\tau(\cdot,z)$ in \eqref{eq:tau}
does not change the axial coordinate 
of $z$, we see\footnote{For instance, using the coarea formula.} that $\tau([1-\lambda_k,1-\lambda_k']\times \gamma_k)
\cap \{t=\eps\}$ has positive $\mathcal H^2$ measure only 
if $\gamma_k\cap\{t=\eps\}$ has positive $\mathcal H^1$ measure.
Again, since also $\projlambdak$ does not change the axial coordinate,  
as before this happens only if $\Psi_k^{-1}  (\widehat \gamma_k\cap\{t=\eps\})$ has positive $\mathcal H^1$-measure, where $\widehat\gamma_k:=\Psi_k(\partial \badset)$;  by Lemma \ref{lem:choice_of_u_k_and_t_k},  this is not possible, since 
we know that $\widehat \gamma_k\cap\{t=\eps\}=
\{ \Psi_k(\widehat x_i)\}$ (see \eqref{points_xi}), and then $\Psi_k^{-1}(\widehat \gamma_k\cap\{t=\eps\})= \{\widehat x_i\}$ which is a finite set. 
 \end{proof}

  We recall from \eqref{eq:S_hat_k} and \eqref{eq:S_k} that  
\begin{align}\label{def_sk}
\piPsiDkXik=(\projlambdak\circ\Psi_k)_\sharp\jump{\badset}+ \Xik.
\end{align}
An immediate consequence of Lemma \ref{lem:Z_k_verticalpart},
formula \eqref{def_sk}, and the fact that $\piPsiDkXik$
is boundaryless in $C_l(1-\lambda_k')$, is the following:
 
 \begin{cor}\label{cor:slice_eps}
We have $\piPsiDkXik\res \{t=\eps\}=0$. In particular 
\begin{align*}
 \partial (\piPsiDkXik\res \{\eps<t<l\})\res(\{t=\eps\})=-\partial (\piPsiDkXik\res \{-1<t<\eps\})\res(\{t=\eps\}) \qquad \text{ in }C_l(1-\lambda_k').
\end{align*}
 \end{cor}

 If $\{E_{k,i}\}_{i\in\mathbb N}$, $E_{k,i}\in C_l$ are the sets which we 
have symmetrized (see \eqref{eq:dec_Ek}), $\mathbb S(E_k)$ is the symmetrized set, and $\mathbb S(\piPsiDkXik)$ is the symmetrized current, 
we have to understand the behaviour of $\mathbb S(\piPsiDkXik)$ on $\{\eps\}\times \R^2$. 
 We have observed that 
$\piPsiDkXik\res (\{\eps\}\times\R^2)=0$ 
because $\mathcal H^2\left((\projlambdak\circ\Psi_k(\badset)\cup W_k)
\cap \{\eps\}\times\R^2\right)=0$. 
The same holds for the symmetrized current, as a particular consequence
of Lemma \ref{lem:sections_of_symm_current}:
 \begin{align*}
\mathbb S(\piPsiDkXik)\res (\{\eps\}\times\R^2)=0.  
 \end{align*}
 
 \subsection{Description of the boundary of the 
current  $\mathbb S(\piPsiDkXik)\res ((-1,\eps)\times B_{1-\lambda_k'})$}
Our first aim is to describe the boundary of  $\mathbb S(\piPsiDkXik)$ 
on $\{\eps\}\times \R^2$ (Corollary \ref{cor_9.5}). 
To do so, let us recall that $\projlambdak$ is given
 in Definition \ref{def:the_projection_pi_k} and that 
the points $x_i$ are defined in  \eqref{points_xi}.

\begin{definition}[\textbf{The current $\scriptHkeps$}]
Recalling \eqref{eq:H_k_eps}, we set
\begin{align}\label{eqn:scriptHkeps}
\scriptHkeps
:=
(\projlambdak\circ\Psi_k)_\sharp\jump{H_{k,\eps}}\in \mathcal D_1(\{\eps\}\times B_{1}),
\end{align}
where $H_{k,\eps}$ is oriented counterclockwise.
\end{definition}

Let us denote by $\{\widetilde x_i\}\subseteq\{x_i\}$ the points which represent the support of the current $\partial \jump{H_{k,\eps}}$.
We can consider 
the orthogonal projection\footnote{Defined  
at least in the region $C_l(1-\lambda_k')\setminus C_l(1-\lambda_k)$.} onto 
the lateral boundary of 
$C_l(1-\lambda_k')$,
and we denote by $L_{k,i}$ the segment connecting $\projlambdak( \Psi_k(\widetilde x_i))$ 
(which belongs to the lateral boundary of 
$C_l(1-\lambda_k)$)
to 
the image point of $ \Psi_k(\widetilde x_i)$ through this projection.

We consider the $1$-integral current in  $\{\eps\}\times B_{1-\lambda_k'}$ given by
\begin{align}\label{eq:def_S_eps_k}
  \scriptHkeps
+\sum_i\jump{L_{k,i}}\in \mathcal D_1(\{\eps\}\times B_{1-\lambda_k'}), 
\end{align}
where $\jump{L_{k,i}}$ are the integrations over the segments $L_{k,i}$ taken with 
suitable orientation in order that 
\begin{equation}
\label{eq:boundary_of_H_keps}
\partial\Big(  \scriptHkeps
+\sum_i\jump{L_{k,i}}\Big)=0 \qquad {\rm in}~ \{\eps\}\times B_{1-\lambda_k'}.
\end{equation}

Before stating the following crucial lemma, we recall that the current $\piPsiDkXik$ is defined in $C_l$ but is supported in $[0,\longR]\times \overline B_{1-\lambda_k'}$.

\begin{lemma}[\textbf{Boundary of $\piPsiDkXik \res((-1,\eps)\times\R^2)$ in 
$\{\eps\}\times B_{1-\lambda_k'}$}]
We have 
\begin{align}\label{slice_eps_S_k}
 \partial\big(\piPsiDkXik\res(
(-1,\eps)\times\R^2)
\big)= \scriptHkeps+\sum_i\jump{L_{k,i}}\qquad{\rm in }~\mathcal D_1(\{\eps\}\times B_{1-\lambda_k'}).
\end{align} 
\end{lemma}
 
 \begin{proof}
 We recall that 
$$
\piPsiDkXik=\piPsiDk+\Xik,
$$  
where $\piPsiDk$ is defined in \eqref{eq:S_hat_k} and,
by  \eqref{eq:def_Xi_k},
$\Xik=\widetilde \tau_\sharp\jump{[1-\lambda_k,1-\lambda_k']\times \partial \badset }$. Observe that 
$$
\partial \left(
\Xik\res ((-1,\eps)\times\R^2)\right)
= 
\sum_i\jump{L_{k,i}} 
\qquad{\rm in~ the~ annulus~ } 
\{\eps\}\times (B_{1-\lambda_k'}\setminus \overline{B}_{1-\lambda_k}).
$$
Indeed, this follows from 
the definition of $L_{k,i}$, the equality\footnote{
Notice that $\partial(\badset \cap \sourcedisk_\eps)=(\partial \badset \cap \sourcedisk_\eps)\cup (\partial \badset \cap \partial \sourcedisk_\eps) \cup (\badset \cap \partial \sourcedisk_\eps)$; recall also that, by Lemma 
\ref{lem:choice_of_u_k_and_t_k} (iv), 
$\partial \badset \cap \partial \sourcedisk_\eps$ consists of a finite set of points.}
$$
\begin{aligned}
\Xik\res ((-1,\eps)\times\R^2)
=&
\widetilde \tau_\sharp\jump{[1-\lambda_k,1-\lambda_k']\times 
(\sourcedisk_\eps \cap
\partial \badset)}
\\
=&
\widetilde \tau_\sharp\jump{[1-\lambda_k,1-\lambda_k']\times \partial 
(\badset \cap \sourcedisk_\eps)}-\widetilde \tau_\sharp\jump{[1-\lambda_k,1-\lambda_k']\times H_{k,\eps}},
\end{aligned}
$$
and \eqref{eq:boundary_of_H_keps}. 
Moreover, from \eqref{eq:S_hat_k}, 
$$
\begin{aligned}
\partial \Big(\piPsiDk\res((-1,\eps)\times\R^2)\Big)\res 
(\{\eps\}\times B_1)
& =
\partial \Big(((\projlambdak\circ \Psi_k)_\sharp \jump{\badset})\res((-1,\eps)\times\R^2)\Big)\res (\{\eps\}\times B_1)
\\
& = 
\partial \Big((\projlambdak\circ \Psi_k)_\sharp \jump{\badset\cap \sourcedisk_\eps }\Big)\res (\{\eps\}\times B_1)
\\
& = 
\Big((\projlambdak\circ \Psi_k)_\sharp \partial \jump{\badset\cap \sourcedisk_\eps }\Big)\res (\{\eps\}\times B_1)
\\
& = 
\scriptHkeps \qquad{\rm on}~ \{\eps\}\times B_1,
\end{aligned}
$$
where in the last equality we use\footnote{
Here we take the boundary of $\badset$ in $\partial \sourcedisk_\eps$ in the sense of currents, so that isolated points are neglected.} 
$\jump{\partial (\badset \cap \sourcedisk_\eps)} = 
\jump{\badset \cap \partial \sourcedisk_\eps}=\jump{H_{k,\eps}}$ 
on $\partial \sourcedisk_\eps$.
\end{proof}

Thanks to Corollary \ref{boundaryat0},
both $\piPsiDkXik$ and $\mathbb S(\piPsiDkXik)$ have no boundary in   $(-\infty,l)\times B_{1-\lambda_k'}$.
Now, we need 
to describe the boundary of the symmetrized current $\mathbb S(\piPsiDkXik)$
restricted to $(-1,\eps) \times B_{1-\lambda_k'}$,
see \eqref{eq:boundary_S_k_epsilon}.
We 
recall the definitions of  $\mathcal X_k$ and 
$\mathcal Y_k$ in \eqref{eq:defX_k} and \eqref{eq:defY_k}, 
and for $t\in (-1,\eps]$ and $\rho\in (1-\lambda_k,1-\lambda_k')$ 
 the function $ \Theta_k(t,\rho)$ defined in \eqref{theta_arcs}. Also in this case 
\begin{equation*}
 \Theta_k(t,\rho)=\Theta_k(t,\varrho)\qquad\text{ for all }\rho,\varrho\in (1-\lambda_k,1-\lambda_k').
\end{equation*}
In cylindrical coordinates, if 
$$
X_1:=(\eps,1-\lambda_k,\Theta_k(\eps,1-\lambda_k)/2),
\qquad 
X_2:=(\eps,1-\lambda_k,-\Theta_k(\eps,1-\lambda_k)/2),
$$
we denote the two $1$-currents 
\begin{align}\label{def_X1X2}
 \mathbb S( L)_1:=\tau(\cdot,X_1)_\sharp\jump{(1-\lambda_k,1-\lambda_k')}\qquad\text{ and }\qquad \mathbb S(L)_2:=\tau(\cdot,X_2)_\sharp\jump{(1-\lambda_k,1-\lambda_k')},
\end{align}
see Fig. \ref{fig:orientation}.
Set 
\begin{equation}
\label{eq:Y_1_Y_2}
Y_1=\tau(1-\lambda_k',X_1), \qquad Y_2=\tau(1-\lambda_k',X_2).
\end{equation}
We know, by construction and definition of $\Theta_k$, that
\begin{align*}
\partial \big( \mathbb S(\piPsiDkXik)\res ((-1,\eps)\times( B_{1-\lambda_k'}\setminus \overline{B}_{1-\lambda_k}))\big)\res(\{\eps\}\times\R^2)=\mathbb S( L)_1-\mathbb S( L)_2
\end{align*}
in $\mathcal D_1(\{\eps\}\times( B_{1-\lambda_k'}\setminus \overline{B}_{1-\lambda_k}))$.
We define
\begin{align}\label{eqn:SscriptHkeps}
\mathbb S(\scriptHkeps):= \partial \big( \mathbb S(\piPsiDkXik)\res ((-1,\eps)\times B_{1-\lambda_k'})\big)\res(\{\eps\}\times\R^2)-\mathbb S( L)_1+\mathbb S( L)_2
\end{align}
in $\mathcal D_1(\{\eps\}\times B_{1-\lambda_k})$, see again Fig. \ref{fig:orientation}.
With these definitions at our disposal we can now write
\begin{equation}\label{eq:boundary_S_k_epsilon}
\begin{aligned}
& \partial\Big(\mathbb S(\piPsiDkXik)\res ((-1,\eps)\times 
B_{1-\lambda_k'})\Big)
\\
=& 
\mathbb S(\scriptHkeps)+\mathbb S( L)_1-\mathbb S( L)_2+\partial\big(\mathbb S(\piPsiDkXik)\res\{t\in(-1,\eps)\}\big)\res(\{-1\}\times 
B_{1-\lambda_k'})\\
=& 
\mathbb S(\scriptHkeps)+\mathbb S( L)_1-\mathbb S( L)_2.
\end{aligned}
\end{equation}
Here we have used once again that $\mathbb S(\piPsiDkXik)$ is supported in 
$[0,\longR]\times B_1$, and then its boundary on $\{t=-1\}$ is always null.

We can clarify the meaning of the last term in formula \eqref{eq:boundary_S_k_epsilon}.

\begin{cor}\label{cor_9.5}We have 
$$
\mathbb S(\scriptHkeps)+\mathbb S( L)_1-\mathbb S( L)_2= -\partial\big(\mathbb S(\piPsiDkXik)\res((\eps,l)\times B_{1-\lambda_k'})\big)\res (\{\eps\}\times \R^2).
$$
\end{cor}

\begin{figure}
	\begin{center}
		\includegraphics[width=0.6\textwidth]{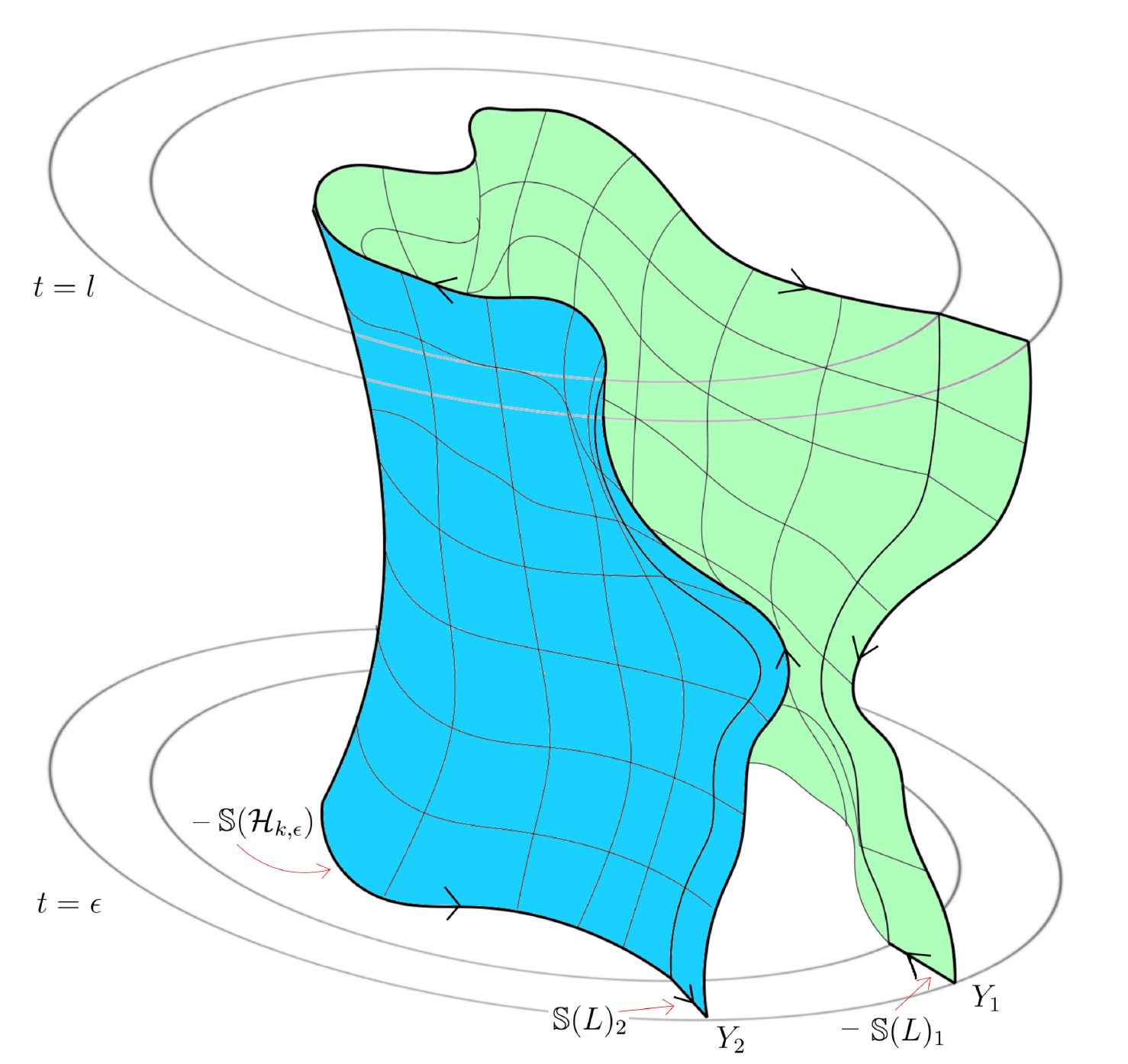}
		\caption{The current $\mathbb S(\piPsiDkXik)\res((\eps,l)\times B_{1-\lambda_k'})$ is depicted. At $t=\eps$ we emphasized the various objects composing its boundary, taken with their orientation. 
		}
		\label{fig:orientation}
	\end{center}
\end{figure}

\begin{proof}
It follows
from \eqref{eq:boundary_S_k_epsilon}, Corollary \ref{cor:slice_eps}, and 
Lemma \ref{lem:sections_of_symm_current}.
\end{proof}

\subsection{Construction of the current $\Vke$}
Let
$\Pi_\eps : \R^3 \to 
\{\eps\}\times \R^2$ be
the orthogonal projection on 
$\{\eps\}\times \R^2$.

\begin{definition}We set
\begin{align}\label{eq:def_Vk}
 \Vke:=(\Pi_\eps)_\sharp\Big(
\mathbb S(\piPsiDkXik)\res ((-1,\eps)\times B_{1-\lambda_k'})\Big) {\in \mathcal D_2(C_l)}.
\end{align}
\end{definition}
\begin{lemma}
We have 
\begin{equation*}
|\currgraphk|_{(\badset\cap \sourcedisk_\eps)\times \R^2} \geq 
 |\Vke| - 
2\pi(\lambda_k-\lambda_k').
\end{equation*}
\end{lemma}

\begin{proof}
By \eqref{eq:def_Vk}, since ${\rm Lip}(\Pi_\eps)=1$, we have,
using \eqref{def_sk}, \eqref{eq:def_Xi_k},
\begin{equation}
\begin{aligned}\label{estimate_of_Vk}
 |\Vke|
= &  
|\Vke|_{(-1,\eps)\times (B_{1-\lambda_k'}\setminus B_{1-\lambda_k})}
+
|\Vke|_{(-1,\eps)\times B_{1-\lambda_k}}
\\
\leq &  
|\Vke|_{(-1,\eps)\times (B_{1-\lambda_k'}\setminus B_{1-\lambda_k})}
+
|\piPsiDkXik|_{(-1,\eps)\times B_{1-\lambda_k}}
\\
= &  
|\Vke|_{(-1,\eps)\times (B_{1-\lambda_k'}\setminus B_{1-\lambda_k})}
+
|\piPsiDk|_{(-1,\eps)\times B_{1-\lambda_k}}
\\
\leq & 
2\pi(\lambda_k-\lambda_k')
+ 
|\currgraphk|_{(\badset\cap \sourcedisk_\eps)\times \R^2},
\end{aligned}
\end{equation}
where we have also 
used a localized version of \eqref{eq:same_inequality_for_the_mass} 
in $(-1,\eps) \times 
B_{1-\lambda_k}$.
\end{proof}

By Corollary \ref{cor_9.5} it  holds\footnote{Recall that $\partial (\Pi_\eps)_\sharp\Big(
\mathbb S(\piPsiDkXik)\res\{t<\eps\}\Big)= (\Pi_\eps)_\sharp \partial \Big(
\mathbb S(\piPsiDkXik)\res\{t<\eps\}\Big)$ and that the map $\Pi_\eps$ 
does not move the plane where $\partial (\mathbb S(\piPsiDkXik)\res \{t<\eps\})$ is supported. }
\begin{align}\label{eq:boundary_Vk}
 \partial \Vke=
\mathbb S(\scriptHkeps)
+
\mathbb S( L)_1-\mathbb S( L)_2\qquad \text{ in }\{\eps\}\times B_{1-\lambda_k'}.
\end{align}
Clearly the above
current is boundaryless in $\{\eps\}\times B_{1-\lambda_k'}$; more
precisely it is an oriented  curve connecting $Y_2$ to $Y_1$ 
(defined in \eqref{eq:Y_1_Y_2}) as soon as $Y_2 \neq Y_1$, 
with
$\mathbb S(\scriptHkeps)$ clockwise oriented\footnote{
When looking at the plane $\{\eps\}\times \R^2$ from 
$t>\eps$.}.
If we extend $\Vke$ to $0$ on the whole plane $\{\eps\}\times \R^2$ 
(keeping the same notation) we have 
\begin{align}\label{eqn:9.16}
 \partial \Vke
=
\mathcal L_k
+\mathbb S(\scriptHkeps)+\mathbb S( L)_1-\mathbb S( L)_2\qquad \text{ on }\{\eps\}\times \R^2,
\end{align}
for some current  $\mathcal L_k$ supported on $\{\eps\}\times\partial B_{1-\lambda_k'}$ and whose boundary is two deltas,  with suitable signs, on $Y_1$ and $Y_2$. In particular $\mathcal L_k$
is the integration between $Y_1$ to $Y_2$
on the circle $\{\eps\}\times \partial B_{1-\lambda_k'}$. 

However there are two arcs which connect these two points, namely (in cylindrical coordinates)
\begin{align}\label{eq:def_Y12}
\{\eps\}\times \{1-\lambda_k' \}
\times \left[
\frac{-\Theta_k(\eps,1-\lambda_k')}{2},\frac{\Theta_k(\eps,1-\lambda_k')}{2}\right] 
\end{align}
oriented clockwise
and 
\begin{align}\label{eq:def_Y12bis}
\left(
\{\eps\}\times \partial B_{1-\lambda_k'}\right)
\setminus 
\left\{
\{\eps\}\times \{1-\lambda_k' \}
\times \left[
\frac{-\Theta_k(\eps,1-\lambda_k')}{2},\frac{\Theta_k(\eps,1-\lambda_k')}{2}\right]\right\} 
\end{align}
oriented counterclockwise.
We have to identify $\mathcal L_k$ with the integration over one of 
these two arcs.

\begin{prop}\label{prop:V_k}
 $\mathcal L_k$ 
is the 
counterclockwise
integration over the arc connecting $Y_1$ and $Y_2$ given by 
\eqref{eq:def_Y12bis}. 
\end{prop}
Before proving this proposition we anticipate a
useful observation.

\begin{remark}\label{rem:V_k}
We set 
$$
\mathbb S(E_k)_\eps:=\mathbb S(E_k)\cap \{t=\eps\}. 
$$
Since $\mathbb S(\piPsiDkXik)$ is the boundary of the integration over $\mathbb S(E_k)$,
the current 
$\mathbb S(\piPsiDkXik)\res((-1,\eps)\times B_{1-\lambda_k'})+\jump{\mathbb S(E_k)_\eps}$ is 
boundaryless in $\mathcal D_2(C_l(1-\lambda_k'))$ 
(with $\mathbb S(E_k)_\eps$ suitably oriented). It follows, invoking Corollary \ref{cor_9.5}, that 
\begin{align*}
 \partial \jump{\mathbb S(E_k)_\eps}=-\mathbb S(\scriptHkeps)-\mathbb S( L)_1+\mathbb S( L)_2\qquad \text{ in }\{\eps\}\times B_{1-\lambda_k'}.
\end{align*}
The fact that 
\begin{align}\label{eq:boundary_Vkbis}
 \partial \Vke= \mathcal L_k+\mathbb S(\scriptHkeps)+\mathbb S( L)_1-\mathbb S( L)_2\qquad \text{ in }\{\eps\}\times \R^2,
\end{align}
(where $\mathcal L_k$ is as in Proposition 
\ref{prop:V_k}) means that 
$\Vke$ 
is the integration over the set $$B_{1-\lambda_k'}\setminus \mathbb S(E_k)_\eps.$$
 In particular $\Vke$ 
has coefficient $1$ in $B_{1-\lambda_k'}\setminus \mathbb S(E_k)_\eps$ and zero in $\mathbb S(E_k)_\eps$.
 On the other hand, if $\mathcal L_k$ were the integration over
\eqref{eq:def_Y12} oriented clockwise,
then we would have that 
$\Vke$ had coefficient $-1$ in $\mathbb S(E_k)_\eps$ and $0$ in $B_{1-\lambda_k'}\setminus \mathbb S(E_k)_\eps$.
\end{remark}

We can now prove Proposition \ref{prop:V_k}.

\begin{proof}
Appealing to Remark \ref{rem:V_k}, it is sufficient to show that 
the coefficient of $\Vke$ is 
$1$ in  $B_{1-\lambda_k'}\setminus \mathbb S(E_k)_\eps$. Equivalently we can show that this coefficient is zero in $B_{1-\lambda_k'}\cap \mathbb S(E_k)_\eps$.

Let us recall, by definitions \eqref{eq:defY_k} and \eqref{eq:defX_k}, 
\begin{align}
\label{eq:slice_Y_t}
 &(\mathcal Y_k)_t
=\widetilde \tau_\sharp\jump{[1-\lambda_k,1-\lambda_k']\times ((\Omega\setminus \badset)\cap \partial \sourcedisk_t)} 
\qquad {\rm for~a.e.}~t\in (0,\eps],
\\
\label{eq:slice_X_t}
 &(\mathcal X_k)_t=\jump{\{t\}\times (B_{1-\lambda_k'}\setminus \overline {B}_{1-\lambda_k})}-(\mathcal Y_k)_t
\qquad \ \ \qquad {\rm for~a.e.}~t\in (0,\eps].
\end{align}
Recalling
Lemma 
\ref{lem:choice_of_u_k_and_t_k}(i), we now divide our analysis in two cases:
\begin{itemize}
 \item[(1)] $|u_k(0)|<1-\lambda_k$.
 \item[(2)] $|u_k(0)|>1-\lambda_k$.
\end{itemize}
We notice that, in both cases, by continuity of $u_k$, for 
all $\delta \in (0,1)$ there is $t_k^\delta>0$ such that 
\begin{align}\label{inclusion_ball}
u_k(\sourcedisk_t)\subset B_\delta(u_k(0)) \qquad
\forall t\in (0,{t_k^\delta}].
\end{align}

\textit{Case (1):}
If $\delta$ is sufficiently small, we can also assume that 
\begin{equation}\label{eq:B_delta_contained_in_B_1_minus_lambda_k}
B_\delta(u_k(0))\subset B_{1-\lambda_k}(0),
\end{equation}
and therefore
\begin{align}
\label{inclusions_case1}
 u_k(\sourcedisk_{t})
\subset B_{1-\lambda_k}(0) \qquad
\forall t\in (0,{t_k^\delta}].
\end{align}
In this case it turns out that if $t\leq t_k^\delta$ then the 
current $(\mathcal Y_k)_t$ in \eqref{eq:slice_Y_t} is null, 
because $\vert u_k\vert^+(t) < 1 - \lambda_k$, hence
$(\Omega\setminus \badset)\cap \partial \sourcedisk_t=\emptyset$
by \eqref{uk_outside_Dk}. In particular, by \eqref{eq:slice_X_t}, 
\begin{align*}
 (\mathcal X_k)_t=\jump{\{t\}\times (B_{1-\lambda_k'}\setminus \overline {B}_{1-\lambda_k})}
\qquad {\rm for~a.e.~} t\leq t_k^\delta. 
\end{align*}
Eventually, since $\overline{\sourcedisk}_t \subset \badset$ 
for any $t \in [0,t_k^\delta]$,
from \eqref{inclusion_ball}, we also deduce
$
u_k(\overline \sourcedisk_{t^\delta_k}\cap \badset)
=u_k(\overline \sourcedisk_{t_k^\delta})\subset B_\delta(u_k(0))$, so that 
\begin{align}\label{inclusions_case1_bis}
 \Psi_k(\badset)\cap([0, t_k^\delta]\times B_1) = 
 \projlambdak\circ\Psi_k(\badset)\cap([0, t_k^\delta]\times B_1)
\subset [0,t_k^\delta]\times B_\delta(u_k(0)).
\end{align}

Now, consider 
the decomposition \eqref{eq:dec_Ek} of $\mathcal E_k$.
 By the crucial identification \eqref{crucial_identification} and \eqref{eq:B_delta_contained_in_B_1_minus_lambda_k} we infer that there must be a set $ E_{k,h}\in \{E_{k,i}\}_{i\in\mathbb N}$ with\footnote{Since the decomposition in \eqref{eq:dec_Ek} is done in undecomposable components, such a set is unique.}
\begin{align*}
 \mathcal X_k\res ((-1,t_k^\delta)\times B_{1-\lambda_k'})&=\jump{(-1,t_k^\delta)\times(B_{1-\lambda_k'}\setminus \overline {B}_{1-\lambda_k})}\nonumber\\
 &=\jump{E_{k,h}\cap \big((-1,t_k^\delta)\times(B_{1-\lambda_k'}\setminus \overline {B}_{1-\lambda_k})\big)}.
\end{align*}
 Therefore 
\begin{align*}
 E_{k,h}\cap \big((-1,t_k^\delta)\times(B_{1-\lambda_k'}\setminus \overline {B}_{1-\lambda_k})\big)=(-1,t_k^\delta)\times(B_{1-\lambda_k'}\setminus \overline {B}_{1-\lambda_k}).
\end{align*}
This has the following consequence: denoting as usual 
$\mathbb S(E_{k,h})$ the cylindrical symmetrization of $E_{k,h}$ we infer 
\begin{align*}
 \mathbb S(E_{k,h})\cap \big((-1,t_k^\delta)\times(B_{1-\lambda_k'}\setminus \overline {B}_{1-\lambda_k})\big)=(-1,t_k^\delta)\times(B_{1-\lambda_k'}\setminus \overline {B}_{1-\lambda_k}),
\end{align*}
and since $\mathbb S(E_{k,h})\subset \mathbb S(E_k)$ we also have
\begin{align}\label{lastpointequation}
 \mathbb S(E_k)\cap \big((-1,t_k^\delta)\times(B_{1-\lambda_k'}\setminus \overline {B}_{1-\lambda_k})\big)=(-1,t_k^\delta)\times(B_{1-\lambda_k'}\setminus \overline {B}_{1-\lambda_k}).
\end{align}
We now consider two subcases.
\begin{itemize}
 \item[(1A)] $\mathcal H^2\Big(\big(\{\eps\}\times(B_{1-\lambda_k'}\setminus B_{1-\lambda_k})\big) \setminus\mathbb S(E_k)_\eps\Big) >0$. To conclude the proof it is sufficient to show that 
\begin{align}\label{claim_case(a)}
\text{the multiplicity of }\Vke \text{ on } \big(\{\eps\}\times(B_{1-\lambda_k'}\setminus B_{1-\lambda_k})\big) \setminus\mathbb S(E_k)_\eps\text{ is }1, 
\end{align}
 because $(\{\eps\}\times B_{1-\lambda_k'})\setminus \mathbb S(E_k)_\eps$ is, 
by definition, outside the finite perimeter set $\mathbb S(E_k)$.

We argue by slicing, and consider the lines $l_{\rho,\theta}$ in $\R^3$ given by $l_{\rho,\theta}=\R\times \{\rho\}\times\{\theta\}$, with $\rho$ and $\theta$ fixed.
Consider any point $p_0$ of coordinates $\rho\in (1-\lambda_k,1-\lambda_k')$ and $\theta\in (-\pi,\pi]$ such that 
\begin{align}\label{slicing1d}
p_0\in  (\{\eps\}\times B_{1-\lambda_k'})\setminus \mathbb S(E_k)_\eps.
\end{align}
For a.e. such $\rho\in (1-\lambda_k,1-\lambda_k')$ and $\theta\in (-\pi,\pi]$ 
the slice of $\piPsiDkXik\res ((-1,\eps)\times B_{1-\lambda_k'})$ with respect to this line is the sum of some Dirac deltas with 
suitable signs, according to the orientation of $\piPsiDkXik$. Indeed $\piPsiDkXik$ is the integration over the boundary of the finite perimeter set $\mathbb S(E_k)$, so it turns out that, for a.e. $\rho\in (1-\lambda_k,1-\lambda_k')$ and $\theta\in (-\pi,\pi]$ the slice of $\jump{\mathbb S(E_k)}$ with respect to the line $l_{\rho,\theta}$ is exactly 
\begin{align}\label{firstpointcondition}
 \jump{\mathbb S(E_k)_{\rho,\theta}}=\jump{\mathbb S(E_k)\cap l_{\rho,\theta}},
\end{align}
that is the integration over some disjoint intervals. If $p_1,p_2,\dots p_m$ are the intervals endpoints (written in order\footnote{ $p_1$ is the point closer to $\{\eps\}\times \R^2$} on $l_{\rho,\theta}$) and if we assume that the last  interval between the points $p_1$ and $p_0=(\eps,\rho,\theta)$  is outside $\mathbb S(E_k)$,  then it results 
\begin{equation}\label{correctsum}
\partial\jump{\mathbb S(E_k)_{\rho,\theta}}=-\sum_{\substack{i>0\\i\textrm{ even}}}\delta_{p_i}+\sum_{\substack{i>0\\i\textrm{ odd}}}\delta_{p_i}. 
\end{equation}
 If instead the last interval $[p_1,p_0]$ is inside $\mathbb S(E_k)$ we have
\begin{equation}\label{wrongsum}
\partial\jump{\mathbb S(E_k)_{\rho,\theta}}=\sum_{\substack{i>0\\i\textrm{ even}}}\delta_{p_i}-\sum_{\substack{i>0\\i\textrm{ odd}}}\delta_{p_i}. 
\end{equation}

Let us now prove claim \eqref{claim_case(a)}.
We have obtained that, for a.e.
$\rho\in (1-\lambda_k,1-\lambda_k')$ and 
any $\theta\in (-\pi,\pi]$ such that  \eqref{slicing1d} holds,
 the slice $\partial\jump{\mathbb S(E_k)_{\rho,\theta}}$ is the sum in \eqref{correctsum}, and thanks to \eqref{lastpointequation} we deduce that the total number of points involved in \eqref{correctsum} must be odd. As a 
consequence, the push-forward by $\Pi_\eps$ of $ \partial\jump{\mathbb S(E_k)_{\rho,\theta}}$ is a Dirac delta with coefficient $-1$. Since this holds for a.e. $\rho\in (1-\lambda_k,1-\lambda_k')$ and any $\theta\in (-\pi,\pi]$, the conclusion follows.

\item[(1B)] Suppose $\mathcal H^2\Big(\big(\{\eps\}\times (B_{1-\lambda_k'}\setminus B_{1-\lambda_k})\big) \setminus\mathbb S(E_k)_\eps\Big)=0$.
In this case we pass to the complementary set; namely, if $\{\eps\}\times (B_{1-\lambda_k'}\setminus B_{1-\lambda_k}) =\mathbb S(E_k)_\eps\cap\big(\{\eps\}\times  (B_{1-\lambda_k'}\setminus B_{1-\lambda_k})\big)$, up to $\mathcal{H}^2$-negligible sets, we show that the multiplicity of $\Vke$ on this set is null. To do so it is sufficient to repeat the slicing argument 
 above for a.e. $(\rho,\theta)$ such that $p_0=(\eps,\rho,\theta)\in \{\eps\}\times (B_{1-\lambda_k'}\setminus B_{1-\lambda_k})$. For these points  \eqref{wrongsum} takes place, since by \eqref{lastpointequation} the number of points involved in the sum is even. The conclusion follows.
\end{itemize}

\textit{Case (2):} Choosing $\delta \in (0,1)$ small enough,
\begin{align}\label{inclusions_case2}
\Psi_k(\sourcedisk_{t_k^\delta})\subset [0,t_k^\delta]\times B_\delta(u_k(0)),
\end{align}
and, using $|u_k(0)|>1-\lambda_k$,
$$\projlambdak\circ \Psi_k(\badset\cap\sourcedisk_{t_k^\delta})\subset [0,t_k^\delta]\times (B_{1-\lambda_k'}\setminus B_{1-\lambda_k}).$$
Recalling the definition of $\piPsiDkXik$, it is not difficult to see that the current $\piPsiDkXik\res ((-1,t_k^\delta)\times B_{1-\lambda_k'})$ is supported in $[0,t_k^\delta]\times (\overline{B}_{1-\lambda_k'}\setminus {B}_{1-\lambda_k})$. 
By the properties of cylindrical symmetrization, we have also that $\mathbb S(\piPsiDkXik)\res ((-1,t_k^\delta)\times B_{1-\lambda_k'})$ is supported in $[0,t_k^\delta]\times (\overline{B}_{1-\lambda_k'}\setminus {B}_{1-\lambda_k})$.

Obviously, being $\mathcal Y_k$ null on $(-1,0)\times B_{1-\lambda_k'}$, we 
have
$$\mathcal X_k\res 
\Big((-1,0)\times (B_{1-\lambda_k'}\setminus B_{1-\lambda_k})
\Big)=\jump{(-1,0)\times (B_{1-\lambda_k'}\setminus B_{1-\lambda_k})},$$
and we find a set $E_{k,h}$, such that 
\begin{align*}
 \mathcal X_k\res 
\Big((-1,0)\times (B_{1-\lambda_k'}\setminus B_{1-\lambda_k})
\Big)=\jump{E_{k,h}\cap((-1,0)\times (B_{1-\lambda_k'}\setminus B_{1-\lambda_k}))}.
\end{align*}
If we pass to the  symmetrized set, arguing as in case (1A), we infer
\begin{align*}
\mathbb S(E_k)\cap \big((-1,0)\times (B_{1-\lambda_k'}\setminus B_{1-\lambda_k}) \big)=(-1,0)\times (B_{1-\lambda_k'}\setminus B_{1-\lambda_k}). 
\end{align*}
In other words, 
$(-1,0)\times (B_{1-\lambda_k'}\setminus B_{1-\lambda_k})$ is contained 
in $\mathbb S(E_k)$, and since the support of $\partial^* \mathbb S(E_k)$ does not intersect the set $(-1,0)\times B_{1-\lambda_k'}$, we infer that also 
\begin{align}\label{9.39}
 (-1,0)\times B_{1-\lambda_k'}\subset \mathbb S(E_k).
\end{align}
We now decompose $\{\eps\}\times B_{1-\lambda_k}$ as
\begin{align*}
 \{\eps\}\times B_{1-\lambda_k}=\big((\{\eps\}\times B_{1-\lambda_k})\cap \mathbb S(E_k)_\eps\big) \cup \big((\{\eps\}\times B_{1-\lambda_k})\setminus  \mathbb S(E_k)
 _\eps\big),
\end{align*}
and one of these two sets on the right-hand side must have positive $\mathcal H^2$-measure. 
Assume that $\mathcal 
H^2((\{\eps\}\times B_{1-\lambda_k})\setminus  \mathbb S(E_k)_\eps)>0$.
Then we will prove that the multiplicity of $\Vke$ on this set is $1$ (if instead $(\{\eps\}\times B_{1-\lambda_k})\setminus  \mathbb S(E_k)_\eps$ has zero measure then it is sufficient to prove that $\Vke$ has zero multiplicity on  $(\{\eps\}\times B_{1-\lambda_k})\cap \mathbb S(E_k)_\eps$; we drop this case being completely similar to the former). 

Therefore we now proceed as in case (1), 
slicing with respect to lines $l_{\rho,\theta}$ 
with $(\eps,\rho,\theta)\in \{\eps\}\times (\{\eps\}\times B_{1-\lambda_k})\setminus  \mathbb S(E_k)_\eps$. Since the last point 
$p_0=(\eps,\rho,\theta)$ does not belong to $\mathbb S(E_k)_\eps$, we are concerned with the sum in \eqref{correctsum}, and by \eqref{9.39} we infer that the number of $\{p_i\}$ involved in the sum is odd. The conclusion follows as in case (1).
\end{proof}

\begin{figure}
\begin{center}
    \includegraphics[width=0.75\textwidth]{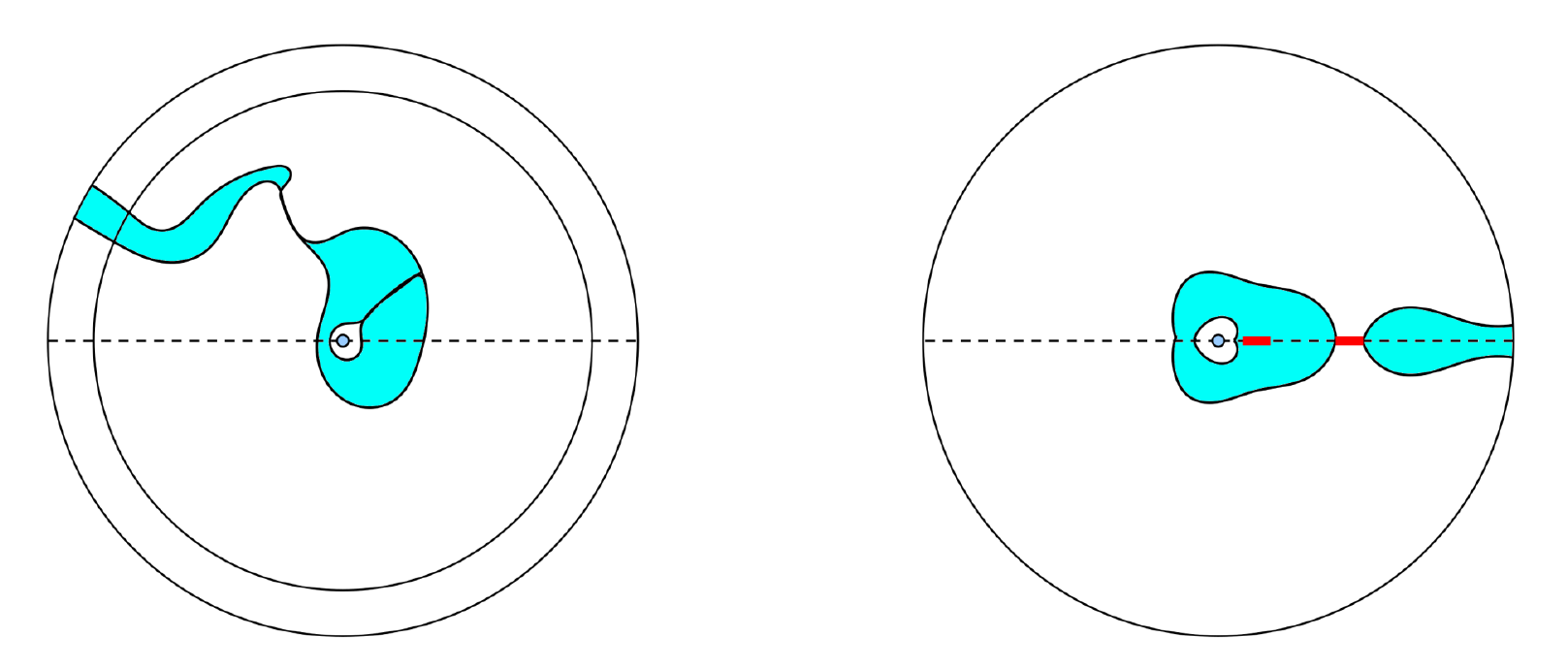}
 \caption{We represent the 
symmetrization of a general closed current in $B_1$. On the left it is visible that on two parts the curve overlaps itself in such a way that the multiplicity of the associated current is zero. In the symmetrized set, on the right picture, we have 
emphasized 
in bold 
 the corresponding set $\Jkzerotwopi$ in \eqref{J^i}. 
}
\label{fig:set_cancellation}
\end{center}
\end{figure}

\section{Gluing rectifiable sets}
\label{subsec:gluing}
In this section we 
show that, up to adding to $\partial \mathbb S(E_k)$ a rectifiable set 
with small $\mathcal H^2$-measure, $\partial \mathbb S(E_k)$ can be 
described as a polar graph of a suitable modification 
of the function $\Fke$ over a subset\footnote{called  $S^{(4)}_{k,\eps}$, see \eqref{newJ_k}.} of the rectangle\footnote{In cartesian coordinates.} 
$(0,\longR)\times [0,1]\times \{0\}\subset \R^3$,
and with Dirichlet boundary conditions independent of $k$. In Section 
\ref{sec:lower_bound} we will 
reduce the estimate of the area of the graph of $u_k$ 
to an estimate for a non-parametric Plateau problem which in turn will be independent of $k$. 

First we remark that $\mathbb S(E_k)
\subseteq C_l(1-\lambda_k')$ and $\mathbb S(\piPsiDkXik) = \partial 
\mathbb S(\mathcal E_k)$ in $C_l(1-\lambda_k')$, see
\eqref{symmetrization_of_Sk}. If we look at $\mathbb S(E_k)$ as a subset of $C_l$, we cannot conclude $\partial^*\mathbb  S(E_k)=\mathbb S(\piPsiDkXik)$ in $C_l$, and $\mathbb S(\piPsiDkXik)$ is not a closed current in $C_l$. For this reason we have to identify the boundary of $\mathbb S(\piPsiDkXik)$ in $C_l$. 

Recalling Corollary \ref{cor_9.5} and Definition \ref{def:Sigmake}, 
\begin{align*}
&{\partial \Big(} \big(\GFkethree+
\GminusFkethree+\jump{\Sigmake}\big)\res((\eps,l)\times \R^2){\Big)}\res(\{\eps\}\times \R^2)\\
=&\partial\big(\mathbb S(\piPsiDkXik)\res((\eps,l)\times \R^2)\big)\res(\{\eps\}\times \R^2)= -\mathbb S({\mathcal  H}_{k,\eps})-\mathbb S(L)_1+\mathbb S(L)_2-\jump{\overline{Y_1Y_2}},
\end{align*} 
in $\mathcal D_2((-\infty,l)\times \R^2)$, where $\jump{\overline{Y_1Y_2}}$ is the integration on ${\overline{Y_1Y_2}}$ (see \eqref{eq:def_Y12}) oriented from $Y_1$ to $Y_2$.
As a consequence, from \eqref{eq:boundary_Vkbis}, we obtain
\begin{align*}
 \partial \big(\GFkethree+
\GminusFkethree+\jump{\Sigmake}+\Vke\big)\res(\{\eps\}\times \R^2)=\jump{\{\eps\}\times \partial B_{1-\lambda_k'}} \qquad \text{ in }\mathcal D_2((-\infty,l)\times \R^2),
\end{align*}
where $\partial B_{1-\lambda_k'}$ is counterclockwisely oriented.

\begin{figure}
\begin{center}
    \includegraphics[width=0.6\textwidth]{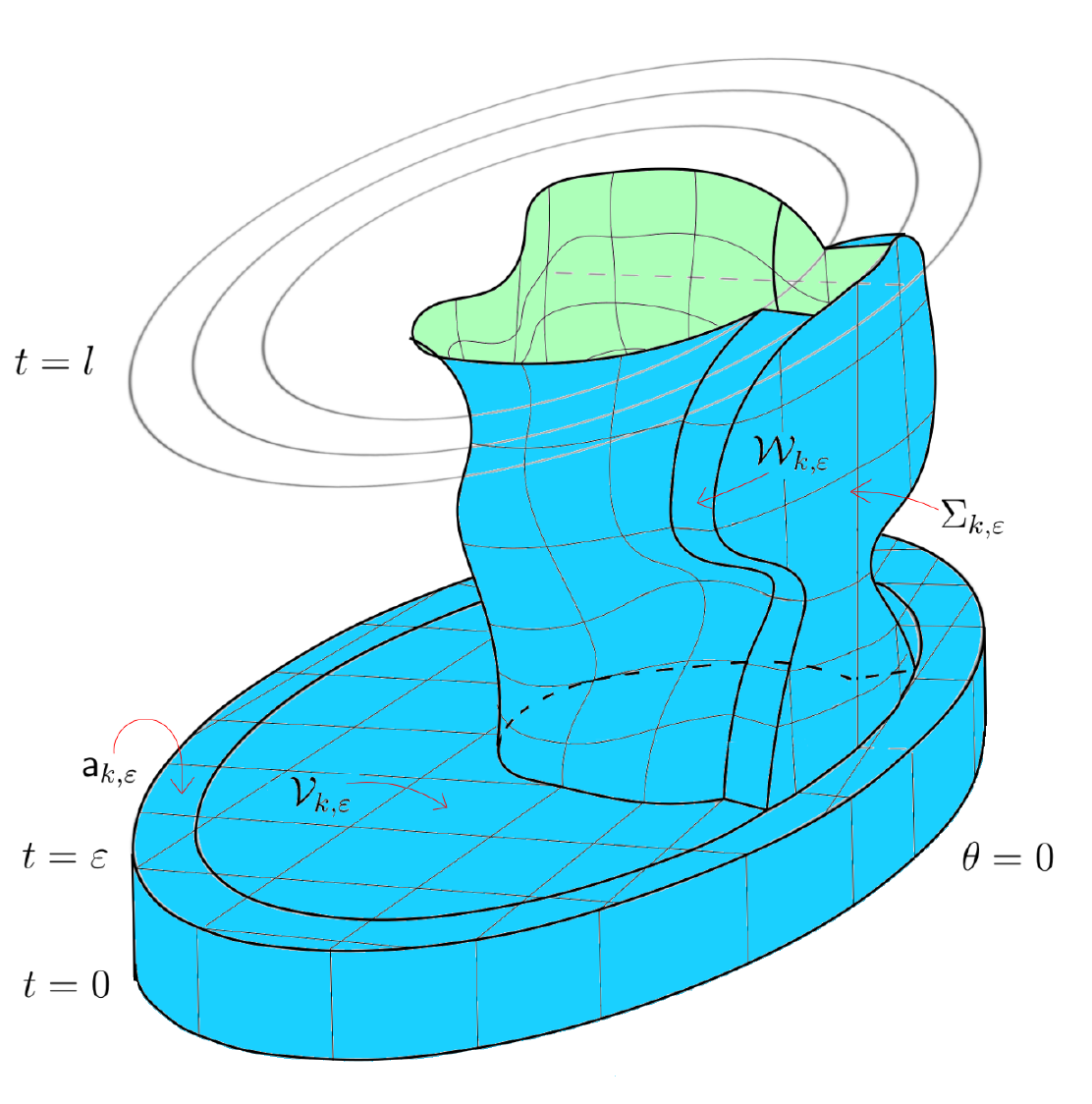}
 \caption{The largest (resp. smaller) basis circle 
has radius $1$ (resp. $1-\lambda_k'$). The smallest 
top circle has radius $1-\lambda_k$. The symbol $\mathcal W_{k,\eps}$
denotes the restriction of $\mathcal W_k$ to $\overline C_\eps^l(1-\lambda_k')$, after
symmetrization. Note that $ \GFkethree+
\GminusFkethree$ does not include $\Sigma_{k,\eps}$ and ${\mathcal V}_{k,\eps}$; see \eqref{eq:G_three_plus_G_minus_three}.
}
\label{fig:tricilindro}
\end{center}
\end{figure}
\subsection{Enforcing boundary conditions at $\{0\}\times \R^2$; 
a modification of $\Fke$}
Let $\intannulus$ denote the integration over the annulus $\{\eps\}\times( B_{1}\setminus B_{1-\lambda_k'})$, in such a way that $$\partial 
\intannulus
=\jump{\{\eps\}\times \partial B_{1}}-\jump{\{\eps\}\times \partial B_{1-\lambda_k'}},
$$
see Fig. \ref{fig:tricilindro}.
Then 
\begin{align}\label{eq:estimateUpsilon}
 |\intannulus|=\pi(1-(1-\lambda_k')^2)\leq 2\pi\lambda_k',
\end{align}
and 
\begin{align*}
 \partial \big(\GFkethree+
\GminusFkethree+\jump{\Sigmake}+\Vke+\intannulus
\big)=\jump{\{\eps\}\times \partial B_{1}} \qquad \text{ in }\mathcal D_2((-\infty,l)\times \R^2).
\end{align*}
Finally, we add to the current $\GFkethree+
\GminusFkethree+\jump{\Sigmake}+\Vke+\intannulus
$ the integration over the lateral boundary of the cylinder $(0,\eps)\times B_1$, so that the resulting current 
\begin{align}
\label{boundary_of_Ok}
 \GFkethree+
\GminusFkethree+\jump{\Sigmake}+\Vke+\intannulus
+\jump{(0,\eps)\times \partial B_1}\in \mathcal D_2((0,\longR)\times\R^2),
\end{align}
satisfies
\begin{align*}
\partial \Big(
\GFkethree+
\GminusFkethree+\jump{\Sigmake}+\Vke+\intannulus
+\jump{(0,\eps)\times \partial B_1}\Big)=\jump{\{0\}\times \partial B_{1}} \text{ in }\mathcal D_2((-\infty,l)\times\R^2);
\end{align*}
in particular it is boundaryless in 
$\mathcal D_2((0,\longR)\times\R^2)$. 

Now, we want to identify the solid
region that we can call the ``inside'' of the current in \eqref{boundary_of_Ok}.
\begin{definition}[\textbf{The sets $\Oke$}]\label{def:O_k}
We let
\begin{equation}\label{eq:def_O_k}
\Oke
:=\big(\mathbb S(E_k)\cap ((\eps,l)\times \R^2)\big)\cup ((0,\eps]\times B_1)\subset [0,\longR] \times\R^2.
\end{equation}
\end{definition}
A direct check shows that the current built in 
\eqref{boundary_of_Ok} is the integration over the  boundary of $\jump{\Oke}$. 
Indeed, by Lemma \ref{lemma10.3}(iii) and Definition \ref{def:Sigmake}
 we see that the integration over
$\mathbb S(E_k)\cap ((\eps,l)\times \R^2)$ has as boundary  $\GFkethree+
\GminusFkethree+\jump{\Sigmake}$ in $(\eps,l)\times \R^2$, 
whereas  $(0,\eps]\times B_1$ trivially
 has boundary $(0,\eps)\times \partial B_1$ in $(0,\eps)\times \R^2$. 
The current $\Vke+\intannulus$ represents the boundary of $\jump{\Oke}$ concentrated on the plane $\{\eps\}\times \R^2$.
 In turn we will see (formulas \eqref{subgraph1extended}, 
\eqref{subgraph2extended}) that $\Oke$ is the polar subgraph of 
a suitable modification of $\Fke$.
Thus we are going to introduce the new  extra ``strip'' (recalling
the definition of $\striptwo$ in \eqref{eq:strip_two} and of $\JQke$ in 
\eqref{strip_Jextension1}):
\begin{equation}
\label{newJ_k}
\begin{aligned}
 \stripfour:&=
\striptwo
\cup \JQke
\cup ((\eps,l)\times [1-\lambda_k',1]\times\{0\})
\\
 &=\{(t,\rho,\theta):t\in(\eps,l),\rho\in[\minuk(t)\wedge (1-\lambda_k),1],\theta=0\},
\end{aligned}
\end{equation}
see Fig. \ref{fig_q4} (and also Figs. \ref{fig_q}, \ref{fig_q3}).

\begin{figure}
	\begin{center}
		\includegraphics[width=0.7\textwidth]{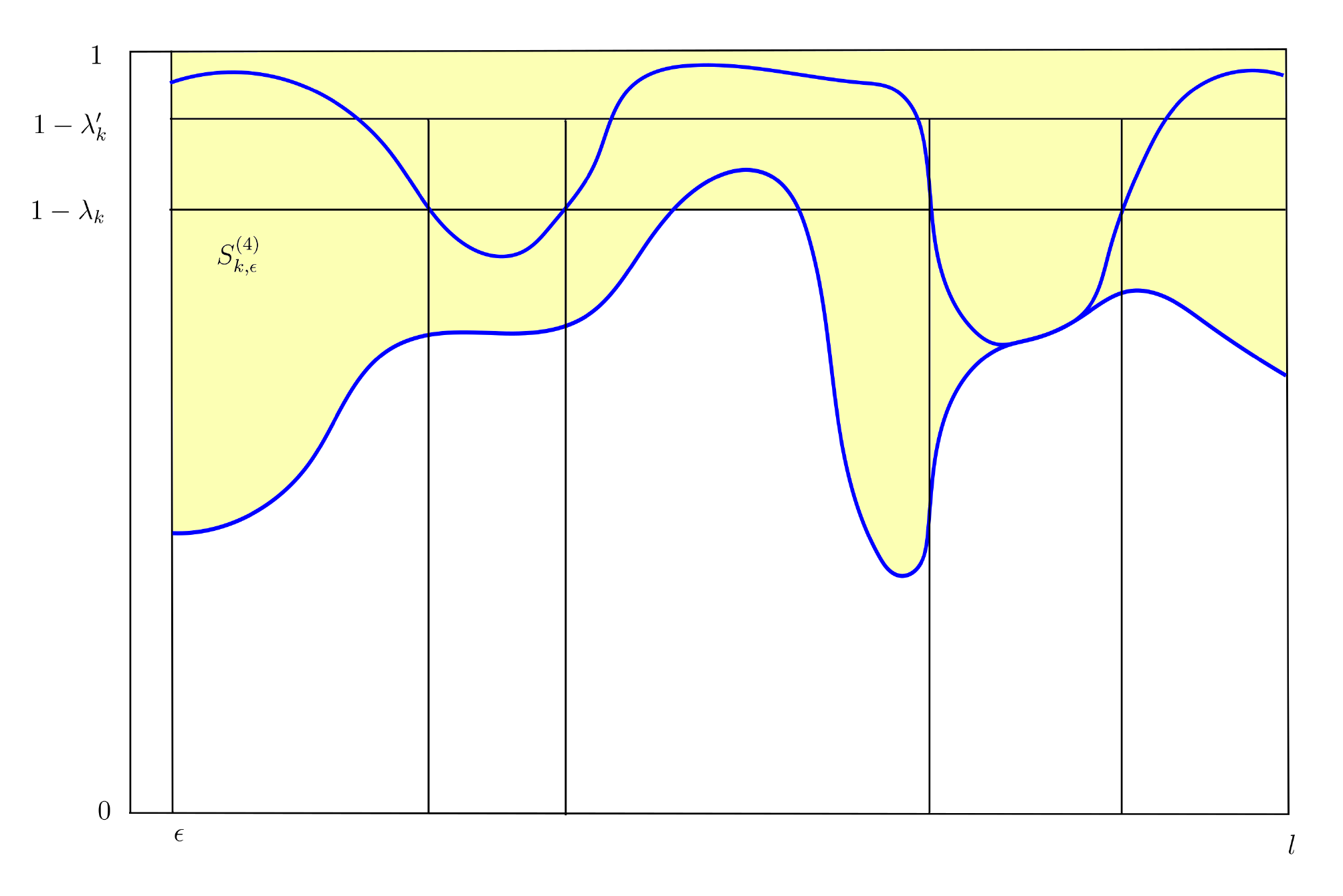}
		\caption{The graphs of the functions $\maxuk$ and $\minuk$ 
			and the set $\stripfour$ in \eqref{newJ_k}.
See also Fig. \ref{fig_q}.
		}
		\label{fig_q4}
	\end{center}
\end{figure}

\begin{definition}[\textbf{The function $\hatFke$}]
\label{def:the_function_hatFke}
We define
$\hatFke : (0,\longR)\times [0,1]\times \{0\}\to \R$ as 
\begin{align}\label{strip_Jextension2}
 \hatFke:=\begin{cases}
\Fke &\text{ in } 
(\eps,l)\times [0,1-\lambda_k']\times \{0\}
\\
                 0 &\text{ in }(\eps,l)\times [1-\lambda_k',1]\times \{0\}\\
                 \pi&\text{ in }(0,\eps]\times [0,1]\times \{0\}.
                \end{cases}
\end{align}
\end{definition}
 Accordingly, we extend the currents $\GFkethree$ and 
$\GminusFkethree$ as follows:
As in \eqref{subgraph1} we fix $\eta \in (0,\frac{\pi}{4})$, 
and set 
\begin{align*}
&SG^{\rm pol}_{\hatFke}:=\{(t,\rho,\theta)\in(0,\longR)\times [0,1]\times \{0\}:\theta\in (-\eta,\hatFke(t,\rho,0))\},\\
&UG^{\rm pol}_{-\hatFke}:=\{(t,\rho,\theta)\in (0,\longR)\times [0,1]\times \{0\}:\theta\in (-\hatFke(t,\rho,0),\eta)\}.
\end{align*}

\begin{remark}\rm
By construction, 
\begin{align}
\label{subgraph1extended}
 &SG^{\rm pol}_{\hatFke}\cap\{\theta\in (0,\pi)\}=\Oke\cap \{\theta\in (0,\pi)\},\\&\label{subgraph2extended}
 UG^{\rm pol}_{-\hatFke}\cap\{\theta\in (-\pi,0)\}=\Oke\cap \{\theta\in (-\pi,0)\},
\end{align}
where the set $\Oke$ is defined in \eqref{eq:def_O_k}.
\end{remark}

The next currents are constructed to reach the 
segment $(0,\longR)\times \{1\}\times \{0\}$.

\begin{definition}[\textbf{The currents $\mathcal G_{\pm
\widehat \vartheta_{k,\eps}}^{(4)}$}]\label{def:G^4_keps}
We define the currents 
\begin{equation}\label{eq:def_generalizedgraphs_Fextended}
\begin{aligned}
 &\GhatFkfour
:=(\partial 
\jump{SG^{\rm pol}_{\hatFke}})
\res\{\theta\in (0,\pi)\}+\jump{G^{\rm pol}_{\hatFke\res \big(\big\{\hatFke \in\{ 0,\pi\}\big\}\cap \stripfour\big)  
}},
\\
 &\GminushatFkfour
:=(\partial \jump{UG^{\rm pol}_{- \hatFke}})\res
\{\theta\in (-\pi,0)\}+\jump{G^{\rm pol}_{-\hatFke\res \big(\big\{\hatFke \in\{ 0,\pi\}\big\}\cap \stripfour\big)
}}.
\end{aligned}
\end{equation}
\end{definition}
In other words, the support of $\GhatFkfour$ coincides with the 
generalized polar graph of $\hatFke$ restricted to $\stripfour\times [0,\pi]$.
Notice that also in this case $\jump{G^{\rm pol}_{-\hatFke
\res
\big(\big\{\hatFke \in\{ 0,\pi\}\big\}\cap \stripfour\big)    
}}+\jump{G^{\rm pol}_{\hatFke
\res \big(\big\{\hatFke \in\{ 0,\pi\}\big\}\cap \stripfour\big)    
}}=0$,
and 
\begin{align}
\label{eq:int_O_k}
 \GhatFkfour
+\GminushatFkfour
=\jump{\partial^* \Oke}\text{ in }(0,\longR)\times \R^2.
\end{align}
Moreover, by \eqref{eq:def_generalizedgraphs_Fextended} 
and \eqref{eq:def_generalizedgraphs_F}, 
\begin{align}\label{sum_generalizedgraphextended}
  |\GhatFkfour|+|\GminushatFkfour|=|
\GhatFkfour
+\GminushatFkfour|+2\mathcal H^2\left(\big\{\hatFke \in\{ 0,\pi\}\big\}\cap \stripfour\right). 
\end{align}
Finally
\begin{align}\label{graph_of_F_ktotal}
\GhatFkfour=&
\GFkethree+\jump{(\eps,l)\times [1-\lambda_k',1]\times \{0\}}+\jump{\Sigmake\cap \{0\leq\theta\leq\pi\}}+\Vke\res\{0\leq\theta\leq\pi\}\nonumber\\
 &+{\intannulus
\res \{0\leq\theta\leq\pi\}}+\jump{((0,\eps)\times \partial B_1)\cap \{0\leq\theta\leq\pi\}},
\end{align}
and
\begin{align*}
 \GminushatFkfour=&
\GminusFkethree-\jump{(\eps,l)\times [1-\lambda_k',1]\times \{0\}}+\jump{\Sigmake\cap \{-\pi\leq\theta\leq0\}}+\Vke\res\{-\pi\leq\theta\leq0\}\nonumber\\
 &+{\intannulus
\res \{-\pi\leq\theta\leq0\}}+\jump{((0,\eps)\times \partial B_1)\cap \{-\pi\leq\theta\leq0\}},
\end{align*}
so that $$\GhatFkfour+
\GminushatFkfour
=\GFkethree+
\GminusFkethree+\jump{\Sigmake}+\Vke+\intannulus
+\jump{(0,\eps)\times \partial B_1}.$$

\begin{remark}
 The function $\hatFke$ is defined on the whole domain $ (0,\longR)\times [0,1]\times \{0\}$, but it might take values in $(0,\pi)$ only in 
$\striptwo$, see Remark \ref{rmk:8.6}(v). Moreover, referring also to Remark \ref{rmk:8.8}, we see that the currents $\GhatFkfour$ and $
\GminushatFkfour$ neglect the generalized polar  graph of $\hatFke$ 
(defined in \eqref{eq:generalized_polar_graph})
on $ ((0,\longR)\times [0,1]\times \{0\})\setminus \stripfour$, with the only 
exception of the ``vertical'' part $\jump{(0,\eps)\times \partial B_1}$.
\end{remark}

An important step in the proof of the estimate
from below in Theorem \ref{teo:step1} is given by the
next inequality.

\begin{prop}[\textbf{Estimate from below in terms of the mass of $\GhatFkfour$ and $\GminushatFkfour$}]
Let $\eps$ be fixed as in \eqref{eq:H1}  and \eqref{eq:H2}.
 The following inequality holds:
 \begin{align}\label{crucial_ineq}
|\currgraphk|_{\badset\times \R^2} 
\geq 
|\GhatFkfour|+|\GminushatFkfour|-
\pi\eps-\frac{C}{\eps n} -o_k(1) 
 \end{align}
for an absolute constant $C>0$, where the sequence 
$o_k(1)\geq0$ depends on $\eps$ and $n$, and is infinitesimal as $k\rightarrow +\infty$.
\end{prop}

\begin{proof}
By \eqref{graph_of_F_ktotal} we get
\begin{align*}
 |\GhatFkfour|\leq&|
\GFkethree|+\lambda_k'l+|\jump{\Sigmake\cap \{0<\theta<\pi\}}|+|\Vke\res\{0<\theta<\pi\}|\nonumber\\
 &+|{\intannulus
\res \{0\leq\theta\leq\pi\}}|+|\jump{((0,\eps)\times \partial B_1)\cap \{0\leq\theta\leq\pi\}}|.
\end{align*}
A similar estimate holds for  $|\GminushatFkfour|$ so that 
\begin{equation*}
 |\GhatFkfour|
+|\GminushatFkfour|
\leq  |\GFkethree|
+|\GminusFkethree|
+|\jump{\Sigmake}|+|\Vke|
 +|\intannulus
|+|\jump{(0,\eps)\times \partial B_1}|
+2\lambda_k'l.
\end{equation*}
Coupling the above inequality with 
\eqref{123} and \eqref{124} gives
\begin{align*}
& |\GhatFkfour|+|\GminushatFkfour|
\\
\leq &
|\currentgengraphFkJkzero|+|\currentgengraphminusFkJkzero|+|\jump{\Sigmake}|+|\Vke|+|\intannulus
|+|\jump{(0,\eps)\times \partial B_1}|
+2\lambda_k'l+\frac{1}{\pi\eps n}\nonumber\\
\leq& \int_{\badset\cap 
(\Omega \setminus \overline \sourcedisk_\eps)
}|J(\projlambdak\circ\Psi_k)|dx+|\jump{\Sigmake}|+|\Vke|+|\intannulus
|+|\jump{(0,\eps)\times \partial B_1}|\nonumber\\
&\quad+2\lambda_k'l+\frac{1}{\pi\eps n}+\frac1n\nonumber\\
\leq
& \int_{\badset\cap (\Omega \setminus \overline \sourcedisk_\eps)}|J(\projlambdak\circ\Psi_k)|dx+ 
|\currgraphk|_{(\badset\cap \sourcedisk_\eps)\times \R^2}+\pi\eps+\frac{C}{\eps n}+o_k(1),
\end{align*}
where the second inequality follows from 
\eqref{eq:estimate_useful}, the last inequality follows from
\eqref{mass_Sigma}, \eqref{eq:estimateUpsilon} 
and \eqref{estimate_of_Vk}, and $C>0$ is an absolute constant. Here $o_k(1)$ is a nonnegative quantity 
infinitesimal as $k\rightarrow +\infty$, depending on $\eps$ and $n$.
In conclusion 
\begin{align*}\label{main_estimate}
& |\GhatFkfour|+|\GminushatFkfour|-\pi\eps
-\frac{C}{\eps n}
-o_k(1)
 \\
\leq &
|\currgraphk|_{(\badset\cap (\Omega \setminus \sourcedisk_\eps))
\times \R^2}+|\currgraphk|_{(\badset\cap \sourcedisk_\eps)\times \R^2}
=|\currgraphk|_{\badset\times \R^2}.
\end{align*}
\end{proof}

\section{Three examples}\label{sec:three_examples}
Before concluding the proof of the lower bound, 
we aim to explain the various geometric objects 
introduced in the previous sections (for a recovery sequence) through three
interesting examples of sequences converging to  the vortex map $\vortexmap$. 
We warn the reader that the sequences of Sections 
\ref{subsec:an_approximating_sequence_of_maps_with_degree_zero:cylinder}
 and \ref{sec:catenoid_with_a_flap}, as well as the one of Section 
\ref{sec:smoothing_by_convolution} for $\longR$
small, are not recovery sequences; nevertheless, we believe 
it is useful to describe the various quantities introduced in 
Sections \ref{sec:the_maps}-\ref{subsec:gluing} in correspondence
to these sequences, since this shades some light
on the proof of the lower bound.

\subsection{An approximating sequence of maps with degree zero: cylinder} 
\label{subsec:an_approximating_sequence_of_maps_with_degree_zero:cylinder}
In \cite{AcDa:94} the authors describe 
an approximating sequence $(u_k)$ of smooth maps taking values
in $\mathbb S^1$  
which, in our context,
are defined in polar coordinates as follows:
 \begin{equation}
\label{eqn:u_k_Cylinder}
\veps(r,\theta):=
\begin{cases}
     (\cos \theta, \sin \theta) & \textrm{ \qquad 
in } \Omega_1:=\BallR \setminus (\Balleps \cup \{\theta \in (-\theta_k, \theta_k)\}),\\
     (\cos (\frac{r}{\reps}(\theta -\pi)+\pi),\sin (\frac{r}{\reps}(\theta -\pi)+\pi)) & \textrm{ \qquad in }\Balleps \setminus \{\theta \in (-\theta_k, \theta_k)\},\\
    (\cos (\frac{\theta_k -\pi}{\theta_k}\theta +\pi),\sin(\frac{\theta_k -\pi}{\theta_k}\theta +\pi)) & 
\textrm{ \qquad in } \{\theta \in [0, \theta_k)\}\setminus \Balleps,\\
(\cos (\frac{-\theta_k +\pi}{-\theta_k}\theta +\pi),\sin(\frac{-\theta_k +\pi}{-\theta_k}\theta +\pi)) & 
\textrm{ \qquad in } \{\theta \in (-\theta_k,0)\}\setminus \Balleps,\\
   (\cos( \frac{r}{\reps}(\frac{\theta_k -\pi}{\theta_k}\theta )+\pi),\sin ( \frac{r}{\reps}(\frac{\theta_k -\pi}{\theta_k}\theta )+\pi)) & \textrm{ \qquad in }\Omega_4:=\Balleps \cap \{\theta \in [0, \theta_k)\},\\
   (\cos( \frac{r}{\reps}(\frac{-\theta_k +\pi}{-\theta_k}\theta )+\pi),\sin ( \frac{r}{\reps}(\frac{-\theta_k +\pi}{-\theta_k}\theta )+\pi)) & \textrm{ \qquad in }\Omega_4:=\Balleps \cap \{\theta \in (-\theta_k, \theta_k)\},
    \end{cases} 
\end{equation}
where $(\reps)$ and $(\theta_k)$ are two infinitesimal sequences 
of positive numbers; see Fig. \ref{fig:Omega_l_uk}.
Recall that in the previous sections
we introduced the number $\eps$; 
we may assume here that  
$\reps<<\eps$
for all $k \in \mathbb N$. 
Notice that
$u_k(0,0)= (-1,0)
=u_k(r,0)$ for $r \in (0,\longR)$.
Moreover for $t \in (0,\longR)$  we have  $u_k(\partial \sourcedisk_t)
=\partial B_1 \setminus \{\theta \in (-\theta_k, \theta_k) \}$, and the degree
of $u_k$ is zero.
 
Now we fix an infinitesimal sequence $(\lambda_k)$ of positive
numbers. Inspecting \eqref{eqn:u_k_Cylinder},
it turns out that the set $\badset$ defined in \eqref{eq:D_k_once_for_all}
and \eqref{eq:D_k_delta} satisfies 
$\badset\subsetneq \Balleps \cup \{\theta\in (-\theta_k, \theta_k)\}\cap
\Om$, 
see Fig.  \ref{fig:D_k_in_Cylinder_EX}.

\begin{figure}
	\begin{center}
		\includegraphics[width=0.9\textwidth]{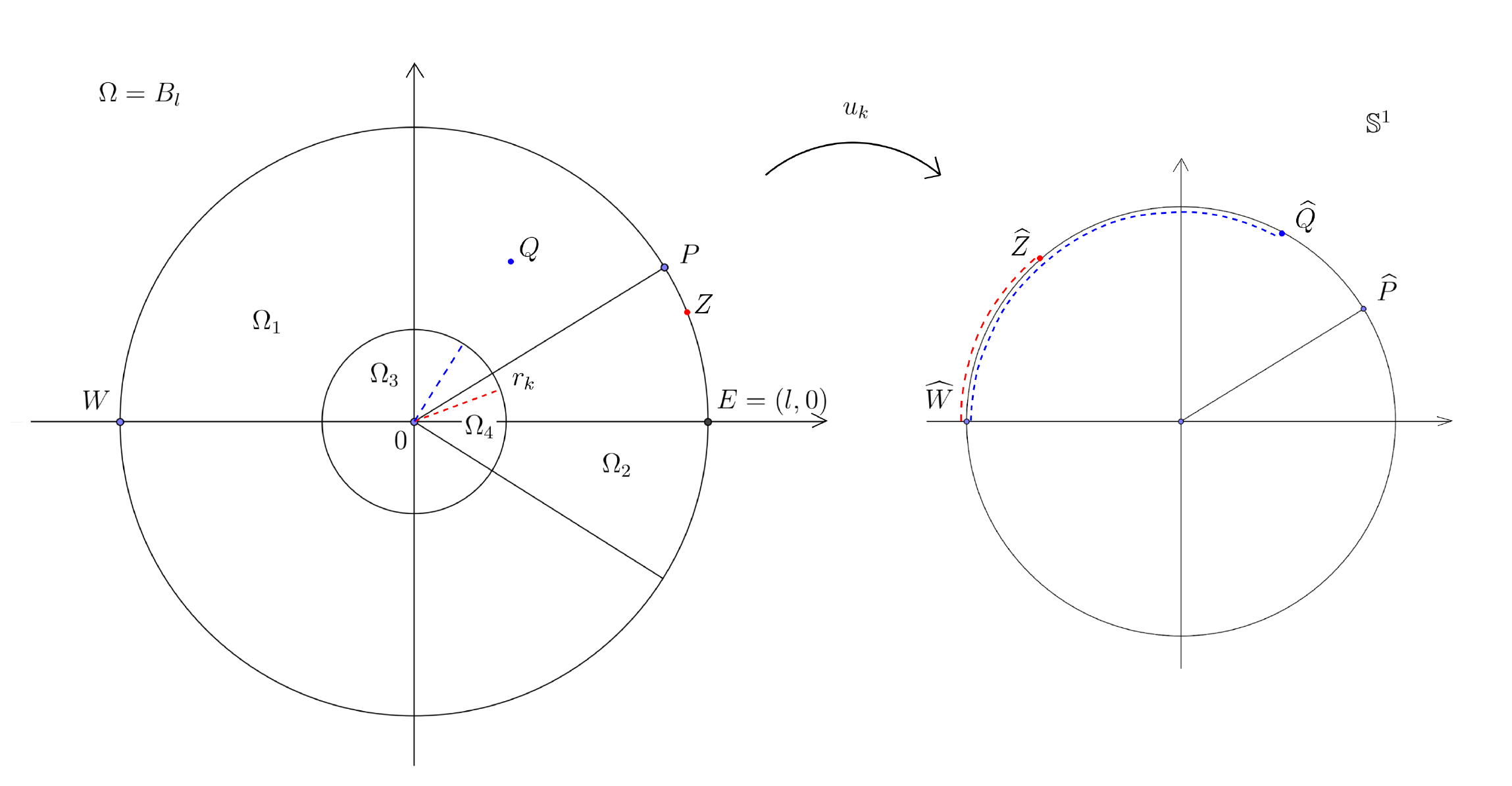}
		\caption{
The map $u_k$ in \eqref{eqn:u_k_Cylinder}. 
We set $\hat P := P/\vert P\vert = \theta_k$, 
$\hat Q := Q/\vert Q\vert$. All points in $\Omega_1 \cup \Omega_2$ are retracted
on $\mathbb S^1$. 
The image of $\Omega_3$ through $u_k$ is as follows:
$u_k$ sends the generic dotted segment onto the 
(long) dotted arc on $\mathbb S^1$.
Finally, the image of $\Omega_4$ through $u_k$ is as follows:
$u_k$ sends the generic dotted segment onto the 
(short) dotted arc on $\mathbb S^1$.
		}
		\label{fig:Omega_l_uk}
	\end{center}
\end{figure}

\begin{figure}
\begin{center}
    \includegraphics[width=0.5\textwidth]{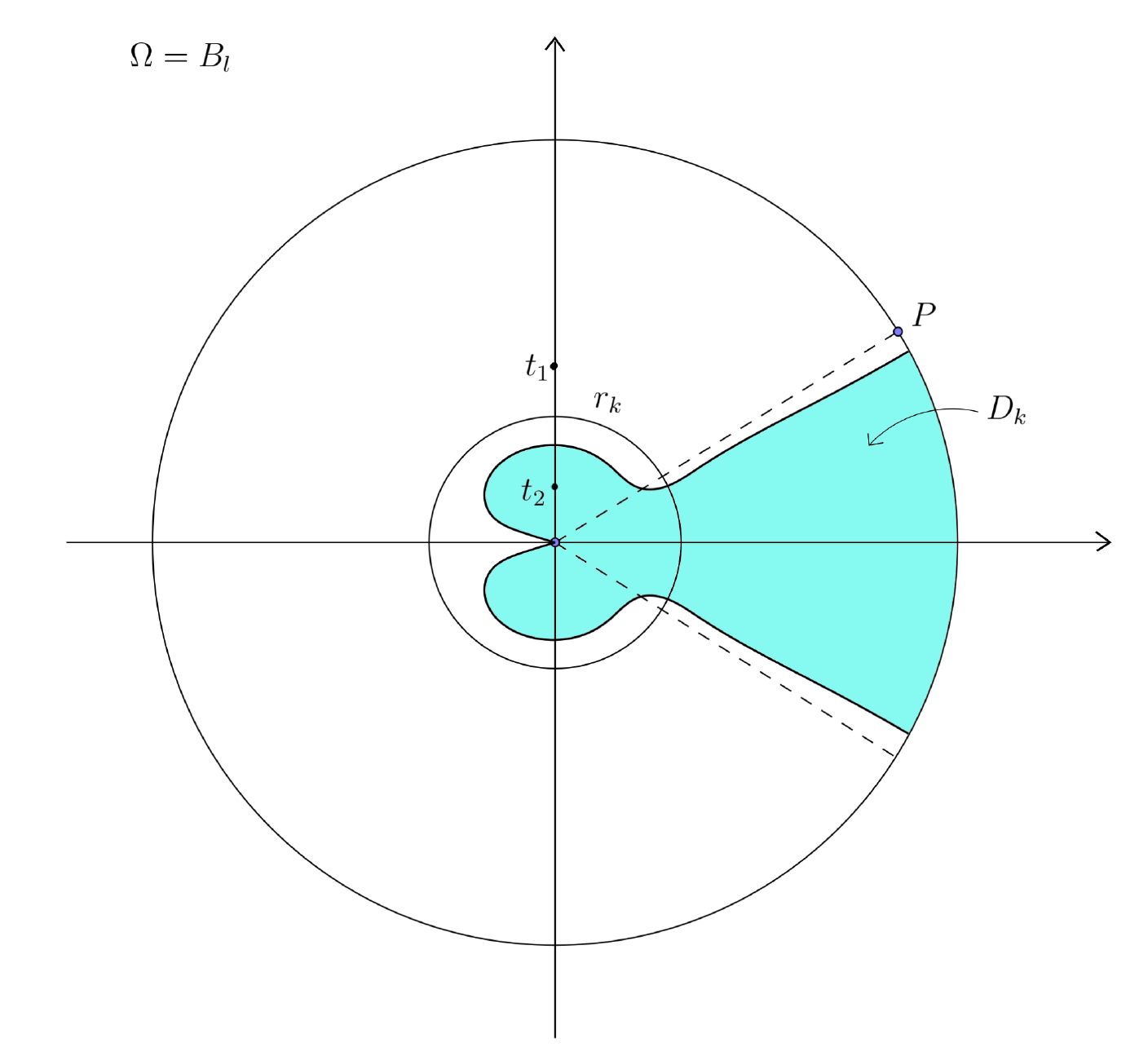}
 \caption{ The set emphasized is $\badset$ 
for the sequence in \eqref{eqn:u_k_Cylinder}.
}
\label{fig:D_k_in_Cylinder_EX}
\end{center}
\end{figure}

Notice that ${\rm spt}({u_k}_\sharp \jump{(\Omega\setminus \badset) \cap \partial 
\sourcedisk_t}) \subseteq  u_k((\Omega\setminus \badset) 
\cap \partial \sourcedisk_t)$ with strict inclusion 
whenever $t$ is such that $(\Omega\setminus \overline \badset) \cap\{\theta 
= \pm \theta_k\}  \cap \partial \sourcedisk_t\neq \emptyset$. 
For instance, for $t\in (\reps, l),$ we have  $u_k((\Omega\setminus \badset) 
\cap \partial \sourcedisk_t)
=\partial B_1 \setminus \{\theta \in (-\theta_k, \theta_k) \}$ and 
$u_k(\badset \cap \partial \sourcedisk_t) 
= \partial B_1 \setminus \{\theta \in (-\theta_k-a_k, \theta_k+a_k) \}$,  
where $a_k$ is a positive small angle;
on the other hand\footnote{This is due to the fact that 
$u_k\vert _{\Omega\setminus \badset}$ covers the arcs 
$(\partial B_1)\cap\{ \theta \in (\theta_k, \theta_k+a_k) 
\cup (-\theta-a_k, -\theta)\}$ twice, with opposite orientations.} 
${\rm spt} ({u_k}_\sharp\jump{(\Omega \setminus \badset) 
\cap \partial \sourcedisk_t})= \partial B_{1} \setminus \{\theta \in (-\theta_k-a_k, \theta_k+a_k) \}={\rm spt} ({u_k}_\sharp\jump{\badset \cap \partial \sourcedisk_t})$.  Similarly we have 
${\rm spt}({u_k}_\sharp \jump{\badset \cap \partial \sourcedisk_t}) 
\subseteq  u_k(\badset \cap \partial \sourcedisk_t)$ 
with strict inclusion\footnote{This could only happen for $t\in (0,\reps)$.} whenever 
$\badset \cap\{\theta = \pm \theta_k\}  \cap \partial \sourcedisk_t \neq \emptyset$ (thus, $t \in (0,r_k)$).

The support of the current $\piPsiDk
=(\projlambdak\circ \Psi_k)_\sharp\jump{\badset}$
in \eqref{eq:S_hat_k}
is a connected set contained in $\partial_{{\rm lat}}(
C_l(1-\lambda_k))$. 
But in this 
example we also have ${\rm spt}(\piPsiDk)
={\rm spt} \left(
(\projlambdak\circ \Psi_k)_\sharp\jump{\Omega\setminus \badset}\right)$; 
moreover 
$$
{\rm spt}(\piPsiDk) \cap 
C^\eps_l
=
 (\eps,l)\times \{1-\lambda_k\}\times \left(
(-\pi,\pi)\setminus (-\theta_k-a_k, \theta_k+a_k)\right).
$$
The support of the current 
  $ \Xik$ in \eqref{eq:def_Xi_k}, an orthogonal ``wall'' 
over $C_l(1-\lambda_k)$ 
of 
height $\lambda_k - \lambda_k'$,
 built on $\partial \piPsiDk$,
divides  $C_l(1-\lambda_k^\prime) \setminus 
C_l (1-\lambda_k)$ into two connected sets; one of them, 
${\rm spt}(\mathcal{Y}_k)$ 
(see \eqref{eq:defY_k}), 
has ${\rm spt} (\projlambdak\circ \Psi_k)_\sharp\jump{\Omega\setminus \badset}$ 
as part of its boundary, and the other one, 
${\rm spt} (\mathcal{X}_k)$ (see \eqref{eq:defX_k}), has $\partial C_l(1-\lambda_k) \setminus 
{\rm spt} ((\projlambdak\circ \Psi_k)_\sharp\jump{\Omega\setminus \badset})$ 
as part of its boundary. 
Hence the support of $\piPsiDkXik=\piPsiDk+\Xik$ (Definition \ref{def:the_current_S_k}) divides  $C_l(1-\lambda_k^\prime)$ 
into two connected sets, one of them, $E_k$, is $C_l(1-\lambda_k) \cup {\rm spt} (\mathcal{X}_k)$, and the second one is 
its complement in $C_l(1-\lambda_k')$ and equals
${\rm spt}( \mathcal{Y}_k)$.
In this particular example the cylindrical Steiner symmetrization introduced in Section \ref{sec:cylindrical_Steiner_symmetrization} is unnecessary, since 
$\mathbb{S}(E_k)= E_k$ and 
hence $\mathbb{S}(\piPsiDkXik)= \piPsiDkXik$. 
Thus the function $\Fke$ in \eqref{eq:Fke} reads as 
\begin{equation*}
\Fke(t,\rho)=
\frac{\Theta_k(t,\rho)}{2}= \frac{1}{2\rho}\Hone((E_k)_{t,\rho})=\begin{cases}
\theta_k+a_k &\qquad \text{for }(t,\rho)\in (\eps,l)\times(1-\lambda_k,1-\lambda_k'],\\
\pi &\qquad  \text{for }(t,\rho)\in (\eps,l)\times(0,1-\lambda_k].
\end{cases}
\end{equation*}
Note that 
\begin{align*}
&{\rm spt }(\piPsiDkXik)
\cap C^\eps_l
=\left({\rm spt}(\piPsiDk) \cup  {\rm spt}(\Xik)\right)
\cap C^\eps_l\\
&=\left({\rm spt}(\piPsiDk) \cup \Big((\eps,l)\times [1-\lambda_k,1-\lambda_k^\prime]\times \{-\theta_k-a_k, \theta_k+a_k\}\Big )
\right)\cap C^\eps_l,
\end{align*}
and
that 
${\rm spt}(\piPsiDkXik) \cap (\{0\} \times \R^2)$ is the segment
$\{0\} \times [1-\lambda_k, 1-\lambda_k'] \times \{\pi\}$.

Concerning the functions 
in \eqref{eq:q_-q_+}, we have $|u_k|^+=|u_k|^-=1$, thus $\pi_0^{\rm pol}(\projlambdak \circ \Psi_k(\Omega \setminus \overline \sourcedisk_\eps))=(\eps,l)\times \{1-\lambda_k\}\times\{0\}$, and the sets $Q_{k,\eps}$
in \eqref{Q_k} and $\JQke $ in \eqref{strip_Jextension1}  are empty. 
Hence, for $\striptwo$
in \eqref{eq:strip_two}, we have  $\striptwo=(\eps,l)\times [1-\lambda_k,1-\lambda_k']\times\{0\}$.
Moreover, for $\Theta_k$ in \eqref{eq:Theta_k},
 we have $\Theta_k(t,\rho)\in (0,\pi)$ for  $(t,\rho,0)\in \striptwo$; thus,
recalling Definitions \ref{def:the_currents_jump_gengraphplusminusFkJkzero} and \ref{Def: The current g3}, 
and since $\jump{G^{\rm pol}_{\pm\Fke\res \striptwo\cap \{\Theta_k \in \{0,2\pi\}\}}}=0$,  
\begin{align*}
\GplusminusFkethree
=\pmcurrentgengraphFkJkzero &= \piPsiDkXik \res \{(t,\rho,\theta): t \in (\eps, l),~\pm \theta \in (0,\pi)\},
\end{align*}
where the last equality follows by construction (see \eqref{symmetrizedset} and \eqref{symmetrization_of_Sk}).

The set $\Sigmake$ in Definition \ref{def:Sigmake} equals 
$$\Sigmake =  (\eps, l)\times\{ 1-\lambda_k^\prime\}\times (-\theta_k-a_k, 
\theta_k+a_k),$$ 
hence the current $\GFkethree +\GminusFkethree + \jump{\Sigmake} = \piPsiDkXik + \jump{\Sigmake}$ is boundaryless in $C^\eps_l$ 
(with $\Sigma_{k,\eps} $ suitably oriented).

Recalling \eqref{eq:def_Vk}, we have 
 $$ \Vke =\jump{ \{\eps\} \times [1-\lambda_k,1-\lambda_k^\prime]\times 
((-\pi,\pi]\setminus (-\theta_k-a_k, \theta_k+a_k))},$$ 
hence
$$
\partial \Vke  = \mathbb S(\scriptHkeps)+\mathbb{S}(L)_1 - \mathbb{S}(L)_2 +  \mathcal L_k,
$$
see \eqref{eqn:scriptHkeps}, \eqref{eqn:SscriptHkeps}, 
\eqref{def_X1X2}, \eqref{eqn:9.16}, and Proposition  \ref{prop:V_k},
\begin{align*}
 \mathbb S(\scriptHkeps) &=
   \scriptHkeps = (\projlambdak\circ\Psi_k)_\sharp\jump{\partial \sourcedisk _\eps \cap \badset}
  =(\projlambdak\circ\Psi_k)_\sharp\jump{\{\eps\}\times (-\theta+a_k, \theta-a_k)}\\
&=\jump{\{\eps\}\times \{1-\lambda_k\}\times ((-\pi,\pi]\setminus (-\theta_k-a_k, \theta_k+a_k))} \text{ (clockwise oriented
when looking at}\\&\text{ the plane $\{\eps\}\times \R^2$ from $t>\eps$)},\\
\mathbb{S}(L)_1 &=\jump{\{\eps\}\times [1-\lambda_k,1-\lambda_k^\prime]\times \{ \theta_k+a_k\} },\\
\mathbb{S}(L)_2&= \jump{\{\eps\}\times [1-\lambda_k,1-\lambda_k^\prime]\times \{-\theta_k-a_k\} },\\
 \mathcal L_k &= \jump{\{\eps\}\times\{ 1-\lambda_k^\prime\}\times  ((-\pi,\pi]\setminus (-\theta_k-a_k, \theta_k+a_k))}\\&\text{ 
(counterclockwise oriented when looking at the plane $\{\eps\}\times \R^2$ 
from $t>\eps$)}.
\end{align*}
Notice that
$$\partial \left(\GFkethree +\GminusFkethree + \jump{\Sigmake}\right)
\res \left(\{\eps\}\times\R^2\right) = - \Big ( \mathbb
S(\scriptHkeps)+\mathbb{S}(L)_1 - \mathbb{S}(L)_2 \Big) +
\jump{\{\eps\}\times \partial B_{1-\lambda_k^\prime}}- \mathcal L_k,
$$
Hence
$$\partial \left(\GFkethree +\GminusFkethree + \jump{\Sigmake} +\Vke\right) 
\res
\left(\{\eps\}\times\R^2\right) = \jump{\{\eps\}\times \partial
B_{1-\lambda_k^\prime}}
$$
Thus, for $\hatFke$ in \eqref{strip_Jextension2},
\begin{equation*}
\hatFke(t,\rho)=\begin{cases}
\theta_k+a_k &\qquad \text{in } (\eps,l)\times(1-\lambda_k,1-\lambda_k']\times \{0\},
\\
0 &\qquad \text{in } (\eps,l)\times(1-\lambda_k',1] \times \{0\},
\\
\pi &\qquad  \text{in } ((0,\eps]\times(0,1]) \cup ((\eps,l)\times 
(0,1-\lambda_k]) \times \{0\},
\end{cases}
\end{equation*}
and
$\GplusminushatFkfour$ in \eqref{eq:def_generalizedgraphs_Fextended}
is given by
\begin{align*}
\GplusminushatFkfour=&
\GplusminusFkethree+
\Big(
\jump{\Sigmake} + \Vke + \intannulus +\jump{(0,\eps)\times \partial B_1}
\Big) \res \{0\leq\pm\theta\leq\pi\}
\pm\jump{(\eps,l)\times [1-\lambda_k',1]\times \{0\}},
\end{align*}
where $\intannulus= \jump{\{\eps\}\times ( B_{1}\setminus B_{1-\lambda_k'})}$.
Observe that
$$
(\{0\} \times \{1\} \times \{\theta\in [0,\pi]\}) \cup\{[0,\longR]\times
\{1\}\times\{0\}\} \subset {\rm supp} (\partial
\GplusminushatFkfour) .$$
Hence
$$\GhatFkfour + \GminushatFkfour = \GFkethree+\GminusFkethree+
\jump{\Sigmake} + \Vke + \intannulus +\jump{(0,\eps)\times \partial B_1},$$
and
$$\partial (\GhatFkfour + \GminushatFkfour )\res \{t<l\}=
\jump{\{0\}\times \partial B_1}
$$
\begin{remark}
$(u_k)$ is not a recovery sequence, 
due to Theorem \ref{Thm:maintheorem}.
We have 
$$
\lim_{k \to +\infty}  
\area(u_k,\Omega) =  \int_\Om|\mathcal M(\nabla \vortexmap)|~dx
+2\pi l,
$$
and $2 \pi l$ has the meaning of the lateral area of the cylinder of 
height $\longR$ and basis the unit disc. This surface is not a minimizer
of the problem on the right-hand side of \eqref{infimum_pb} 
(where it corresponds to $h \equiv 1$).
\end{remark}

\subsection{An approximating sequence of maps
with degree zero: catenoid union a flap}
\label{sec:catenoid_with_a_flap}
In this section we 
discuss another example of a possible approximating sequence $(u_k)$. 
We replace the cylinder lateral surface\footnote{In polar coordinates.} $[0,\longR]\times \{1\}\times(-\pi,\pi]$, which 
contains the image of 
$(r_k, l) \times (-\theta_k,\theta_k)$ through the map 
$\Psi_k$
in the 
example of Section 
\ref{subsec:an_approximating_sequence_of_maps_with_degree_zero:cylinder}, with 
half\footnote{
For convenience, we consider the doubled segment $[0,2l]$,
in order to define the catenoid; then we restrict
the construction to $(0,\longR)$.} 
of a catenoid union a flap (see Fig. \ref{fig:cat_flap}): calling this union $CF \res (0,\longR)\times \R^2$, 
we have
$$ CF=: \{(t, \overline\rho(t),\theta): t\in [0,2l],~ \theta \in (-\pi,\pi] \} \cup \{(t,r,0): t\in (0,2l),~r\in [\overline\rho(t),1]\},$$ 
where $~\overline \rho(t):=a \cosh(\frac{t-l}{a}),$ and $a>0$ is such that $\overline\rho(0)=1$
(and $\overline \rho(2l)=0$).

Notice that $CF$ ``spans'' 
$\Big(\{0,2l\}\times \{1\}\times(-\pi,\pi]\Big) \cup \Big(  [0,2l]\times \{1\}\times\{0\} \Big)$, 
which is the union of two unit circles connected by a segment. 

Let $\reps>0, \theta_k>0$, $\overline \theta_k >\theta_k$ be
such that   $\reps, \theta_k, (\overline \theta_k-\theta_k) \to 0^+$ 
as $k\to +\infty$.
Set $$\rho(t):= \overline \rho 
\left(\frac{t-\reps}{l-\reps}l\right), \qquad t\in (\reps,l).$$
We define $u_k := u$ in $\Omega \setminus \Big( \sourcedisk_{\reps}  \cup \{\theta \in (-\overline \theta_k,\overline \theta_k)\}\Big)$,
in particular
$$
u_k(\partial \sourcedisk_t \setminus \{\theta\in (-\overline\theta_k, 
\overline\theta_k) \})= \partial B_1 \setminus \{\theta \in (-\overline \theta_k, \overline \theta_k) \}, \qquad t \in (r_k,l).
$$
On $ \{\theta \in (-\overline \theta_k,\overline \theta_k)\} \setminus  \sourcedisk_{\reps} $ we define $u_k$ in such a way that for each $t\in (\reps,l)$ we have 
\begin{align*}
u_k\left(\partial \sourcedisk_t \cap 
\{\pm\theta\in (\theta_k, \overline{\theta_k}) \}\right)
&=\partial B_1 \cap  \{\pm\theta \in (0, \overline \theta_k) \} ,\\
u_k\left(\partial \sourcedisk_t \cap \{\pm\theta\in (0,\theta_k) \}\right)
&=\{(r,0)\in B_1: r\in [\rho(t),1]\} \cup\Big(\partial B_{\rho(t)} \cap  \{\pm\theta \in (0, \pi) \} \Big).
\end{align*}
See Fig. \ref{fig:D_k_in_Catenoid_EX} for a representation of the map $u_k$.
\begin{figure}
\begin{center}
    \includegraphics[width=0.9\textwidth]{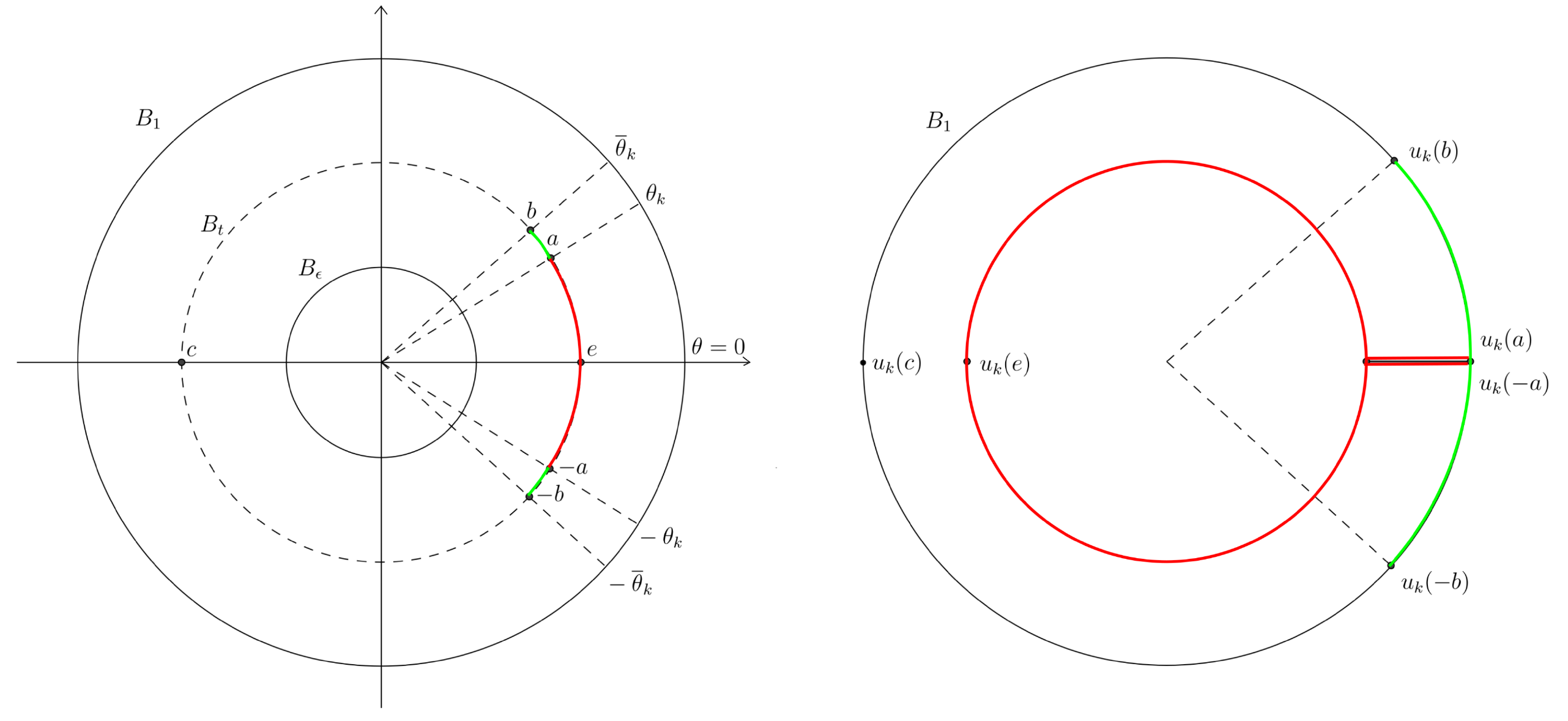}
 \caption{Source and target of the map $u_k$ in the example of Section \ref{sec:catenoid_with_a_flap}. The small interior circle
in the right figure is a $t$-slice of a catenoid, whereas the horizontal segment is the $t$-section of the flap.}
\label{fig:D_k_in_Catenoid_EX}
\end{center}
\end{figure}
To define $u_k$ on $\sourcedisk_\reps$ we adopt a construction similar to
the one in \eqref{eqn:u_k_Cylinder}. First of all, $u_k(0,0):= (-1,0)$.
Then,
in 
$\sourcedisk_\reps \cap \{\theta\in (-\pi,\pi) \setminus(-\overline \theta_k, \overline \theta_k)\}$ we impose $u_k$ as in 
\eqref{eqn:u_k_Cylinder} with $\overline \theta_k$ replacing $\theta_k$. 
In 
$\sourcedisk_\reps \cap \{\theta\in (-\overline \theta_k, \overline \theta_k)\}$ we require
$$u_k([0,r_k],\alpha):= \partial B_1 \cap \{\pm \theta \in (u_k(r_k,\alpha),\pi) \},
\qquad \pm \alpha \in (0,\pi].
$$
\begin{figure}
	\begin{center}
		\includegraphics[width=0.5\textwidth]{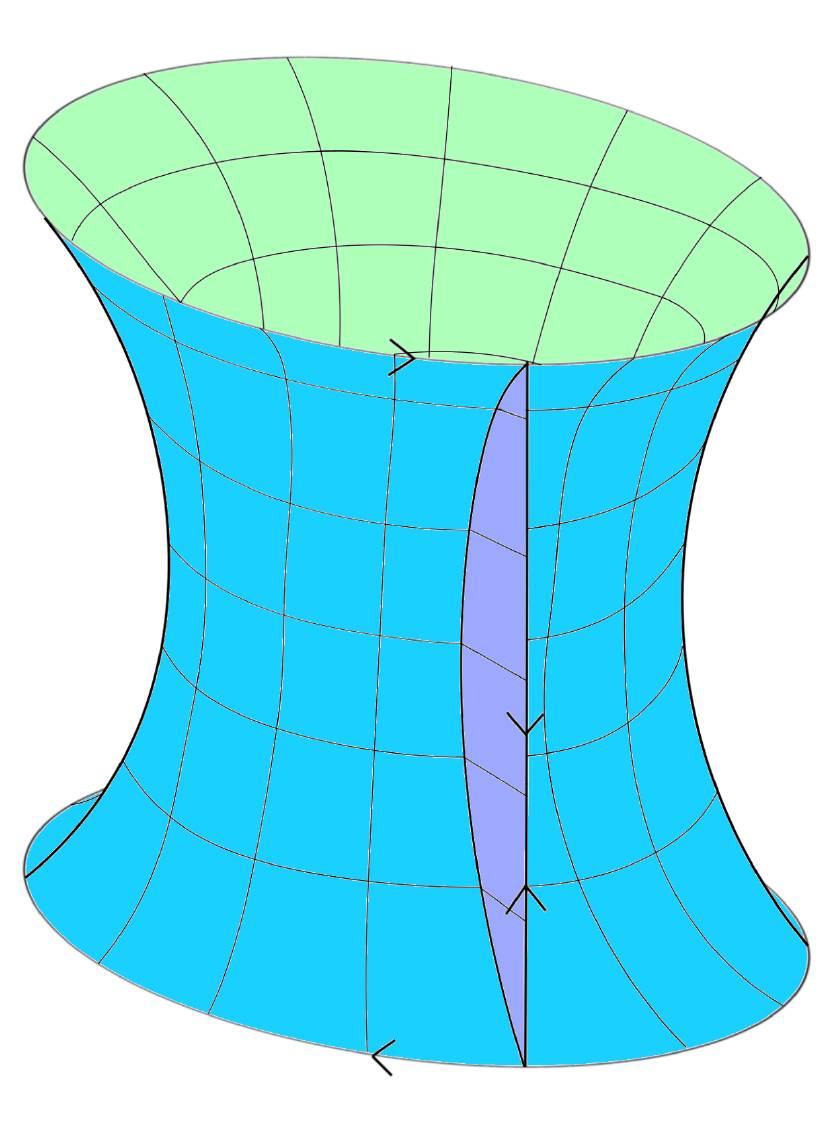}
		\caption{Catenoid union a flap (Section \ref{sec:catenoid_with_a_flap}). This is the set $CF$, 
namely the limit (as $k\rightarrow+\infty$) of the image by $\pi_{\lambda_k}\circ\Psi_k$ of $D_k$.}
		\label{fig:cat_flap}
	\end{center}
\end{figure}
Hence 
\begin{equation*}
u_k(\partial \sourcedisk_t) \begin{cases}
 \subsetneq\partial B_1\qquad  &\text{if } t \in (0,\reps],\\
 = (\partial B_1) \cup \{(r,0)\in \sourcedisk_1: r\in [\rho(t),1]\} \cup (\partial B_{\rho(t)})\qquad  &\text{if } t \in (\reps,l).
\end{cases}
\end{equation*}
\begin{remark}\rm
For an explicit construction, see Section \ref{sec:upper_bound} with $\hstar=\overline\rho$ and 
\begin{equation*}
\psionestar(s,t)=\begin{cases}
0 &\qquad \text{for  } t \in (0,\longR), ~s\in [-1,-\overline\rho(t)],\\
\sqrt{(\overline\rho(t))^2-s^2}&\qquad \text{for  } t\in (0,\longR),~s\in [-\overline\rho(t),\overline\rho(t)].
\end{cases}
\end{equation*}
\end{remark}

Now we fix an infinitesimal sequence $(\lambda_k)$ of positive numbers; 
we may also assume that $\vert (\cos \overline\theta_k, \sin \overline\theta_k)- (1,0)\vert << \lambda_k$. 
Hence $\badset\subsetneq (\Balleps \cup \{\theta\in (-\theta_k, \theta_k)\})$.

By construction there exists $t_k \in (\reps,\eps)$ 
such that $\rho(t_k)=1-\lambda_k$. 
Hence for $t\in [t_k,l)$ we have 
\begin{align*}
{\rm spt}\Big((u_k)_\sharp 
(\partial \sourcedisk_t \cap(\Omega\setminus \badset))\Big)&= 
\partial B_1,\qquad \partial B_1 \subsetneq  
u_k (\partial \sourcedisk_t \cap(\Omega\setminus \badset)),\\
{\rm spt}\Big((u_k)_\sharp (\partial \sourcedisk_t \cap\badset)\Big)&
= \partial B_{\rho(t)},\qquad \partial B_{\rho(t)} \subsetneq u_k(\partial \sourcedisk_t \cap\badset).
\end{align*}
Notice that the above strict inclusions are due to the fact that the segments $\{(r,0)\in B_1: r\in [\rho(t),1]\} \cap u_k(\badset)$ and $\{(r,0)\in B_1: r\in [\rho(t),1]\} \cap u_k(\Omega \setminus \badset)$
are covered twice with opposite orientation. 

On the other hand, for $t\in [\reps,t_k)$ we have 
\begin{align*}
\partial B_1 &\subset {\rm spt}\Big((u_k)_\sharp 
(\partial \sourcedisk_t \cap(\Omega\setminus \badset))\Big), \qquad 
{\rm spt}\Big((\projlambdak\circ \Psi_k)_\sharp 
(\partial \sourcedisk_t \cap(\Omega\setminus \badset))\Big)\subset \partial B_{1-\lambda_k}\\
\partial B_{\rho(t)} &\supset 
{\rm spt}\Big((u_k)_\sharp (\partial \sourcedisk_t \cap \badset)\Big), 
\qquad 
{\rm spt}\Big((\projlambdak\circ \Psi_k)_\sharp (\partial \sourcedisk_t \cap \badset)\Big)\subset \partial B_{1-\lambda_k},
\end{align*}
where the second inclusion is due to the fact that $(\projlambdak\circ \Psi_k)(\partial \sourcedisk_t \cap(\Omega\setminus \badset))$ covers two arcs of $\partial B_{1-\lambda_k}$ twice with opposite orientation.
Notice also that the only cancellation that could happen on $(\projlambdak\circ \Psi_k) (\partial \sourcedisk_t \cap \badset)$ due to covering more than one with opposite orientation is along the 
segment $\{(r,0)\in B_1: r\in [\rho(t),1]\} $, 
in particular we have   
$$ 
{\rm spt}\Big((\projlambdak\circ \Psi_k)_\sharp (\partial \sourcedisk_t \cap(\Omega\setminus \badset))\Big)={\rm spt}\Big((\projlambdak\circ \Psi_k)_\sharp (\partial \sourcedisk_t \cap \badset)\Big).
$$
By definition the above equality holds also for $t\in (0,\reps)$; moreover we have 
$$
{\rm spt}\Big((\projlambdak\circ \Psi_k)_\sharp (\partial \sourcedisk_{t_1} \cap \badset)\Big)\subsetneq {\rm spt}\Big((\projlambdak\circ \Psi_k)_\sharp (\partial \sourcedisk_{t_2} \cap \badset)\Big),\qquad t_1<t_2 , ~ t_1, t_2\in (0,\reps),$$
where
$\projlambdak\circ \Psi_k(0) = \{0\}\times\{1-\lambda_k\}\times\{\pi\}$.

From the above discussion we can see that the support of the current $\piPsiDk$ is a connected set contained in $\overline C_l(1-\lambda_k)$
and the support of $(\projlambdak\circ \Psi_k)_\sharp\jump{\Omega\setminus \badset}$ is a connected set contained in $\partial_{{\rm lat}}C_l(1-\lambda_k)$, in particular we have 
 \begin{align*}
  &{\rm spt}\Big(\piPsiDk \Big)\res C_{t_k}
={\rm spt} \Big((\projlambdak\circ \Psi_k)_\sharp\jump{\Omega\setminus \badset}\Big)\res C_{t_k}~\subsetneq \partial C_{t_k}(1-\lambda_k),\\
 &{\rm spt}\Big(\piPsiDk\Big) \res C^{t_k}_l 
=\{(t,r,\theta): t\in (t_k,l),~ r=\rho(t),~ \theta \in (-\pi,\pi]\}\subsetneq C_l^{t_k}(1-\lambda_k),\\
&{\rm spt}\Big(\partial \piPsiDk\Big)\subset \partial C_{t_k}(1-\lambda_k),\\
 &{\rm spt} \Big((\projlambdak\circ \Psi_k)_\sharp\jump{\Omega\setminus \badset}\Big)\res C^{t_k}_l 
= \partial C_l^{t_k}(1-\lambda_k).
 \end{align*}
 The current 
  $ \Xik$, a normal wall over $\partial \piPsiDk$ over $C_l(1-\lambda_k)$ of 
height $\lambda_k - \lambda_k'$,
 built on $\partial \piPsiDk$,
divides  $C_l(1-\lambda_k^\prime) \setminus 
C_l (1-\lambda_k)$ into two connected sets; one of them, 
${\rm spt} (\mathcal{Y}_k)$, contains $C^{t_k}_l(1-\lambda_k^\prime) \setminus 
C^{t_k}_l (1-\lambda_k)$ and  has ${\rm spt} 
\left((\projlambdak\circ \Psi_k)_\sharp\jump{\Omega\setminus \badset}\right)$ 
as part of its boundary. The other one, 
${\rm spt} (\mathcal{X}_k)$ (see \eqref{eq:defX_k}), 
contains $C_{0}(1-\lambda_k^\prime) \setminus 
C_{0} (1-\lambda_k)$ and has $\partial C_l(1-\lambda_k) \setminus 
{\rm spt} (\projlambdak\circ \Psi_k)_\sharp\jump{\Omega\setminus \badset}$ as part of its boundary. 
 Hence $\piPsiDkXik=\piPsiDk+\Xik$ divides  $C_l(1-\lambda_k^\prime)$ into two connected sets, one of them is
   $$E_k=C_{t_k}(1-\lambda_k) \cup {\rm spt} (\mathcal{X}_k) \cup \{(t,r,\theta)\in C^{t_k}_{l}(1-\lambda_k): r\in(0,\rho(t)]\} , $$ 
  and the second one is 
its complement in $C_l(1-\lambda_k')$ and contains
${\rm spt}( \mathcal{Y}_k)$.
 Note that $
 {\rm spt}(\Xik) \subset [0,t_k]\times [1-\lambda_k,1-\lambda_k^\prime]\times(-\pi,\pi]$ and  that ${\rm spt}(\piPsiDkXik) \cap (\{0\} \times \R^2)$ is the segment
$\{0\} \times [1-\lambda_k, 1-\lambda_k'] \times \{\pi\}$.

In this particular example when we apply the cylindrical Steiner symmetrization introduced in Section \ref{sec:cylindrical_Steiner_symmetrization} nothing changes, {\it i.e.,} $\mathbb{S}(E_k)= E_k$ and 
hence $\mathbb{S}(\piPsiDkXik)= \piPsiDkXik$. 
Observe that 
\begin{align*}
|u_k|^+=1 
,\qquad |u_k|^-=\begin{cases}
1, &  r\leq\reps,\\
\rho(t), & r>\reps.
\end{cases}
\end{align*}
Thus $$\pi_0^{\rm pol}(\projlambdak \circ \Psi_k(\Omega \setminus \overline \sourcedisk_\eps))=\{(t,r,0):t\in (\eps,l)~ r\in[\rho(t),1-\lambda_k]\},$$
and the sets $Q_{k,\eps}$
in \eqref{Q_k} and $\JQke $ in \eqref{strip_Jextension1}  are empty. 
Hence $$\striptwo=\{(t,r,0):t\in (\eps,l)~ r\in[\rho(t),1-\lambda_k^\prime]\}.$$ 
Moreover we have $\Theta_k(t,\rho)=0$ for  $(t,\rho,0)\in \striptwo$ thus,
recalling Definitions \ref{def:the_currents_jump_gengraphplusminusFkJkzero} 
and \ref{Def: The current g3}, 
\begin{align*}
\GplusminusFkethree
=\pmcurrentgengraphFkJkzero &= \piPsiDkXik \res \{(t,\rho,\theta): t \in (\eps, l),~\pm \theta \in (0,\pi)\}+\jump{G^{\rm pol}_{\pm\Fke\res \striptwo\cap \{\Theta_k \in \{0,2\pi\}\}}}\\
&= \piPsiDkXik \res \{(t,\rho,\theta): t \in (\eps, l),~\pm \theta \in (0,\pi)\}\pm \jump{\{(t,r,0):t\in (\eps,l),~r\in [\rho(t),1-\lambda_k^\prime]\}}.
\end{align*}
The set $\Sigmake$ in Definition \ref{def:Sigmake} is empty,
hence $\GFkethree +\GminusFkethree + \jump{\Sigmake} = \piPsiDkXik$ which is boundaryless in $C^\eps_l$.

We have, recalling \eqref{eq:def_Vk}, 
 $$ \Vke =\jump{ \{\eps\} \times [\rho(\eps),1-\lambda_k^\prime]\times (-\pi,\pi]},$$ 
{\it i.e.,}
$$
\partial \Vke  = \mathbb S(\scriptHkeps)+\mathbb{S}(L)_1 - 
\mathbb{S}(L)_2 +  \mathcal L_k=-\jump{\{\eps\}\times \partial B_{\rho(\eps)}}+\jump{\{\eps\}\times \partial B_{1-\lambda_k^\prime}},
$$
where 
(see \eqref{eqn:SscriptHkeps} \eqref{eqn:scriptHkeps}, \eqref{def_X1X2}, \eqref{eqn:9.16}, and Proposition  \ref{prop:V_k})
\begin{align*}
 \mathbb S(\scriptHkeps) &=
   \scriptHkeps = (\projlambdak\circ\Psi_k)_\sharp
\jump{\badset \cap
\partial \sourcedisk _\eps}\text{ (oriented counterclockwise)} =
-\jump{\{\eps\}\times \partial B_{\rho(\eps)}},\\
\mathbb{S}(L)_1&=\mathbb{S}(L)_2= \jump{\{\eps\}\times [1-\lambda_k,1-\lambda_k^\prime]\times \{0\} },\\
 \mathcal L_k &=\jump{\overline{Y_2Y_1} }= \jump{\{\eps\}\times \partial B_{1-\lambda_k^\prime}}.
\end{align*}
Notice that
\begin{align*}
\partial (\GFkethree +\GminusFkethree ) \res \{\eps\}\times\R^2
&=\jump{\{\eps\}\times \partial B_{1-\lambda_k}} \\
&=  - \Big ( \mathbb S(\scriptHkeps)+\mathbb{S}(L)_1 - \mathbb{S}(L)_2
\Big) + \jump{\{\eps\}\times \partial B_{1-\lambda_k^\prime}}-
\mathcal L_k.
\end{align*}
Thus
$$\partial (\GFkethree +\GminusFkethree +\Vke) \res (\{\eps\}\times\R^2)
= \jump{\{\eps\}\times \partial B_{1-\lambda_k^\prime}},$$
and
\begin{align*}
\GplusminushatFkfour=&
\GplusminusFkethree+
\Big(\Vke + \intannulus +\jump{(0,\eps)\times \partial B_1}
\Big) \res \{0\leq\pm\theta\leq\pi\}
\pm\jump{(\eps,l)\times [1-\lambda_k',1]\times \{0\}},
\end{align*}
where $\intannulus= \jump{\{\eps\}\times ( B_{1}\setminus B_{1-\lambda_k'})}$;
observe that
$$(\partial B_{1} \cap\{0\leq\pm\theta\leq\pi\}) \cup\{[0,\longR]\times
\{1\}\times\{0\}\} \subset {\rm supp} (\partial
\GplusminushatFkfour).$$
Hence
$$\GhatFkfour + \GminushatFkfour = \GFkethree+\GminusFkethree + \Vke +
\intannulus +\jump{(0,\eps)\times \partial B_1},$$
and
$$\partial (\GhatFkfour + \GminushatFkfour )\res \{t<l\}=
\jump{\{0\}\times \partial B_1}.
$$

\begin{remark}		
$(u_k)$ is not a recovery sequence, 
due to Theorem \ref{Thm:maintheorem}.
We have 
$$
\lim_{k \to +\infty}  
\area(u_k,\Omega) =  \int_\Om|\mathcal M(\nabla \vortexmap)|~dx
+\mathcal H^2
({\rm catenoid})+2\, \mathcal H^2({\rm flap}).
$$
This surface is not a minimizer
of problem on the right-hand side of \eqref{infimum_pb}.
\end{remark}

\subsection{Smoothing by convolution: the case of the two discs}
\label{sec:smoothing_by_convolution}
In \cite{AcDa:94}, the authors describe
a sequence $(u_k)$ of maps converging to  
the vortex map $\vortexmap$,  simply defined as follows:
 \begin{equation}\label{eqn:u_k_two_disks}
\veps(r,\theta):=\phi_k(r)u(r,\theta),
\end{equation}
where $\phi_k:[0,\longR]\to [0,1]$ is a smooth function 
such that $\phi_k=0$ in $[0,\frac{1}{k^2}]$, 
$\phi_k=1$ in $[\frac{1}{k},l ]$, 
and $0\leq \phi_k^\prime\leq 2k$. Hence $|u_k -u|=1-\phi_k $.
We shall assume that for all $k>0$, we have $\frac{1}{k} << \eps$. 

Now we fix an infinitesimal sequence $(\lambda_k)$ of positive
numbers, hence for any $k \in \mathbb N$ there exists $r_k\in (0,1/k)$ 
such that $\phi_k(r_k)=1-\lambda_k$, and we have  
\begin{equation}
\label{eq:easy_equality}
\badset = \sourcedisk_{r_k}.
\end{equation}
Notice that $u_k(\partial \sourcedisk_r)=\partial B_{\phi_k(r)}$,
\begin{align*}
&
{\rm spt}(\piPsiDk)
=\{(t,r,\theta):t\in [1/k^2,r_k], ~r =\phi_k(t),\; \theta \in (-\pi,\pi]\}\subsetneq \overline C_{r_k}(1-\lambda_k),\\
&{\rm spt}((\projlambdak\circ \Psi_k)_\sharp \jump{\Omega \setminus \badset })=\projlambdak\circ \Psi_k(\Omega \setminus \badset )=\partial C^{r_k}_l(1-\lambda_k).
\end{align*}
Also, using \eqref{eq:easy_equality},
 $\projlambdak\circ \Psi_k(\badset ) 
\setminus {\rm spt}(\piPsiDk)=
[0,1/k^2)\times\{0\}\times\{0\} = 
\projlambdak\circ \Psi_k(\sourcedisk_{1/k^2})$ and,
 unlike the examples in Sections 
\ref{subsec:an_approximating_sequence_of_maps_with_degree_zero:cylinder}, \ref{sec:catenoid_with_a_flap}, there is no cancellation due to 
covering the same $2$-dimensional
set with two opposite orientations; the fact 
that $\projlambdak\circ \Psi_k(\badset ) \setminus {\rm spt}(\piPsiDk)$
is nonempty is due to the fact that 
$\projlambdak\circ \Psi_k(\sourcedisk_{1/k^2})$ is one-dimensional.
Moreover we have $${\rm spt}((\projlambdak\circ \Psi_k)_\sharp \jump{\partial \badset } )={\rm spt}(\piPsiDk) \cap {\rm spt}((\projlambdak\circ \Psi_k)_\sharp \jump{\Omega \setminus \badset }) = \{r_k\}\times \{1-\lambda_k\}\times(-\pi,\pi].$$
 Hence 
$${\rm spt}(\Xik)= \{r_k\}\times [1-\lambda_k,1-\lambda_k^\prime]\times(-\pi,\pi],$$
 {\it i.e., } the current 
  $ \Xik$
divides  $C_l(1-\lambda_k^\prime) \setminus 
C_l (1-\lambda_k)$ into two connected sets; one of them, 
${\rm spt} (\mathcal{Y}_k)= C^{r_k}_l(1-\lambda_k^\prime) \setminus 
C^{r_k}_l (1-\lambda_k) $, and the other one, 
${\rm spt} (\mathcal{X}_k)=C_{r_k}(1-\lambda_k^\prime) \setminus 
C_{r_k} (1-\lambda_k)$.
Therefore $\piPsiDkXik=\piPsiDk+\Xik$ divides  $C_l(1-\lambda_k^\prime)$ into two connected sets, one of them is 
 $$E_k=C_{r_k}(1-\lambda_k^\prime) \setminus \{(t,r,\theta): t\in (1/k^2,r_k),~ r\in(0, \phi_k(t)), ~ \theta \in (-\pi,\pi]\}  , $$
and the second one is 
its complement in $C_l(1-\lambda_k')$ and contains
${\rm spt}( \mathcal{Y}_k)$.

Also in this example,
 when we apply the cylindrical Steiner symmetrization 
introduced in Section \ref{sec:cylindrical_Steiner_symmetrization}, nothing changes, {\it i.e.,} $\mathbb{S}(E_k)= E_k$ and 
hence $\mathbb{S}(\piPsiDkXik)= \piPsiDkXik$. 

Note also that $|u_k|^+=|u_k|^-=1$ in $(r_k,l)$, thus $\pi_0^{\rm pol}(\projlambdak \circ \Psi_k(\Omega \setminus \overline \sourcedisk_\eps))=(\eps,l)\times \{1-\lambda_k\}\times\{0\}$, and the sets $Q_{k,\eps}$
in \eqref{Q_k} and $\JQke $ in \eqref{strip_Jextension1}  are empty. 
Hence $\striptwo=(\eps,l)\times [1-\lambda_k,1-\lambda_k']\times\{0\}$.
We have $\Theta_k(t,\rho)=0$ for  $(t,\rho,0)\in \striptwo$ thus,
recalling Definitions \ref{def:the_currents_jump_gengraphplusminusFkJkzero} 
and \ref{Def: The current g3}, 
\begin{align*}
\GplusminusFkethree
=\pmcurrentgengraphFkJkzero &=\pm\jump{ \striptwo}= \pm\jump{(\eps,l)\times [1-\lambda_k,1-\lambda_k']\times\{0\}}.
\end{align*}
The set $\Sigmake$ in Definition \ref{def:Sigmake} is empty,
hence $\GFkethree +\GminusFkethree + \jump{\Sigmake} = 0$.
We have, recalling \eqref{eq:def_Vk}, 
 $$ \Vke =\jump{ \{\eps\} \times \overline B_{1-\lambda_k^\prime}},$$ 
and 
$$
\partial \Vke  = \mathbb S(\scriptHkeps)+\mathbb{S}(L)_1 - \mathbb{S}(L)_2 +  \mathcal L_k= \jump{\{\eps\}\times \partial B_{1-\lambda_k^\prime}},
$$
where (see \eqref{eqn:SscriptHkeps} \eqref{eqn:scriptHkeps}, \eqref{def_X1X2}, \eqref{eqn:9.16}, and Proposition  \ref{prop:V_k})
\begin{align*}
& \mathbb S(\scriptHkeps) =
   \scriptHkeps = 0,\\
&\mathbb{S}(L)_1 =\mathbb{S}(L)_2=\jump{\{\eps\}\times [1-\lambda_k,1-\lambda_k^\prime]\times \{ 0\} },\\
& \mathcal L_k = \jump{\{\eps\}\times \partial B_{1-\lambda_k^\prime}}.
\end{align*}
Notice that
$$\partial (\GFkethree +\GminusFkethree ) \res (\{\eps\}\times\R^2) = -
\Big ( \mathbb S(\scriptHkeps)+\mathbb{S}(L)_1 - \mathbb{S}(L)_2 \Big)
+ \jump{\{\eps\}\times \partial B_{1-\lambda_k^\prime}}- \mathcal
L_k=0, $$
and
$$\partial (\GFkethree +\GminusFkethree  +\Vke) \res
(\{\eps\}\times\R^2) = \jump{\{\eps\}\times \partial
B_{1-\lambda_k^\prime}}.$$
Finally we have
\begin{align*}
\GplusminushatFkfour=&
\GplusminusFkethree+
\Big(\Vke + \intannulus +\jump{(0,\eps)\times \partial B_1}
\Big) \res \{0\leq\pm\theta\leq\pi\}
\pm\jump{(\eps,l)\times [1-\lambda_k',1]\times \{0\}}
\\
=&
 \jump{ \{\eps\} \times \overline B_{1}}+\jump{(0,\eps)\times \partial
B_1}\pm\jump{(\eps,l)\times [1-\lambda_k,1]\times \{0\}},
\end{align*}
where $\intannulus= \jump{\{\eps\}\times ( B_{1}\setminus B_{1-\lambda_k'})}$.
Notice that

$$(\partial B_{1} \cap\{0\leq\pm\theta\leq\pi\}) \cup\{[0,\longR]\times
\{1\}\times\{0\}\} \subset {\rm supp} (\{\partial
\GplusminushatFkfour\}).
$$
Hence
$$\GhatFkfour + \GminushatFkfour =  \jump{ \{\eps\} \times \overline
B_{1}}+\jump{(0,\eps)\times \partial B_1},$$
and
$$\partial (\GhatFkfour + \GminushatFkfour )\res \{t<l\}=
\jump{\{0\}\times \partial B_1}.
$$
\begin{remark}
$(u_k)$ is a recovery sequence for $\longR$ 
sufficiently large, due to \cite[Lemma 4.2]{AcDa:94}.
We have 
$$
\lim_{k \to +\infty}  
\area(u_k,\Omega) =  \int_\Om|\mathcal M(\nabla \vortexmap)|~dx
+\pi,
$$
and $\pi$ has the meaning of the area of the unit disc.
This surface, for $\longR$ sufficiently large, is a minimizer
of problem on the right-hand side of \eqref{infimum_pb} 
(where it corresponds to $h \equiv -1$).
\end{remark}

\section{Lower bound}
\label{sec:lower_bound}	
In this section we reduce the analysis of
$\GplusminushatFkfour$ in Definition \ref{def:G^4_keps} 
 to a non-parametric Plateau-type problem with a sort of free boundary. 
Precisely, after suitable projections, we will arrive to a Plateau-type problem on 
the closed
rectangle $\overline R_l$, where
$$
R_l:=(0,\longR)\times (-1,1)\times \{0\}
$$
in Cartesian coordinates, equivalently
$R_l = 
\{t\in (0,\longR),\;\rho\in [0,1),\;\theta=0\}\cup\{t\in (0,\longR),\;\rho\in 
[0,1),\;\theta=\pi\}$ in cylindrical coordinates. 
The rectangle $R_l$ will be 
often identified with 
$(0,\longR)\times (-1,1)$, thus neglecting the third coordinate.
We will impose a Dirichlet boundary condition $\varphi$ on a part 
\begin{equation}\label{eq:Dirichlet_part_of_boundary}
\partial_DR_l:=(\{0\}\times [-1,1])\cup([0,\longR]\times \{-1\})
\end{equation}
of $\partial R_l$, while no conditions will be imposed on $\{l\}
\times (-1,1)$; more involved conditions will be assigned
on $(0, \longR) \times \{1\}$,
see the mutual relations between $\psi$ and $h$ in \eqref{eq:X_l}
(see also the problem on the right-hand side of    
\eqref{infimum_pb} and  Section \ref{sec:structure_of_minimizers}). 

Then the strategy to estimate from below the relaxed area of the graph of the 
vortex map $\vortexmap$  will  be the following:
We split
\begin{align*}
 \area(u_k,\Om)=\int_{\Omega\setminus \badset}|\mathcal M(\nabla u_k)|~dx+\int_{\badset}|\mathcal M(\nabla u_k)|~dx.
\end{align*}
In order to estimate the $\liminf_{k \to +\infty}$ of the first term 
on the right-hand side
we will employ \eqref{estimate_Dc}, 
whereas, in order to pass to the limit as $k\rightarrow +\infty$ in 
the second term, we will use
\eqref{crucial_ineq}, so that we first want to render the right-hand side of this latter
inequality independent of $k$. This will be done with the aid 
of the non-parametric Plateau-type problem studied in this section 
and in Section \ref{sec:structure_of_minimizers}.

\begin{definition}[\textbf{The projection $p$}]
We let $p:C_l\cap \{t\geq0\} \to 
\overline R_l$ be the othogonal projection.
\end{definition}
Recall that
$\mathcal G_{\pm
\widehat \vartheta_{k,\eps}}^{(4)}$
 are defined in \eqref{eq:def_generalizedgraphs_Fextended}, 
that $\GhatFkfour$ is the generalized polar graph of $\hatFke$ on its domain of definition (see \eqref{strip_Jextension2}), and that $\hatFke$ takes values in $[0,\pi]$. 
We first prove the following preliminary result:
\begin{lemma}\label{lemma_domain}
Let $\eps \in (0,1)$ be as in \eqref{eq:H1}, \eqref{eq:H2},
 and $k \in \mathbb N$.
Then there is a negligible set $C_{k,\eps}\subset (0,\longR)$ such that 
for all $t\in (0,\longR)\setminus C_{k,\eps}$ 
 \begin{equation}\label{conclusion_lemma_domain}
p\Big({\rm spt}(\GhatFkfour)
\Big)\cap (\{t\}\times \R^2)
 \end{equation}
 is a subinterval of the segment $\overline R_l\cap(\{t\}\times\R^2)
= \{t\} \times [-1,1] \times \{0\}$ 
with  one endpoint $(t,1,0)$.
Moreover $p({\rm spt}(\GhatFkfour))=
p({\rm spt}(\GminushatFkfour))$.
\end{lemma}
\begin{proof}
The latter  assertion follows by symmetry. To prove the former,
we argue by slicing. 
For a.e. $t\in (0,\longR)$ the set ${\rm spt}(\GhatFkfour)\cap (\{t\}\times \R^2)$ coincides with support of the current
$(\GhatFkfour)_t$, see \cite[Def. 7.6.2]{Krantz_Parks:08}. 
First notice that for all  $t\in(0,\eps)$ the 
conclusion 
follows by construction\footnote{In this set 
we have $\Fke=\pi$ and the current $(\GhatFkfour)_t$ is the integration over the 
half-circle $\{t\}\times ((\partial B_1)\cap \{\theta\in(0,\pi)\})$,
whose projection through $p$ is 
the whole interval with endpoints $(t,1,0)$ and  
$(t,1,0)$ (in cylindrical coordinates).}.

It remains to consider the case $t\in (\eps,l)$.
{}Recall that 
the set $\stripfour$ in \eqref{newJ_k} has the form
\begin{align*}
 \{t\in(\eps,l),\;\rho\in [\minuk(t)\wedge (1-\lambda_k),1],\;\theta=0\}.
\end{align*}
Therefore, for a.e. $t\in (\eps,l)$ the slice $(\GhatFkfour)_t$ 
is the integration over ${\rm spt}(\GhatFkfour)$ 
restricted to the plane $\{t\}\times\R^2$, which in turn is the 
integration over the generalized polar graph (see \eqref{eq:generalized_polar_graph})
of $\hatFke$ restricted to the closed set 
$$
\{t\}\times [\minuk(t)\wedge (1-\lambda_k),1]\times [0,\pi].
$$
Namely $$(\GhatFkfour)_t=\jump{{\rm spt}(\GhatFkfour)\cap\big(\{t\}\times [\minuk(t)\wedge (1-\lambda_k),1]\times [0,\pi]\big)},$$
so that the support $\sigma_t$ of $(\GhatFkfour)_t$ can also be 
obtained as
\begin{align}\label{sigma_t}
\sigma_t=\bigcap_{h=1}^{+\infty}\sigma_t^h,
\end{align}
where 
$$\sigma_t^h:={\rm spt}(\GhatFkfour)\cap\Big(\{t\}\times [\big(\minuk(t)\wedge (1-\lambda_k)\big)-\frac1h,1]\times [0,\pi]\Big).$$

For $h\in \mathbb N$ large enough, let 
$$
U_h:=\{t\}\times \big((\minuk(t)\wedge (1-\lambda_k))-\frac1h,1\big)\times (-\frac1h,\pi+\frac1h),
$$ which is a relatively open set in $\{t\}\times B_1$,
and let $(\GhatFkfour)_t$ be the slice of $\GhatFkfour$ on $\{t\}\times B_1$.
We have
\begin{align}\label{boundary_of_sigmath}
{\rm spt}((\GhatFkfour)_t\res U_h)\subset \sigma_t^h\subset\{t\}\times [\big(\minuk(t)\wedge (1-\lambda_k)\big)-\frac1h,1]\times [0,\pi].
\end{align}
 On the other hand
since $\GhatFkfour\res(\{t\}\times B_1)$ is the boundary of the subgraph of $\hatFke$ in $\{t\}\times B_1$, it is a closed $1$-integral current in $U_h$ and 
in $(\{t\}\times (B_1\setminus \overline{B}_{(\minuk(t)\wedge (1-\lambda_k))-\frac1h})$,
so that the boundary $\partial((\GhatFkfour)_t\res U_h) $ in $\mathcal D_1(\{t\}\times B_1)$ is 
supported on $(\partial U_h)
\cap ((\partial B_1)\cup \partial B_{(\minuk(t)\wedge (1-\lambda_k))-\frac1h})$.
{}From \eqref{boundary_of_sigmath},  the fact that $\hatFke=0$ 
at $(t,1)$ and that $\hatFke$ is constant on the segment $\big((\minuk(t)\wedge (1-\lambda_k))-\frac1h,\minuk(t)\wedge (1-\lambda_k)\big)$ 
with value either $0$ or $\pi$, we deduce that 
\begin{align}\label{boundary_of_sigmath2}
&{\rm spt}(\partial((\GhatFkfour)_t\res U_h))
\nonumber\\
&\subset\big(\{t\}\times \{(\minuk(t)\wedge (1-\lambda_k))-\frac1h\}\times\{0,\pi\}\big)\bigcup\big(\{t\}\times \{1\}\times\{0\}\big).
\end{align}
Moreover, if we set $$P_1:=(t,1,0)\text{ and } P_2^h:=\big(t,(\minuk(t)\wedge (1-\lambda_k))-\frac1h,\hatFke((\minuk(t)\wedge (1-\lambda_k))-\frac1h)\big),$$ 
from \eqref{boundary_of_sigmath2} it follows that 
$$
\partial((\GhatFkfour)_t\res U_h)=\delta_{P_2^h}-\delta_{P_1}.
$$
By decomposition of the integral $1$-current $(\GhatFkfour)_t\res U_h$ (see \cite[Section 4.2.25]{Federer:69}), 
there are at most countable Lipschitz curves $\{\alpha_i^h\}$ such that $\alpha^h_0$ connects $P_2^h$ to $P_1$, and $\alpha^h_i$ is closed for $i>0$. 
We claim that there cannot be closed curves $\alpha^h_i$, namely $\{\alpha^h_i\}_{i\in\mathbb N}=\{\alpha^h_0\}$.
 Indeed, since $\alpha_0^h$ connects $P_1$ and $P_2^h$, we see that $(\{t\}\times \partial B_\rho)\cap \alpha_0^h$ 
consists of at least one point for $\mathcal H^1$-a.e. $\rho\in \big((\minuk(t)\wedge (1-\lambda_k))-\frac1h,1\big)$. 
On the other hand, $(\{t\}\times \partial B_\rho)\cap \sigma_t^h$ consists of only one point\footnote{Because $\sigma_t^h$ is the support of a polar graph; the points where this intersection is not a 
singleton coincide with the values of $\rho$ where $\hatFke$ has a jump.} for $\mathcal H^1$-a.e. $\rho\in \big((\minuk(t)\wedge (1-\lambda_k))-\frac1h,1\big)$. So there cannot be other curves $\alpha_i^h$ otherwise the last condition will be violated.

From the claim we deduce that the current $ (\GhatFkfour)_t\res U_h$ is the integration over a simple curve $\alpha^h_0$ connecting $P_2^h$ and $P_1$, and
its support coincides with $\sigma_t^h$. Now, 
from \eqref{sigma_t} and the fact that $\sigma_t$ is a segment on $\{t\}\times \big((\minuk(t)\wedge (1-\lambda_k))-\frac1h,\minuk(t)\wedge (1-\lambda_k)\big)$, we conclude that also $\sigma_t$ must be a 
unique curve, say $\alpha_0$, 
connecting $P_1$ to $P_2:=\lim_{h\rightarrow\infty}P_2^h$.
By continuity of the projection by $p$, $\alpha_0$ is an interval with one endpoint in $p(P_1)=(t,1,0)$,
for a.e. $t\in (\eps,l)$.
\end{proof}

{\bf The new coordinates $(w_1,w_2,w_3)$.}
In what follows, it is convenient to revert the rectangle with respect to its second
coordinate: 
if $(t,\rho,\theta)\in [0,\longR]\times [0,1]\times (-\pi,\pi]$ are the 
cylindrical coordinates in the cylinder $C_l$ exploited so far, we introduce 
Cartesian coordinates $(w_1,w_2,w_3) \in [0,\longR] \times [-1,1] \times [-1,1]$ 
 defined as
\begin{align}\label{cartesian_coordinates}
 w_1:=t,\;\;w_2:=-\rho\cos\theta,\;\;w_3:=-\rho \sin\theta,
\end{align}
in such a way that the segment 
$\{0\leq t\leq l,
\rho=1,
\theta=0\}$ 
coincides with the bottom edge $[0,\longR]\times \{-1\}\times \{0\}$ of the rectangle
$\overline R_l$.

Thanks to Lemma \ref{lemma_domain} we are allowed to give the following
\begin{definition}[\textbf{The function $\hke$}]
\label{def:h_k}
Let $\eps \in (0,1)$ be as in \eqref{eq:H1}, \eqref{eq:H2},
 and $k \in \mathbb N$.
We define $\hke:[0,\longR]\rightarrow [-1,1]$ as
$$\hke (\axialcoordofcylinder):=\mathcal H^1\Big(p\big({\rm spt}(\GhatFkfour)
\big)\cap (\{\axialcoordofcylinder\}\times \R^2)\Big)-1.$$
\end{definition} 
For all $\axialcoordofcylinder\in(0,\longR)$ 
for which Lemma \ref{lemma_domain} is valid,  
we have that 
$1+\hke(\axialcoordofcylinder)$ equals 
the length of the interval in \eqref{conclusion_lemma_domain}. 
Now the content of Lemma \ref{lemma_domain} is that the $p$-projection of 
${\rm spt}(\GhatFkfour)$ on $\overline R_l$ is of the form
\begin{align}\label{SGh_support}
 p\Big({\rm spt}(\GhatFkfour)\Big)
=SG_{\hke}
:=
\{(w_1,w_2) \in R_l: w_1\in(0,\longR),
w_2\in(-1,\hke(w_1))\},
\end{align}
up to a set of zero $\mathcal H^2$-measure.
The function $\hke$ is built 
in such a way that
$(w_1,-1)$ and $(w_1,\hke(w_1))$ 
are the endopoints of the interval 
$p({\rm spt}(\GhatFkfour))\cap (\{w_1\}\times \R^2)$
for almost every $w_1 \in (0,\longR)$. 
Observe that 
$$\hke\geq -1+\lambda_k' \qquad {\rm in}~ (\eps,l),
$$ 
and 
$$
\hke=1 \qquad{\rm in}~  (0,\eps).
$$
Indeed, from Definition \ref{def:the_function_hatFke}, equation \eqref{newJ_k}
and Definition \ref{def:G^4_keps}, we see that the set $\Big((0,\longR) \times 
[1-\lambda_k',1] \times \{0\}\Big) \cup \Big((0,\eps) \times [-1,1] \times \{0\}\Big)$ 
is contained in $ p({\rm spt}(\GhatFkfour))$.

We have built $\Oke$ in \eqref{eq:def_O_k} as the set enclosed between 
$\GminushatFkfour$ and 
$\GhatFkfour$, see formula \eqref{eq:int_O_k}.
We now perform a (classical) Steiner 
symmetrization\footnote{Despite $\Oke$ is obtained 
by cylindrical symmetrization, it still can have 
``holes'' (see Fig. \ref{fig:set_cancellation} for a slice), that 
disappear
when further performing the Steiner symmetrization.} of the set $\Oke$ with respect to the plane $\{w_3=0\}$. 
We denote by $\mathbb S_{\rm cl}(\Oke)$ the symmetrized set. 
\begin{remark}\label{2.14_bisrem}
We emphasize that the set $\Oke$ in $(0,\eps)\times\R^2$ is exactly $(0,\eps)\times B_1$, and is already symmetric with respect to the plane containing $R_l$. 
For this reason $\Oke$ does not change (in that region)
after Steiner  symmetrization,  
 \begin{align}\label{semicircle}
  \Oke\cap \{w_1\in (0,\eps)\}=\mathbb S_{\rm cl}(\Oke)\cap \{w_1\in (0,\eps)\}.
 \end{align}
\end{remark}
Since
the perimeter does not increase when symmetrizing, from \eqref{eq:int_O_k} 
we conclude
\begin{align}\label{ineq_steinerclassic}
 |\GhatFkfour+\GminushatFkfour|\geq \mathcal H^2(\partial^* \mathbb S_{\rm cl}(\Oke)\cap ((0,l)\times \R^2)).
\end{align}
\begin{definition}[\textbf{The function $\psike$}]
We introduce the function $\psike:R_l\rightarrow[0,+\infty)$ as
\begin{equation}\label{eq:psike}
\psike(w_1,w_2):=\frac12\mathcal H^1(\{w_3:(w_1,w_2,w_3)\in \Oke\}),
\qquad (w_1,w_2) \in R_l.
\end{equation}
\end{definition}

We stress that the set where $\psike >0$ is contained, up to
$\mathcal H^2$-negligible sets,
in the region
$SG_{\hke}$ defined in \eqref{SGh_support}.
Notice also that $\psike$ may take the value $0$ in $SG_{\hke}$ on a set
of positive $\mathcal H^2$-measure. 

\begin{remark}\label{rem_datialbordo}
\begin{itemize}

\item[(i)] By definition of classical Steiner symmetrization, 
\begin{align*}
 \mathbb S_{\rm cl}(\Oke)&=\{w = (w_1,w_2,w_3)\in R_l\times \R:w_3\in (-\psike(w_1,w_2),\psike(w_1,w_2))\}\\
 &=\{w = (w_1,w_2,w_3)\in SG_{\hke}\times \R:w_3\in (-\psike(w_1,w_2),\psike(w_1,w_2))\},
\end{align*}
up to Lebesgue-negligible sets, the second equality 
following from the fact that 
$\psike=0$ almost everywhere in $R_\longR \setminus SG_{\hke}$; 

\item[(ii)] since $\Oke$ has finite perimeter, it follows that 
$\psike \in BV(R_l)$; 

\item[(iii)] since 
$\Oke \res ([0,\eps) \times \R^2) = C_l 
\res ([0,\eps) \times \R^2)$ 
and 
$\Oke \res 
([\eps,l) \times \R^2)$ 
is contained in 
$C_l(1-\lambda_k')\res ([\eps,l) \times \R^2)$ 
(as a consequence 
of \eqref{strip_Jextension2}), it follows that $\psike$ has null trace on the segments $(0,\longR)\times \{-1\}$ and $(0,\longR)\times \{1\}$.
\end{itemize}
\end{remark}

We can split $\partial^*\mathbb S_{\rm cl}(\Oke)$ as 
\begin{equation}\label{eq:we_can_split}
\partial^*\mathbb S_{\rm cl}(\Oke)=((\partial^*\mathbb S_{\rm cl}(\Oke))
\cap \{w_3> 0\})\bigcup 
((\partial^*\mathbb S_{\rm cl}(\Oke))\cap \{w_3<0\})=:(\partial^*\mathbb S_{\rm cl}(\Oke))^+\cup(\partial^*\mathbb S_{\rm cl}(\Oke))^-
\end{equation}
up to a set of $\mathcal H^2$-measure zero,
in such a way that 
\begin{equation}\label{uno}
 (\partial^* \mathbb S_{\rm cl}(\Oke))^+=
(\partial^* SG_{\psike})\cap \big(R_l\times (0,+\infty)\big),
\qquad (\partial^* \mathbb S_{\rm cl}(\Oke))^-
=(\partial^* UG_{-\psike})\cap \big(R_l\times (-\infty,0)\big),
\end{equation}
where $SG_{\psike}$ and $UG_{-\psike}$ are, respectively, 
the (standard) generalized subgraph and epigraph of $\pm\psike$ in $R_l\times \R$. 
Notice that, since $\psike\geq0$, 
\begin{equation}\label{3.150}
\begin{aligned}
 (\partial^* SG_{\psike})&\cap(R_l
\times [0,+\infty))
\\
&=
(\partial^* \mathbb S_{\rm cl}(\Oke))^+
\cup \{(w_1,w_2,0)\in SG_{\hke}:\psike=0\}\cup (R_l\setminus SG_{\hke}),
\\
 (\partial^* UG_{-\psike})&\cap(R_l\times (-\infty,0])
\\
 &=(\partial^* \mathbb S_{\rm cl}(\Oke))^-\cup \{(w_1,w_2,0)\in  SG_{\hke}:\psike=0\}\cup (R_l\setminus SG_{\hke}),
\end{aligned}
\end{equation}
up to $\mathcal H^2$-negligible sets.

We are ready to prove the following:

\begin{lemma}\label{lemma_9}
 We have
\begin{equation}\label{eq:lem9}
 \begin{aligned}
 |\GhatFkfour|+|\GminushatFkfour| \geq& 
  \mathcal H^2\left((\partial^* SG_{\psike})\cap
 (R_l\times [0,+\infty)\right)
+\mathcal H^2\left((\partial^* UG_{-\psike})\cap (R_l\times (-\infty,0])
\right)\\&-2\mathcal H^2(R_l\setminus SG_{\hke}).
 \end{aligned}
\end{equation} 
 Moreover, 
 $\mathcal H^2\left((\partial^* SG_{\psike})
\cap (R_l\times [0,+\infty))\right)=\mathcal H^2\left(
(\partial^* UG_{-\psike})\cap (R_l\times (-\infty,0])\right)$.
\end{lemma}
\begin{proof}
The last assertion follows by symmetry. Let us prove the former:
 By \eqref{3.150} we have
 \begin{align*}
  \mathcal H^2(\partial^* SG_{\psike}\cap(R_l\times [0,+\infty)))=&
\mathcal H^2\big((\partial^* \mathbb S_{\rm cl}(\Oke))^+\big)+\mathcal H^2( \{(w_1,w_2)\in SG_{\hke}:\psike=0\})\\&+\mathcal H^2(R_l\setminus SG_{\hke}),\\
 \mathcal H^2(\partial^* UG_{-\psike}\cap(R_l\times (-\infty,0]))=&\mathcal H^2\big((\partial^* \mathbb S_{\rm cl}(\Oke))^-\big)+ \mathcal H^2(\{(w_1,w_2)\in SG_{\hke}:\psike=0\})\\&+\mathcal H^2(R_l\setminus SG_{\hke}).
 \end{align*}
Taking the sum of these two expressions 
and using \eqref{ineq_steinerclassic}, \eqref{eq:we_can_split},
we obtain
\begin{align*}
 &\mathcal H^2(\partial^* SG_{\psike}\cap(R_l\times [0,+\infty)))+\mathcal H^2(\partial^* UG_{-\psike}\cap(R_l\times (-\infty,0]))\\&\leq |\GhatFkfour+\GminushatFkfour|+2\mathcal H^2( \{(w_1,w_2)\in SG_{\hke}
:\psike=0\})+2\mathcal H^2(R_l\setminus SG_{\hke}).
\end{align*}
 Recalling \eqref{newJ_k},
we now claim that, up to $\mathcal H^2$-negligible sets, 
\begin{align}\label{claim_lemma_315}
  \{(w_1,w_2)\in SG_{\hke}:\psike(w_1,w_2)=0\}
\subset
\big\{
\hatFke =0\big\}\cap \stripfour,
\end{align}
see Fig. \ref{fig:set_cancellation}.
From the claim it follows that 
$$\mathcal H^2( \{(w_1,w_2)\in SG_{\hke}:\psike=0\})\leq \mathcal H^2\big(\big\{\hatFke \in\{ 0,\pi\}\big\}\cap \stripfour\big),$$ and hence by \eqref{sum_generalizedgraphextended} we conclude
\begin{align*}
 &\mathcal H^2(\partial^* SG_{\psike}\cap(R_l\times [0,+\infty)))+\mathcal H^2(\partial^* UG_{-\psike}\cap(R_l\times (-\infty,0]))\\
 &\leq |\GhatFkfour|+|\GminushatFkfour|+2\mathcal H^2(R_l\setminus SG_{\hke}),
\end{align*}
that is \eqref{eq:lem9}. 
It remains to show \eqref{claim_lemma_315}. 
As usual, we argue by slicing; hence for almost all 
$w_1\in (0,\longR)$ we will show that  
\eqref{claim_lemma_315} holds (up to $\mathcal H^1$-negligible sets). Notice that both the left and right-hand sides of \eqref{claim_lemma_315} are empty for $w_1<\eps$, so we assume $w_1>\eps$. 
Therefore, fix $(\widetilde w_1,\widetilde w_2)\in SG_{\hke}$ (with $\widetilde w_1>\eps$) such 
that $\psike(\widetilde w_1,\widetilde w_2)=0$ and assume, by contradiction, that $\hatFke(\widetilde w_1,\widetilde w_2)>0$. 
In a first step we will suppose $\widetilde w_2<0$.
We might further assume that $\widetilde w_2$ is a Lebesgue point for 
the function $\hatFke(\widetilde w_1,\cdot)$. 
Hence 
in any left-neighbourhood of this point $\hatFke$ is strictly positive on a set of positive meausure, \textit{i.e.},
we can find positive numbers $\delta_1,\delta_2$  such that 
for all $\delta\in (0,\delta_1)$, 
there exists a set $B 
\subset (\widetilde w_2-\delta,\widetilde w_2)$, 
of positive measure such that 
\begin{align}\label{11.7}
\hatFke(\widetilde w_1,w)>\delta_2>0
\qquad \forall w \in B.
\end{align}
If $\pi_0^{\rm pol}$ is the projection in Definition \eqref{def_pi0}, since $\widetilde w_2<0$ for $\delta_3>0$ small enough the  segment $I:=\{(\widetilde w_1,\widetilde w_2,w_3):w_3\in (0,\delta_3)\}$ satisfies
\begin{align*}
 I_0:=\pi_0^{\rm pol}(I)\subset\{(\widetilde w_1, w_2,0): w_2\in (\widetilde w_2-\delta_2,\widetilde w_2)\}.
\end{align*}
We have that $\pi_0^{\rm pol}:I\rightarrow I_0$ is a homeomorphism.
Now, if $\psike(\widetilde w_1,\widetilde w_2)=0$ the segment $I$ 
cannot intersect the subgraph of $\hatFke$ (on a set of positive $\Hone$-measure), and thus 
\begin{align}\label{limitnull}
 \hatFke(\widetilde w_1,w_2)\leq \theta\big(({\pi_0^{\rm pol}}_{\vert I_0})^{-1}(\widetilde w_1,w_2,0)\big)\qquad \text{ for $\mathcal H^1$-a.e. }(\widetilde w_1,w_2,0)\in I_0,
\end{align}
where $\theta$ represents the usual angular coordinate. Since 
$\theta\big(({\pi_0^{\rm pol}}_{\vert I_0})^{-1}(\widetilde w_1,w_2,0)\big)$ is infinitesimal as $w_2\rightarrow \widetilde w_2^-$, 
condition \eqref{limitnull} contradicts \eqref{11.7}.

Let us now treat the case $\widetilde w_2>0$. 
This is much simpler to deal with, 
up to noticing that $\hatFke$ is defined on $S^{(4)}_{k,\eps}\subset\{(w_1,w_2,w_3):w_2\in [-1,0]\}$. The fact that $\psike(\widetilde w_1,\widetilde w_2)=0$ means 
that the line $(\widetilde w_1,\widetilde w_2)\times \R$ 
does not intersect $\Oke$ on a set of positive $\mathcal H^1$-measure 
 but this contradicts the fact that  $(\widetilde w_1,\widetilde w_2)\in SG_{\hke}$.  
 Indeed since $(\widetilde w_1,\widetilde w_2)\in SG_{\hke}$ hence there exists $w_2>\widetilde w_2$ such that $\psike(\widetilde w_1,w_2)>0$. Let $A:=\Oke \cap (\widetilde w_1, w_2)\times \R)$ then a suitable 
rotation of $A$
around the axis of the cylinder shall meet $(\widetilde w_1, \widetilde w_2)\times \R$ on a set $ \widetilde A$ of positive $\Hone$-measure (note that $\widetilde A\subset \Oke$), which contradicts $\psike(\widetilde w_1,\widetilde w_2)=0.$ 
\end{proof}

\begin{remark}\label{2.14rem}
By \eqref{semicircle}, \eqref{eq:psike} and \eqref{eq:def_O_k}, we deduce
\begin{align}\label{datum_1}
{\rm the~ trace~ of~ } \psike {\rm~ on~ } \overline R_l \cap \{w_1=0\}
{\rm ~is~} 
\sqrt{1-w_2^2},\quad{\rm for  }\; w_2\in[-1,1].
\end{align}
Moreover, by construction and by 
Remark \ref{rem_datialbordo} (iii),
\begin{align}\label{datum_2}
 {\psike}(w_1,-1) = 0\;\;\;\text{ and }\;\;\; 
{\psike}(w_1,1) = 0,\qquad w_1\in(0,\longR).
\end{align}
\end{remark}

\begin{remark}\label{rem_12.9}
 We can write \cite{Giusti:84}
 \begin{align}\label{expression_area}
  \mathcal H^2\Big((\partial^* SG_{\psike})\cap (R_l\times [0,\infty))\Big)=
\areaonecod(\psike, R_l),
 \end{align}
where 
\begin{align*}
 \areaonecod(\psike, R_l)=\int_{R_l}\sqrt{1+|\nabla \psike|^2}~dx+|D^s\psike|( R_l)
\end{align*}
is the classical area of the graph of the $BV$-function $\psike$ in $R_l$.
Moreover, by \eqref{datum_2}, it follows
$|D^s\psike|( R_l)
=|D^s\psike|\big(\overline R_l\setminus (\{w_1=0\}\cup\{w_1=l\})\big)$ 
and hence
\begin{align*}
 \areaonecod(\psike, R_l)=\areaonecod\big(\psike,
\overline R_l\setminus (\{w_1=0\}\cup\{w_1=l\})\big).
\end{align*}
\end{remark}
Recalling the expression \eqref{eq:Dirichlet_part_of_boundary} of 
$\partial_D R_l$, 
define $\varphi: \partial_D R_l \to [0,1]$ as 
\begin{equation}\label{eq:old_phi}
\varphi(w_1,w_2) := \begin{cases}
\sqrt{1-w_2^2}  & {\rm if}~ (w_1,w_2)\in \{0\}\times [-1,1], 
\\
0 & {\rm if} ~ (w_1,w_2)\in (0,\longR)\times \{-1\}. 
\end{cases}
\end{equation}
\begin{definition}[\textbf{The functional $\FBl$}]
\label{def:the_functional_Fhat}
Given 
$h\in L^\infty([0,\longR],[-1,1])$ and $\psi \in \BV(R_l;[0,1])$
we define 
\begin{equation}\label{eq:F}
  \FBl(h,\psi):=\areaonecod
(\psi, R_l)-\mathcal H^2(R_l\setminus SG_h)+\int_{\partial_DR_l}|\psi-\varphi|~d\mathcal H^1+\int_{(0,\longR)\times \{1\}}|\psi|~d\mathcal H^1.
\end{equation}
\end{definition}

We further define 
\begin{equation}\label{eq:X_l}
\oldDom
:=\{(h,\psi):
h\in L^\infty([0,\longR],[-1,1]), \psi \in \BV(R_l,[0,1]),\psi=0 ~{\rm in ~}R_l\setminus SG_h
\}.
\end{equation}
\begin{Remark}\label{remark12.11}\rm
\begin{itemize}
	\item[(i)] The Borel function $h_{k,\eps}:[0,\longR]\rightarrow[-1,1]$ satisfies $h_{k,\eps}=1$ in $[0,\eps)$, and $\psi_{k,\eps}\in  BV([0,\longR]\times [-1,1])$ is such that $\psi_{k,\eps}=0$
almost everywhere in  $R_l \setminus
	SG_{\hke}$. Moreover $\psi_{k,\eps}(w_1,w_2)
	= \sqrt{1-w_2^2} {\rm ~for~} (w_1,w_2)\in (0,\eps)\times [-1,1]$, 
and $\psike(\cdot,-1)=0 ~{\rm in~ } [0,\longR]$.
In particular
$$(\hke,\psike)\in \oldDom.$$
\item[(ii)]
if $(h,\psi) \in \oldDom$,
and if $h$ is smaller than $1$ almost everywhere on $(0,l)$
then
 the last
addendum on the right-hand side of \eqref{eq:F} vanishes.
\item[(iii)] Thanks to \eqref{datum_1} and \eqref{datum_2}, 
it follows from Remark \ref{rem_12.9} that 
\begin{equation*}
 \begin{aligned}
&   \mathcal H^2\big((\partial^* SG_{\psike})\cap (R_l\times [0,+\infty)\big)
-\mathcal H^2(R_l\setminus SG_{\hke})
\\
=&
\areaonecod(\psike, \overline R_l\setminus \big(\{w_1=0\}\cup\{w_1=l\}\big))-\mathcal H^2(R_l\setminus SG_{\hke})
=\FBl(h_{k,\eps},\psi_{k,\eps}).
 \end{aligned}
\end{equation*}
As a consequence, from Lemma \ref{lemma_9} we have
\begin{align}\label{12.20}
 |\GhatFkfour|+|\GminushatFkfour| ~\geq~ 
& 
 2\FBl(h_{k,\eps},\psi_{k,\eps}).
\end{align}
\end{itemize}
\end{Remark}
Notice that in minimizing $\FBl$ we have a free boundary condition on the edge $\{l\}\times [-1,1]$.
By Remark \ref{remark12.11} (i) and \eqref{12.20}
we have 
\begin{align}\label{infimum_pb}
|\GhatFkfour|+|\GminushatFkfour| \geq 
2 \inf_{(h,\psi)\in \oldDom} \FBl(h,\psi),
\end{align}
which leads to the investigation of the minimum problem on the right-hand side.

\begin{remark}\label{leb_point}
 Let $(h,\psi) \in \oldDom$. 
If $t_0\in(0,\longR)$ is a Lebesgue point for $h$, and if $h(t_0)<1$, then the trace of $\psi$ over the segment $\{w_1=t_0,\;h(t_0)\leq w_2\leq 1\}$ 
vanishes. Indeed
for any $\eta>0$ we can find $\delta_\eta>0$ such that  
 \begin{align}\label{markov1}
  \frac{1}{2\delta}\int_{t_0-\delta}^{t_0+\delta}|h(w_1)-h(t_0)|~dw_1<\eta
\qquad 
\forall
\delta \in (0,\delta_\eta).
 \end{align}
Let now $s_0\in (-1,1)$ be such that $h(t_0)<s_0\leq 1$ ({\it i.e.}, $(t_0,s_0)\in \{w_1=t_0,\;w_2> h(t_0)\}$), and set
$2\Delta:=s_0-h(t_0)$.
By 
Chebyschev inequality and \eqref{markov1} it follows that the set $B_\Delta:=\{w_1\in(t_0-\delta,t_0+\delta):|h(w_1)-h(t_0)|>\Delta\}$ satisfies
\begin{align}\label{markov2}
 \mathcal H^1(B_\Delta)\leq \frac{2\delta\eta}{\Delta}.
\end{align}
Then, for any $\xi\in(0,\Delta)$ we 
infer\footnote{In the first inequality we have used that $0 \leq \psi\leq1$; 
in the second inequality 
that $SG_h$ is the subgraph of $h$ in $(0,\longR)\times(-1,1)$;
 in the third inequality we have used that $s_0-h(t_0)=2\Delta$ and that $\xi<\Delta$. }
\begin{equation}
\label{markov3}
\begin{aligned}
 &\frac{1}{2\delta}\int_{t_0-\delta}^{t_0+\delta}
\int_{s_0-\xi}^{s_0+\xi}
\psi(w_1,w_2)~dw_2 dw_1 
\leq  \frac{1}{2\delta}\int_{t_0-\delta}^{t_0+\delta}
\int_{s_0-\xi}^{s_0+\xi}\chi^{}_{\{\psi>0\}}(w_1,w_2)~dw_2 dw_1
\\
\leq & \frac{1}{2\delta}\int_{t_0-\delta}^{t_0+\delta}
\int_{s_0-\xi}^{s_0+\xi}\chi^{}_{SG_h}(w_1,w_2)~dw_2dw_1
\leq 
\frac{\xi}{\delta}\int_{t_0-\delta}^{t_0+\delta}
\chi_{B_\Delta}(w_1)
dw_1\leq \frac{2\xi\eta}{\Delta},
\end{aligned}
\end{equation}
where the penultimate inequality follows from the inclusions 
\begin{align*}
SG_h\cap\big(
[t_0-\delta,t_0+\delta]\times[s_0-\xi,s_0+\xi]\big)&\subseteq SG_h\cap \big([t_0-\delta,t_0+\delta]\times[s_0-\Delta,s_0+\Delta]\big)\\
&\subseteq  B_\Delta\times[s_0-\Delta,s_0+\Delta],
\end{align*}
and the last inequality follows from \eqref{markov2}. 
Now \eqref{markov3} 
entails the claim by the arbitrariness of $\eta>0$ and since $\psi\geq0$.
\end{remark}

We now refine the choice of the class of pairs $(h,\psi)$ where the infimum in \eqref{infimum_pb} is computed.

\begin{definition}[\textbf{The classes $\Hricc$ and 
 $\Wspace$}]\label{def:the_class_W}
We set 
\begin{equation*}
\begin{aligned}
&\Hricc:=\{h\in L^\infty([0,\longR],[-1,1]): \;h\text{ convex and 
nonincreasing in }[0,\longR], h(0)=1\}, 
\\
& \Wspace:=\{(h,\psi) \in X_l:
~ h \in \Hricc\}.
\end{aligned}
\end{equation*}
\end{definition}

\begin{prop}[\textbf{Convexifying $h$}]\label{modifications_of_h}
We have
\begin{align}\label{get_rid_of_eps}
 \inf_{(h,\psi)\in \oldDom} 
\FBl(h,\psi)=\inf_{(h,\psi)\in \Wspace}\FBl(h,\psi).
\end{align}
\end{prop}
\begin{proof}
It is enough to show
the inequality ``$\geq$''.
By extending $\psi$ outside $ R_l$ as
$\psi:=0$ in $((0,\longR)\times \R)\setminus R_l$,  
we see that 
\begin{equation}\label{eq:by_extending}
\FBl(h,\psi)=
\areaonecod
\left(\psi, \overline R_l\setminus \big(\{w_1=0\}\cup\{w_1=l\}\big)\right)
-\mathcal H^2(R_l\setminus SG_h)
+
\int_{\{0\}\times[-1,1]}|\psi^--\varphi|~d\mathcal H^1,
\end{equation}
where, with a little abuse of notation,
$$\areaonecod\left(\zeta, \overline R_l\setminus \big(\{w_1=0\}\cup\{w_1=l\}\big)\right)
=\areaonecod
(\zeta,  {R}_l)+\int_{(0,\longR)\times \{1,-1\}}|\zeta^-|~d\mathcal H^1,$$
$\zeta^-$ being the trace of $\zeta \in BV(R_l)$ on $(0,\longR) \times\{1,-1\}$.

The thesis of the proposition will follow from the next three observations:
\begin{itemize}
 \item[(1)] If $h\in \Hricc$ 
is such that  $h(t_0)=-1$ for some Lebesgue point $t_0\in(0,\longR)$, 
then the subgraph $SG_h$ 
of $h$
splits in two mutually disjoint components: 
$SG_h^-=SG_h\cap \{w_1<t_0\}$ 
and $SG_h^+=SG_h\cap \{w_1>t_0\}$. 
Let  $\psi\in BV(R_l,[0,1])$ be such that 
\begin{align*}
 \psi=0 \qquad \text{~a.e.~in~ } R_l \setminus
SG_h.
\end{align*}
The trace of $\psi$ over the segment $\{w_1=t_0,\;h(t_0)\leq w_2\leq 1\}$ is $0$, as  a consequence of Remark \ref{leb_point}.
Then the function $\modpsi: R_l \to [0,1]$ defined as
\begin{align*}
 \modpsi(w_1,w_2):=\begin{cases}
                            \psi(w_1,w_2)&\text{if }w_1<t_0,\\
                            0&\text{otherwise,}
                           \end{cases}
\end{align*}
 still satisfies $(h,\modpsi)\in 
\oldDom$, and
\begin{align*}
\FBl(h,\modpsi)\leq \FBl(h,\psi).
\end{align*}
Being $\modpsi$ identically zero in $\{w_1> t_0\}$, in particular 
in $SG_h\cap\{w_1>t_0\}$, we can introduce
\begin{align*}
 \modh(w_1):=\begin{cases}
                             h(w_1)&\text{if }w_1<t_0,\\
                            -1&\text{otherwise,}                       
                      \end{cases}
\end{align*}
so that $(\modh,\modpsi)\in \oldDom$
and we easily see that
$\FBl(\modh,\modpsi) 
 \leq \FBl(h,\modpsi)$;  
hence
\begin{equation*}
 \FBl(\modh,\modpsi) \leq \FBl(h,\psi).
\end{equation*}

\item[(2)] More generally, let $( h,\psi)\in
\oldDom$ and let $t_0 \in (0,\longR)$ 
be any Lebesgue point of $h$; we can also suppose that $h(t_0)<1$. 
Consider 
\begin{align}\label{modif_omega1}
 \modh(w_1):=\begin{cases}
                             h(w_1)&\text{if }w_1<t_0,\\
                             h(w_1)\wedge h(t_0)&\text{otherwise,}                       
                      \end{cases}
\end{align}
\begin{align*}
 \modpsi(w_1,w_2):=\begin{cases}
                            \psi(w_1,w_2)&\text{if }w_1<t_0,\\
                            \psi(w_1,w_2)&\text{if }w_1\geq t_0,\;w_2\leq h(t_0),\\
                            0&\text{otherwise. }
                           \end{cases}
\end{align*}
We
claim that $\FBl(\modh,\modpsi) \leq \FBl( h,\psi).$
Define \begin{equation*}
  U:=\{(w_1,w_2) \in (0,\longR)\times(-1,1): w_1>t_0,\; h(t_0)<w_2< h(w_1)\},
\end{equation*}
that is the set where we have replaced $\psi$ by $0$. 
To prove the claim, using \eqref{eq:by_extending} and
the equalities
$$
\int_{\{0\}\times[-1,1]}|\psi^--\varphi|~d\mathcal H^1 = 
\int_{\{0\}\times[-1,1]}|{\modpsi}^--\varphi|~d\mathcal H^1,
$$
$$
\mathcal H^2
(R_l \setminus SG_{\modh}) = 
\mathcal H^2
(U \cup (R_l \setminus SG_{h}))=
\mathcal H^2
(U) +
\mathcal H^2(R_l \setminus SG_{h}),
$$
 we have to show that
\begin{align}\label{conclusion1}
\areaonecod
\left(\modpsi,
\overline R_l\setminus \big(\{w_1=0\}\cup\{w_1=l\}\big)
\right)\leq \areaonecod
\left(\psi,
\overline R_l\setminus \big(\{w_1=0\}\cup\{w_1=l\}\big)\right)
+\mathcal H^2( U).
\end{align}
Assume that $U$ is non-empty and that $\mathcal H^2(U)>0$.
It is convenient to introduce
$$V:=\{(w_1,w_2) \in R_l:t_0 < w_1<l,\; h(w_1)\vee h(t_0)\leq w_2< 1\},$$
so that 
$U\cup V=\{(w_1,w_2): w_1>t_0,\; h(t_0)<w_2< 1\}$ is an open rectangle. 
Since we have modified $\psi$ only in $U$, inequality 
\eqref{conclusion1} is equivalent to 
\begin{equation}
\label{conclusion1ter}
\begin{aligned}
& 
\areaonecod
(\modpsi,{U\cup V})+\int_{(t_0,l)\times\{h(t_0)\}}|{\modpsi}^+-{\modpsi}^-|d\mathcal H^1+\int_{(t_0,l)\times\{1\}}|{\modpsi}^-|d\mathcal H^1
\\
\leq
& 
\areaonecod
(\psi,{U\cup V})+\int_{(t_0,l)\times\{h(t_0)\}}|\psi^+-\psi^-|d\mathcal H^1+\int_{(t_0,l)\times\{1\}}|\psi^-|d\mathcal H^1
+\mathcal H^2( U),
\end{aligned}
\end{equation}
with 
$\psi^\pm$ (resp.
${\modpsi}^\pm$)
the external and internal traces of $\psi$ 
(resp. $\modpsi$) on $\partial (U\cup V)$;
here we have used from Remark \ref{leb_point}
that the trace of $\psi$ on $\{t_0\}\times (h(t_0),1)$ 
is zero (hence $
\int_{\{t_0\} \times (h(t_0),1)} \vert \psi^+ - \psi^-\vert 
d\mathcal H^1=
\int_{\{t_0\}\times (h(t_0),1)} 
\vert {\modpsi}^+ - {\modpsi}^-\vert d\mathcal H^1=
0$)
 and that the external traces $\psi^+$, ${\modpsi}^+$ 
on $(t_0,l)\times\{1\}$ vanish as well. 
Hence, exploiting that $\modpsi=0$ on $U\cup V$, so that 
$\areaonecod(\modpsi, U \cup V) = 
\mathcal H^2(U)+ 
\mathcal H^2(V)$, and that $\modpsi = \psi$ on $R_l \setminus (U\cup V)$,
inequality \eqref{conclusion1ter} is equivalent to
\begin{equation}\label{eq:conclusion2}
\begin{aligned}
&\mathcal H^2(V)+\int_{(t_0,l)\times\{h(t_0)\}}|\psi^+|d\mathcal H^1
\\
\leq &\areaonecod
(\psi,{U\cup V})+\int_{(t_0,l)\times\{h(t_0)\}}|\psi^+-\psi^-|d\mathcal H^1+\int_{(t_0,l)\times\{1\}}|\psi^-|d\mathcal H^1.
\end{aligned}
\end{equation}
We split
$$(t_0,l)=H_1\cup H_2\cup H_3,$$
with $H_1:=\{w_1\in (t_0,l):h(w_1)=1\}$, 
$H_2:=\{w_1\in (t_0,l):h(t_0)\leq h(w_1)<1\}$, and 
$H_3:=\{w_1\in (t_0,l):h(w_1)< h(t_0)\}$.
Since $\mathbb A(\psi;U\cup V)=\mathcal H^2(\mathcal G_\psi\cap ((U\cup V)\times\R))$, by slicing and looking at $\mathcal G_\psi$ as an 
integral current, we have\footnote{Here we use that $D_{w_2}\psi=0$ in $V$.}
\begin{align*}
 \areaonecod
(\psi, U\cup V)&\geq \int_{(t_0,l)}
\mathcal H^1\Big((\mathcal G_{\psi})_{t} {\cap 
((t_0,l)\times (h(t_0),1)\times \R})\Big)~dt
\\
 &\geq
\int_{(t_0,l)}\int_{(h(t_0),1)}|D_{w_2}\psi(t,s)|~\;dt
+\mathcal H^2(V) 
\\
 &=
\int_{H_1\cup H_2}\int_{(h(t_0),1)}|D_{w_2}\psi(t,s)|~\;dt
+\mathcal H^2(V)
\\
 &\geq
 \int_{H_2}|\psi^-(t,h(t_0))|~dt+
\int_{H_1}|\psi^-(t,h(t_0))-\psi^-(t,1)|~d t
+\mathcal H^2(V)
\\
&\geq\int_{H_1\cup H_2}|\psi^-(t,h(t_0))|~d t 
-\int_{H_1}|\psi^-(t,1)|~d t+\mathcal H^2(V)
\\
 &=\int_{(t_0,l)}|\psi^-(t,h(t_0))|~d t 
-\int_{H_1}|\psi^-(t,1)|~d t+\mathcal H^2(V)
\\
 &=\int_{(t_0,l)}|\psi^-(t,h(t_0))|~dt 
-\int_{(t_0,l)}|\psi^-(t,1)|~dt+\mathcal H^2(V),
\end{align*}
where $(\mathcal  G_{\psi})_{t}$ is the slice of $\mathcal G_\psi$ 
on the plane $\{w_1=t\}$, 
that is the generalized graph of the function $\psi\res\{w_2=t\}$. 
From the above expression, the 
triangular inequality 
implies \eqref{eq:conclusion2}.

\item[(3)] Let $(h,\psi)\in
\oldDom$. Let $t_1,t_2\in (\eps,l)$ 
be Lebesgue points for $h$ with  $t_1<t_2$, 
and let $r_{12}(t):=h(t_1)+\frac{h(t_2)-h(t_1)}{t_2-t_1}(t-t_1)$. We consider the following modifications of $h$ and $\psi$:
\begin{align*}
 \modhbis(w_1):=\begin{cases}
                            h(w_1)&\text{if }0<w_1<t_1\text{ or }l>w_1>t_2,\\
                            h(w_1)\wedge r_{12}(w_1)&\text{otherwise,}                       
                      \end{cases}
\end{align*}
and 
\begin{align*}
 \modpsibis(w_1,w_2):=\begin{cases}
                            \psi(w_1,w_2)&\text{if }0<w_1<t_1
\text{ or }l>w_1>t_2,\\
                            \psi(w_1,w_2)&\text{if }w_1\in[t_1,t_2] ~\textrm{and} ~w_2\leq r_{12}(w_1),\\
                            0&\text{otherwise. }
                           \end{cases}
\end{align*}
In other words we set $\psi$ equal to $0$ above the segment $L_{12}$ connecting $(t_1,h(t_1))$ to $(t_2,h(t_2))$.
Also in this case we have
\begin{align}\label{ineq_conv}
 \FBl(\modhbis,\modpsibis) \leq \FBl(h,\psi).
\end{align}
Indeed, if $h(t_1)=h(t_2)$ the proof is 
identical to the case (2). Otherwise, it can be obtained by slicing as well, 
parametrizing $L_{12}$ by an arc length parameter, 
then slicing the region $\{(w_1,w_2):w_1\in  (t_1,t_2),\;w_2\in (\ell_{12}(w_1),1)\} $\footnote{$\ell_{12}$ represents the affine function whose graph is $L_{12}$.} by lines perpendicular to $L_{12}$, and 
exploiting the fact that $\psi$ equals zero on the segments $\{t_i\}\times (h(t_i),1)$.
\end{itemize}
 
Let  $(h,\psi)\in \oldDom$ be given; 
from (3) we
can always replace $h$ by its convex envelope 
and modifying accordingly 
$\psi$, we get two functions $\modhbis$ and $\modpsibis$ 
such that \eqref{ineq_conv} holds. Moreover, by (2), 
if $t_0\in(0,\longR)$ is a Lebesgue point for $h^\#$, we 
can always replace $h^\#$ by $\modh$ in \eqref{modif_omega1}, 
so that $\modh$ turns out to be nonincreasing. 
The assertion of the proposition follows.
\end{proof}

Let us rewrite the functional $\FBl$ in a convenient way. 
Let  $(h,\psi)\in \Wspace$, and let $G_h=
\{(w_1,h(w_1)):w_1\in (0,\longR)\}\subset \overline{R}_l$ 
be the graph of $h$. We have, using \eqref{eq:F},
\begin{equation}\label{eq:F_ter}
  \FBl(h,\psi)
=
  \areaonecod
(\psi, SG_h)+\int_{\graphh \setminus \{h=-1\}}|\psi|~d\mathcal H^1+\int_{\partial_DR_l}|\psi-\varphi|~d\mathcal H^1,
\end{equation}
where, in the integral over $\graphh$, 
we consider the trace of $\psi \res SG_h$ on $\graphh$.

\begin{cor}\label{cor:bound_from_below_using_F} 
Let $\eps \in (0,1)$ and $n \in \mathbb N$.
Then for any $k \in \mathbb N$ we have
\begin{align}\label{crucial_ineq3}
|\currgraphk|_{\badset\times\R^2}
\geq 
 2 
\inf_{(h,\psi)\in \Wspace}\FBl(h,\psi)
-\pi \eps-\frac{C}{\eps n}-o_k(1),
\end{align}
for an absolute constant $C>0$, and where the sequence $o_k(1)$ 
depends on $\eps$ and $n$ and is infinitesimal as $k \to +\infty$.
\end{cor}
\begin{proof}
From \eqref{infimum_pb} and Proposition \ref{modifications_of_h}, we get
\begin{align}\label{crucial_ineq2}
|\GhatFkfour|+|\GminushatFkfour| \geq 
 2 
\inf_{(h,\psi)\in \Wspace}\FBl(h,\psi).
\end{align}
Combining  \eqref{crucial_ineq2} with \eqref{crucial_ineq},
inequality \eqref{crucial_ineq3} follows.
\end{proof}

\subsection{Lower bound: reduction to a Plateau-type 
problem on the rectangle $R_l$}\label{subsec:first_lower_bound}
We now state and prove our first main result.
\begin{theorem}[\textbf{Lower bound for the area of the vortex map}]\label{teo:step1}
 The relaxed area of the graph of the vortex map $\vortexmap$ satisfies
 \begin{align}\label{eq:first_lower_bound}
  \relarea(\vortexmap,\Omega)\geq \int_\Om|\mathcal M(\nabla \vortexmap)|~dx
+2\inf_{(h,\psi)\in \Wspace}\FBl(h,\psi).
 \end{align}
\end{theorem}
\begin{proof}
 We write
 $$\area(u_k,\Om)=
\area(u_k,\Om\setminus \badset)
+
\area(u_k,\badset)
= \int_{\Omega\setminus \badset}|\mathcal M(\nabla u_k)|~dx
+\int_{\badset}|\mathcal M(\nabla u_k)|~dx.
$$
Therefore
\begin{equation}\label{eq:liminf_sum_sum_liminf}
  \relarea(\vortexmap,\Omega)\geq \liminf_{k\rightarrow+\infty}\int_{\Omega\setminus \badset}|\mathcal M(\nabla u_k)|~dx+\liminf_{k\rightarrow+\infty}\int_{\badset}|\mathcal M(\nabla u_k)|~dx.
 \end{equation}
Given $\eps \in (0,\longR)$ and $n \in \mathbb N$, from \eqref{estimate_Dc} it follows
\begin{equation}\label{eq:first_liminf}
 \liminf_{k\rightarrow+\infty}\int_{\Omega\setminus \badset}|\mathcal M(\nabla u_k)|~dx\geq \int_{\Omega\setminus \overline{\sourcedisk}_\eps}|\mathcal M(\nabla u)|~dx-\frac1n-\frac{2}{\eps n}.
\end{equation}
{}From \eqref{crucial_ineq3} we have 
\begin{equation}\label{eq:second_liminf}
\int_{\badset}|\mathcal M(\nabla u_k)|~dx=|\currgraphk|_{\badset\times\R^2}\geq 2
\inf_{(h,\psi)\in \Wspace}\FBl(h,\psi)-\pi\eps-\frac{C}{\eps n}-o_k(1).
\end{equation}
Exploiting the fact that the right-hand side of  \eqref{crucial_ineq2}
 does not depend on $k$, we can pass to the liminf as $k\rightarrow+\infty$ in the above expression, to obtain
\begin{align}\label{eq:third_liminf}
\liminf_{k\rightarrow+\infty}\int_{\badset}|\mathcal M(\nabla u_k)|~dx \geq 2
\inf_{(h,\psi)\in \Wspace}\FBl(h,\psi)-\pi \eps-\frac{C}{\eps n}.
\end{align}
{}From \eqref{eq:liminf_sum_sum_liminf}, \eqref{eq:first_liminf}
and \eqref{eq:third_liminf} we obtain
\begin{equation}
 \relarea(\vortexmap,\Om)\geq \int_{\Omega\setminus \overline{\sourcedisk}_\eps}|\mathcal M(\nabla u)|~dx+2
\inf_{(\psi,h)\in \Wspace}\FBl(\psi, h)-\pi\eps - \frac{C+2}{\eps n}
-\frac1n,
\end{equation}
for all $n\in \mathbb N$ and $\eps\in(0,\longR)$.
Letting $n \to +\infty$ 
and then $\eps \to 0^+$, 
by the  dominated convergence theorem 
(since $\Omega\setminus \overline{\sourcedisk}_\eps\rightarrow\Om$ as 
$\eps\rightarrow 0^+$)  we get
\begin{align*}
 \relarea(u,\Om)&\geq \liminf_{\eps\rightarrow 0^+}\Big(\int_{\Om\setminus \overline{\sourcedisk}_\eps}|\mathcal M(\nabla u)|~dx+2
\inf_{(h,\psi)\in \Wspace}\FBl(h,\psi)-\pi \eps\Big)\\
 &=\int_{\Om}|\mathcal M(\nabla u)|~dx+2
\inf_{(h,\psi)\in \Wspace}\FBl(h,\psi).
\end{align*}
\end{proof}

\section{Structure of minimizers of $\FB$}
\label{sec:structure_of_minimizers}
In this section we analyse the minimum problem on the
right-hand side of \eqref{get_rid_of_eps}.
We prove the existence of minimizers, and exploy it to show 
that the inequality \eqref{eq:first_lower_bound}
in Theorem \ref{teo:step1} is optimal. 
First it is convenient to write the analogue of $\FBl$ 
in a doubled rectangle, see \eqref{eqn:FB}.
 
We start by introducing some notation. We denote $\doubledrectangle$ 
the open doubled 
rectangle, $\doubledrectangle
:=(0,2l)\times (-1,1)$, and define its Dirichlet boundary\footnote{Note
that $\partial_D \doubledrectangle$ consists of three edges of 
$\partial \doubledrectangle$,
while $\partial_D R_l$ (see \eqref{eq:Dirichlet_part_of_boundary}) 
consists of two edges of $\partial R_l$.} $\partial_D 
\doubledrectangle\subset \partial R_{2l}$ as
\begin{align*}
 \partial_D \doubledrectangle:=\big(\{0,2l\}\times [-1,1]\big)\cup \big((0,2l)\times\{-1\}\big),
\end{align*}
so that 
$\partial \OmAlaa \setminus \partialbar\OmAlaa = (0,2l)\times \{1\}$.	

\begin{definition}
We set
\begin{equation}
\label{eq:the_space_of_symmetric_convex_traces}
\begin{split}
\Hspace=\big \{\h : [0,2 \longR] \to [-1,1], 
~ \h {\rm~ convex}, ~
\h(\axialcoordofcylinder)=\h(2\longR-\axialcoordofcylinder) 
~\forall \axialcoordofcylinder \in [0,2\longR]
\big \}.
\end{split}
\end{equation}
\end{definition}

For each $h\in \Hspace$, we further define
 $$
\graphh :=\{(\axialcoordofcylinder, h(\axialcoordofcylinder)) 
: \axialcoordofcylinder \in (0,2\longR)\}, \qquad
\subgraphh 
:= \{(\axialcoordofcylinder,\scoord)\in \OmAlaa: \scoord<\h(\axialcoordofcylinder)\},
$$
where $\subgraphh:= \emptyset$ 
in the case $h\equiv-1$.
We set
\begin{equation}\label{eq:Lh}
\Lh:=
\Big(\{0\} \times (h(0),1)\Big) \cup 
\Big(\{2l\} \times (h(2l),1)\Big),
\end{equation}
which is either empty,
or the union of two equal intervals,
see Fig. \ref{fig:graphofh}.

\begin{figure}
\centering
\includegraphics[scale=0.5]{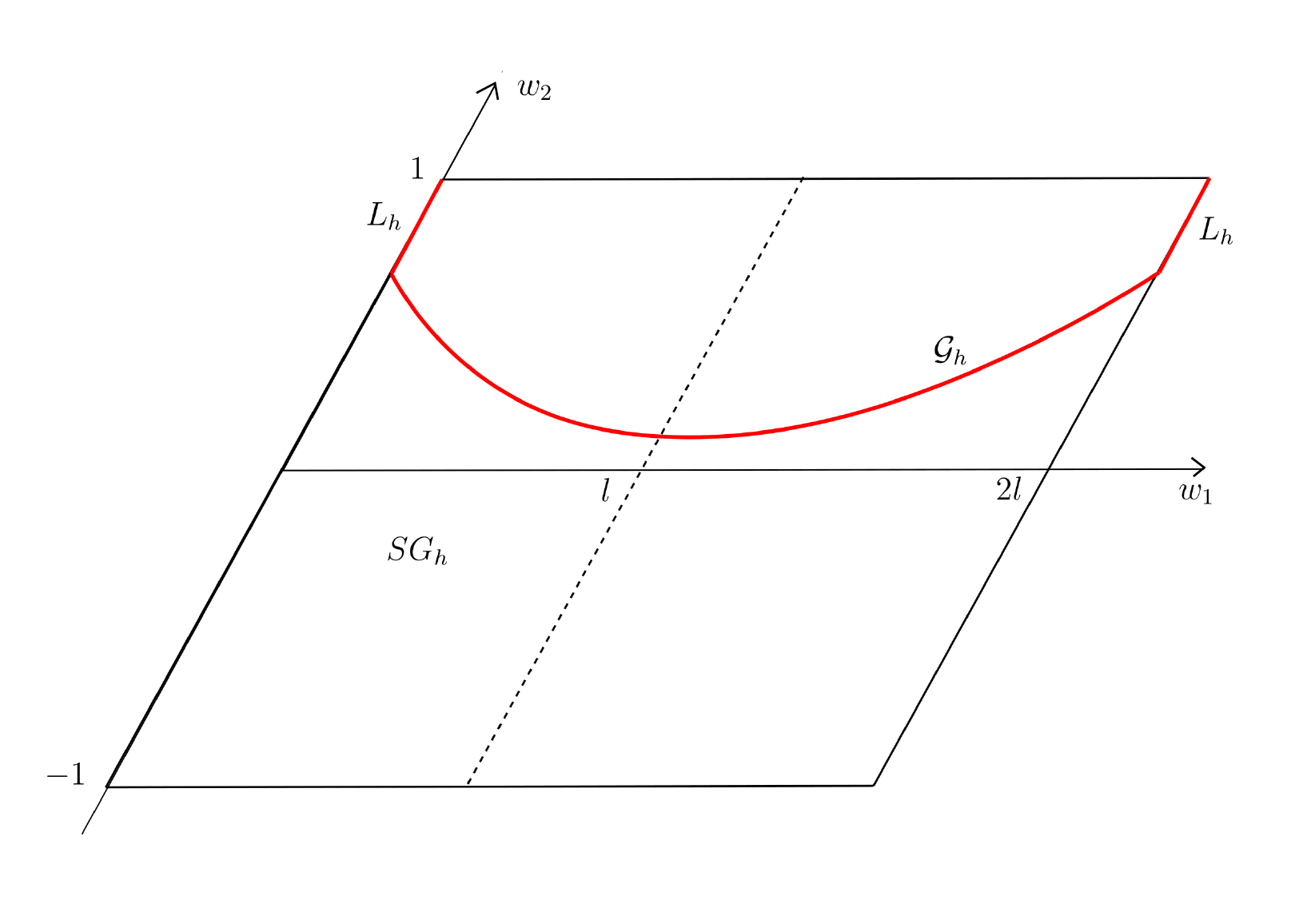}
\caption{the graph of a convex symmetric 
function $\h \in \Hspace$;
$\Lh$, defined in \eqref{eq:Lh},  
consists of  the two vertical segments  
over the boundary of $(0,2\longR)$, from $h(0) = h(2\longR)$ to $1$.}
\label{fig:graphofh}
\end{figure}

Define 
\begin{equation}\label{eq:varphi_doubled}
\dirdatum:\partial_D \OmAlaa \to [0,1],\qquad 
\dirdatum(\axialcoordofcylinder,\scoord):=
\begin{cases}
\sqrt{1-\scoord^2} & \text{ if }
(\axialcoordofcylinder,\scoord)
 \in \{0,2\longR\} \times [-1,1],
\\
0 & \text{ if }(\axialcoordofcylinder,\scoord) 
\in (0,2\longR) \times \{-1\}.
\end{cases}
\end{equation}
The graph of $\dirdatum$ on $\{0,2\longR\} \times [-1,1]$ 
consists of two half-circles of radius $1$ 
centered at $(0,0)$ and $(2 \longR,0)$ respectively, 
see Fig. \ref{fig:graphofphi}. 
We notice that such a $\varphi$ 
extends definition \eqref{eq:old_phi}.
\begin{figure}
\centering
\includegraphics[scale=0.50]{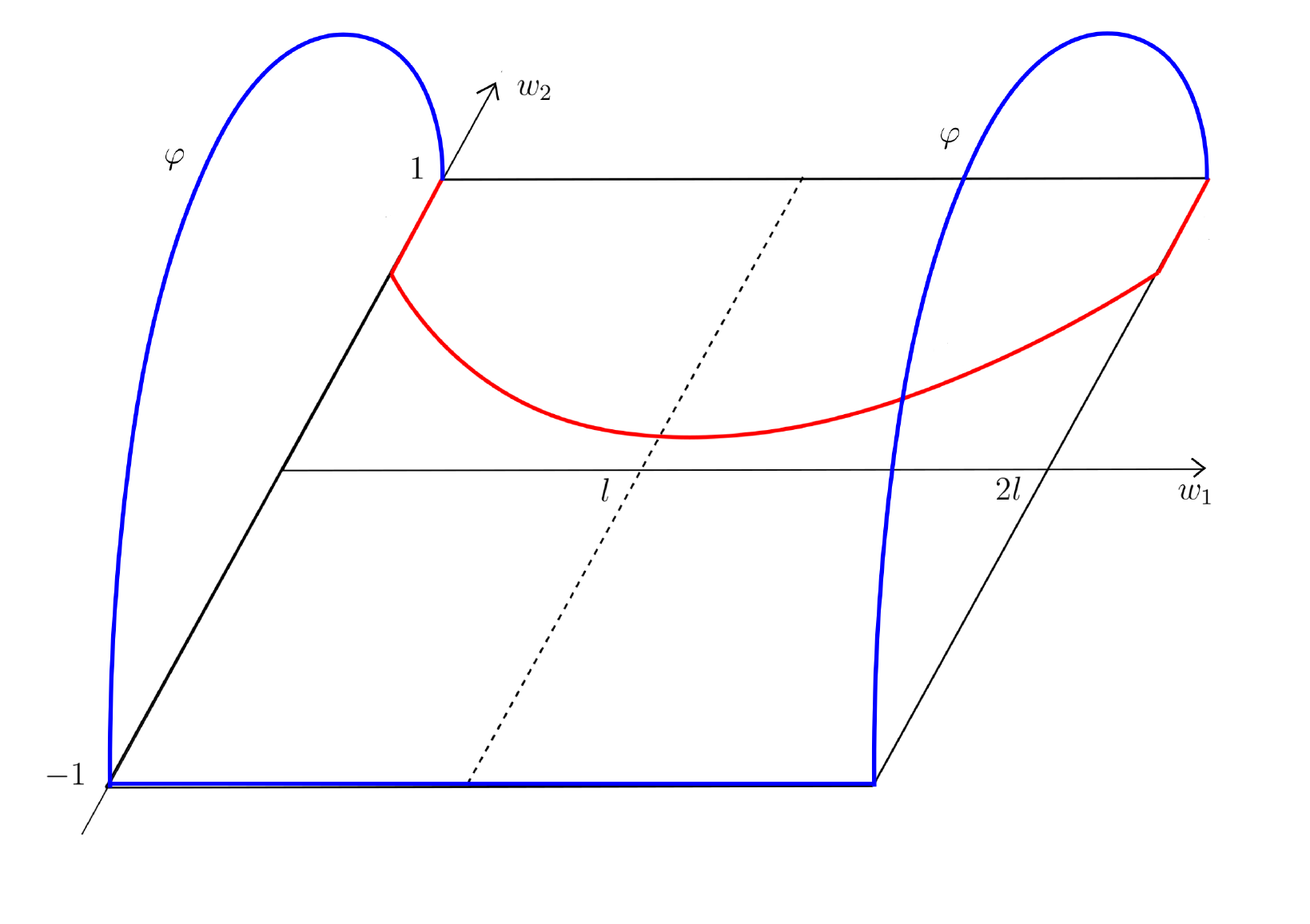}
\caption{the graph of the boundary condition function $\dirdatum$
in \eqref{eq:varphi_doubled} on the Dirichlet
boundary of $R_{2\longR}$. We also draw the graph of a function 
$h \in \mathcal H_{2\longR}$, and the
two segments $L_h$.
}
\label{fig:graphofphi}
\end{figure}

\begin{definition}[\textbf{The functional $\FB$}] \label{def:FB}
We define 
$$
\domalaa
 :=
\left\{(h,\psi): h \in \Hspace, \psi\in BV(\doubledrectangle, [0,1]),\psi = 0 \ 
{\rm on}~ \OmAlaa \setminus \subgraphh \right\}, 
$$
and for any $(h,\psi) \in 
\domalaa$, 
\begin{equation}
\FB(\h,\psione):=\Aone(\psione;\OmAlaa)-
\Htwo(\OmAlaa \setminus \Omegah)+\int_{\partialbar\OmAlaa}  
\vert \psione  -\dirdatum\vert d \Hone 
+ \int_{
\partial \OmAlaa \setminus \partialbar\OmAlaa
}  \vert \psione \vert 
~d \Hone.
\label{eqn:FB}
\end{equation}
\end{definition}
\begin{Remark}\rm \label{rmk:FAremarks}
\begin{itemize}
\item[(i)] The only case in which
the last addendum on the right-hand side of \eqref{eqn:FB} may be positive
is when $h$ is identically $1$ on $\partial \OmAlaa \setminus 
\partial_D \OmAlaa$;
 \item[(ii)] 
We have 
\begin{equation}
\FB(\h,\psione)=\Aone(\psione,\Omegah)
+
\int_{\partialbar \Omegah}  \vert \psione  -\dirdatum\vert ~d \Hone 
+ 
\int_{\graphh \setminus\{w_2=-1\}}\vert \psione^- \vert ~d \Hone 
+
\int_{\Lh} \dirdatum ~d \Hone, 
\label{eqn:FA}
\end{equation}
where
\begin{equation}\label{eq:partial_D_G_h}
\partial_D \Omegah:= (\partial_D R_{2l}) \cap \partial \Omegah,
\end{equation}
and $\psione^-$ denotes the trace of $\psione$ from the side of $\Omegah$.
To show \eqref{eqn:FA}, we start to observe that,  
 using that $\psione=0$ on $\OmAlaa \setminus \Omegah$, it follows 
$\int_{\partial \doubledrectangle \setminus 
\partial_D \doubledrectangle} \vert \psi\vert~d \mathcal H^1 = 
\int_{G_h\cap\{w_2=1\}} \vert \psi\vert~d \mathcal H^1$. This last term is nonzero only if $h\equiv 1$, in which case $L_h$ is empty, and the equivalence between \eqref{eqn:FB} and \eqref{eqn:FA} easily follows. If instead $h$ is not identically $1$, then, using again that $\psione=0$ 
on $\OmAlaa \setminus \Omegah$, we see that the last term on 
the right-hand side of \eqref{eqn:FB} is null,  
and 
\begin{equation}\label{eq:AA}
\Aone(\psione,\Omegah)=
\Aone(\psione,\OmAlaa)-\Htwo(\OmAlaa \setminus \Omegah) -
\int_{\graphh \cap \OmAlaa} \vert\psione^- \vert d \Hone.
\end{equation}
Hence, if $h$ is not identically $1$,
inserting
\eqref{eq:AA} into \eqref{eqn:FB}, 
we obtain, splitting $\partial_D \OmAlaa
  = (\partial_D \Omegah) \cup L_h \cup (G_h\cap \{w_2=-1\})$, and using
that $\dirdatum =0$ on $(0,2\longR)\times\{-1\}$,
\begin{equation*}
\begin{aligned}
\FB(\h,\psione)
=&
\Aone(\psione,\OmAlaa)-\Htwo(\OmAlaa \setminus \Omegah) 
+\int_{\partialbar \OmAlaa}  \vert \psione -\dirdatum\vert d \Hone 
\\
=&\Aone(\psione,\Omegah)
+\int_{\partialbar \OmAlaa}  \vert \psione -\dirdatum\vert d \Hone
+\int_{\graphh \cap \OmAlaa} \vert\psione^- \vert d \Hone
\\
=&
\Aone(\psione,\Omegah)+\int_{\partialbar \Omegah}  
\vert \psione -\dirdatum\vert d \Hone + 
\int_{\Lh} \vert  \dirdatum\vert d \Hone 
+ \int_{G_h\cap\{w_2=-1\}} \vert \dirdatum\vert d \Hone
\\
& 
+\int_{\graphh \cap \OmAlaa} \vert\psione^- \vert d \Hone
\\
=& \Aone(\psione,\Omegah)+
\int_{\partialbar \Omegah}  \vert \psione  -\dirdatum\vert d \Hone 
+\int_{\graphh\setminus \{w_2=-1\}}\vert 
\psione^-\vert d\Hone + \int_{\Lh} \dirdatum d \Hone.
\end{aligned}
\end{equation*}
\item[(iii)] We have 
 \begin{equation}  \label{eqn:equiv_of_inf_on_good_set}
 \begin{aligned}
& \inf_{(h,\psi)\in
 \domalaa} \FB(\h,\psione)
\\
=&
\inf \big \{ \Aone(\psione,\Omegah)
+
\int_{\partialbar \Omegah}  \vert \psione  -\dirdatum\vert ~d \Hone 
+ 
\int_{\graphh \setminus \{w_2=-1\}}\vert \psione^-\vert ~d \Hone 
+
\int_{\Lh} \dirdatum ~d \Hone
\\
 & \quad \quad :\h \in \Hspace \setminus \{\h\equiv -1\},~\psione \in  
\BV (\Omegah,[0,1]) \big \} .
\end{aligned} 
\end{equation}
 \item[(iv)] If $h>-1$ everywhere, then  $\Omegah$ is connected,
$\partial_D \Omegah = \partial_D R_{2l} \setminus \Lh$,
 and the sum of the first three terms on 
the right-hand side of \eqref{eqn:FA} gives the area 
of the graph of $\psi$ 
on $\overline{\Omegah}$, with the boundary condition $\dirdatum$ set to be $0$ on $\graphh$. 
 
\item[(v)]  Our aim is to have a surface in 
$\overline R_{2l} \times \R \subset \R^3=
\R^2_{(\axialcoordofcylinder,\scoord)}\times \R$, of graph type,  whose boundary consists of the union of the graph of $\dirdatum$ and 
the graph of a convex function $\h \in \Hspace$. The last three terms in \eqref{eqn:FA} are an area penalization to force the solution to attain these boundary conditions by filling, with vertical walls, the gap between the boundary of any 
competitor surface (the generalized graph of $\psione$) and the required boundary conditions. In particular the presence of the last term of \eqref{eqn:FA} is explained as follows: 
assume that $\h(0)<1$, \textit{i.e.}, $\Lh \neq \emptyset$; the graph of 
any $\psione \in \BV(\Omegah,[0,1])$ does not
reach the graph of $\dirdatum \vert_{\Lh}$ (simply because $L_h
\cap \overline \Omegah=\emptyset$). To overcome this, 
the graph of $\psione$ is glued to the wall
consisting of the subgraph of 
$\dirdatum \vert_{\Lh}$ (inside $\overline R_{2\longR}$).

\item[(vi)] Take $\hn:=-1+\frac{1}{n}$, and $\psionen:=c>0$ on $\Omegahn$,  
then $\displaystyle
\lim_{n\to +\infty} \areaonecod(\psi_n, \Omegahn)=0$, 
$\displaystyle
\lim_{n\to +\infty} \int_{\partialbar \Omegahn}  
\vert \psione_n  -\dirdatum\vert ~d \Hone 
=2cl$, and 
$\displaystyle
\lim_{n\rightarrow+\infty} \int_{G_{h_n} \setminus \{h_n=-1\}}\vert \psione \vert ~d \Hone =2cl$,
$\displaystyle
\lim_{n\to +\infty} \int_{L_{h_n}} \dirdatum~  d\Hone=\pi$,
hence
$$
\FB(-1,0)=\pi < \lim_{n\to +\infty} \FB(h_n,\psi_n) = 4c\longR +\pi,
$$
that is the functional $\FB$ in some sense forces a minimizing  sequence to attain the boundary conditions as much as possible. 
  \end{itemize}
 \end{Remark}

By symmetry, we easily infer
\begin{align}\label{doubling}
 2\inf_{(h,\psi)\in \Wspace}
\FBl(h,\psi)=\inf_{(h,\psi)\in
 \domalaa}
\FB(h,\psi),
\end{align}
therefore
we can now restate the content of Theorem \ref{teo:step1} as follows:
\begin{equation}\label{eq:restate_theorem}
\relarea(\vortexmap,\Omega)
\geq \int_\Om|\mathcal M(\nabla \vortexmap)|~dx
+
\inf_{(h,\psi)\in \domalaa}
\FB(h,\psi).
 \end{equation}

\begin{Remark}[\textbf{Two explicit estimates from above}]
\label{rem:Alaa}\rm 
Let $\h\equiv 1$ and 
$\psione(\axialcoordofcylinder,\scoord):=\sqrt{1-\scoord^2} =:
\psi_s(w_1,w_2)$ for any $(w_1,w_2) \in R_{2\longR}$. 
Then $(\h, \psione)$ is one of the competitors  in \eqref{doubling} and 
therefore 
$$
\inf_{(h,\psi)\in
 \domalaa} \FB(h,\psi)\leq 
\FB(1,\psione_s)=2 \pi \longR \qquad \forall \longR >0,
$$
which is the lateral area of the cylinder $(0,2\longR) \times B_1$.
Also, $\FB$ is well defined for $h\equiv -1$, 
in which case $\Omegah = \emptyset$, $\psi \equiv 0$ in $R_{2l}$,
and therefore
\begin{equation}\label{eq:FB_minus_one_zero} 
\FB(-1, 0)=\int_{\{0,2\longR\}\times(-1,1)}\dirdatum~ d\Hone =\pi,
\end{equation}
which is the area of  the two 
half-disks joined by the segment $(0,2 \longR) \times \{-1\} $,
see Fig. \ref{fig:graphofphi}.  In particular
\begin{equation} \label{eqn:upperboundofA}
\inf_{(h,\psi)\in
 \domalaa}
\mathcal{F}(h,\psi)\leq \pi \qquad 
\forall
 \longR >0.
\end{equation}
\end{Remark}

In the next section we shall 
prove the existence and regularity of minimizers for the minimum
problem on the right-hand side of \eqref{eq:restate_theorem}.

\subsection{Existence of a minimizer of  $\FB$}
The construction of a suitable recovery 
sequence for the relaxed area of the graph of the vortex map 
$\vmap$ depends on the existence and regularity 
of one-codimensional area minimizing surfaces of graph type for problem
\eqref{doubling}. 
We start by analysing the features of the space 
$\Hspace$ in \eqref{eq:the_space_of_symmetric_convex_traces}.
Clearly the graph of $h \in \Hspace$ is symmetric with 
respect to $\{\axialcoordofcylinder=\longR\}$; also, the convexity of $\h$ implies $\h
\in {\rm Lip}_{{\rm loc}}((0,2\longR))$, 
and $\h$ has a  
continuous extension on $[0, 2\longR]$.

\begin{lemma}[\textbf{Compactness of $\Hspace$}]\label{lem:compactnessofH}
Every sequence $(\h_k) \subset \Hspace$ has a subsequence 
converging uniformly on compact subsets of $(0,2\longR)$ 
to some element of $\Hspace$.
\end{lemma}

\begin{proof} See for instance \cite[Sec. 1.1]{Hormander:94}.
\end{proof}

It is convenient to extend $\varphi$ in the doubled rectangle
$\overline R_{2l}$ by defining the extension $\widehat \varphi$
as:
\begin{equation}\label{eq:widehat_varphi}
\widehat \varphi(w_1,w_2)=
\widehat 
\varphi(0,w_2):=\sqrt{1-w_2^2}\qquad \forall (w_1,w_2)\in \overline R_{2l}.
\end{equation}
In the rest of this section we want to prove the following result.
\begin{theorem}[\textbf{Minimizing pairs}]
\label{Thm:existenceofminimizer}
There exists $(\hstar,\psionestar) \in  \domalaa$  such that 
\begin{equation}\label{eqn:B}
\FB(\hstar, \psionestar)=\min \big \{ \FB(\h,\psione):  (\h,\psione) \in  \domalaa
 \big \},
\end{equation}
 and $\psionestar$ is symmetric with respect to 
$\{\axialcoordofcylinder=\longR\} \cap R_{2l}$.
Moreover, if $\hstar$ is not identically $-1$, then
\begin{itemize}
\item[(i)] $\hstar(0)=1=\hstar(2\longR)$, and 
$\hstar >-1$ in $(0,2\longR)$;
\item[(ii)] $\psionestar$ is locally Lipschitz (hence analytic) 
and strictly positive
in $\Omegahstar$; 
\item[(iii)] $\psionestar$ 
is continuous up to the boundary of $\Omegahstar$, 
and attains the boundary conditions, \textit{i.e.}, for $(\tcoord, \scoord)
\in 
\partial SG_{h^*}$,  
\begin{equation}\label{eqn:psionestarboundary}
 \psionestar(\axialcoordofcylinder,\scoord)=\begin{cases}
 0& \text{ if}\quad \scoord=-1 \text{ or }\scoord=\hstar(\axialcoordofcylinder),\\
 \sqrt{1-\scoord^2}&\text{ if}\quad \axialcoordofcylinder=0 \text{ or } \axialcoordofcylinder=2\longR,
 \end{cases} 
 \end{equation}
 hence 
\begin{equation}\label{eqn:FAequalareaofhstar}
\FB(\hstar,\psionestar)=\Aone(\psionestar, \Omegahstar);
\end{equation}
\item[(iv)] we have 
\begin{equation}\label{eq:psi_star_below_cylinder}
\psi^\star < \widehat \varphi {\rm ~in}~ R_{2l}.
\end{equation}
\end{itemize}
\end{theorem}

The rest of this section is devoted 
to prove this theorem; we start with some preparation.

\begin{definition}[\textbf{Convergence in $ \domalaa$}]
\label{def:convergence_in_Dom_FB}
We say that a sequence $((\hn, \psionen))\subset 
 \domalaa$  
converges to $(\h,\psione)\in  \domalaa$, if 
\begin{itemize}
\item[-] $(\hn)$ converges 
to $\h$
uniformly on compact subsets of $(0,2\longR)$; 
\item[-] $(\psionen)$ converges 
to $\psione$ in $\Lone(\OmAlaa)$.
\end{itemize}
\end{definition}\label{def:convinHxBV}

\begin{lemma}[\textbf{Closedness of $ \domalaa$}]\label{lem:closednessinB}
Let $((\hn,\psionen))\subset  \domalaa$ be a sequence such that 
$(\hn)$ converges 
to $\h \in \Hspace$
uniformly on compact subsets of $(0,2\longR)$,
 and $(\psionen)$ converges 
to $\psione \in \BV(\OmAlaa)$
in $\Lone(\OmAlaa)$.
Then $(\h, \psione)\in  \domalaa$.
\end{lemma}

\begin{proof}
Possibly passing to a (not relabelled) 
subsequence, we can assume that 
$(\psionen)$ 
converges to $\psione$ pointwise in $A \subseteq \OmAlaa$, with 
$\Htwo (\OmAlaa \setminus A)=0$, and
$\psionen =0$ in  $A \cap 
(\OmAlaa \setminus \Omegahn)$
for all $n \in \mathbb N$.
We only have to show that
$\psione =0$ in $A \cap (\OmAlaa \setminus \Omegah)$. We can also assume
that $A$ does not intersect the graph of $h$. 
If $(\axialcoordofcylinder,\scoord) \in A \cap (\OmAlaa \setminus \Omegah)$, 
then $\scoord>\h(\axialcoordofcylinder)$. {}From 
the local uniform convergence of $(\hn)$ to $\h$ 
in $(0,2\longR)$ it follows that
$\scoord>\hn(\axialcoordofcylinder)$  for $n$ large enough, {\it i.e.,} 
$(\axialcoordofcylinder,\scoord)\in A\cap (\OmAlaa \setminus \Omegahn)$,
and the assertion follows.
\end{proof}

\begin{lemma}[\textbf{Lower semicontinuity of $\FB$}]
\label{lem:lowersemicontinuityofFB}
Let $((\hn, \psionen))\subset  \domalaa$ be a sequence 
converging  to $(\h,\psione) \in  \domalaa$ in the sense of Definition \ref{def:convergence_in_Dom_FB}. Then 
\begin{equation}\label{eqn:partofF(B)}
\FB(\h,\psione) \leq \liminf_{n\to +\infty}\FB(\hn,\psionen).
\end{equation}
\end{lemma}

\begin{proof}

\smallskip
It is standard\footnote{Indeed, let $\tilde{\dirdatum}:\partial\OmAlaa \to [0,1]$ be defined as 
$\tilde{\dirdatum}:=\dirdatum$  on $\partialbar \OmAlaa$, and 
$\tilde\dirdatum := 0$ on $\partial \OmAlaa \setminus \partialbar \OmAlaa$.
Let $B\subset \R^2$ be an open disc containing $\overline \OmAlaa$. 
We extend $\tilde\dirdatum $ to a $W^{1,1}$ function 
in $B \setminus \overline \OmAlaa$,
 \cite[Thm. 2.16] {Giusti:84},
and we still denote by $\tilde\dirdatum$ such an extension.
For every $\psione \in BV(\OmAlaa)$, define 
$\widehat \psione
:=\psione$ in $\OmAlaa$ and
$\widehat \psione :=
\tilde\dirdatum$ in $B \setminus\OmAlaa$.
We have 
$$
\begin{aligned}
 \Aone(\psione,\OmAlaa)+\int_{\partialbar \OmAlaa}  \vert \psione  -\dirdatum\vert d \Hone +\int_{\partial 
\OmAlaa\setminus \partialbar \OmAlaa}  \vert \psione \vert d \Hone
=  \Aone(\psione,\OmAlaa)+\int_{\partial \OmAlaa}  \vert \psione  -
\tilde \dirdatum\vert d \Hone 
=  \Aone(\widehat \psione,B) -\Aone(\tilde\dirdatum,B\setminus \overline \OmAlaa),
\end{aligned}
$$
where the last equality follows from \cite[(2.15)] {Giusti:84}. 
Thus the lower semicontinuity of the functional in \eqref{eq:first_three_terms} 
follows from the $\Lone(B)$-lower semicontinuity of the area functional.}
 to show that
the functional 
\begin{equation}\label{eq:first_three_terms}
\psi \in BV(\OmAlaa, [0,1]) \to \Aone(\psione,\OmAlaa)+\int_{\partialbar \OmAlaa}  
\vert \psione  -\dirdatum\vert~ d \Hone +\int_{\partial \OmAlaa\setminus \partialbar \OmAlaa} 
 \vert \psione \vert ~d \Hone
\end{equation}
is $\Lone(\OmAlaa)$-lower semicontinuous.
Since $(\hn)$ converges to $\h$ pointwise in $(0,2\longR)$, we also have 
$\lim_{n \to +\infty} \Htwo (\OmAlaa \setminus \Omegahn) = \Htwo (\OmAlaa \setminus \Omegah)$.
The assertion follows.
\end{proof}

\begin{prop}[\textbf{Existence of a minimizer of \eqref{eqn:B}}]\label{Thm:existenceofB}
There exists $(\hstar,\psionestar)\in  \domalaa$ satisfying \eqref{eqn:B}.
\end{prop}

\begin{proof}
Note that 
$$ 
\h(\axialcoordofcylinder):=-1,\quad 
\psione(\axialcoordofcylinder, \scoord):=0, \qquad (\axialcoordofcylinder,\scoord)\in \OmAlaa,
$$
is a competitor in \eqref{eqn:B}. Hence, 
for a minimizing sequence $((\hn, \psionen))\subset  \domalaa$, 
recalling \eqref{eq:FB_minus_one_zero} 
we have
\begin{equation}
\lim_{n\rightarrow +\infty} \FB(\hn, \psionen)=\inf \big \{ \FB(\h,\psione): (\h,\psione) \in 
 \domalaa
 \big \} \leq \pi.
\end{equation}
Thus 
$\sup_{n \in \mathbb N}
\vert D \psionen \vert (\OmAlaa) < +\infty$, and 
there exists $\psionestar \in BV(\OmAlaa, [0,1])$ such that, up to a 
(not relabelled) subsequence, $(\psionen)$ 
converges to $\psionestar$ in $\Lone(\Omega)$. 

Using Lemmas \ref{lem:compactnessofH} and 
\ref{lem:closednessinB}, we may assume that $(\h_n)$ 
converges locally uniformly to some $\hstar \in \Hspace$,  
and $\psionestar=0$ in $\OmAlaa \setminus \Omegahstar$.
The assertion then follows from 
Lemma \ref{lem:lowersemicontinuityofFB}.
\end{proof}

We now turn to the regularity of minimizers.

\begin{prop}[\textbf{Analyticity and positivity of a minimizer}]
\label{prop_reg}
Suppose that $(\h,\psione)$ is a 
minimizer of \eqref{eqn:B}, and that $h$ is not identically $-1$.
Then $\psione$ is analytic 
in $\Omegah$, and solves the equation
\begin{align}\label{div_eq_bis}
{\rm div} \Big(\frac{\nabla \psione}{\sqrt{1+|\nabla \psione|^2}}\Big)=0
\qquad {\rm in}~ \Omegah.
\end{align}
Moreover $\psione>0$ in $\Omegah$.
\end{prop}

\begin{proof}
Since by assumption $h$ is not identically $-1$,
we have that $\Omegah$ is nonempty. Moreover minimality ensures
$$
\int_{\Omegah} \sqrt{1+|D \psione|^2} \leq 
\int_{\Omegah} \sqrt{1+|D \psione_1|^2} 
 $$ 
for any $\psi_1 \in BV(\Omegah)$, ${\rm spt}(\psi
-\psi_1) \subset \subset \Omegah$.
Thus, 
by  \cite[Thm
14.13]{Giusti:84}, $\psione$ is locally Lipschitz, and hence analytic, in
$\Omegah$, and
\eqref{div_eq_bis} follows. 
Now, let $z\in \Omegah$ and take a 
disc $B_\eta(z)\subset \subset \Omegah$. Since
$\psione \geq 0$ on $\partial B_\eta(z)$ we find, by
the strong maximum principle  \cite[Thm. C.4]{Giusti:84}, that 
either $\psione$ is identically zero  in $B_\eta(z)$, or
$\psione>0$ in
$B_\eta(z)$. Hence from the analyticity of $\psione$ and the
arbitrariness of
$z$,  we have that either $\psione$ is identically
zero in $\Omegah$ or $\psione>0$ in $\Omegah$. 
Now 
$\FB(h,0)=|\Omegah|+\pi > \FB(-1, 0)=\pi$, see
\eqref{eq:FB_minus_one_zero}.  Thus $(h,0)$ is not a minimizer, and the positivity of $\psione$ in $SG_h$ is achieved. 

\end{proof}

\begin{lemma}[\textbf{Symmetric minimizers}]\label{lem:symmetry_min}
 There exists a minimizer $(\h,\psione)$ of \eqref{eqn:B} 
such that 
$h(\cdot) = h(2l - \cdot)$ and $\psione$ is
 symmetric with respect to $\{\axialcoordofcylinder=\longR\} \cap R_{2l}$.
\end{lemma}
\begin{proof}
 Let $(\h, \psione)$ be a minimizer of \eqref{eqn:B}. 
Let $I \subset (0,2\longR)$ be an open interval; consistently
with \eqref{eqn:FB}, and since $\psi$ is continuous in $SG_h$, we set
\begin{align*}
\FB(\h, \psione; I):= &	
\Aone(\psione, I\times(-1,1))-
\Htwo\Big(I\times(-1,1) \setminus \Omegah\Big)+
\int_{(\partialbar\OmAlaa) \cap (\overline I \times [-1,1)) }  \vert \psione  -\dirdatum\vert ~d \Hone \nonumber
\\
&+ \int_{(\partial \OmAlaa \setminus \partialbar\OmAlaa) 
\cap (\overline I \times (-1,1])}  
\vert \psione \vert~ d\Hone.
\end{align*}  
Recall that $\h\in \Hspace$, hence  its graph 
is symmetric with respect to $\{\axialcoordofcylinder=\longR\}
\cap R_{2l}$. 
Define
$\tilde \psione := \psione $ on $(0,\longR)\times (-1,1)$ 
and  $\tilde \psione (\axialcoordofcylinder,\scoord) := \psione (2\longR -
\axialcoordofcylinder,\scoord)$ 
for  $(\axialcoordofcylinder,
\scoord) \in (\longR, 2\longR)\times (-1,1)$, in particular 
the graph of 
$\tilde\psione$ is symmetric with 
respect to $\{\axialcoordofcylinder=\longR\} \cap 
R_{2l}$. 
Since $\FB 
(\h , \psione;(0,\longR))= 
\FB
(\h ,\psione;(\longR,2\longR))$, it follows 
$\FB(\h , \tilde\psione)= \FB(\h, \psione)$ 
for, if $\FB 
(\h , \psione;(0,\longR))<\FB
(\h ,\psione;(\longR,2\longR))$, then 
$\FB(\h , \tilde\psione)< \FB(\h, \psione)$ which contradicts the 
minimality of $(h,\psione)$. 
\end{proof}

\begin{lemma}\label{lemma_reg2_bis}
Suppose that $(\h,\psione)$ is a 
minimizer of \eqref{eqn:B} such 
that:
\begin{itemize}
\item[(i)]
 $h(\cdot) = h(2l - \cdot)$;
\item[(ii)]
 $\psione$ is
 symmetric with respect to $\{\axialcoordofcylinder=\longR\} \cap R_{2l}$;
\item[(iii)] 
$h$ is not identically $-1$.
\end{itemize}
Then $$h(\axialcoordofcylinder)>-1
\qquad \forall\axialcoordofcylinder\in [0,2\longR].
$$
\end{lemma}
\begin{proof}
By the symmetry of $\h$ and $\psione$ with respect to 
$\{\axialcoordofcylinder=\longR\} \cap \doubledrectangle$
(Lemma \ref{lem:symmetry_min}) we may restrict our argument to $[0,\longR]$.
  Assume by contradiction that there exists $\overline w_1 
\in (0,\longR]$ such that $\h( \overline w_1)=-1$. 
Recall that $h$ is convex,
 nonincreasing in $[0,\longR]$ and continuous at $\longR$.
Let $$w_1^0:=
\min\{\axialcoordofcylinder \in (0,\longR]:h(\axialcoordofcylinder)=-1\}.$$ 
Since $\h$ is not identically $-1$,
 we have $w_1^0>0$, and $h$ is strictly decreasing in $(0,w_1^0)$. Let 
  $$h^{-1}:[-1,\h(0)]\rightarrow [0,w_1^0]$$
be the inverse of $\h\vert _{[0,w_1^0]}$.
We have,
using \eqref{eqn:FA} and \eqref{eq:partial_D_G_h},
 \begin{align*}
 &
\frac12\FB(\h,\psione)
= \frac12\left[ \Aone(\psione, \Omegah)
+
\int_{\partialbar \Omegah}  \vert \psione  -\dirdatum\vert ~d \Hone 
+ 
\int_{\graphh \setminus \{w_2=-1\}}\vert \psione^- \vert ~d \Hone 
+
\int_{\Lh} \dirdatum ~d \Hone \right] 
\\
= & \frac12 \Aone(\psione, \Omegah)
+
\int_{(-1,\h(0))}  \vert \psione(0,\scoord)
  -\dirdatum(0,\scoord)\vert ~d \scoord 
+
\int_{(0,w_1^0)}  \vert \psione 
(\axialcoordofcylinder,-1)  -\dirdatum(\axialcoordofcylinder,-1)\vert 
~d \axialcoordofcylinder 
\\
&\qquad
\ \ \qquad 
+\int_{G_{h{\llcorner{(0,w_1^0)}}}}
\vert \psione^- \vert d \Hone
+
\int_{(\h(0),1)} \dirdatum(0,\scoord) d \scoord.
\end{align*}
Now, 
we argue by slicing 
the rectangle $R_l=(0,\longR)\times (-1,1)$
 with lines $\{\axialcoordofcylinder=\tau\}, \tau \in (0,\longR)$. 
Recalling the expression of $\Omegah$ (which is non empty 
by assumption (iii)), and neglecting the third addendum,
\begin{equation*}
\begin{aligned}
\frac12\FB(\h,\psione)
=&   
\int_{0}^{w_1^0}\int_{-1}^{h(\axialcoordofcylinder)} 
\sqrt{1+\vert \nabla \psione\vert^2}~
d\scoord d\axialcoordofcylinder
\\
&+ 
\int_{(-1,\h(0))}  \vert \psione(0,\scoord)  
-\dirdatum(0,\scoord)\vert d \scoord  +
\int_{(0,w_1^0)}  
\vert \psione (\axialcoordofcylinder,-1)  
-\dirdatum(\axialcoordofcylinder,-1)\vert ~d \axialcoordofcylinder 
\\
&
+\int_{G_{h{\llcorner{(0,w_1^0)}}}}\vert \psione \vert d \Hone
+
\int_{(\h(0),1)} \dirdatum(0,\scoord) d \scoord
\\ 
\geq 
&\int_{0}^{w_1^0}\int_{-1}^{h(\axialcoordofcylinder)} \sqrt{1+\vert\nabla \psione\vert^2}
d\scoord d\axialcoordofcylinder
+
\int_{(-1,\h(0))}  \vert \psione(0,\scoord)  -\dirdatum(0,\scoord)\vert d \scoord 
\\
&+ \int_{G_{h{\llcorner{(0,w_1^0)}}}}\vert \psione^- \vert d \Hone
+
\int_{(\h(0),1)} \dirdatum(0,\scoord) d \scoord
\\
> & 
\int_{0}^{\axialcoordofcylinder_0}
\int_{-1}^{h(\axialcoordofcylinder)} \vert D_\axialcoordofcylinder 
\psione (\axialcoordofcylinder,\scoord) \vert 
d\scoord d\axialcoordofcylinder
+
\int_{(-1,\h(0))}  \vert \psione(0,\scoord) 
 -\dirdatum(0,\scoord)\vert d \scoord 
\\
&+ \int_{G_{h{\llcorner(0,w_1^0)}}}\vert \psione^- \vert d \Hone
+
\int_{(\h(0),1)} \dirdatum(0,\scoord) d \scoord,
 \end{aligned}
\end{equation*}
where $D_{w_1}$ stands for the partial derivative with respect to $w_1$.
Neglecting  $\sqrt{1+(\frac{d}{ds} h^{-1})^2}$ in the third addendum on the
right-hand side, 
we deduce
\begin{equation*}
\begin{aligned}
\frac12\FB(\h,\psione)
> 
&\int_{-1}^{\h(0)}\int_{0}^{\h^{-1}(\scoord)} \vert 
D_\axialcoordofcylinder  \psione (\axialcoordofcylinder,\scoord) 
\vert ~
d\axialcoordofcylinder
d\scoord 
+ 
\int_{(-1,\h(0))}  \vert \psione(0,\scoord)  -\dirdatum(0,\scoord)\vert d \scoord 
\\
& +\int_{(-1,h(0))}\psione (h^{-1}(\scoord),\scoord)d\scoord 
+\int_{(\h(0),1)} \dirdatum(0,\scoord) d \scoord
\\
\geq &
\int_{-1}^{\h(0)}\Big\vert\int_{0}^{\h^{-1}(\scoord)}  
D_\axialcoordofcylinder  \psione (\axialcoordofcylinder,\scoord) 
d\axialcoordofcylinder \Big\vert d\scoord
-
\int_{(-1,\h(0))} \psione(0,\scoord) d \scoord
\\
&  +\int_{(-1,h(0))}\psione (h^{-1}(\scoord),\scoord)d\scoord 
+\int_{(-1,1)} \dirdatum(0,\scoord) d \scoord
\\
\geq &
\int_{(-1,\h(0))}\lvert  \psione (\h^{-1}(\scoord),\scoord) -\psione (0,\scoord)\rvert d\scoord
-
\int_{(-1,\h(0))} \psione(0,\scoord) d \scoord 
\\
& +\int_{(-1,\h(0))}\psione (h^{-1}(\scoord),\scoord)d\scoord 
+ \int_{(-1,1)} \dirdatum(0,\scoord) d \scoord
 \nonumber\\
 \geq & \int_{(-1,1)}\varphi(0,\scoord)d\scoord =\frac12\FB(-1,0).
 \nonumber 
\end{aligned}
\end{equation*}
Hence the value of $\FB$ on 
the pair $\h\equiv-1$, $\psi \equiv 0$  
is smaller than $\FB(\h,\psi)$, thus
contradicting the minimality of $(h,\psi)$. 
\end{proof}

We now prove point (iii) of Theorem \ref{Thm:existenceofminimizer}:
 this will be a consequence of 
the next lemma and Theorem \ref{prop_zero}.

\begin{lemma}\label{lem:attaining}
 Let $(h,\psi)$ be as in Lemma \ref{lemma_reg2_bis}.
Then $\psi$ attains the boundary condition on 
$\partial_D\Omegah$.
\end{lemma}
\begin{proof}
The result follows from \cite[Theorem 15.9]{Giusti:84},  
since $\partial_D\Omegah$ is union of three segments. 
\end{proof}

\begin{remark}\label{rem:h_id_1}
 In the hypotheses of Lemma 
\ref{lemma_reg2_bis}, if $h\equiv1$ then the graph of $h$ is a 
segment and, as in Lemma \ref{lem:attaining},  $\psi=0$ on $G_h$.
\end{remark}

The conclusion of the proof of Theorem \ref{Thm:existenceofminimizer} (iii)
is given
by the following delicate result.

\begin{theorem}\label{prop_zero}
 Let $(h,\psi)$ be as in Lemma \ref{lemma_reg2_bis}.
Then there exists a solution $(\widetilde h,\widetilde\psi) \in 
X_{2\longR}^{{\rm conv}}$ of the minimum problem 
\eqref{eqn:B} such that $L_{\widetilde h}$ is empty, $\widetilde \psi$ is continuous up
to  $G_{\widetilde h}$, and 
$$
\widetilde\psi = 0 \quad{\rm on}~ G_{\widetilde h}.
$$
\end{theorem}

\begin{proof}
 By Remark \ref{rem:h_id_1}, we can assume that $h$ is not
 identically $1$ and, by Lemma \ref{lemma_reg2_bis}, also that 
$h(\axialcoordofcylinder) \geq h(l)>-1$ for any $\axialcoordofcylinder\in 
[0,2\longR]$.
 Therefore, fix a number $\bar s \in (-1,h(l))$ and
set
$$K:=(0,2\longR)\times (\bar s, 1) \subset R_{2\longR}.
$$
We extend $\psi$ in  $\R^2 \setminus R_{2\longR}$ as follows:
we define $\extpsi : \R^2 \to [0,1]$, $\extpsi := \psi$ in $R_{2\longR}$, and 
\begin{align*}
 \extpsi(
\axialcoordofcylinder,\scoord):=\begin{cases}
                \varphi(\scoord) &\text{if } \axialcoordofcylinder
<0 \text{ or } \axialcoordofcylinder>2\longR, \text{ and }|\scoord|\leq 1,\\
                0&\text{if }|\scoord|>1.
               \end{cases}
\end{align*}
In this way  $\extpsi$ is continuous in 
$\R^2\setminus \overline R_{2\longR}$.

Now, we divide the proof
into six steps.
We start by regularizing
$\extpsi$ (step 1)
in order that the regularized functions have smooth graphs (hence of 
disc-type\footnote{We expect the graph of $\extpsi$, considering
also a possible vertical part over the graph of $h$, to be
a surface of disc-type; however, we miss the proof
of this fact, mainly due to possible high
degree of irregularity of the trace of $\extpsi$
over $G_h$.}). Next (step 2), 
we will compare these graphs with the solution of 
a suitable disc-type Plateau problem.

\textit{Step 1: Approximation of $\extpsi$.} 
Let $n>0$ be a natural number
  (that will be sent to $+\infty$ later) 
such that $\bar s+\frac1n<h(l)$,
and consider the enlarged rectangle
\begin{equation}\label{eq:enlarged_rectangle}
K_n:=\left(-\frac1n,2l+\frac1n\right)\times \left(\bar s, 1+\frac1n\right),
\end{equation}
see Fig. \ref{fig:curves_Frechet}.
Note that 
$$
\extpsi {\rm ~ is~ continuous~
on~ } \partial K_n.
$$
Given $n \in \mathbb N$, we claim that 
we can build a sequence 
$(\psi_k^n)_{k\in \mathbb N}$ (depending on $n$) 
which satisfies the following properties:
\begin{equation}\label{eq:properties_of_psi_n}
\begin{aligned}
&\psi_k^n \in C^\infty(K_n,[0,1]) \cap 
C(\overline K_n,[0,1]) \qquad \forall k >0,
\\
 &\psi^n_k\rightharpoonup \extpsi \text{ weakly}^\star \text{ in }BV(K_n)
\text{ as } k\rightarrow +\infty,
\\
& \int_{K_n}|\grad \psi^n_k|~dw \rightarrow |D\extpsi|(K_n)
\text{ as } k\rightarrow +\infty,\\
 &\psi_k^n=\extpsi \text{ on }\partial K_n\qquad\forall k>0.
\end{aligned}
\end{equation}
In order to obtain these features for $\psi^n_k$ we use standard arguments 
(details can be found in \cite[Thm. 3.9]{AmFuPa:00} or \cite[Thm. 1, Section 4.1.1]{GiMoSu:98}). 
To the aim of our discussion, we just recall that we proceed by constructing an 
increasing sequence 
$(U_i)_{i\geq 1}$ of subsets of $K_n$,
$U_i = U_{i,n}$,
 $U_i \subset \subset U_{i+1}
\subset \subset K_n$, $\cup_i U_i= K_n$ (for 
$i \geq 1$ we take
$U_i:=\{x\in\R^2:\textrm{dist}(x,\R^2\setminus K_n)>
\frac{1}{i+n}\}$
for definitiveness) and with the aid of a
partition  
of unity
$(\eta_i)$ 
associated to $V_1:=U_2$, 
$V_i= 
V_{i,n}
:=U_{i+1}\setminus \overline U_{i-1}$ for $i\geq2$, we 
mollify $\widehat \psi$ accordingly in $V_i$. For our purpose we 
choose\footnote{
We need the full set $\overline V_i$ as support 
in order that the argument to detect the behaviour of $h_n$ 
(defined in \eqref{eq:def_h_n}) in 
$[-\frac1n,0]$ applies.} 
$\eta_i$ in such a way that 
\begin{align}\label{support_eta_i}
\supp(\eta_i)=\overline V_i.
\end{align}
Since $\psi_k^n$ is obtained by mollification we have $\psi_k^n\in 
C^\infty(K_n)$ and moreover 
$\psi^n_k\in  C(\overline K_n)$ 
because it attains the continuous boundary datum $\extpsi$
on $\partial K_n$.
Notice that we use the same mollifier $\rho\in C_c^\infty (B_1)$ in each $V_i$, choosing $\rho_i(w) = \rho_{i,k}(w):=\rho(w/r_{i,k})$ with  
$r_{i,k}:=r_i/k>0$,  $r_i$ decreasing with 
respect to $i\geq1$, with\footnote{$\rho_{i,k}$ and 
$r_{i,k}$ depend on $n$. We could 
take $r_i=\frac{1}{i+2+n}$. 
}  
$r_{i}\rightarrow 0^+$ as $i\rightarrow+\infty$. 
Finally, $[0,2l]\times[\bar s+\frac1n,1]\subset U_1\subset V_1$, 
and $V_{i}\cap \left(
[0,2l]\times[\bar s+\frac1n,1]\right)=\emptyset$ for $i \geq 2$.
It follows 
\begin{equation}\label{mollification_in_K}
\psi_k^n=\widehat\psi\star \rho_{1,k}\qquad \text{ in }[0,2l]\times\Big[
\bar s+\frac1n,1\Big] \qquad \forall
n \in \mathbb N. 
\end{equation}
Using \cite[Prop. 3 Sec. 4.2.4 pag. 408, and Th. 1 Sec. 4.1.5 pag. 331]{GiMoSu:98} 
we infer
\begin{align}\label{14.34}
 \areaonecod(\psi_k^n,K_n)\rightarrow \areaonecod(\extpsi, K_n)
\qquad \text{ as } k\rightarrow +\infty. 
\end{align}
Now that properties \eqref{eq:properties_of_psi_n} are achieved, 
by a diagonal argument we select 
functions $\psi_{n}:=\psi_{k_n}^n\in (\psi^n_k)$ such that
\begin{equation}\label{eq:diag}
\begin{aligned}
 &\psi_n\rightharpoonup \extpsi \text{ weakly}^* \text{ in }BV(K)
\text{ as } n\rightarrow +\infty,\\
& \int_{K_n}|\grad \psi_n|~dw\rightarrow |D\extpsi|(\overline{K})
\text{ as } n\rightarrow +\infty,\\
 &\psi_n=\extpsi \text{ on }\partial K_n \qquad \forall n\in\mathbb N.
\end{aligned}
\end{equation}
On the basis of \eqref{14.34} and  \eqref{eq:diag}, 
we can also ensure\footnote{To prove claim \eqref{conv_areas_Kn}, 
fix $m\in \mathbb N$, and set 
$\tilde \psi_n:=\extpsi$ outside $K_n$ and $\tilde\psi_n = \psi_n$
in $K_n$, so 
that 
\begin{align*}
&\tilde \psi_n\rightharpoonup \extpsi \text{ weakly}^* \text{ in }BV(K_m)
\text{ as } n\rightarrow + \infty,
\\
& |\grad \tilde \psi_n|(K_m)\rightarrow|D\extpsi|(K_m)= |D\extpsi|(\overline{K})+|D\extpsi|(K_m\setminus \overline{K})
\text{ as } n\rightarrow + \infty.
\end{align*}
Then   
$\limsup_{n\rightarrow+\infty} \areaonecod(\psi_n, K_n)\leq 
\limsup_{n\rightarrow+\infty} \areaonecod(\tilde \psi_n, K_m)= \areaonecod(\extpsi, K_m)
= \areaonecod(\extpsi, \overline{K})+ \areaonecod(\extpsi, K_m\setminus \overline{K})$,
the first equality following from the strict convergence 
of $\tilde \psi_n$ to $\widehat \psi$
\cite[Prop. 3 Sec. 4.2.4 pag. 408 and 
Thm. 1 Sec. 4.1.5 pag. 371]{GiMoSu:98}. 
Taking the limit as $m\rightarrow +\infty$, 
since 
$\widehat \psi \in W^{1,1}(K_m\setminus \overline{K})$
we conclude
$\limsup_{n\rightarrow+\infty} \areaonecod
(\psi_n, K_n)\leq \areaonecod(\psi, \overline{K})$.
Then 
\eqref{conv_areas_Kn} follows by 
lower-semicontinuity.} 
 that
\begin{align}\label{conv_areas_Kn}
 \areaonecod(\psi_n, K_n)\rightarrow \areaonecod(\extpsi, \overline K) \qquad{\rm 
as}~ n \to +\infty.
\end{align}
Here, by $\mathbb A(\extpsi, \overline K)$ we mean the 
area of the graph of $\extpsi$ relative to the 
closed rectangle $\overline K$, which, recalling also 
Proposition \ref{prop_reg},
 reads as
\begin{equation}\label{eq:area_closure}
\areaonecod(\extpsi, \overline K)=\areaonecod(\extpsi, K)+
\int_{\{0\}\times(\bar s,1)}|\extpsi^--\varphi|~d\mathcal H^1+
\int_{\{2l\}\times(\bar s,1)}|\extpsi^--\varphi|~d\mathcal H^1, 
\end{equation}
where $\extpsi^-$ denotes 
the trace of $\extpsi$ on $\partial K$.

\smallskip

We  now construct functions
$h_n: (-\frac1n,2l+\frac1n) \to (\bar s, 1+\frac1n)$ 
such that
\begin{equation}\label{eq:constr_h_n}
\begin{aligned}
& h_n(\cdot) = h_n(2\longR - \cdot),
\\
&  \psi_n=0 {\rm ~}~ K_n \setminus SG_{h_n},
\\
& h_n {\rm ~is~nondecreasing~in~} \big[-\frac{1}{n}, \longR\big], 
\\
& \mathcal H^2(S_{h_n})\rightarrow\mathcal H^2(\Omegah\cap K),
\end{aligned}
\end{equation}
where
\begin{align}\label{topology_Kn+}
S_{h_n}:=\left\{(\axialcoordofcylinder,\scoord):
\axialcoordofcylinder\in \left(-\frac1n,2l+\frac1n\right),\; \scoord
\in (\bar s,h_n(\axialcoordofcylinder))\right\},
\end{align}
and finally
$$
\lim_{n \to +\infty}\areaonecod(\psi_n,  S_{h_n}) = 
\mathcal F_{2l}(h,\psi)-\areaonecod(\psi, R_{2l}\setminus K).  
$$

For any $n\in \mathbb N$ 
we define
\begin{equation}\label{eq:def_h_n}
\begin{aligned}
 h_n(\axialcoordofcylinder):=&
\sup\left\{\scoord\in \Big(\bar s,1+\frac1n\Big):
\psi_n(\axialcoordofcylinder,\scoord)>0\right\}
\qquad \forall \axialcoordofcylinder
\in \Big(-\frac1n,2l+\frac1n\Big).
\\
\widehat h(\axialcoordofcylinder):=&
\sup\left\{\scoord\in \Big(\bar s,1+\frac1n\Big):
\widehat \psi(\axialcoordofcylinder,\scoord)>0\right\}
\qquad \forall \axialcoordofcylinder
\in \Big(-\frac1n,2l+\frac1n\Big).
\end{aligned}
\end{equation}
Since (see Proposition \ref{prop_reg}) 
$\extpsi$ is positive in $\Omegah\cup((-\frac1n,0)\times (\bar s,1))\cup((2l,2l+\frac1n)\times (\bar s,1))$ 
it turns out, recalling also that 
$\psi_n$ is obtained by mollification, that
\begin{equation}\label{eq:lines}
\begin{aligned}
 & -1 < h(\axialcoordofcylinder)<h_n(\axialcoordofcylinder)<1+\frac1n \qquad 
\forall \axialcoordofcylinder\in (0,2l),
\\
 &1<h_n(\axialcoordofcylinder)<1+\frac1n \qquad \forall
\axialcoordofcylinder\in \Big(-\frac1n,0\Big]\cup\Big[2l,2l+\frac1n\Big).
\end{aligned}
\end{equation}
Also, $\widehat h = h$ in $[0,2l]$, and $\widehat h = 1$
in $(-1/n,0) \cup (2l, 2l + 1/n)$.
Moreover, again the positivity of $\extpsi$ implies
that 
\begin{align}\label{topology_Kn+}
\psi_n>0 \ \  {\rm ~in}~
 \ \ S_{h_n} \subset K_n,
\end{align}
whereas
\begin{align}\label{eq:psi_n_zero}
 \psi_n(\axialcoordofcylinder, \scoord)=0 \qquad \text{if }
\tcoord \in \Big(-\frac{1}{n}, 2 \longR + \frac{1}{n}\Big), 
~\scoord\in \Big[h_n(\axialcoordofcylinder),1+\frac1n\Big),
\end{align}
because $\extpsi(\axialcoordofcylinder,
\scoord)=0$ 
if $\axialcoordofcylinder\in[0,2l]$, $\scoord>h(\axialcoordofcylinder)$ and if $\scoord>1$.
Exploiting \eqref{mollification_in_K}, and the fact that $h$ is nonincreasing 
(resp. nondecreasing)  in $[0,\longR]$ (resp. in $[l,2l]$), one checks\footnote{Let us show for instance that $h_n$ is decreasing in 
$[0,\longR]$.
Recall that the function $\widehat \psi$ vanishes above the graph of $h$, 
which is decreasing in $[0,\longR]$. Now, take a point $(w_1,w_2)\in K_n$, $w_1 
\in [0,\longR)$,  $w_2 > h(w_1)$; suppose first that $w_1 \geq r_1$.
If ${\rm dist}((w_1,w_2), {\rm graph}(h))
> r_1$, then $\psi_n(w_1,w_2)
=\widehat \psi\star \rho_1(w_1,w_2)
=0$, and  
if ${\rm dist}((w_1,w_2), {\rm graph}(h))
< r_1$, then $\psi_n(w_1,w_2)=\widehat \psi\star \rho_1(w_1,w_2)
>0$.
Hence,  if $\widehat \psi\star \rho_1(w_1,w_2)=0$
then also $\widehat \psi\star \rho_1(w_1+\eps,w_2)=0$ 
for $\eps>0$ small enough, because 
${\rm dist}((w_1+\eps,w_2), {\rm graph}(h)) >
{\rm dist}((w_1,w_2), {\rm graph}(h))$, 
being $h$ decreasing in $[0,\longR]$.
This argument applies also when $w_1 \in [0,r_1)$,
by \eqref{mollification_in_K}, since 
$\overline h$ is nonincreasing also in $(-1/n, l)$.
}
that also $h_n$ is nonincreasing in  $[0,\longR]$
(and nondecreasing in $[l,2l]$). 
Concerning the behaviour of $h_n$ in $(-\frac1n,0]$ (and similarly in $[2l,2l+\frac1n)$), we see that in $V_i = V_{i,n}$ ($i>1$), we are mollifying 
with $\rho_{i,k_n}$ 
whose radius of mollification is $r_i/k_n$, so that 
$\widehat \psi\star \rho_{i,k_n}$  equals $0$ 
on the line $\{w_2=1+\frac{r_i}{k_n}\}$, and nonzero below inside $K_n$ 
(this follows from the fact that $\widehat \psi$ is  $0$
 on the line $\{w_2=1\} \cap K$ and nonzero below).
We have defined the radii $r_i$ 
in such a way that they are decreasing 
with respect to $i$, so that, $\psi_n$ being the 
sum
of $\widehat \psi\star \rho_{i,k_n}$  (whose support is 
${\overline V}_{\!i,n}$ by \eqref{support_eta_i}), 
it turns out that $\psi_n$ is $0$ on $\{w_2=1+\frac{r_i}{k_n}\}$ 
and nonzero below\footnote{Notice that in $V_{i,n}
\setminus V_{i-1,n}$ only $\widehat \psi\star \rho_{i,k_n}$ and 
$\widehat \psi\star \rho_{i+1,k_n}$, are nonzero.}
in $V_{i,n}\setminus V_{i-1,n}$ 
(from this it follows that $h_n=1+\frac{r_i}{k_n}$ in $(-\frac1n+\frac{1}{i+n+1},-\frac1n+\frac{1}{i+n}]$). 
In particular $h_n$ is piecewise constant and nondecreasing 
in $(-\frac1n,0]$.

Therefore we have
\begin{equation}\label{h_n_BV}
h_n\in BV\Big((-\frac1n,2l+\frac1n)\Big). 
\end{equation}
Finally, it is not difficult to see that the functions $h_n$ 
converge 
to $h$ in $L^1((0,2l))$ 
as $n\rightarrow \infty$, 
and
\begin{align}
 \mathcal H^2(S_{h_n})\rightarrow\mathcal H^2(\Omegah\cap K).\label{eqn:25}
\end{align}
From this, \eqref{conv_areas_Kn}, Lemma \ref{lem:attaining}, \eqref{eq:area_closure} and
\eqref{eqn:FB}
we deduce
\begin{align}\label{convergenza_plateau}
\areaonecod(\psi_n,  S_{h_n})=
 \areaonecod(\psi_n, K_n)-\mathcal H^2(K_n\setminus S_{h_n})
\rightarrow\mathcal F_{2l}(h,\psi)-\areaonecod(\psi, R_{2l}\setminus K).  
\end{align}

\textit{Step 2: Comparison with a Plateau problem.} In this step we 
want to compare the graph of $\psi_n$ with the solution of a
 disc-type Plateau problem. In particular we will 
obtain a disc-type surface $\dtsnp$ 
whose area is smaller than 
or equal to the area of the graph of $\psi_n$,
see
\eqref{key_ineq_D+}.
 In step 3 (see \eqref{ineq_withF}) we will compare this surface with the graph of $\psi$ on $K$. 

We recall that $\psi_n$ is continuous in 
$\overline K_n$, it is positive on the bottom edge 
$[-\frac1n,2l+\frac1n]\times \{\bar s\}$
of $K_n$ (see \eqref{eq:diag}), 
it is zero on the top edge $[-\frac1n,2l+\frac1n]\times \{1+\frac1n\}$
by \eqref{eq:lines}, and on the lateral edges of $K_n$ it coincides with $\widehat \psi$; more specifically
\begin{equation*}
\begin{aligned}
 &\psi_n\Big(-\frac1n,\scoord\Big)=\psi_n\Big(2l+\frac1n,\scoord\Big)
=\varphi(0,\scoord)>0\qquad \text{for } \scoord\in [\bar s,1),\nonumber\\
 &\psi_n\Big(-\frac1n,\scoord\Big)
=\psi_n\Big(2l+\frac1n,\scoord\Big)=0\qquad \text{for }\scoord\in \Big[1,1+\frac1n\Big).
\end{aligned}
\end{equation*}
Define
$$\partial_DK_n:=\Big(\Big[-\frac1n,2l+\frac1n\Big]
\times \{\bar s\}\Big)\cup\Big(\Big\{-\frac1n,2l+\frac1n\Big\}\times [\bar s,1]\Big).
$$
{}From \eqref{eq:diag}, we see that $\psi_n$ coincides with 
$\extpsi$ over  $\partial_DK_n$, and its graph over this set is a 
curve,
that we 
denote by $\Gamma_n^+$. This curve, excluding its endpoints 
$P_n=(-\frac1n,1,0)$ and $Q_n=(2l+\frac1n,1,0)$, 
is contained in the half-space 
$\{w_3>0\}$, while 
$P_n, Q_n \in \{w_3=0\}$.
We further denote by $\Gamma_n^-$ the
 symmetric of $\Gamma_n^+$ with respect to the plane $\{w_3=0\}$, 
so that 
$$
\Gamma_n:=\Gamma_n^+\cup\Gamma_n^-
$$
is a Jordan curve in $\R^3$, see
Fig. \ref{fig:curves_Frechet}.
We can now solve the disc-type
Plateau problem with boundary $\Gamma_n$ \cite{Hil1} and 
call $\dtsn \subset \R^3$ one of its  
solutions\footnote{$\dtsn$ is the image of an area-minimizing
map from the unit disc into $\R^3$.}. 
Finally, we can assume that $\dtsn$ is symmetric with respect to $\{w_3=0\}$ and that 
\begin{align*}
 \mathcal H^2(\dtsnp)=\mathcal H^2(\dtsnm),
\end{align*}
with $\dtsn^\pm:=\dtsn\cap \{w_3\gtrless0\}$, respectively (see Fig. \ref{fig:curves_Frechet}).

\begin{figure}
	\begin{center}
		\includegraphics[width=0.8\textwidth]{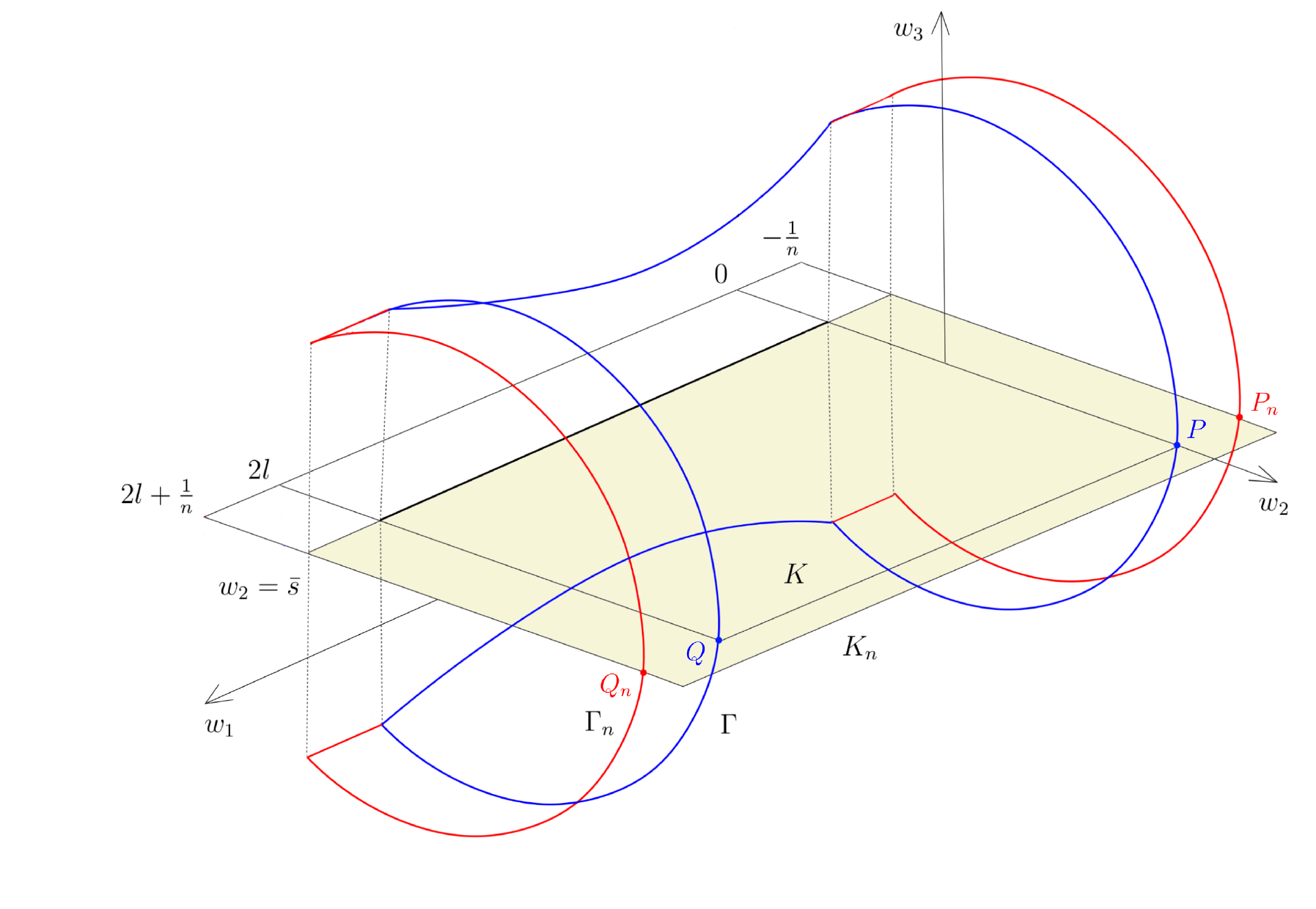}
		\caption{The rectangle $K_n$ 
in \eqref{eq:enlarged_rectangle} is colored, 
and the rectangle inside is $K$. $\Gamma$ is the
curve passing through $Q$ and $P$, the curves $\Gamma_n$
(which pass through $Q_n$ and $P_n$) 
approach the curve $\Gamma$ ($\Gamma$ and $\Gamma_n$ coincide and overlap on the graph of $\psi$ over the bold segment $\{w_2=\bar s\}\cap K$). 
		}
		\label{fig:curves_Frechet}
	\end{center}
\end{figure}

We now want to compare the area of the graph
of $\psi_n$ in 
$S_{h_n}$ with 
$\mathcal H^2(\dtsnp)$. 
To this aim we start by observing that $\psi_n$, 
being smooth in $K_n$ and continuous in $\overline K_n$, is such that 
its graph over $S_{h_n}$ has the topology of 
$S_{h_n}$, 
that is the topology of the 
disc\footnote{
$S_{h_n}$ is
bounded by construction, and it is open
from 
\eqref{topology_Kn+}, \eqref{eq:psi_n_zero}. 
In addition, it is connected and simply connected. Indeed,
take any continuous curve $\gamma:S^1\rightarrow S_{h_n}$.
Using \eqref{eq:lines}, 
let $\widehat s\in (\bar s, 1)$ 
be such that $\{w_2=\widehat s\}\cap K_n\subset S_{h_n}$; hence we can 
(vertically) contract $\gamma$ 
continuously to its projection on the line $\{w_2=\widehat s\}$, 
and then contract it continuously to the middle point of $\{w_2=\widehat s\}\cap K_n$, showing that $\gamma$ is homotopic to the constant curve.
 Hence,
by the Riemann mapping theorem, $S_{h_n}$ is biholomorphic to the 
open unit disc, and 
$\overline S_{h_h}$
is homeomorphic to the closure of the disc, thanks to the fact that $\partial S_{h_n}$ is a Jordan curve, due to the BV-regularity of $h_n$.
}. 
Denoting by $\mathcal G_{\psi_n}^+$ the graph of $\psi_n$ over $S_{h_n}$, 
we consider the graph $\mathcal G_{\psi_n}^-$  of $-\psi_n$ over $S_{h_n}$, 
and observe that the closure of 
$\mathcal G_{\psi_n}^+\cup \mathcal G_{\psi_n}^-$ is a disc-type surface with boundary $\Gamma_n$. Therefore, 
by minimality,
\begin{align}
\label{key_ineq_D+}
 \areaonecod(\psi_n, S_{h_n})=\mathcal H^2(\mathcal G_{\psi_n}^+)\geq \mathcal H^2(\dtsnp).
\end{align}

\textit{Step 3: Passing to the limit as $n\rightarrow+\infty$: the surface $\Sigma$.}
The graph of $\psi$ over the segment $[0,2l]\times \{\bar s\}$ 
and the graph of $\varphi$ over the two segments 
$\{0,2l\}\times [\bar s,1]$ form a simple continuous curve $\Gamma^+$ which,
excluding the two endpoints $P=(0,1,0)$ and $Q=(2l,1,0)$, 
is contained in the half-space $\{w_3>0\}$, while
$P,Q \in \{w_3=0\}$ (see
Fig. \ref{fig:curves_Frechet}).
If we consider
$$
\Gamma:=\Gamma^+\cup\Gamma^-,
$$
with $\Gamma^-$ the symmetric of $\Gamma^+$ with respect to $\{w_3=0\}$, a direct check shows that the curves 
$\Gamma_n$ converge to the curve $\Gamma$
in the sense of 
Frechet \cite{Nitsche:89}, as $n\rightarrow +\infty$. 
As a consequence, the area-minimizing disc-type surfaces $\dtsn$ satisfy $\mathcal H^2 (\dtsn)\rightarrow\mathcal H^2(\dts)$ (see \cite[Paragraphs 301, 305]{Nitsche:89}), with 
$\Sigma$ a disc-type area-minimizing surface 
spanned by $\Gamma$. It follows
\begin{equation}\label{eq:convergence_of_areas_Sigma_n}
 \mathcal H^2(\dtsnp)
\rightarrow \mathcal H^2(\dtsp) \qquad {\rm as}~ n \to +\infty,
\end{equation}
where $\dtsp:=\dts\cap \{w_3>0\}$.  
{}From \eqref{eq:convergence_of_areas_Sigma_n}, 
\eqref{key_ineq_D+}, \eqref{convergenza_plateau}
we deduce
$$
\begin{aligned}
\mathcal H^2(\dtsp)  = &
\lim_{n \to +\infty} 
\mathcal H^2(\dtsnp)
\leq 
\lim_{n \to +\infty} 
 \areaonecod(\psi_n, S_{h_n})
\\
=&
\lim_{n \to +\infty}\left(
 \areaonecod(\psi_n, K_n)-\mathcal H^2(K_n\setminus S_{h_n})\right)
=
\mathcal F_{2l}(h,\psi)-\areaonecod(\psi, R_{2l}\setminus K).  
\end{aligned}
$$ 
Since $\psi_n = \psi$ on $R_{2\longR} \setminus K$, we 
get 
\begin{align}
\label{ineq_withF}
\lim_{n \to +\infty} \left(
\areaonecod(\psi_n, S_{h_n})+\areaonecod(\psi, R_{2l}\setminus K)
\right)=  \FB(h,\psi)\geq \mathcal H^2(\dtsp)+ \areaonecod(\psi, R_{2l}\setminus K).  
\end{align}

Let $\Phi=(\Phi_1,\Phi_2,\Phi_3):\overline B_1\rightarrow \dts \subset \R^3$ be a parametrization of $\Sigma$, which is 
analytic and conformal in the open unit disc
$B_1$ and continuous up to 
$\partial B_1$ with $\Phi(\partial B_1)=\Gamma$. 
Exploiting the results in \cite{Meeks_Yau:82} (see also \cite[pag. 343]{Hil1}) 
we know that %
\begin{equation}
\label{eq:Phi_is_an_embedding}
\Phi {\rm~ is~ an~ embedding}, 
\end{equation}
since $\Gamma$ is a simple curve on
 the boundary of the convex set $K\times \R$.

\medskip
Now, we need to prove several qualitative properties of $\Sigma$.
\smallskip

\textit{Step 4: 
$\dts \cap \{w_3=0\}$ is a simple curve 
connecting $P$ and $Q$.}
 This can be seen as follows:
Assume $\Phi(p_0)=P$ and $\Phi(q_0)=Q$ for two distinct points $p_0,q_0\in \partial B_1$. 
By standard arguments\footnote{See also step 5 where a similar statement is proved.}, 
the disc $B_1$ is splitted into two 
connected components $\{x\in B_1:\Phi_3(x)\geq 0\}$ and $\{x\in B_1:\Phi_3(x)<0\}$ and the set 
$\{\Phi_3=0\}$ must be a simple curve in $B_1$ connecting $p_0$ and $q_0$  (here we use that the points $p_0$ and $q_0$ are the unique points in $\partial B_1$ where $\Phi_3=0$ and that the two arcs in $\partial B_1$ with extreme points $p_0$ and $q_0$ are mapped in $\{w_3>0\}$ and $\{w_3<0\}$ respectively).
By the injectivity of $\Phi$ (property
\eqref{eq:Phi_is_an_embedding}) we conclude that 
\begin{equation}\label{eq:Gamma_0}
\Gamma_0:=\Phi(\{\Phi_3=0\})
\end{equation}
is a simple curve connecting $P$ and $Q$ on the plane $\{w_3=0\}$, 
and more specifically $\Gamma_0 \subset K$. 

\medskip

\medskip

In the next step we show that, due to the particular
form of $\Gamma$, the surface $\Sigma$ admits a semicartesian
parametrization \cite{BePaTe:16}, namely that if we slice
$\Sigma$ with a plane orthogonal to the first coordinate (in 
$(0,2l)$) then the intersection is a curve connecting
the two corresponding points on $\Gamma$; in addition, in this
present case, 
this curve turns out to be simple. We will also show 
that the free part of $\Sigma$, {\textit{i.e.}, $\Gamma_0$,}  leaves a trace on $R_{2l}$
which is the graph of a convex function (of one variable). 

\medskip

\textit{Step 5:
The projection $p(\Sigma)$ of $\Sigma$ on 
the plane $\{w_3=0\}$ is the subgraph of a convex 
function $\widetilde h\in \mathcal H_{2l}$.} 

We first show that $p(\Sigma)$ is the subgraph of a 
function $\widetilde h$, and then we prove that 
$\widetilde h\in \mathcal H_{2l}$. 
 Take a point $W = (W_1,W_2,W_3)\in \Sigma\setminus \Gamma$; 
by the strong maximum principle, $p(W) \notin \partial K$ (this follows since points in $\Sigma\setminus \Gamma$ are in the interior of the convex envelope of $\Gamma$, see \cite{Hil1}). 
Consider the (unique) point $x\in B_1$ 
such that $\Phi(x)=W$. Due to the particular structure of $\Gamma$,
one easily checks that 
 $\partial B_1=\Phi^{-1}(\Gamma)$ splits into two 
connected components, $\Phi_1^{-1}((W_1,2l])\cap \partial B_1$ 
and $\Phi_1^{-1}([0,W_1])\cap \partial B_1$,
 since $\Phi_1^{-1}(\{W_1\})\cap \partial B_1$ consists
of two points $q_1$, $q_2$ in $\partial B_1$. 
In particular, the continuous function $\Phi_1(\cdot)-W_1$ changes 
sign only twice on $\partial B_1$, 
namely at $q_1$ and $q_2$. 
From Rado's lemma \cite[Lemma 2, pag. 295]{Hil1} 
it follows that there are no points on $\Sigma \cap\{w_1=W_1\}$ 
where the two area-minimizing surfaces $\Sigma$ and the plane $\{w_1=W_1\}$ 
are tangent to each other\footnote{If $\pointP$ is a tangence point, 
then the differential of $\Phi_1$ must vanish at $\Phi^{-1}(\pointP)\in B_1$.}.  
It follows that, if $\pointP\in (\Sigma\setminus \Gamma)\cap \{w_1=W_1\}$, 
then the set $(\Sigma\setminus \Gamma)\cap \{w_1=W_1\}$ is, 
in a neighbourhood of $\pointP$, an analytic curve, 
see again 
\cite[Lemma 2, pag. 295]{Hil1}. 
Hence, $\{\Phi_1=W_1\}\cap B_1$ is, 
in a neighbourhood of $\Phi^{-1}(\pointP)$, an analytic curve.
If $\gamma_I:I\rightarrow B_1$ 
is a parametrization of this curve,
$I=(a,b)$ a bounded open interval, 
we see that the limits as 
$t\rightarrow a^+$ and $t \to b^-$
of 
$\gamma_I(t)$ 
exist\footnote{$\overline B_1$ is compact, 
hence $\gamma_I(t)$ has some accumulation point
as $t\rightarrow a^+$. Notice that $I$ and $\gamma_I(I)$ 
are homeomorphic by contruction; in turn $\gamma_I(I)$ is 
homeomorphic to the analytic curve 
$\Phi\circ\gamma_I(I)$.
Assume $x$ is an accumulation point for $\gamma_I(t)$ as $t\rightarrow a^+$. If $x\in B_1$, there is a small neighborhood $U$ of $x$ such that $\sigma:=\Phi(U)\cap \{w_1=W_1\}$ is an analytic curve. Then $\gamma_I$, in a right neighbourhood $J$ of $a$, is homeomorphic to the analytic curve $\Phi\circ\gamma_{I}(J)\in \R^3$ emanating from $\Phi(x)$, which in turn is the restriction of $\sigma$. In particular  $\gamma_I(I)$ is a curve emanating from $x$ and the limit as $t\rightarrow a^+$ of $\gamma_I(t)$ is $x$.
If instead $x\in \partial B_1$ then $x$ must be the unique accumulation point. 
Indeed, 
$\lim_{t\rightarrow a^+}\Phi_1\circ\gamma_I(t)=W_1$, 
and then $x=q_1$ or $x=q_2$ (say $x=q_1$).
Assume there is 
another accumulation point $y$ as $t\rightarrow a^+$; then $y\notin B_1$, otherwise we fall in the previous case, 
and therefore necessarily $y=q_2$. But in this case, we 
see that there must be another accumulation point $z\in B_1$ 
(as $t\rightarrow a^+$, we move between a neighbourhood $U$ of $x$ 
and a neighbourhood $V$ of $y$ frequently, 
so that there should be some 
other accumulation point in $\overline B_1\setminus (U\cup V)$)
leading us to the previous case again.} 
and are points in $ \overline {B_1}$. If 
$\lim_{t \to a^+} \gamma_I(t)$ belongs to
$\partial B_1$, 
it must be either $q_1$ or $q_2$, if instead it is in $B_1$, 
then we can always extend $\gamma_I$ in a neighbourhood of $a$ 
and find a larger interval $J\supset I$ on which $\gamma_I$ can be extended. 
Similarly, for 
$\lim_{t \to b^-} \gamma_I(t)$.
Let now 
$I_m=(a_m,b_m)$ be a maximal interval on which $\gamma_I$ is defined, 
so that, by maximality, the limits as $t \to a_m^+$ and $t \to b_m^-$ 
are $q_1$ and $q_2$, respectively.
We can then consider the closure $\overline I_m$ of $I_m$ 
and we have that $\gamma_{\overline I_m}(\overline I_m)$ 
is a curve in $\overline B_1$ joining $q_1$ and $q_2$.
Thus we have proved that $\sigma_W:=\Sigma\cap \{w_1=W_1\}$ equals
$\Phi(\gamma_{\overline I_m}(\overline I_m))$.
In particular $\sigma_{W}$ is a curve in $\R^3$ contained 
in the plane $\{w_1 = W_1\}$
and connecting the points $Q_1:=\Phi(q_1)\in \Gamma$ 
and  $Q_2:=\Phi(q_2)\in \Gamma$. But we know that $p(Q_1)=p(Q_2)=
(W_1,\bar s,0)$, 
so $p(\sigma_W)$ is a segment in $R_{2l}$  
with endpoints $(W_1,\bar s,0)$ and 
$(W_1,s^+,0)$ for some $s^+>\bar s$, and $s^+\geq W_2$. 
 In particular the whole 
segment ``below'' $p(W)$, 
namely the one with endpoints $(W_1,\bar s,0)$ and $(W_1,W_2,0)$, belongs
 to $p(\Sigma)$, and $p(\Sigma)$ is then the subgraph of some function $\widetilde h$. 
As a remark, due to the symmetry of the curve $\Gamma$, we can 
assume $\widetilde h$ is symmetric with respect to $\{w_1=l\}$, namely $\widetilde h (\cdot)=\widetilde h (2l-\cdot)$.
 
 Now we show that $\widetilde h$ is convex. 
Assume it is not, and 
take two points $(t_1,\widetilde h(t_1),0), (t_2,\widetilde h(t_2),0)\in 
R_{2l}$,  
such that there is a third point $(t_3,\widetilde h(t_3),0)$, 
with $t_1<t_3<t_2$, which is strictly above the segment 
$l_{12}$ in $R_{2l}$ joining $(t_1,\widetilde h(t_1),0)$ 
and $(t_2,\widetilde h(t_2),0)$. Let $f:\R^3\rightarrow \R$ be a 
nonzero affine function\footnote{Take  
the signed distance from the plane.} 
vanishing on the plane passing through $l_{12}$ 
and orthogonal to $\{w_3=0\}$, and assume that 
$f$ is positive at $(t_3,\widetilde h(t_3),0)$. Let $\pointQ\in \Sigma$ be 
such that $p(\pointQ)=(t_3,\widetilde h(t_3),0)$. Then $f\circ \Phi:B_1\rightarrow \R$ is harmonic, and by the maximum principle there is a
continuous curve $\gamma_\pointQ$ in $B_1$ 
joining $\Phi^{-1}(\pointQ)$ to $\partial B_1$ such that 
$f\circ \Phi$ is always positive on $\gamma_\pointQ$. But now, the continuous
curve $p\circ \Phi(\gamma_\pointQ)$ joins $(t_3,\widetilde h(t_3),0)$ 
to $p(\Gamma)$ and remains, in $R_{2l}$, strictly above the segment 
$l_{12}$. This is a contradiction,
 because $p\circ \Phi(\gamma_\pointQ)$ must be in the interior of
 the subgraph of $\widetilde h$.

\smallskip

Before passing to step 6, recall the definition of $\Gamma_0$ in 
\eqref{eq:Gamma_0}, and  observe
that the Jordan curve $\Gamma^+\cup\Gamma_0$ is the boundary of the disc-type surface $\dtsp$.
Let us denote by $U\subset K$ the connected component of $K\setminus \Gamma_0$ with boundary $\Gamma_0\cup (\{0\}\times [\bar s,1])\cup ([0,2l]\times \{\bar s\})\cup (\{2l\}\times [\bar s,1])$.

We are now in a position to show that $\Sigma^+$ admits
a non-parametric description over the plane $\{w_3=0\}$.

\medskip
\textit{Step 6: Graphicality of $\dtsp$: the disc-type surface $\dtsp$ 
can be written as a graph over the plane $\{w_3=0\}$ 
of a $W^{1,1}$ function 
$\widetilde \psi:U\rightarrow [0,+\infty)$}. 
At first we observe that if $\Sigma^+$ is not Cartesian with 
respect to $\{w_3=0\}$, then there is 
some point $\pointP\in \Sigma^+\setminus \partial \Sigma^+$ 
where the tangent plane to $\Sigma^+$ is vertical\footnote{This can be 
seen as follows: as shown in step 5,  
the intersection between $\Sigma^+$ and any plane $\{w_1=\textrm{cost}\}$,
${\rm cost}\in (0,2l)$, is a 
simple curve with endpoints in $\partial \Sigma^+$. If $\Sigma^+$ is not Cartesian, one of these curves $\gamma$ 
is not Cartesian, and then there is a point where 
the tangent vector to $\gamma$ is vertical. 
At such a point the tangent plane to $\Sigma^+$  
is vertical. 
}, 
that is, it contains the 
line $\{\pointP+(0,0,w_3):\;w_3\in \R\}$.
We will show,
with an argument similar to 
the one needed to prove Rado's Lemma \cite[Lemma 2, pag. 295]{Hil1},
 that any vertical plane is tangent to $\Sigma$ 
in 
at most one point. 
\begin{itemize}
	\item[Claim:] If $\Pi$ is a vertical plane which is tangent to $\Sigma$, then there is at most one point where 
$\Pi$ and $\Sigma$ are tangent.
\end{itemize}
Assume $\Pi$ intersects the relative interior of $\Sigma$. It is easy to 
see that the intersection between $\Pi$ and the Jordan curve $\Gamma$ consists at most 
of four points\footnote{A vertical plane intersects $K$ on a straight segment. In turn, this segment intersects $\partial K$ in two points. If a vertical plane intersects $\Gamma$ in a point $(W_1,W_2,W_3)$, then $(W_1,W_2,0)\in \partial K$. Moreover this plane intersects $\Gamma$ also at $(W_1,W_2,-W_3)$. Thus, the points of intersection are at most four. The degenerate cases in which $\Pi$ contains a full $\mathcal H^1$-measured part of $\Gamma$ are excluded by this analysis, because in these cases $\Pi$ does not intersect the interior of $\Sigma$. Instead, the cases in which the intersection consists of $2$ or $3$ points are easier to treat, and we detail only the $4$-points case (notice that by the geometry of $\Gamma$, the case of  $3$ points occurs when this plane is tangent to $\Gamma$ at one of the points $(0,1,0)$ or $(2l,1,0)$).} $p_i$, $i=1,2,3,4$.
Let $f$ be a linear function on $\R^3$ vanishing on $\Pi$.
Then  
 $f\circ \Phi$ is harmonic  in $B_1$ 
and continuous in $\overline B_1$; in addition, it vanishes
at $\{p_i,\;i=1,2,3,4\}$, and alternates its sign on the four arcs
$\overline{p_i p_{i+1}}$
on $\partial B_1$
with endpoints $p_i$.
With no loss of generality, we may assume $f\circ\Phi>0$ on $\overline{p_1p_2}$ and $\overline{p_3p_4}$. 
By harmonicity of $f\circ\Phi$, 
any connected component of the region $\{x\in \overline B_1:f\circ\Phi(x)>0\}$  must contain part of  $\overline{p_1p_2}$ or $\overline{p_3p_4}$, so that we deduce that these connected components are at most two. 

Assume now by contradiction that there are two distinct 
points $\pointP$ and $\pointQ$ of $\Sigma$ such that $\Pi$ 
is tangent to $\Sigma$ at $\pointP$ and $\pointQ$.   
Since $f\circ\Phi$ has null differential at $\Phi^{-1}(\pointP)$ 
and $\Phi^{-1}(\pointQ)$, the set $\{f\circ \Phi=0\}$, 
in a neighbourhood of $\Phi^{-1}(\pointP)$, consists of $2m_p$ analytic 
curves crossing at $\Phi^{-1}(\pointP)$, whereas  in a neighbourhood of 
$\Phi^{-1}(\pointQ)$, it consists of $2m_q$ analytic curves crossing at 
$\Phi^{-1}(\pointQ)$. Therefore, in a neighbourhood of $\Phi^{-1}(\pointP)$, 
the set $\{f\circ \Phi>0\}$ counts at least $2$ open regions 
(and similarly at $\Phi^{-1}(\pointQ)$). Let us call the two of 
these regions $A_i$, $i=1,2$ and $B_i$, $i=1,2$ ($A_i$'s 
around $\Phi^{-1}(\pointP)$ and $B_i$'s around $\Phi^{-1}(\pointQ)$). 
By harmonicity each $A_i$ and $B_i$ must be connected to one of the 
arcs   $\overline{p_1p_2}$ or $\overline{p_3p_4}$. 
Hence some of these regions must belong to the same connected component
of $\{f\circ\Phi>0\}$. Then we are reduced to two following
cases (see Fig. \eqref{fig:twocases}): 
\begin{itemize}
	\item[(Case A)] $A_1$ and $A_2$ belong to the same connected 
component
(say the one containing $\overline{p_1p_2}$). Hence we can construct two disjoint curves in $\{f\circ \Phi>0\}$, both joining $\Phi^{-1}(\mathcal P)$ to a point in   $\overline{p_1p_2}$, emanating from $\Phi^{-1}(\mathcal P)$, one in 
region $A_1$ and one in region $A_2$. This contradicts the 
maximum principle, because these two curves would enclose a region 
where $f\circ\Phi$ takes also negative values, whereas its boundary is in $\{f\circ\Phi>0\}$.
	\item[(Case B)] $A_1$ and $B_1$ are joined 
to $\overline{p_1p_2}$ and $A_2$ and $B_2$ are joined to $\overline{p_3p_4}$. 
In this case we can construct four curves in $\{f\circ\Phi>0\}$:  $\sigma_1$ and $\sigma_2$ emanating from $\Phi^{-1}(\mathcal P)$ in regions $A_1$ and $A_2$ and reaching $\overline{p_1p_2}$ and $\overline{p_3p_4}$, respectively; $\beta_1$ and $\beta_2$ emanating from $\Phi^{-1}(\mathcal Q)$ in regions $B_1$ and $B_2$ and reaching $\overline{p_1p_2}$ and $\overline{p_3p_4}$, respectively.
	The region enclosed between these $4$ curves has boundary contained in $\{f\circ\Phi>0\}$ and necessarily inside it 
the function $f\circ\Phi$ takes also negative values, again in contrast with the 
maximum principle.
\end{itemize} 
From the above discussion our claim follows.

We are now ready to conclude the proof of step 6:
suppose by contradiction  that $\Sigma^+$ is not Cartesian
with respect to $\{w_3=0\}$, and take a point $P^+ 
\in \Sigma^+ \setminus \Gamma$
where the tangent plane $\Pi$ to $\Sigma^+$ at $P^+$ is vertical. By symmetry of $\Sigma$, the point $P^-$, 
defined as the symmetric of $P^+$ with respect to the 
rectangle $R_{2l}$, belongs to $\Sigma^-$, and the tangent plane to $\Sigma^-$ at $P^-$ is the same plane $\Pi$. This contradicts the claim.
We eventually observe that $\widetilde \psi$ is analytic on the subgraph of $\widetilde h$, since its graph is $\Sigma^+$. We conclude that $\widetilde \psi$ belongs to $W^{1,1}(SG_{\widetilde h})$, since also 
its total variation is bounded by the area of its graph, which is finite.

\nada{
Setting $\widetilde \dts:=\partial \mathbb 
S_{cl}(E)\setminus (\partial K\times \R)$, we see that
$\widetilde \dts$ is a disc-type 
surface whose boundary is $\Gamma$ (here we use 
that $\Gamma$ is symmetric and that $\Gamma^+$ and $\Gamma^-$ are Cartesian, 
with respect to $\{w_3=0\}$). Indeed, $\widetilde \Sigma$ is 
a Cartesian surface, graph of some function defined on the subgraph of 
$\widetilde h$, which has the topology of the disc. 
	
We deduce
\begin{equation}\label{strict_ineq}
\mathcal H^2(\widetilde \dts)<\mathcal H^2(\dts),
\end{equation}
 contradicting the minimality of $\dts$.  
This proves our claim, and
therefore there exists 
$$\widetilde \psi\in BV(U,[0,+\infty))$$
such that its 
graph over $\overline{U}$ coincides with $\dtsp$.
Hence we can write
\begin{align}
 \mathcal H^2(\dtsp)= \areaonecod(\widetilde \psi, \overline U),
\end{align}
and from \eqref{ineq_withF} we conclude
\begin{align}\label{ineq_withF2}
 \areaonecod(\widetilde \psi, \overline U)+
\areaonecod(\psi, R\setminus K)\leq \FB(h,\psi).
\end{align}
}

\begin{figure}
	\begin{center}
		\includegraphics[width=0.7\textwidth]{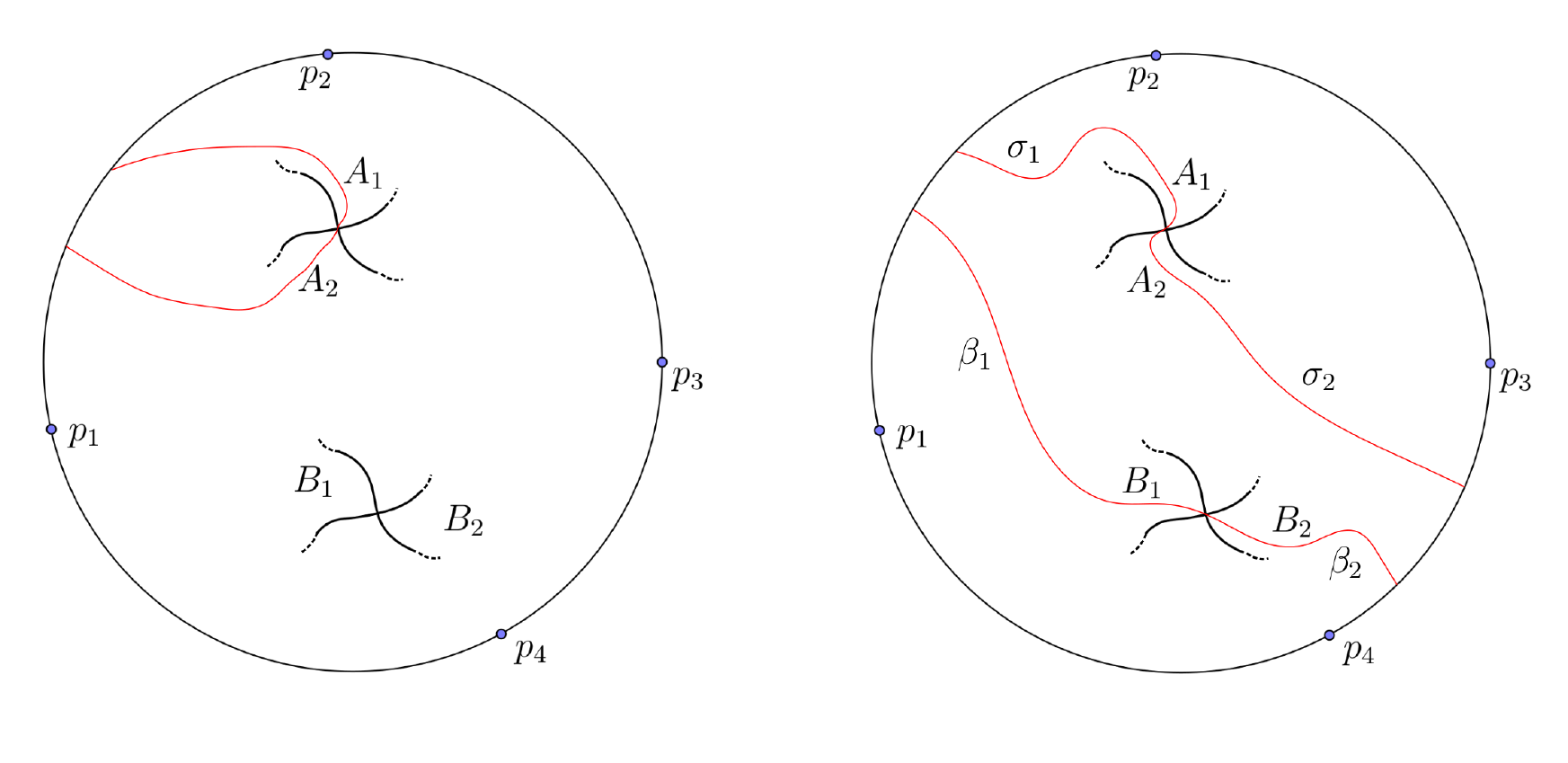}
		\caption{On the left it is represented 
case A in step 6 of the proof of Theorem \ref{prop_zero}. The point $\Phi^{-1}(\mathcal P)$ is in the cross where the two emphasized paths start from. These curves stand in the region $\{f\circ\Phi>0\}$ and join $\Phi^{-1}(\mathcal P)$ with the boundary arc $\overline{p_1p_2}$. The picture on the right represents instead case B. The two cross points are $\Phi^{-1}(\mathcal P)$ and $\Phi^{-1}(\mathcal Q)$ and the paths $\sigma_1$, $\sigma_2$, $\beta_1$, $\beta_2$ are depicted.
		}
		\label{fig:twocases}
	\end{center}
\end{figure}

\medskip

\textit{Step 7: the pair 
$(\widetilde h, \widetilde \psi)$ is an admissible competitor for $\FB$.}
To see this, we recall that in step 5 we 
proved that $\widetilde h$ is convex and 
$\widetilde h(\cdot)=\widetilde h(2l-\cdot)$, \textit{i.e.}  $\widetilde h\in \mathcal H_{2l}$.
Furthermore $\Sigma^+$ is the graph of $\widetilde \psi$, and its projection on the plane $\{w_3=0\}$ is the subgraph of $\widetilde h$. It follows that the area of the graph of $\widetilde \psi$ is exactly the area of $\Sigma^+$ upon $SG_{\widetilde h}$. Let us also recall the 
$W^{1,1}$ regularity of $\widetilde \psi$ proved in 
step 6.
Setting $\widetilde \psi:=\psi$ in $R\setminus K$ we infer the 
admissibility of $(\widetilde h,\widetilde \psi)$.


\medskip
\textit{Step 8: Conclusion.} From \eqref{ineq_withF} we deduce
\begin{align}
 \FB(h,\psi)\geq \mathcal H^2(\dtsp)+ \areaonecod(\psi, R\setminus K)= \FB(\widetilde h,\widetilde \psi),
\end{align}
where the last equality follows from the fact that $\widetilde \psi$ 
is continuous on $\partial_DR_{2l}$. Hence, also $(\widetilde h,\widetilde \psi)$ is a minimizer for $\FB$.
We
 now show that $\widetilde \psi$ is continuous and equals $0$ on $G_{\widetilde h}$. Indeed $\dts=\dtsp\cup \dtsm$ is 
analytic, hence also $\widetilde h$ is smooth (and convex). 
Moreover we know that $\widetilde \psi$ is smooth in 
$SG_{\widetilde h}$. If its  trace $\widetilde \psi^+$ 
on $G_{\widetilde h}$ is strictly positive somewhere, 
we infer that the vertical subset of $\dtsp$ defined as
$$\{(w_1,w_2,w_3):(w_1,w_2)
\in G_{\widetilde h},\;w_3\in (0,\widetilde \psi^+(w_1,w_2))\},$$
has positive $\mathcal H^2$-measure and 
cannot have zero mean curvature (the only case in which 
its mean curvature vanishes is when $\widetilde h$ is linear, 
but in this case $\dtsp$ must be contained in a plane 
containing $G_{\widetilde h}$ which is impossible,
since $\Gamma^+$ is not).
We conclude $\widetilde \psi^+=0$ on $G_{\widetilde h}$.
Finally, $L_{\widetilde h} = \emptyset$ for, if 
not, the vertical part of $\dtsp$ obtained on $L_{\widetilde h}$ 
is flat and then, by analyticity, also $\dtsp$ is, a contradiction.
The thesis of  the theorem, and hence of 
Theorem \ref{Thm:existenceofminimizer} (iii), is achieved.
\qed 

\medskip

To conclude the proof of 
Theorem \ref{Thm:existenceofminimizer},
it remains to show (iv). 
The pair $(h \equiv 1, \widehat \varphi)$ , where 
the function $\widehat \varphi$ is as in \eqref{eq:widehat_varphi},
is one of the competitors for
problem \eqref{eqn:B} (notice that $\widehat \varphi$ attains the boundary condition); in addition,
its subgraph is strictly convex 
(see Fig. \ref{fig:graphofphi}), hence\footnote{As observed, 
the minimal surface $\Sigma^+$ is the graph of $\psi^\star=\widetilde \psi$, and it must be contained in the convex envelope of $\Gamma$, \textit{i.e.}, inside the subgraph of $\widehat\varphi$.}  necessarily
$\psi^\star\leq \widehat \varphi$ in $\overline{R_{2l}}$ (where we have taken $\psi^\star=\widetilde \psi$, the solution given by Theorem \ref{prop_zero}).

Eventually, the strict inequality in \eqref{eq:psi_star_below_cylinder}
is a consequence of the strong maximum principle: indeed, the internal points of a minimal surface are always strictly inside the convex hull of its boundary, with the only exception in the case of part of a plane 
(see \cite[pag 63, section 70]{Nitsche:89}); so that internal points of $\Sigma^+$ are strictly inside the  the graph $G_{\widehat \varphi}$ of $\widehat \varphi$ (that is  half of the lateral boundary of a cylinder).
\end{proof}

We conclude this section by observing a consequence of 
Theorem \ref{prop_zero}: 
Let $G_w$ be the graph in $R_{2l}$ of a 
function $w\in C([0,2l],(-1,1])$ such that $w(0)=w(2l)=1$,
 and consider the curve $\Gamma_w$ obtained by concatenation 
of $G_w$ with the graph of $\varphi$ over $\partial_DR_{2l}$. 
\begin{cor}\label{main_cor}
We have
\begin{align}\label{equiv_plateau}
\mathcal F_{2l}(h,\psi)=\inf\mathcal P_{\Gamma_w}(X_{{\rm min}}),
\end{align}
where $(h,\psi)\in X_{2l}^{\rm{conv}}$ 
is a minimizer of $\mathcal F_{2l}$, $X_{{\rm min}}$ 
is a parametrization of a disc-type 
area-mininizing solution of 
the Plateau problem spanning  $\Gamma_w$ (see \eqref{plateau}), 
and the infimum is computed over all functions $w$ as above.
\end{cor}
The proof of this corollary
 can be achieved by adapting the proof of Theorem \ref{prop_zero}, 
which shows that the solution to the 
Plateau problem in \eqref{equiv_plateau} is Cartesian and the optimal $w$ is convex.

\nada{
 and, 
precisely it is a subset of the boundary of convex hull of the graph 
of $\psi^\star$\draftRRR{I think the convex hull is exactly half cylinder, whereas the graph of $\widehat \varphi$ is its lateral boundary}. As a consequence, if $\psi^\star(w_1,w_2)
=\widehat 
\varphi(w_1,w_2)$ for some $(w_1,w_2)\in R_{2l}$ 
it follows that the mean curvature of $G_{\psi^\star}$ cannot be zero at the point $(w_1,w_2,\psi^\star(w_1,w_2))$, because $G_{\psi^\star}$ is tangent to $G_{\widehat\varphi}$ (which has strictly positive curvature) and it holds $\psi^\star\leq 
\widehat
\varphi$. \draftRRR{Probably we can directly appeal to a result which states that the internal points of a minimal surface are always strictly inside the convex hull of its boundary, with the only exception in the case of part of a plane. 
(see \cite[pag 63, section 70]{Nitsche:89})}\draftAAA{I agree}}

\section{Upper bound}\label{sec:upper_bound}
A minimizer $(h^\star, \psi^\star)$ of \eqref{eqn:B}
needs to be used 
for constructing a recovery sequence $(u_k) \subset {\rm Lip}(\Om, \R^2)$,
see formulas \eqref{vepsonCeps} and \eqref{vepsonCeps-}:
we know that 
$\psi^\star$ 
is locally Lipschitz, but not Lipschitz, in $R_{2l}$; therefore we need first a regularization 
procedure. 

Let $(h^\star,\psi^\star)$ be a minimizer provided by 
Theorem \ref{Thm:existenceofminimizer}, and assume that $h^\star$ is not identically $-1$. 
We fix an integer $m>0$
and, recalling the definition of $\widehat \varphi$ in 
\eqref{eq:widehat_varphi}, define
\begin{align}\label{def_varphi_m}
\varphi_m:=\Big(\widehat 
\varphi-\frac2m\Big)\vee0 
\qquad {\rm in}~ \overline R_{2l}.
\end{align}
We observe that  $\varphi_m$ is Lipschitz continuous in $\overline R_{2l}$.
We then set
\begin{align}\label{def_psi_m}
\psi^\star_m:=\Big(\big(\psi^\star-\frac1m)\vee0\big)\Big)\wedge\varphi_m
\qquad {\rm in}~ R_{2l}.
\end{align}
Since $\psi^\star$ is locally Lipschitz in $R_{2l}$, an easy check
shows that $\psi^\star_m$ is Lipschitz continuous in $R_{2l}$ for
any $m$ (with an unbounded Lipschitz constant as $m\rightarrow+\infty$).
 This follows from the fact that 
$\psi^\star$ is continuous up to the boundary of $R_{2l}$
(see Theorem  \ref{Thm:existenceofminimizer} (iii)) 
and  hence 
$\psi^\star_m$ coincides with 
either $0$ or $\varphi_m$ in a 
neighbourhood of $(\partial_D R_{2l})\cup G_{h^\star}$ in $R_{2\longR}$. Furthermore still $\psi^\star_m=0$ on the upper graph $\overline{R_{2l}}\setminus SG_{h^\star}=\{(w_1,w_2)\in \overline{R_{2l}}:w_2\geq h^\star(w_1)\}$ of $h^\star$.

\begin{lemma}[\textbf{Properties of $\psi_m^\star$}]
\label{lem:properties_of_psi_m}
	Let $(h^\star,\psi^\star)$ be a minimizer of $\FB$ as in 
Theorem \ref{Thm:existenceofminimizer} and assume $h^\star$ is not identically $-1$. 
For all $m>0$ let $\psi^\star_m$ be defined as in \eqref{def_psi_m}.
	Then: 
	\begin{itemize}
		\item[(i)] $\psi^\star_m$ is 
Lipschitz continuous in $\overline{SG_{h^\star}}$,  $\psi^\star_m=0$ 
on $([0,2l]\times\{-1\}) \cup (\overline{R_{2l}}\setminus SG_{h^\star})$, 
and $\psi^\star_m(0,\cdot)=\varphi_m(0,\cdot)$, so that  $|\partial_{w_2}\psi^\star_m(0,\cdot)|\leq|\partial_{w_2}\varphi(0,\cdot)|=|\partial_{w_2}\psi^\star(0,\cdot)|$ a.e. in $[-1,1]$; \
		
		\item[(ii)] $(\psi^\star_m)$ 
converges to $\psi^\star$ uniformly on $\{0,2l\}\times [-1,1]$ as $m \to +\infty$;
		
		\item[(iii)] we have 
\begin{align}\label{eq:lim_m}
		\lim_{m\rightarrow +\infty} 
\areaonecod(\psi^\star_m, SG_{h^\star})=
		\areaonecod(\psi^\star, SG_{h^\star}).
		\end{align}
\end{itemize}
	As a consequence $\FB(h^\star,\psi^\star_m)\rightarrow \FB(h^\star,\psi^\star)$ as $m\rightarrow +\infty$.
\end{lemma}
\begin{proof} (i) and (ii) are direct consequences of the definitions. 
To show (iii) 
	we start to observe that 
$\psi^\star_m\rightarrow \psi^\star$ 
pointwise in $R_{2l}$: indeed, 
this follows from the definitions 
of $\varphi^\star_m$ and $\psi^\star_m$ up to noticing that 
$\varphi_m\rightarrow \widehat \varphi$ pointwise in $R_{2l}$ 
as $m\rightarrow +\infty$, 
and $\psi^\star\leq \widehat 
\varphi$ on $R_{2l}$. 
From Theorem \ref{Thm:existenceofminimizer} (iv) it follows that, 
at any point $(w_1,w_2)\in R_{2l}$, 
for $m$ large enough 
$\varphi_m(w_1,w_2)>\psi^\star(w_1,w_2)$ 
(since $\widehat \varphi(w_1,w_2)>\psi^\star(w_1,w_2)$), 
so that 
$\psi^\star_m(w_1,w_2)=\psi^\star(w_1,w_2)-\frac1m$. 
	As a consequence the set $A_m:=\{0<\psi^\star-\frac1m<
	\varphi_m\}$ satisfies 
$$
\lim_{m \to +\infty} \mathcal H^2(SG_{h^\star}\setminus A_m)=0,
$$
and on $A_m$ it holds
 $\psi^\star_m=\psi^\star-\frac1m$ and  $\nabla \psi^\star_m=\nabla \psi^\star$. Moreover,
 on $SG_{h^\star}\setminus A_m$, 
either $\psi^\star_m=0$ (and hence $\nabla \psi_m^\star=0$) or $\psi^\star_m=\varphi_m$ (and hence $\nabla \psi^\star_m=\nabla 
\varphi_m$). Therefore
$$
\int_{SG_{h^\star}\setminus A_m}\sqrt{1+|\nabla\psi^\star_m|^2}~dx
\leq \int_{SG_{h^\star}\setminus A_m}\sqrt{1+|\nabla\varphi_m|^2}~dx
$$
and 
$$
\lim_{m \to +\infty}
\int_{SG_{h^\star}\setminus A_m}\sqrt{1+|\nabla\psi^\star_m|^2}~dx
\leq \lim_{m \to +\infty}
\int_{SG_{h^\star}\setminus A_m}\sqrt{1+|\nabla\varphi_m|^2}~dx
= 0,$$
because $\vert \nabla \varphi_m\vert$ are 
uniformly bounded in $L^1(R_{2l})$.
Also

	\begin{align*}
	\areaonecod(\psi^\star_m, SG_{h^\star})=\int_{A_m}\sqrt{1+|\nabla\psi^\star|^2}~dx+\int_{SG_{h^\star}\setminus A_m}\sqrt{1+|\nabla\psi^\star_m|^2}~dx,
	\end{align*}
and \eqref{eq:lim_m} follows.
\end{proof}

The main result of this section reads
as follows.

\begin{theorem}[\textbf{Upper bound for the area of the vortex map}]
\label{Thm:maintheorem}
The relaxed area of the graph of the vortex map $\vmap$ satisfies 
\begin{equation}
\label{eq:upper_bound_recovery}
\relarea(\vmap,\BallR)\leq 
\int_{\BallR}\vert \M(\nabla \vmap)\vert ~dx  + 2
\min \big \{ \FBl(\h,\psione):  
(\h,\psione) \in  \Wspace \big \}.
\end{equation}
\end{theorem}

\begin{proof}[Proof of theorem \ref{Thm:maintheorem}]
To prove the theorem, we need to construct a 
sequence $(\veps) \subset \textrm{Lip}(\BallR, \R^2)$ 
converging to $\vmap$ in $L^1(\BallR,\R^2)$ such that
\begin{equation*}
\lim_{k \to +\infty}\area(\veps, \BallR)\leq
\int_{\BallR}\vert \M(\nabla \vmap)\vert dx + \FB(\hstar,\psi^\star),
\end{equation*}
where $(\hstar, \psi^\star)$ is a pair
minimizing $\FB$ as in Theorem \ref{Thm:existenceofminimizer}. 
We can assume that $h^\star$ is not identically
$-1$, otherwise the result follows from \cite{AcDa:94}.

We will specify various subsets of $\BallR$ and define the sequence 
$(\veps)$ on each of these sets (see Fig. \ref{fig:picture_alaa_last_section}). More precisely, we will define $u_k$ as a map into $\mathbb S^1$ in the largest sector (step 1). This construction is similar 
to the one in \cite{AcDa:94} (see also Remark \ref{rem:below} below). The contribution of the area in this sector will equal, as $k\rightarrow\infty$, the first term in \eqref{eq:upper_bound_recovery}. The second term 
will be instead provided by the contribution of $u_k$ in region $C_k\setminus B_{r_k}$ (step 2), where we will need the aid of 
the functions $(\hstar, \psi^\star)$  (suitably regularized, in order to render $u_k$ Lipschitz continuous). The other regions surrounding $C_k\setminus B_{r_k}$ are needed to glue $u_k$ between the aforementioned regions. This is done in steps 3, 4 and 5, 
where it is also proven that the corresponding 
area contribution is negligible. Finally, in steps 6 and 7 
we show the crucial estimates to prove \eqref{eq:upper_bound_recovery}. 
In Fig. \ref{fig:picture_alaa_last_section} 
this subdivion of the domain $\Omega$ is drawn.

\begin{Remark}\label{rem:below}\rm
Our construction 
differs from the one in \cite{AcDa:94}, even when in place of 
$(\hstar,\psi^\star)$ we use $(1, \sqrt{1-s^2})$ 
(\textit{i.e.},
the one in Section \ref{subsec:an_approximating_sequence_of_maps_with_degree_zero:cylinder})
in the following sense.
We use the full 
graph of $\pm\psi^\star$ to construct 
$u_k$ (and therefore, in the case  
when $(\hstar,\psi^\star)$ is replaced by $(1, \sqrt{1-s^2})$, 
the image of $u_k$ covers the whole cylinder and not
only a part of it). Since $h^\star$ {\it may be
not identically $1$} (and actually is not explicit in general), 
the presence of a new set $T_k$
is now needed, as an intermediate region to glue the 
trace of $u_k$ along two segments $\{\theta=  \pm \overline\theta_k\}$.
The image set $u_k(T_k)$ covers a small
part of the unit circle.
See Fig. \ref{fig:picture_alaa_last_section}.
\end{Remark}
Let $k \in \mathbb N$ and let $(\reps),(\thetaeps), (\thetaepsbar)$ 
be infinitesimal sequences of positive numbers
such that $\thetaepsbar-\thetaeps=:\deltaeps>0$. 
We shall suppose\footnote{This assumption will be used only in step 7.} 
\begin{equation}\label{eq:k_theta_k_to_zero}
\lim_{k \to +\infty}(\theta_k k)=0.
\end{equation}
Let  $\Balleps$ be the 
open disc centered at the origin with radius $\reps$, and 
\begin{equation}\label{eq:C_k}
\Coneeps:=\{(r,\theta)\in 
[0,\longR) \times [0,2\pi):
\theta \in [0,\thetaeps]\cup[2\pi-\theta_k,2\pi)\},
\end{equation}
be the half-cone in $\Omega$, 
with vertex at the origin and aperture equal to $2\thetaeps$, see
Fig. \ref{fig:picture_alaa_last_section}. Let  
\begin{equation}\label{eq:T_k}
\Teps:=\{(r,\theta)\in [0,\longR) \times [0,2\pi):\theta \in [\thetaeps,\thetaepsbar]\cup[2\pi-\thetaepsbar,2\pi-\thetaeps]\}.
\end{equation} 
We set 
$$
C_k^+:=C_k\cap\{\theta\in[0,\theta_k]\}, \quad
C_k^-:=C_k\cap\{\theta\in[2\pi-\theta_k,2\pi]\},
$$
and 
divide $\Coneeps\cap (\BallR\setminus\Balleps)$ into two sets
\begin{equation}\label{eq:C_k_plus_minus}
\begin{aligned}
C_k\setminus \Balleps:=
\big(C_k^+ \setminus\Balleps\big)\cup\big(
C_k^- \setminus\Balleps\big).
\end{aligned}
\end{equation}
\begin{figure}
\begin{center}
    \includegraphics[width=0.9\textwidth]{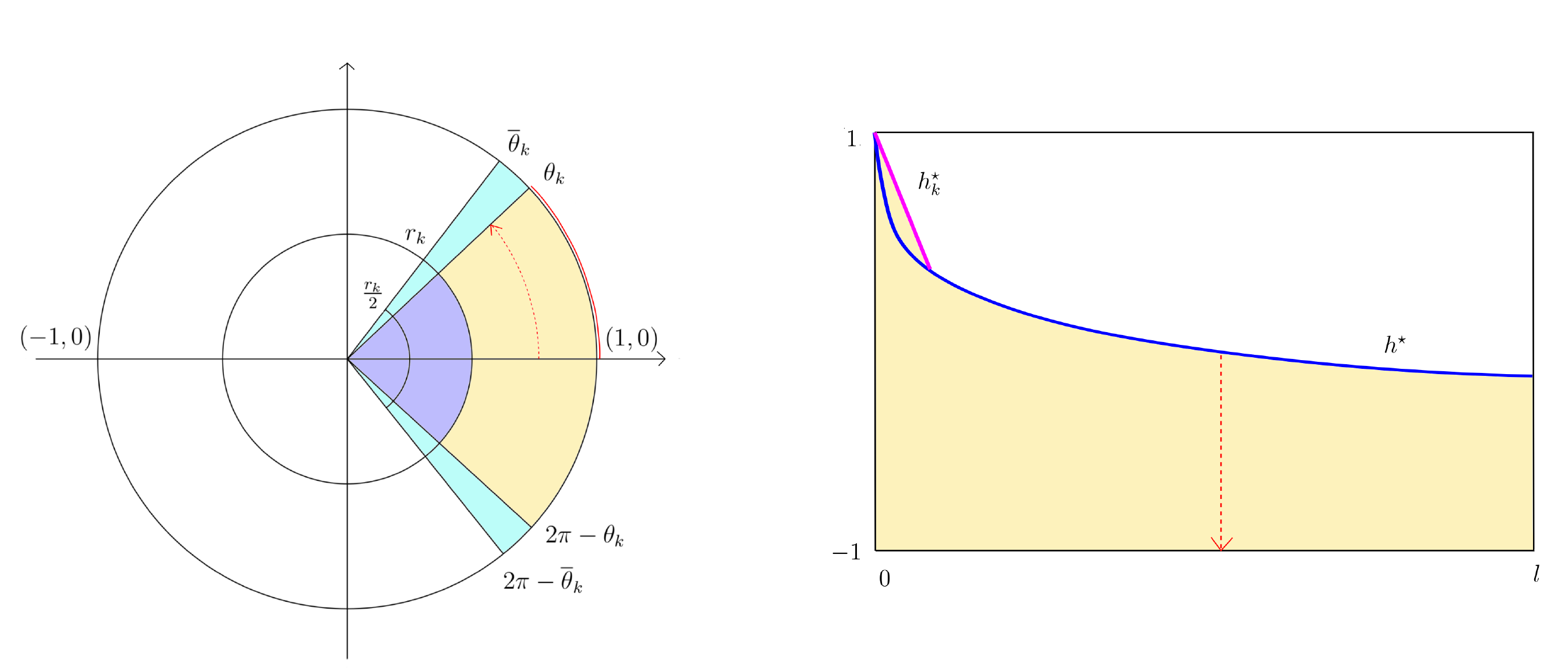}
\caption{On the left the subdivision of $\sourcedisk_l$ in sectors. 
Specifically, the sectors 
$C_k^+\setminus \Balleps$ 
and $C_k^-\setminus \Balleps$ are emphasized in 
light grey. The map $\mathcal T_k$ defined
	in \eqref{eqn:changeofvariableT} sends $C_k^+\setminus \Balleps$ 
 in the (reflected) subgraph of $h^\star_k$ in $R_\longR$, depicted on the right. This parametrization maps the segment joining $(r_k,0)$ to $(1,0)$ onto the graph of $h^\star_k$, and the radius corresponding to $\theta=\theta_k$ to the basis of $R_l$, following the orientation emphasized by the dashed arrow. The graph of $h^\star_k$ starts 
linearly from the point $(0,1)$ with negative derivative,
then joins
 and next coincides 
 with the graph of $h^\star$.  The definition of  $u_k$ in $C_k^+\setminus \Balleps$
 makes use of this parametrization of $SG_{h^\star_k}$ (see \eqref{vepsonCeps}). This parametrization needs a reflection,
in order to glue $u_k$ on the 
horizontal segment $\{\theta=0\}$ with the definition of $u_k$ 
in $C_k^-\setminus \Balleps$. Note that 
$\Psi_{k}(C_k^+\setminus \Balleps) \subset \R^3$,
 with $\Psi_k$ defined in \eqref{eq:Psi_k}, 
 is the graph of $\psi_k^*(w_1,-w_2)$ 
above the 
reflection of the two-dimensional region in between the graph of 
$h_k^*$ and the horizontal segment $\{w_2=-1\}$.}
\label{fig:picture_alaa_last_section}
\end{center}
\end{figure}
\smallskip
{\it Step 1.} 
Definition of $\veps$ in $\overline{\BallR \setminus ( \Coneeps \cup \Teps)}$.

In this step our 
construction is similar to the one in \cite[Lem. 5.3]{AcDa:94},
see also \eqref{eqn:u_k_Cylinder}; in order to define $u_k$, 
in the source we use polar  
coordinates $(r,\theta)$ and in the target Cartesian coordinates.
 Define
\begin{equation}\label{eqn:vepsoutofthecone}
\veps(r,\theta):=
\begin{cases}
      \vmap(r,\theta) = (\cos \theta, \sin \theta), & 
(r,\theta)
\in (\overline{\BallR \setminus (\Coneeps\cup \Teps)})\setminus 
\sourcedisk_{r_k/2},
\\
     \Big(\cos (\frac{2r}{\reps}(\theta -\pi)+\pi),
\sin (\frac{2r}{\reps}(\theta -\pi)+\pi)\Big), & 
(r,\theta) \in
\overline{\sourcedisk_{r_k/2} \setminus (\Coneeps\cup \Teps)}.
    \end{cases} 
\end{equation}
Observe that 
\begin{align}
&\veps(0,0)=(-1,0) = u_k(r,\pi), \qquad r \in [0,\longR);
\nonumber
\\
&\veps(r,\thetaepsbar)=(\cos \thetaepsbar, \sin  \thetaepsbar), \qquad 
\veps(r,2\pi-\thetaepsbar)=(\cos \thetaepsbar, \sin  (-\thetaepsbar)),
\qquad r \in (\reps/2,l);
\nonumber
\\
&\veps(r,\thetaepsbar)=\Big(
\cos(\frac{2r}{\reps}(\thetaepsbar -\pi)+\pi)~,~ 
\sin (\frac{2r}{\reps}(\thetaepsbar -\pi)+\pi)\Big),
\qquad r \in [0,\reps/2],
\nonumber
\\
&\veps(r,2\pi-\thetaepsbar)=\Big(
\cos(\frac{2r}{\reps}(\pi-\thetaepsbar )+\pi)~,~ 
\sin (\frac{2r}{\reps}(\pi-\thetaepsbar)+\pi)\Big),\qquad 
r \in [0,\reps/2].
\label{eqn:vepsboundaryinBeps}
\end{align}

\smallskip
The relevant contribution to the area of the graph of $u_k$ is the one
in region $C_k$, and more specifically in $C_k \setminus B_{r_k}$;
it is in this region that we need to use a minimizing pair of $\FB$.

{\it Step 2.} 
Definition of $\veps$ on $\Coneeps\setminus \Balleps$.

We first need a regularization of $h^\star$:  assuming without 
loss of generality $1/k < \longR$,
we define 
\begin{equation}\label{def_h_k}
h^\star_k (w_1):=
\begin{cases}
h^\star(w_1) & {\rm for~  }w_1\in [\frac1k,\longR],
\\
k\left(h^\star (\frac1k)- \hstar(0)\right)w_1+\hstar(0)& {\rm for~  }w_1\in [0,\frac1k),
\end{cases}
\end{equation}
where we recall that $h^\star(0)=1$ (see Theorem \ref{Thm:existenceofminimizer}), 
and we set $h^\star_k(w_1):=h^\star_k(2\longR -w_1)$ for $w_1\in 
[\longR,2\longR]$ (see Fig. \ref{fig:picture_alaa_last_section}, right). Notice that $\hstark(0)=1$, 
$\h^\star_k \in \Lip([0,2\longR])$ and
the convexity of $h^\star$ implies that also $h^\star_k$ is convex,
$h^\star_k\geq h^\star$, and therefore by Lemma \ref{lem:properties_of_psi_m} (i) we see that
$(h^\star_k, \psi^\star_k) 
\in X_{2\longR}^{{\rm conv}}$,
where 
$\psi^\star_k$ is the approximation of $\psi^\star$ in Lemma 
\ref{lem:properties_of_psi_m} (with $k=m$), see formula
\eqref{def_psi_m}.
Again by Lemma
\ref{lem:properties_of_psi_m}, 
$\FB(h^\star_k,\psi^\star_k)=
\FB(h^\star,\psi^\star_k)+ \int_0^{2\longR} 
\left( h^\star_k(w_1) -\h^\star(w_1) \right)~dw_1\rightarrow \FB(h^\star,\psi^\star)$ as $k\rightarrow +\infty$. 

We start with the construction of $u_k$ on $C_k^+\setminus \Balleps$. Set
\begin{align}
&\teps:[\reps,\longR]\to[0,\longR],&\qquad &\teps(r):=\frac{\longR}{\longR-\reps}(r-\reps), \label{eqn:teps}\\
&\seps:[\reps,\longR]\times[0,\thetaeps]\to [-1,1],&\qquad &\seps(r,\theta):=\frac{1 + \hstark(\teps(r))}{\thetaeps}\theta -\hstark(\teps(r)).\label{eqn:seps}
\end{align}
Note that $\seps(r,\cdot):[0,\thetaeps]\to [-\hstark(\teps(r)),1]$ is a 
bijective increasing function, for any $r \in [r_k, \longR]$. 
Thus 
\begin{align}
&\seps(r,0)=-\hstark(\teps(r)) \quad{\rm for~any}~ 
r \in [r_k,\longR],  \text{ in particular }\seps(\reps,0)=-1,
\label{eqn:sepsboundarythetazero}
\\
&\seps(r,\thetaeps)=1, \qquad r\in[\reps,l]
\label{eqn:sepsboundarythetaeps},
\\
& \seps(r_k,\theta)=\frac{2\theta}{\theta_k} -1, \qquad \theta\in[0,\theta_k].
\label{eq:s_k_r_k_theta}
\end{align}
We have, for all $r\in[\reps,l]$ and $\theta\in[0,\theta_k]$,
\begin{align}
&
\teps'(r)
=\frac{\longR}{\longR-\reps},\label{eqn:tepsderiv}\\
&\partial_\theta \seps(r,\theta)=\frac{1 +\hstark(\teps(r))}{\thetaeps},
\label{eqn:sepsthetaderiv}
\end{align}
and, for almost every $r\in [\reps,l]$ and
all $\theta\in[0,\theta_k]$, 
\begin{align}
\partial_r \seps(r,\theta) =\left(\frac{\theta}{\thetaeps}-1\right)
\teps'(r)  \hstark'(\teps(r)) = 
\frac{\longR}{\longR-\reps}
\left(\frac{\theta}{\thetaeps}-1\right)  
\hstark'(\teps(r)).
\label{eqn:sepsrderiv}
\end{align}
Moreover  we define
\begin{equation}\label{eqn:rfn}
\invr : [0,\longR]\to[\reps,\longR], \qquad
\invr(\tcoord) :=\frac{\longR-\reps}{\longR}\axialcoordofcylinder+\reps ~
\end{equation}
to be the inverse of $\tau_k$ and, recalling that 
$\overline R_l = [0,\longR] \times [-1,1]$, 
\begin{equation}
\label{eqn:thetafn} 
\invtheta
:\Omegahstark \cap \overline R_l  \to [0,\thetaeps],
\qquad
\invtheta(\axialcoordofcylinder,\scoord) :=\frac{\thetaeps}{1 + \hstark (\axialcoordofcylinder)}
(\hstark(\axialcoordofcylinder)-\scoord).
\end{equation}
Notice that 
$\invtheta(\axialcoordofcylinder,\cdot):
[-1,\hstark(\axialcoordofcylinder)]
\to [0,\thetaeps]$ is a linearly decreasing bijective function. 
	
The map 
\begin{equation}\label{eqn:changeofvariableT}
\mathcal T_k:C_k^+\setminus \Balleps
\to 
\Omegahstark
\cap 
\overline R_l, 
\qquad \mathcal T_k(r,\theta):=(\teps(r),-\seps(r,\theta)),
\end{equation} 
is invertible, and its inverse is the map
\begin{equation}\label{eqn:changeofvariableTinverse}
\mathcal T_k^{-1}:
\Omegahstark\cap 
\overline R_l
\to C_k^+\setminus \Balleps, \qquad 
\mathcal T_k^{-1}(\axialcoordofcylinder,\scoord)
:=(\invr(\axialcoordofcylinder),\invtheta(\axialcoordofcylinder,\scoord)).
\end{equation} 
The modulus of the determinant of the Jacobian of $\mathcal T_k^{-1}$ is given by 
\begin{equation}\label{eqn:changeofvariabledeterminant}
\vert J_{\mathcal T_k^{-1}}\vert=\left( 
\frac{\longR -\reps}{\longR}\right)
\frac{\thetaeps}{1 +\hstark(\axialcoordofcylinder)}.
\end{equation}
We set 
\begin{equation}\label{vepsonCeps}
\veps(r,\theta)
:=\Big(\seps(r,\theta),\psionestar_k
\big(\mathcal T_k(r,\theta)\big)\Big) =
\Big(u_{k1}(r,\theta),
u_{k2}(r,\theta)\Big), \qquad 
r \in [r_k, l], \theta \in [0,\theta_k].
\end{equation}
Observe that, using the definition of $\psi_k^\star$,
\begin{equation}\label{eq_13.28} 
\begin{aligned}
&\veps \in {\rm Lip}(C_k^+\setminus \Balleps, \R^2),
\\
&\veps(r,\thetaeps)=(\seps(r,\thetaeps), 
\psi_k^\star(\mathcal T_k(r, \theta_k)))=(1,0),
\\
&\veps(r,0)=(-\hstark(\teps(r)), \psionestar_k(\teps(r),\hstark(\teps(r))))
=(-\hstark(\teps(r)),0),
\\
&\veps(\reps,\theta)
=(\seps(\reps,\theta),\psionestar_k(0,-\seps(\reps,\theta)))
=(\seps(\reps,\theta),\varphi_k(0,-\seps(\reps,\theta))),
\end{aligned}
\end{equation}
for $r \in [r_k, l]$ and $\theta \in [0, \theta_k]$, 
as it follows from
\eqref{eqn:teps}, 
\eqref{eqn:sepsboundarythetazero},  \eqref{eqn:sepsboundarythetaeps}, and \eqref{eqn:psionestarboundary},
where $\varphi_k$ is defined in \eqref{def_varphi_m} (with $k=m$).

Eventually we define $u_k$ on $C_k^-\setminus \Balleps$ as
\begin{align}\label{vepsonCeps-}
\veps(r,\theta)
:=
(u_{k1}(r,2\pi-\theta),
-u_{k2}(r,2\pi-\theta)),
\qquad
r \in [r_k, l], \theta \in [2\pi-\theta_k, 2\pi). 
\end{align}
It turns out
\begin{align*}
&\veps  \in {\rm Lip}(C_k^-\setminus \Balleps, \R^2),
\\
&\veps(r,2\pi-\thetaeps)=(1,0),
\\
&\veps(r,2\pi)=(-\hstark(\teps(r)), -\psionestar_k(\teps(r),\hstark(\teps(r))))
=(-\hstark(\teps(r)),0),
\\
&\veps(\reps,\theta)=(\seps(\reps,2\pi-\theta), -\psionestar_k(0,-\seps(\reps,2\pi-\theta))),
\end{align*}
for $r\in[r_k,l]$, $\theta\in[2\pi-\theta_k,2\pi)$.

The area of the graph of $u_k$ on $C_k \setminus B_{r_k/2}$
will be computed in step 7.
\smallskip

{\it Step 3.} 
Definition of $\veps$ on $\Coneeps\cap(\overline B_{r_k}\setminus {\rm B}_{r_k/2})$
and its area contribution.  

Let $G_{\psi^\star_k(0,\cdot)}\subset\R^2$ (resp.
$G_{\psi^\star(0,\cdot)}\subset\R^2$)
denote the graph of $\psi^\star_k(0,\cdot)$ (resp. 
of $\psi^\star(0,\cdot)$)
on $[-1,1]$.
We 
introduce the retraction map
$\Upsilon:(\R\times [0,+\infty))\setminus O
\subset \R^2_{{\rm target}}
\rightarrow  G_{\psi^\star(0,\cdot)}\subset \R^2_{{\rm target}}$, $O=(0,0)$, defined by
\begin{align*}
\Upsilon(p)=q:=G_{\psi^\star(0,\cdot)}\cap \ell_{Op}\qquad \forall p\in (\R\times [0,+\infty))\setminus O,
\end{align*}
where $\ell_{Op}$ is the line passing through 
$O$ 
 and $p$. Then  $\Upsilon$ is well-defined and it is 
Lipschitz continuous in a neighbourhood of $G_{\psi^\star(0,\cdot)}$ in $\R\times [0,+\infty)$. 
We also define 
$$
\Upsilon_k:G_{\psi^\star_k(0,\cdot)}\rightarrow G_{\psi^\star(0,\cdot)}
$$
as the 
restriction of $\Upsilon$ to $G_{\psi^\star_k(0,\cdot)}$;
see Fig. \ref{fig:graphof_h}.
 As a consequence, since for $k \in \NN$ large enough $G_{\psi_k^\star(0,\cdot)}$ 
is contained in a neighbourhood of $G_{\psi^\star(0,\cdot)}$, 
we have that $\Upsilon_k$ is Lipschitz continuous with Lipschitz constant 
independent of $k$.
Notice also that $\Upsilon_k((-1,0))=(-1,0)$ and $\Upsilon_k((1,0))=(1,0)$. 

We define $u_k$ on $C_k^+\cap (\overline {\rm B}_{r_k}\setminus {\rm B}_{r_k/2})$ setting, 
for $r\in[\frac{r_k}{2},r_k]$ and $\theta\in[0,\theta_k]$,
\begin{align*}
u_k(r,\theta):=\Big(2-\frac{2r}{r_k}\Big)
\Upsilon_k\big(s_k(r_k,\theta),\psi_k^\star(0,-s_k(r_k,\theta))\big)+
\Big(\frac{2r}{r_k}-1\Big)\big(s_k(r_k,\theta),\psi_k^\star(0,-s_k(r_k,\theta))\big).
\end{align*}
We have
$$
u_k(r_k,\theta)=(s_k(r_k,\theta),\psi_k^\star(0,-s_k(r_k,\theta))),
$$
so that $u_k$ glues, on $C_k^+ \cap \partial {\rm B}_{r_k}$, 
 with the values obtained in step 2 (last formula in \eqref{eq_13.28}), and
$$
u_k(r,\theta_k)=(1,0),\qquad u_k(r,0)=(-1,0).
$$
This formula shows that $u_k$ also glues, 
on $C_k^+ \cap \{(r,\theta): r \in [r_k/2,r_k], \theta \in \{0,\theta_k\}\}$, 
with the values obtained in step 2 (second and third formula in \eqref{eq_13.28}). Moreover
\begin{equation}\label{eq:step_3_r_k_2}
	u_k(r_k/2,\theta)=\Upsilon_k\big(s_k(r_k,\theta),\psi_k^\star(0,-s_k(r_k,\theta))\big),
\qquad \theta \in [0,\theta_k].
\end{equation}
In addition,  the derivatives of $u_k$ satisfy, 
for $r\in(\frac{r_k}{2},r_k)$ and $\theta\in(0,\theta_k)$, using \eqref{eq:s_k_r_k_theta},
\begin{align*}
&\partial_ru_k(r,\theta)=-\frac{2}{r_k}\Upsilon_k\big(s_k(r_k,\theta),\psi_k^\star(0,-s_k(r_k,\theta))\big)+\frac{2}{r_k}\big(s_k(r_k,\theta),\psi_k^\star(0,-s_k(r_k,\theta))\big),\\
&\partial_\theta u_k(r,\theta)=
\Big(2-\frac{2r}{r_k}\Big)\nabla \Upsilon_k\big(s_k(r_k,\theta),\psi_k^\star(0,-s_k(r_k,\theta))\big)
\cdot\Big(\frac{2}{\theta_k},-\frac{2}{\theta_k}\partial_{w_2}\psi_k^\star(0,-s_k(r_k,\theta))\Big)\nonumber\\
&\qquad\qquad\qquad +\Big(\frac{2r}{r_k}-1\Big)\Big(\frac{2}{\theta_k},-\frac{2}{\theta_k}\partial_{w_2}
\psi_k^\star(0,-s_k(r_k,\theta))\Big),
\end{align*}
so that 
\begin{align*}
&|\partial_ru_k(r,\theta)|\leq \frac{4}{r_k},\\
&|\partial_\theta u_k(r,\theta)|\leq 
\frac{2(\widehat C+1)}{\theta_k}(|\partial_{w_2}
\psi_k^\star(0,-s_k(r_k,\theta))|+1),
\end{align*}
where $\widehat C$ is a positive constant independent of $k$, which bounds the gradient of $\Upsilon_k$.
Since $\psi_k^\star$ 
is Lipschitz, we deduce that $u_k$ is Lipschitz 
continuous\footnote{The Lipschitz constant of $u_k$ on this set turns out to be unbounded with respect to $k$.} 
on $C_k^+\cap(\Balleps\setminus {\rm B}_{r_k/2})$.

Furthermore the image of $(\frac{r_k}{2},r_k)\times (0,\theta_k)$ through the map $(r,\theta)\mapsto u_k(r,\theta)$
is  the 
region enclosed by $G_{\psi^\star_k}$ and $G_{\psi^\star}$ (with multiplicity $1$). 
The area of this region is infinitesimal as $k\rightarrow +\infty$, 
so that, by the area formula,
$$\int_{r_k/2}^{r_k}\int_0^{\theta_k}r|Ju_k(r,\theta)|d\theta dr=o(1)
\qquad
{\rm as~} k\rightarrow +\infty.
$$
Hence, using the fact that the gradient in polar
coordinates is $(\partial_r,\frac1r\partial_\theta)$, 
we eventually estimate (see also \eqref{eqn:areapolarexpression})
\begin{align}\label{estimate_area_step3}
\int_{r_k/2}^{r_k}\int_0^{\theta_k}r|\mathcal M(\nabla u_k)|d\theta 
dr&\leq 
\int_{r_k/2}^{r_k}\int_0^{\theta_k}\Big(r+\frac{4r}{r_k}+\frac{C}{\theta_k}|\partial_{w_2}
\psi_k^\star(0,1-\frac{2\theta}{\theta_k})|+\frac{C}{\theta_k}\Big)~d\theta dr+o(1),
\nonumber
\\
&=o(1)+C\frac{r_k}{2\theta_k}\int_0^{\theta_k}|\partial_{w_2}\psi_k^\star(0,1-\frac{2\theta}{\theta_k})|d\theta=o(1)
\end{align}
as $k\rightarrow +\infty$.
In the last equality we use that $|\partial_{w_2}\psi^\star_k(0,\cdot)|
\leq|\partial_{w_2}\psi^\star(0,\cdot)|$, which is integrable via the change of variables $w_2=1-\frac{2\theta}{\theta_k}$ (it also makes $\theta_k$ disappear at the denominator in front of the integral 
in \eqref{estimate_area_step3}).

This proves that 
the contribution of area of the graph of 
$u_k$ over $C_k^+\cap(\Balleps\setminus {\rm B}_{r_k/2})$ is infinitesimal as $k\rightarrow +\infty$.

\begin{figure}
	\centering
	\includegraphics[scale=0.6]{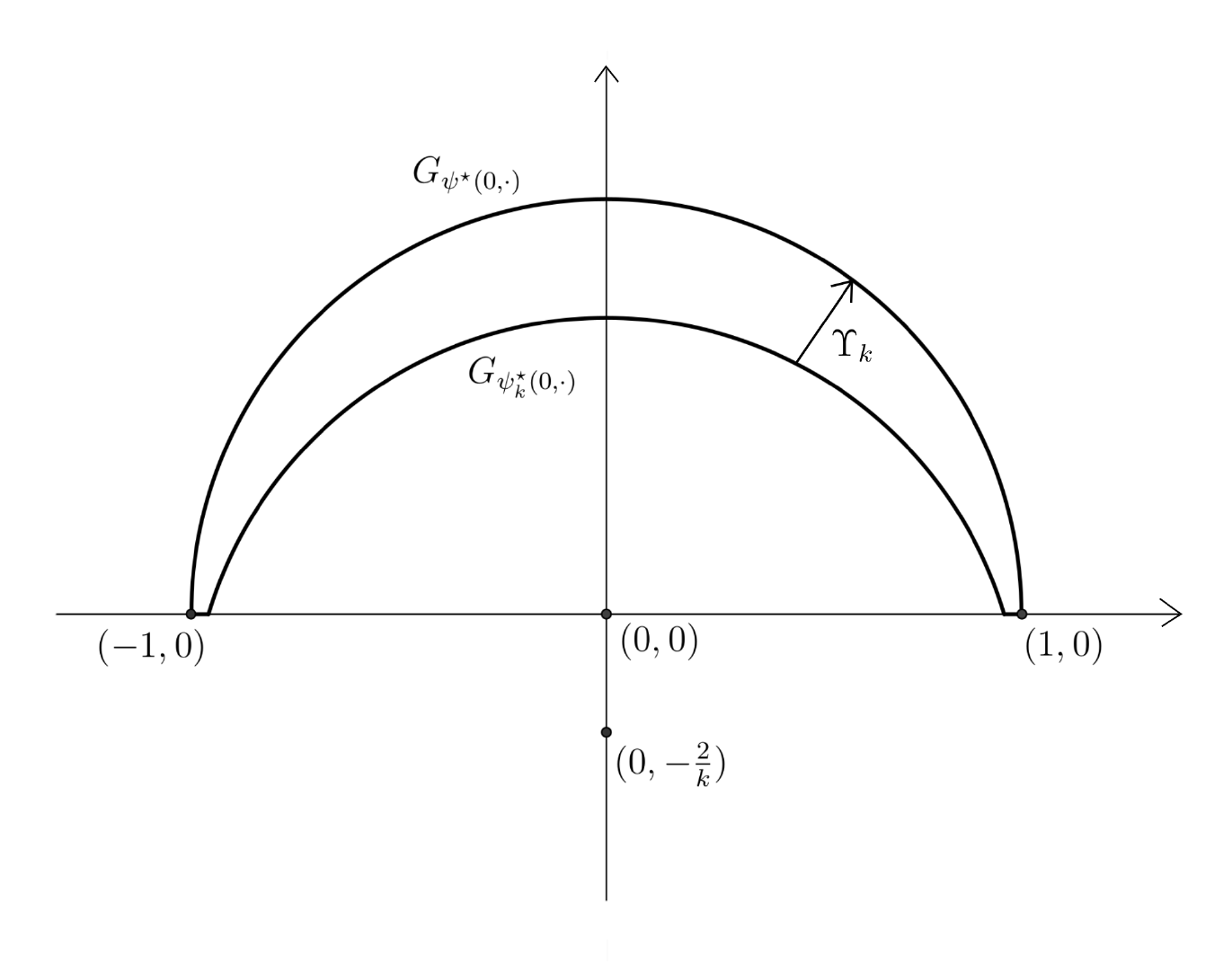}
	\caption{the graphs of the functions $\psi^\star_k(0,\cdot)$ and $\psi^\star(0,\cdot)$; these contain arcs of circle centered at $(0,0)$ and $(0,-\frac2k)$ respectively. The map $\Upsilon_k$ is emphasized. This turns out to be the restriction of $x\mapsto \frac{x}{|x|}$ on $\psi^\star_k(0,\cdot)$.}
	\label{fig:graphof_h}
\end{figure}

Eventually, for $r\in[r_k/2,r_k]$, $\theta\in[2\pi-\theta_k,2\pi)$, we set
\begin{align}
u_k(r,\theta):=(u_{k1}(r,2\pi-\theta),-u_{k2}(r,2\pi-\theta)).
\end{align}
Observe that, thanks to \eqref{vepsonCeps-}, $u_k$ is continuous on $\partial {\rm B}_{r_k}$, and 
similar estimates as in \eqref{estimate_area_step3} for the area hold on $(\sourcedisk_{r_k}
\setminus \sourcedisk_{r_k/2})\cap C_k^-$.
\smallskip

{\it Step 4.} Definition of $\veps$ on $\Coneeps\cap \sourcedisk_{r_k/2}$
and its area contribution.

We start with the construction of $u_k$ on $C_k^+\cap {\rm B}_{r_k/2}$. 
For $r\in[0,r_k/2)$  and $\theta\in [0,\theta_k]$ we set
\begin{align}\label{u_k_step4}
u_k(r,\theta):=\Upsilon_k\Big(\frac{4r\theta}{r_k\theta_k}-1,\psi_k^\star
\big(0,1-\frac{4r\theta}{r_k\theta_k}\big)\Big).
\end{align}
First we observe that
$$
u_k\Big(\frac{r_k}{2},\theta\Big):=\Upsilon_k\Big(\frac{2\theta}{\theta_k}-1,\psi_k^\star
\big(0,1-\frac{2\theta}{\theta_k}\big)\Big), \qquad \theta \in (0,\theta_k),
$$
so that  $u_k$ is continuous on $C_k^+ \cap
\partial {\rm B}_{r_k/2}$ (see \eqref{eq:step_3_r_k_2} and \eqref{eq:s_k_r_k_theta}), and  
\begin{align}\label{eq:u_k_lip_cont}
&u_k(r,\theta_k)=\Upsilon_k\Big(\frac{4r}{r_k}-1,\psi_k^\star(0,1-\frac{4r}{r_k})\Big),
\\
&u_k(r,0)=(-1,\psi_k^\star(0,1))=(-1,0).
\end{align}
Direct computations lead to the following estimates:
\begin{align}
&|\partial_ru_k(r,\theta)|\leq \widehat C\frac{4\theta}{r_k\theta_k}\Big(1+|\partial_{w_2}\psi^\star_k(0,1-\frac{4\theta r}{r_k\theta_k})|\Big),\\
&|\partial_\theta u_k(r,\theta)|\leq \widehat C\frac{4r}{r_k\theta_k}\Big(1+|\partial_{w_2}\psi^\star_k(0,1-\frac{4\theta r}{r_k\theta_k})|\Big),
\end{align}
where $\widehat C$ is the constant bounding the gradient of $\Upsilon_k$ as in 
step 3.
Finally, since by \eqref{u_k_step4} $u_k$ takes values in $\mathbb S^1\subset\R^2$, we have
$Ju_k(r,\theta)=0$ for all $r\in (0,r_k/2)$, $\theta\in[0,\theta_k]$. 
Hence, the area of the graph of $u_k$ on $C_k^+\cap {\rm B}_{r_k/2}$ is 
\begin{align*}
\int_0^{r_k/2}\int_0^{\theta_k}r|\mathcal M(\nabla u_k)(r,\theta)|~d\theta 
dr\leq \int_0^{r_k/2}\int_0^{\theta_k} (r+C) 
+\frac{C}{\theta_k}+\frac{Cr}{r_k}(1+\frac{1}{\theta_k})|\partial_{w_2}\psi^\star_k(0,1-\frac{4\theta r}{r_k\theta_k})|d\theta dr,
\end{align*}
where 
$C$ is a positive constant independent of $k$. Exploiting 
that $|\partial_{w_2}\psi^\star_k(0,\cdot)|\leq|\partial_{w_2}\psi^\star(0,\cdot)|$, we can 
estimate the right-hand side of the previous formula as follows:
\begin{equation}
\label{estimate_area_step4}
\begin{aligned}
&C\int_0^{r_k/2}\int_0^{\theta_k}
\frac{r}{r_k}
\Big(1+\frac{1}{\theta_k}\Big)|\partial_{w_2}\psi^\star(0,1-\frac{4\theta r}{r_k\theta_k})|d\theta dr+o(1)
\\
\leq&C\int_0^{r_k/2}\int_{-1}^1\theta_k
\Big(1+\frac{1}{\theta_k}\Big)|\partial_{w_2}\psi^\star(0,w_2)|dw_2dr+o(1)\\
\leq & C\int_0^{r_k/2}(\theta_k+1)dr+o(1)=o(1),
\end{aligned}
\end{equation}
where $o(1)\rightarrow0$ as $k\rightarrow +\infty$, and $C$ is a positive constant independent of $k$ which might change from line to line.

In $C_k^-\cap \sourcedisk_{r_k/2}$ we set, for $r\in [0,r_k/2)$, $\theta\in[2\pi-\theta_k,2\pi)$, 
\begin{align*}
u_k(r,\theta):=(u_{k1}(r,2\pi-\theta),-u_{k2}(r,2\pi-\theta)).
\end{align*}
 
Similar estimates as in \eqref{estimate_area_step4} 
for the area hold on $C_k^- \cap \sourcedisk_{r_k/2}$.

\medskip

{\it Step 5.} Definition of $\veps$ on $\Teps$
and its area contribution.

We first construct $u_k$ on $\Teps \cap \{(r,\theta): r\in [0,r_k/2],\theta \in [\thetaeps, \thetaepsbar]\}$.
We define $\beta_k: [0,r_k/2]\times [\theta_k,\overline\theta_k]\rightarrow [0,\pi]$ as
\begin{align*}
\beta_k(r,\theta):=\frac{\overline\theta_k-\theta}{\overline\theta_k-\theta_k}\alpha_k(r)+(1-\frac{\overline\theta_k-\theta}{\overline\theta_k-\theta_k})\big(\frac{2r}{r_k}(\overline\theta_k-\pi)+\pi\big),
\end{align*}
where 
$$\alpha_k(r):=\arccos\big(\Upsilon_{k1}(\frac{4r}{r_k}-1,\psi^\star_k(0,1-\frac{4r}{r_k}))\big), \qquad
r \in [0,r_k/2].
$$
Notice that $\alpha_k$ is decreasing and takes values in 
$[0,\pi]$.
Therefore we set
\begin{align*}
u_k(r,\theta):=\big(\cos(\beta_k(r,\theta)),\sin (\beta_k(r,\theta))\big),
\qquad
(r,\theta)\in [0,r_k/2]\times [\theta_k,\overline\theta_k].
\end{align*}
One checks that $\beta_k(r,\theta_k)= \alpha_k(r)$,
$\beta_k(r,\overline \theta_k) = \frac{2r}{r_k}(\overline \theta_k-\pi)+\pi$
(see also \eqref{eqn:vepsoutofthecone}), and   
\begin{align*}
&\alpha_k(r_k/2)=0,\\
&u_k(r_k/2,\theta)=\big(\cos((1-\frac{\overline\theta_k-\theta}{\overline\theta_k-\theta_k})\overline\theta_k),\sin((1-\frac{\overline\theta_k-\theta}{\overline\theta_k-\theta_k})\overline\theta_k)\big),\\
&u_k(r,\theta_k)=\big(\cos(\alpha_k(r)),\sin (\alpha_k(r))\big)=\Upsilon_k(\frac{4r}{r_k}-1,\psi^\star_k(0,1-\frac{4r}{r_k})),\\
&u_k(r,\overline\theta_k)=\big(\cos(\frac{2r}{r_k}(\overline\theta_k-\pi)+\pi),\sin(\frac{2r}{r_k}(\overline\theta_k-\pi)+\pi)\big),
\end{align*}
so that $u_k$ is continuous on $\{\theta \in \{\theta_k, \overline \theta_k\}, \;r\in [0,r_k/2]\} \cap \Omega$,
see \eqref{eqn:vepsboundaryinBeps} and \eqref{eq:u_k_lip_cont}.

Notice also that $u_k$ is continuous 
at $(0,0)\in \R^2$ and $u_k(0,0)=(-1,0)$.
Finally, since $u_k$ takes values in $\mathbb S^1$, the
 determinant of its Jacobian vanishes, 
so that in order to estimate the area contribution
of the graph of $u_k$ in $\Teps \cap \{(r,\theta):r\in [0,r_k/2],\; \theta \in [\thetaeps, \thetaepsbar]\}$ it is sufficient to estimate the derivatives of $u_k$. We have
\begin{align*}
&|\partial_ru_k(r,\theta)|= |\partial_r \beta_k(r,\theta)|\leq |\partial_r\alpha_k(r)|+\frac{2\pi}{r_k},\\
&|\partial_\theta u_k(r,\theta)|=|\partial_\theta\beta_k(r,\theta)|\leq \frac{|\alpha_k(r)|}{\overline\theta_k-\theta_k}+\frac{\pi}{\overline\theta_k-\theta_k}\leq \frac{2\pi}{\overline\theta_k-\theta_k}.
\end{align*}
Therefore 
\begin{align}\label{estimate_area_step5a}
&\int_{0}^{r_k/2}\int_{\theta_k}^{\overline\theta_k}
r|\mathcal M(\nabla u_k)(r,\theta)|d\theta dr\leq \int_{0}^{r_k/2}
\int_{\theta_k}^{\overline\theta_k} 
\Big[\frac{r_k}{2}
(1+|\partial_r\beta_k(r,\theta)|)+|\partial_\theta \beta_k(r,\theta)|\Big]
~d\theta dr\nonumber\\
&\leq o(1)+
\int_{0}^{r_k/2}\int_{\theta_k}^{\overline\theta_k}\left(
\frac{r_k}{2}|\partial_r\alpha_k(r)|+\pi +\frac{2\pi }{\overline\theta_k-\theta_k}\right)d\theta dr=o(1),
\end{align}
with $o(1)\rightarrow0$ as $k\rightarrow +\infty$. 
Notice that the integral of $|\partial_r\alpha_k(r)|$ with respect to $r$ can be computed via the fundamental integration theorem, since $\alpha_k$ 
is monotone.

In $T_k\cap \{(r,\theta):  r\in [0,r_k/2],\theta \in [2\pi-\thetaepsbar,2\pi-\thetaeps]\}$ we set
\begin{align*}
u_k(r,\theta):=(u_{k1}(r,2\pi-\theta),-u_{k2}(r,2\pi-\theta)).
\end{align*}
We now define $u_k$ on $\Teps \cap \{(r,\theta):  r\in(r_k/2,l),\theta \in [\thetaeps, \thetaepsbar]\}$. We set
\begin{equation*}
u_k(r,\theta):=\big(\cos((1-\frac{\overline\theta_k-\theta}{\overline\theta_k-\theta_k})\overline\theta_k),\sin((1-\frac{\overline\theta_k-\theta}{\overline\theta_k-\theta_k})\overline\theta_k)\big).
\end{equation*}
Then  $\veps \in \textrm{Lip}(\Teps,\Sone)$, and 
\begin{align*}
& \veps(r,\thetaeps)=(1,0),\qquad \veps(r,\thetaepsbar) = (\cos \thetaepsbar,\sin \thetaepsbar)\quad \text{ for  } r \in (\reps/2,l),
\\
&\partial_ru_k(r,\theta)= 0,\\
&\partial_\theta u_k(r,\theta)=\frac{\overline \theta_k}{\overline \theta_k-\theta_k}\Big(-\sin((1-\frac{\overline\theta_k-\theta}{\overline\theta_k-\theta_k})\overline\theta_k),\cos((1-\frac{\overline\theta_k-\theta}{\overline\theta_k-\theta_k})\overline\theta_k)\Big).
\end{align*}
Hence 
\begin{align}\label{estimate_area_step5b}
&\int_{r_k/2}^l\int_{\theta_k}^{\overline\theta_k}
r|\mathcal M(\nabla u_k)(r,\theta)|d\theta dr\leq 
\int_{r_k/2}^l\int_{\theta_k}^{\overline\theta_k}\Big(
r+\frac{\overline \theta_k}{\overline \theta_k-\theta_k}\Big)~d\theta dr=o(1)
\end{align}
as $k\rightarrow +\infty$. 

Finally in  
$T_k\cap \{(r,\theta): r\in(r_k/2,l), \theta \in [2\pi-\thetaepsbar,2\pi-\thetaeps]\}$ we set
\begin{align*}
u_k(r,\theta) := (u_{k1}(r,2\pi-\theta),-u_{k2}(r,2\pi-\theta)).
\end{align*}
Similar estimates as in \eqref{estimate_area_step5a}, 
\eqref{estimate_area_step5b} for the area hold on 
$T_k\cap \{(r,\theta): r\in(0,r_k/2), \theta \in [2\pi-\thetaepsbar,2\pi-\thetaeps]\}$,
$T_k\cap \{(r,\theta): r\in(r_k/2,l), \theta \in [2\pi-\thetaepsbar,2\pi-\thetaeps]\}$,
respectively. 

{\it Step 6.} We claim that  
\begin{equation}\label{eqn:integraloutofceps}
\int_{\BallR \setminus 
(C_k\cup T_k)
} \vert \M (\nabla \veps ) \vert dx \longrightarrow \int_\BallR \vert \M (\nabla \vmap)\vert  ~dx \qquad \text{ as }  k \to +\infty,
\end{equation}
where we recall that $C_k\cup T_k = \{(r,\theta)\in \BallR: r\in[0,\longR), \theta \in [0,\thetaepsbar]\cup[2\pi-\thetaepsbar,2\pi)\}.$ 

Indeed, on $\BallR \setminus (C_k\cup T_k)$
the maps $\veps$ and $\vmap$ take values in the circle $\Sone$, hence 
\begin{equation*}
\det (\nabla \veps)=0,\qquad \det (\nabla\vmap )
=0,\qquad \text{ in } \BallR \setminus (C_k\cup T_k).
\end{equation*}
Thus  
\begin{equation*}
\int_{\BallR \setminus (C_k\cup T_k)} 
\vert\M(\nabla \veps)-\M(\nabla \vmap)\vert ~ d x 
\leq \sum _{i=1,2}\int_{\BallR \setminus (C_k\cup T_k)} \vert \nabla (\veps_i - \vmap _i) \vert ~d x.
\end{equation*}
From \eqref{eqn:vepsoutofthecone}, we have 
\begin{equation}\label{eqn:vepsminusvoutofthecone_rderiv}
\begin{aligned}
&\vert \partial_r (\veps -\vmap) \vert
   = 0 \qquad \textrm{ in }\quad\BallR \setminus (\Balleps \cup \Coneeps \cup \Teps),
\\
&\vert \partial_r (\veps -\vmap) \vert\leq \frac{\pi}{\reps} 
\ \quad \textrm{ in }\quad\Balleps\setminus(\Coneeps \cup \Teps),
\\
&\vert \partial_\theta (\veps -\vmap) \vert
   = 0 \qquad \text{ in }\quad\BallR \setminus (\Balleps \cup \Coneeps \cup \Teps),
\\
&\vert \partial_\theta (\veps -\vmap) 
\vert\leq 2 \qquad \textrm{ in }\quad\Balleps \setminus ( \Coneeps\cup \Teps).
\end{aligned}
\end{equation}
Our previous remarks and the fact that $\reps, \thetaeps , (\thetaepsbar -\thetaeps)\to0^+$ as $k \to +\infty$, imply \eqref{eqn:integraloutofceps}. 
\smallskip

{\it Step 7.} 
We know from \eqref{estimate_area_step3}, \eqref{estimate_area_step4}, \eqref{estimate_area_step5a}, and  \eqref{estimate_area_step5b}, that the integral of $\vert\mathcal M(\nabla u_k)\vert$ is infinitesimal as $k\rightarrow 
+\infty$, on the region $(\sourcedisk_{r_k}\cap C_k)\cup T_k$. Therefore it 
remains to compute the area of the graphs of $u_k$ in the region $C_k\setminus {\rm B}_{r_k}$.
We claim that this contribution is 
\begin{equation}\label{eq:this_contribution_will_be}
\lim_{k \to +\infty}\int_{\Coneeps\setminus\Balleps} 
\vert \M (\nabla \veps ) \vert~ dx \leq 2
\FBl(\hstar,\psionestar)=\Aone(\psionestar,\Omegahstar).
\end{equation}
To prove this, we start to compute the area of the graph of $\veps$ restricted to $C_k^+\setminus \Balleps$.
From \eqref{vepsonCeps}, \eqref{eqn:tepsderiv}, \eqref{eqn:sepsrderiv} and \eqref{eqn:sepsthetaderiv}, we have
\begin{equation}\label{eq:comput_first_derivatives}
\begin{aligned}
&\partial_r \veps_1=\Big(\frac{\theta}{\thetaeps}-1\Big)
\teps' 
\hstark'
=
\frac{\longR}{\longR -\reps} 
\Big(\frac{\theta}{\thetaeps}-1\Big) 
\hstark',
\\
& 
\partial_\theta \veps_1=\frac{1 + \hstark}{\thetaeps},
\\
&\partial_r \veps_2=
\teps'\Big[\Big(1-\frac{\theta}{\thetaeps}\Big) 
\hstark' \partial_\scoord \psionekstar 
+ \partial_\axialcoordofcylinder \psionekstar\Big]=
\frac{\longR}{\longR -\reps} 
\Big[ \Big(1-\frac{\theta}{\thetaeps}\Big) 
\hstark'\partial_\scoord \psionekstar + \partial_\axialcoordofcylinder \psionekstar\Big],
\\
&\partial_\theta \veps_2
=
-\Big[\frac{1+ \hstark}{\thetaeps}\Big] 
\partial_\scoord\psionekstar,
\\
& 
\partial_r \veps_1
\partial_\theta \veps_2 - 
\partial_\theta \veps_1
\partial_r \veps_2 = - \left(
\frac{1+\hstark}{\thetaeps}
\right)\frac{\longR}{\longR -\reps} 
\partial_\axialcoordofcylinder\psionekstar,
\end{aligned}
\end{equation}
where $\hstark'$ denotes the 
derivative of $\hstark$ with respect
to $\axialcoordofcylinder$, 
 $\hstark, \hstark'$ 
are evaluated at 
$\teps(r)$, and the two partial derivatives 
$\partial_\scoord\psionekstar$, 
$\partial_\axialcoordofcylinder \psionekstar$
of $\psionekstar$ with respect to $\scoord, \axialcoordofcylinder$ 
are evaluated at $(\teps(r),
-\seps(r,\theta))$. Note carefully
that, in the computation of the Jacobian, the terms containing $\partial_\scoord \psionekstar$ cancel each other.

Notice that, since $\hstark$ is convex, its derivative is nonincreasing, and therefore
$\int_{\reps}^{\longR} \vert \hstark'\vert~dr < +\infty$.
As a consequence of \eqref{eq:comput_first_derivatives}, from \eqref{eqn:areapolarexpression}, we have 
\begin{equation*}
\begin{aligned}
& \area (\veps,C_k^+\setminus \Balleps)
\\
= &
\int_{\reps}^{\longR}\int_{0}^{\thetaeps}r
 \Bigg\{
1
+\left(\frac{\longR}{\longR-\reps}\right)^2
\left(\frac{\theta}{\thetaeps}-1\right)^2(\hstark')^2
\\
& +\left(\frac{\longR}{\longR-\reps}\right)^2
\left[
\big(\frac{\theta}{\thetaeps}-1\big)^2(\hstark')^2(\partial_{\scoord}
\psionekstar)^2 
+
2\big(1-\frac{\theta}{\thetaeps}\big)\hstark'
\partial_{\scoord} \psionekstar
\partial_{\axialcoordofcylinder} \psionekstar
+ (\partial_{\axialcoordofcylinder} \psionekstar)^2
\right]
\\
&
 +\frac{1}{r^2}\left(
\frac{1+ \hstark}{\thetaeps}
\right)^2
\left(
1+(\partial_{\scoord}\psionekstar)^2
+
\left(\frac{\longR}{\longR-\reps}\right)^2
(\partial_{\axialcoordofcylinder} \psionekstar)^2
\right)\Bigg\}^\frac{1}{2}
~  dr d\theta,
\end{aligned}
\end{equation*}
where 
$\partial_\scoord \psionekstar$, $\partial_{\axialcoordofcylinder}
\psionekstar$ 
are evaluated at $(\teps(r),-\seps(r,\theta))$, 
and $\hstark$, $\hstark'$ are evaluated at $\teps(r)$.
Now we use the change of variable \eqref{eqn:changeofvariableT}:
from \eqref{eqn:changeofvariabledeterminant}, we have
\begin{equation*}
\begin{aligned}
& \area (\veps,C_k^+\setminus \Balleps)
\\
=&
\int_{0}^{\longR}
\int^{\hstark(\axialcoordofcylinder)}_{-1}
 \left( \frac{\longR-\reps}{\longR}\right)
 \left( \frac{\thetaeps}{1+\hstark}\right)
 \invr(\axialcoordofcylinder)
 \Bigg\{
1 
+
\big(\frac{\longR}{\longR-\reps}\big)^2
\left(\frac{\invtheta(\axialcoordofcylinder,\scoord)}{\thetaeps}-1\right)^2(\hstark')^2 
\\
&
+\left(\frac{\longR}{\longR-\reps}\right)^2
\Big[
\big(1-\frac{\invtheta(\axialcoordofcylinder,\scoord)}{\thetaeps}\big)^2
(\hstark')^2(\partial_{\scoord} \psionekstar)^2 
+
2\big(1-\frac{\invtheta(\axialcoordofcylinder,\scoord)}{\thetaeps}\big)
\hstark' \partial_{\scoord}\psionekstar
\partial_{\axialcoordofcylinder}\psionekstar
+
(\partial_{\axialcoordofcylinder}\psionekstar)^2
\Big]
\\
&+ \frac{1}{(\invr(\axialcoordofcylinder))^2}
\big(\frac{1+ \hstark}{\thetaeps}\big)^2\Big( 1
+(\partial_{\scoord}\psionekstar)^2
+
\big(\frac{\longR}{\longR-\reps}\big)^2
(\partial_{\axialcoordofcylinder}
\psionekstar
)^2\Big)
\Bigg\}^\frac{1}{2}
~  d\scoord d\axialcoordofcylinder,
\end{aligned}
\end{equation*}
where $\invr(\axialcoordofcylinder)$, 
$\invtheta(\axialcoordofcylinder,\scoord)$ 
are defined in \eqref{eqn:rfn}, \eqref{eqn:thetafn}, $\hstark'$ is evaluated at $w_1$, and $\partial_{w_1} \psi^\star_k$ and $\partial_{w_2} \psi^\star_k$ are evaluated at $(w_1,w_2)$.
Therefore
\begin{equation}\label{last_area}
 \area (\veps,C_k^+\setminus \Balleps)
= 
\int_{0}^{\longR}\int^{\hstark(\axialcoordofcylinder)}_{-1}
\Big\{\textrm{I}_k
+
\textrm{II}_k
+
\textrm{III}_k
+
\textrm{IV}_k
+
\textrm{V}_k
+
\textrm{VI}_k
\Big\}^{\frac{1}{2}} ~d\scoord d\axialcoordofcylinder,
\end{equation}
where
\begin{equation*}
\begin{cases}
\textrm{I}_k = 
\left( \frac{\longR-\reps}{\longR}\right)^2
 \left( \frac{\thetaeps}{1+\hstark}\right)^2(\invr(\axialcoordofcylinder))^2,
\\
\\
\textrm{II}_k = 
 \left( \frac{\thetaeps}{1+\hstark}\right)^2
 \left(1-\frac{\invtheta(\axialcoordofcylinder,\scoord)}{\thetaeps}\right)^2
(\invr(\axialcoordofcylinder))^2(\hstark')^2,
\\
\\
\textrm{III}_k = 
\left( \frac{\thetaeps}{1+\hstark}\right)^2(\invr(\axialcoordofcylinder))^2
\Big[
\big(1-\frac{\invtheta(\axialcoordofcylinder,\scoord)}{\thetaeps}
\big)^2(\hstark')^2(\partial_{\scoord}\psionekstar)^2 
\\
\\
\qquad \qquad 
\qquad \qquad 
\qquad \qquad 
+2\big(1-\frac{\invtheta(\axialcoordofcylinder,\scoord)}{\thetaeps}\big)
\hstark' \partial_{\scoord}\psionekstar
\partial_{\axialcoordofcylinder}
\psionekstar
+
(\partial_{\axialcoordofcylinder}
\psionekstar)^2
\Big],
\\
\\
\textrm{IV}_k  =
\big( \frac{\longR-\reps}{\longR}\big)^2,
\\
\\
\textrm{V}_k  =
\big( \frac{\longR-\reps}{\longR}\big)^2
(\partial_\scoord \psionekstar)^2,
\\ 
\\
\textrm{VI}_k  =
(\partial_{\axialcoordofcylinder}
\psionekstar
)^2.
\end{cases}
\end{equation*}
Since 
$\lim_{k\to \infty} \frac{l-\reps}{l}=1$
and $\lim_{k \to +\infty} \thetaeps=0$, we deduce
from  \eqref{eqn:rfn}, \eqref{eqn:thetafn}, 
\begin{align*}
\lim_{k \to +\infty}\invr(\axialcoordofcylinder) 
= \axialcoordofcylinder, \qquad  
\lim_{k \to +\infty}
\frac{\invtheta(\axialcoordofcylinder,\scoord)}{\thetaeps} = 
\frac{\hstar(\axialcoordofcylinder)-\scoord}{1+\hstar(\axialcoordofcylinder)}.
\end{align*}
%
Therefore we see that 
$$\int_{0}^{\longR}\int^{\hstark(\axialcoordofcylinder)}_{-1}(\textrm{I}_k)^{\frac12}+(\textrm{II}_k)^{\frac12}dw_2dw_1=o(1),$$
as $k\rightarrow +\infty$.
Moreover 
\begin{equation}\label{eq:III_k}
\int_{0}^{\longR}\int^{\hstark(\axialcoordofcylinder)}_{-1}(\textrm{III}_k)^{\frac12}dw_2dw_1=o(1)
\end{equation}
as $k\rightarrow +\infty$.
Indeed 
we may estimate
$$\int_{0}^{\longR}\int^{\hstark(\axialcoordofcylinder)}_{-1}(\textrm{III}_k)^{\frac12}dw_2dw_1\leq C\theta_k\int_{0}^{\longR}\int^{\hstark(\axialcoordofcylinder)}_{-1}|\hstark'(w_1)||\partial_{w_2}\psi^\star_k(w_1,w_2)|+|\partial_{w_2}\psi^\star_k(w_1,w_2)|dw_2dw_1,$$
and using that $|\hstark'(w_1)|\leq 2k$ (see \eqref{def_h_k}), 
if we assume \eqref{eq:k_theta_k_to_zero}, {\it i.e.}, 
$\theta_k k\rightarrow0$, then 
\eqref{eq:III_k} follows, 
since the $BV$-norm of $\psi_k^\star$ is bounded uniformly with respect to  $k$.

Hence, from \eqref{last_area},
\begin{align}\label{eq:almost_the_end}
\area (\veps,C_k^+\setminus \Balleps)
&\leq
\int_{0}^{\longR}\int^{\hstark(\axialcoordofcylinder)}_{-1}
\Big\{
\textrm{IV}_k
+
\textrm{V}_k
+
\textrm{VI}_k
\Big\}^{\frac{1}{2}} ~d\scoord d\axialcoordofcylinder+o(1)\nonumber\\
&\leq\int_0^\longR \int^{\hstark(\axialcoordofcylinder)}_{-1}
\sqrt{1+(\partial_{\axialcoordofcylinder}
	\psionekstar)^2 +(\partial_{w_2}
	\psionekstar)^2}~d\scoord dw_1+o(1)\nonumber\\
&=\areaonecod(\psi^\star_k, SG_{h^\star}\cap R_l)+o(1)=\frac12\areaonecod(\psi^\star_k, SG_{h^\star})+o(1)
\end{align} 
as $k\rightarrow +\infty$.
Then taking the limit as $k \to +\infty$ in \eqref{eq:almost_the_end}, and using Lemma \ref{lem:properties_of_psi_m} (iii),  we get 
\begin{equation}\label{eq:step_8}
\lim_{k \to +\infty} \area(\veps, C_k^+\setminus \Balleps)
 \leq\Aone(\psionestar, \Omegahstar)=\FB(\hstar,\psionestar),
\end{equation}
where the last equality follows from \eqref{eqn:FAequalareaofhstar}.

\smallskip
{\it Step 8.} Conclusion.
Notice that $\veps \in {\rm Lip}(\Omega, \R^2)$,
and $u_k \to \vmap$  in $L^1(\BallR,\R^2)$. 
Inequality 
\eqref{eq:upper_bound_recovery} follows from 
\eqref{eqn:integraloutofceps} (which gives the term $\int_\Omega|\mathcal M(\nabla u)|dx$), from  \eqref{eq:this_contribution_will_be} (which gives the second term in \eqref{eq:upper_bound_recovery}), and from estimates \eqref{estimate_area_step3}, \eqref{estimate_area_step4}, \eqref{estimate_area_step5a}, and \eqref{estimate_area_step5b}, showing that all the other contributions are negligible. 
\end{proof}

\begin{cor}[\textbf{Minimizers for $\longR$ large enough}]
For $\longR$ large enough, a
 solution to
the minimum
problem on the right-hand side of \eqref{doubling} 
is given by $\h\equiv -1$ and $\psi \equiv 0$.
\end{cor}
\begin{proof}
Recall that for $\longR$ large enough, we have
$$\relarea(\vortexmap,\Omega)= \int_{\BallR}\vert \M(\nabla \vmap)\vert d x + \pi, $$
see \cite{AcDa:94}. The assertion then follows
from Theorem \ref{teo:step1} and 
\eqref{eq:restate_theorem}. 
\end{proof}
%

\section*{Acknowledgements}
The first and third authors acknowledge the support
of the INDAM/GNAMPA.
The first two authors are
grateful to ICTP (Trieste), where part of this paper was written.


\addcontentsline{toc}{section}{References}
\begin{thebibliography}{99}

\bibitem{AcDa:94} E. Acerbi and G. Dal Maso,
\emph{New lower semicontinuity results for polyconvex integrals},
Calc. Var. Partial Differential Equations \textbf{2} (1994), 329--371.
%
\bibitem{AmFuPa:00} L. Ambrosio, N. Fusco and D. Pallara, 
``Functions of Bounded Variation and Free Discontinuity Problems'',
 Mathematical Monographs, Oxford Univ. Press, 2000.

\bibitem{Ball:77}
J.M. Ball, \emph{Convexity conditions and existence theorems in nonlinear
elasticity}, Arch. Ration. Mech. Anal. {\bf 63} (1977), 337-403.

\bibitem{Ball_Murat:84}
J.M. Ball and F. Murat, 
$W^{1,p}$-\emph{Quasi-convexity and variational
problems for multiple integrals}, J. Funct. Anal. 
{\bf 58} (1984), 225-253.


\bibitem{BeElPaSc:19} G. Bellettini, A. Elshorbagy, M. Paolini and R. Scala, 
\emph{On the relaxed area of the graph of discontinuous maps from the
plane to the plane taking three values with no symmetry assumptions},
Ann. Mat. Pura Appl. \textbf{199} (2019), 445--477.


\bibitem{BePa:10} G. Bellettini and M. Paolini,
\emph{On the area of the graph of a singluar map from the plane to the plane taking three values},
Adv. Calc. Var. \textbf{3} (2010), 371--386.

\bibitem{BePaTe:15}
G. Bellettini, M. Paolini and L. Tealdi,
\emph{On the area of the graph of a piecewise smooth map from the
plane to the plane with a curve discontinuity},
{ ESAIM: Control, Optimization and Calculus of Variations}
{\bf 22} (2015), 29--63.

\bibitem{BePaTe:16}
G. Bellettini, M. Paolini and L. Tealdi,
\emph{Semicartesian surfaces and  the relaxed area of
 maps from the plane to the plane with a line discontinuity},
Ann. Mat. Pura Appl.
{\bf 195} (2016), 2131-2170.

\bibitem{CPS} F. Cagnetti, M. Perugini and
 D. St\"oger, \emph{Rigidity for perimeter inequality under spherical
symmetrisation}, 
Calc. Var. Partial Differential Equations {\bf 59} (2020), 59-139.

\bibitem{Creutz:20} P. Creutz, 
\emph{Plateau's problem for singular curves}, arXiv:1904.12567.

\bibitem{Dacorogna:89}
G. Dacorogna, ``Direct Methods in the Calculus of Variations'', 
Springer, Berlin-Heidelberg-New York, 1989.

\bibitem{DalMaso:80} 
G. Dal Maso,
\emph{Integral representation on $BV(\Omega)$ of $\Gamma$-limits
of variational integrals},
Manuscripta Math. (1980), 387-416.

\bibitem{DeGiorgi:92} E. De Giorgi, 
\emph{On the relaxation of functionals defined on cartesian manifolds},
In ``Developments in Partial Differential Equations and Applications
in Mathematical Physics'' (Ferrara 1992),
Plenum Press, New York, 1992.

\bibitem{Hil1} U. Dierkes, S. Hildebrandt and F. Sauvigny,
``Minimal Surfaces'', 
Grundlehren der mathematischen
Wissenschaften, Vol. 339, Springer-Verlag, Berlin-Heidelberg, 2010.
%
\bibitem{Federer:69}
H. Federer, 
``Geometric Measure Theory'',
Die Grundlehren der mathematischen Wissenschaften, Vol. 153,
Springer-Verlag, New York Inc., New York, 1969.

\bibitem{Finn:86} R. Finn,
``Equilibrium Capillary Surfaces'',
Die Grundlehren der mathematischen Wissenschaften, Vol. 284,
Springer-Verlag, New York-Berlin-Heidelberg-Tokyo, 1986.


\bibitem{FuHu:95}
N. Fusco and J.E. Hutchinson, 
\emph{A direct proof for lower semicontinuity of polyconvex functionals},
Manuscripta Mat. {\bf 87} (1995), 35-30.

\bibitem{GiMoSu:98} M. Giaquinta, G. Modica and J. Sou\u{c}ek, 
``Cartesian Currents in the Calculus of Variations I. Cartesian Currents'',
Ergebnisse der Mathematik und ihrer Grenzgebiete, Vol. 37,
Springer-Verlag, Berlin-Heidelberg, 1998.

\bibitem{Giusti:84} E. Giusti, 
``Minimal Surfaces and Functions of Bounded Variation'',
Birkh\"auser, Boston, 1984.

\bibitem{GS:64} C. Goffman and J. Serrin, \emph{Sublinear functions of measures and variational integrals}, Duke Math. J. {\bf 31} (1964), 159--178.

\bibitem{Hass:91} J. Hass,
\emph{Singular curves and the Plateau problem}, 
Internat. J. Math. {\bf 2} (1991), 1–16. 

\bibitem{Hormander:94} L. Ho\"rmander, 
``Notions of Convexity'',
Birkh\"auser, Boston, 1994.

\bibitem{Krantz_Parks:08} G. Krantz and R. Parks,
``Geometric Integration Theory'',
 Cornerstones, Birkh\"auser Boston, Inc., Boston, MA, 2008.

\bibitem{Maggi:12} F. Maggi, 
``Sets of Finite Perimeter and Geometric Variational Problems.
An Introduction to Geometric Measure Theory'', 
Cambridge Univ. Press, Cambridge, 2012.

\bibitem{Meeks_Yau:82} W. H. Meeks and
 S. T. Yau, \emph{The classical Plateau problem and the topology of three-dimensional manifolds}, Topology \textbf{21}
(1982), 409-440.

\bibitem{Morrey:66} C.B. Morrey,
``Multiple Integrals in the Calculus of Variations'', 
Grundlehren der mathematischen
Wissenschaften, Vol. 130, Springer-Verlag, New York, 1966.


\bibitem{Nitsche:89} 
J. C. C. Nitsche, ``Lectures on Minimal Surfaces'', Vol. I,
Cambridge University Press, Cambridge, 1989.

\bibitem{Scala:19} R. Scala, \emph{Optimal estimates for the 
triple junction function and other
surprising aspects of the area functional}, { Ann. Sc. Norm. Super. Pisa Cl. Sci.} 
{\bf XX} (2020), 491-564.
\end{thebibliography}
\end{document}